\documentclass[11pt,a4paper]{article}

\usepackage[a4paper,top=0.7in,bottom=0.7in,left=0.8in,right=0.8in]{geometry} 
\usepackage{setspace}
\usepackage{mathrsfs}
\usepackage{natbib}
\setcitestyle{round}
\usepackage{amsmath,amssymb,amsfonts,amsthm,bm,subcaption}
\usepackage{graphicx,color,xcolor}
\usepackage{multirow,booktabs,array}
\usepackage{float,rotating}
\usepackage{algorithm,algorithmicx,algpseudocode}
\usepackage{url,hyperref}
\usepackage{lipsum}
\usepackage{indentfirst}
\usepackage{authblk}
\usepackage{textcomp}
\usepackage{threeparttable}
\usepackage{pdflscape}
\usepackage{enumerate}

\hypersetup{
	colorlinks=true,
	linkcolor=blue,
	anchorcolor=blue,
	citecolor=blue,
	urlcolor=blue,
	CJKbookmarks=true
}

\newtheorem{theorem}{Theorem}

\newtheorem{lemma}{Lemma}
\newtheorem{remark}{Remark}

\allowdisplaybreaks

\newcommand{\Z}{{\bf Z}}
\newcommand{\X}{{\bf X}}
\newcommand{\I}{{\bf I}}
\newcommand{\Y}{{\bf Y}}
\newcommand{\U}{{\bf U}}
\newcommand{\V}{{\bf V}}

\newcommand{\M}{{\bf M}}
\newcommand{\A}{{\bf A}}
\newcommand{\B}{{\bf B}}
\newcommand{\D}{{\bf D}}
\newcommand{\y}{{\bf y}}

\newcommand{\z}{{\bf z}}
\newcommand{\s}{{\bf s}}
\newcommand{\m}{{\bf m}}
\newcommand{\e}{{\bf e}}
\newcommand{\q}{{\bf q}}
\newcommand{\p}{{\bf p}}

\newcommand{\1}{\mathbf 1}

\def\tr{\operatorname{tr}}

\makeatletter \@addtoreset{equation}{section} \makeatother

\author{Yashi Wei}
\author{Jiang Hu}
\author{Zhidong Bai}

\affil{KLASMOE and School of Mathematics and Statistics, Northeast Normal University}

\date{\today}  

\title{Edgeworth corrections for spiked eigenvalues of non-Gaussian sample covariance matrices with applications}

\begin{document}
	
	\maketitle
	
	\begin{abstract}
		\cite{yang2018edgeworth} established an Edgeworth correction for the largest sample eigenvalue in a spiked covariance model under the assumption of Gaussian observations, leaving the extension to non-Gaussian settings as an open problem. In this paper, we address this issue by establishing first-order Edgeworth expansions for spiked eigenvalues in both single-spike and multi-spike scenarios with non-Gaussian data. Leveraging these expansions, we construct more accurate confidence intervals for the population spiked eigenvalues and propose a novel estimator for the number of spikes. Simulation studies demonstrate that our proposed methodology yields more accurate finite-sample approximations across a wide range of settings, particularly when sample sizes are limited.
	\end{abstract}
	
	\textbf{KEY WORDS: Edgeworth expansion; spiked covariance model; largest eigenvalue.
	}
	
	\tableofcontents
	
	\addtocontents{toc}{\setcounter{tocdepth}{2}}
	\section{Introduction}
Empirical evidence from diverse fields, including wireless communication engineering, speech recognition, and genetics, has consistently shown that the extreme eigenvalues of sample covariance matrices often exhibit a distinct separation from the bulk spectrum. This phenomenon motivated the development of the spiked covariance model, first formalized by \cite{johnstone2001distribution}, which has since become a cornerstone of modern high-dimensional statistical theory. The model characterizes covariance structures where a few population eigenvalues (spikes) are well-separated from the rest, yielding significant insights for both theoretical analysis and real-world applications.

The extreme eigenvalues of sample covariance matrices have garnered considerable theoretical interest, with most research focusing on the almost-sure limits of sample eigenvalues, their fluctuations as described by central limit theorems (CLTs), and the distinctive phase transitions at critical eigenvalue thresholds. For instance, \cite{johnstone2001distribution} showed that, after proper centering and scaling,  the limiting distribution of the largest eigenvalue of a Wishart matrix is the Tracy-Widom (T-W) distribution. \cite{peche2003universality} established that this T-W convergence holds more broadly, provided the largest population spike is not excessively large. For complex-valued data, \cite{baik2005phase} and \cite{paul2007asymptotics} confirmed that the largest sample eigenvalues converge to a T-W distribution below a phase transition threshold and to a Gaussian distribution above it. 
Further contributions include the work of \cite{bai2008central}, who provided the limiting distribution under more relaxed conditions, and \cite{bai2012estimation}, who proposed a consistent estimation method for spiked eigenvalues. Approximate distributions for the largest eigenvalue under Gaussianity were derived by \cite{bai2012sample}, \cite{johnstone2017roy}, and \cite{yang2018edgeworth}. More recently, \cite{wang2017asymptotics} extended the results of \cite{paul2007asymptotics} to a more general setting, while \cite{jiang2021generalized} and \cite{jiang2021partial} obtained universality results for the central limit theorem of eigenvalues of high-dimensional covariance matrices.

Moreover, the investigation of spiked models has found extensive applications across diverse disciplines; see, for example, \cite{shen2016statistics}, \cite{wang2017asymptotics}, \cite{huang2017asymptotic}, \cite{johnstone2018pca}, \cite{li2020asymptotic}, \cite{zhang2022heteroskedastic}, \cite{zhang2022asymptotic}, \cite{ding2023spiked}, \cite{dornemann2024detecting}, and \cite{tang2025mode}. Beyond the canonical sample covariance matrix, similar spiked models have been studied for other random matrix ensembles, including sample correlation matrices (\cite{jiang2021limits}), F-matrices (\cite{wang2017extreme}, \cite{xie2021limiting}), Beta-matrices (\cite{lamarre2019edge}), and canonical correlation coefficient matrices (\cite{bao2019canonical}). For each of these models, significant progress has been made in understanding the behavior of limiting eigenvalues, phase transitions, and associated central limit theorems, as well as their practical applications (\cite{jiang2021limits}).

Among the extensive literature, the work of \cite{yang2018edgeworth} is particularly noteworthy for its derivation of an Edgeworth correction for the largest sample eigenvalue under Gaussian observations. Their expansion provides a first-order correction to the limiting distribution, refining the standard central limit theorem and offering deeper insight into the finite-sample behavior of the leading eigenvalue. Such higher-order approximations are valuable for improving inferential accuracy, as they capture distributional features not described by simple Gaussian limits. For instance, a more precise characterization of the spiked eigenvalue's distribution enables more accurate estimation of the number of spikes--a fundamental challenge in signal detection and factor analysis, as demonstrated by \cite{passemier2012determining}, \cite{passemier2017estimation}, \cite{bai2018consistency}, \cite{cai2020limiting}, and \cite{jiang2021generalized}.

While the work of \cite{yang2018edgeworth} provides a comprehensive theory for Gaussian data, its applicability is limited in many real-world scenarios where observations are predominantly non-Gaussian. This paper fills this gap by establishing first-order Edgeworth expansions for sample spiked eigenvalues in a spiked covariance model under non-Gaussian settings. To summarize, the contributions of this paper are as follows:

\begin{itemize}
	\item We develop a unified framework that handles both single-spike and multi-spike scenarios under non-Gaussian settings. Our results resolve an open problem posed by \cite{yang2018edgeworth} and refine the asymptotic distributions of spiked eigenvalues derived in \cite{jiang2021generalized}. Our approach extends the foundational work of \cite{hall2013bootstrap} and \cite{petrov2000classical} on Edgeworth expansions for sums of independent, non-identically distributed random variables, adapting and generalizing these techniques to the complex setting of high-dimensional eigenvalue distributions.
	
	\item We introduce novel confidence interval constructions that leverage our Edgeworth-corrected asymptotic distributions, offering significant improvements over conventional Gaussian-based approximations. The first method extends the work of \cite{jiang2021generalized}, while the second builds upon that of \cite{yang2019edgeworth}. We develop corresponding estimation techniques for the number of spikes, particularly in low-dimensional settings, where our method demonstrates notable precision.
\end{itemize}
Our extensive simulation studies demonstrate that Edgeworth-corrected distributions provide a substantially better fit to the finite-sample behavior of spiked eigenvalues. Furthermore, we confirm the precision of our spike number estimation method in low-dimensional settings through comprehensive simulations across various distributions and covariance structures.

The remainder of the paper is organized as follows. Section \ref{mainresults} presents our main theoretical results. Section \ref{application} demonstrates the utility of our theory through practical applications. Section \ref{simulation} provides numerical simulations to illustrate the practical implications of our findings.  Section \ref{limitingdistribution} outlines the main proof ideas and presents several preliminary results. Section  \ref{proof} contains the proofs of the main theorems. Section \ref{future} concludes the paper by summarizing key contributions and discussing future research directions. Proofs of auxiliary lemmas and additional simulation results are deferred to the Supplementary Material.

\subsection*{Conventions}
In this paper,  we denote the trace of a matrix $\A$ by $\mathrm{tr}(\A)$ and its transpose by $\A'$. The notation $|\cdot|$ denotes the Euclidean norm for vectors,
	and $\|\cdot\|$ denotes the spectral norm for matrices. We also write
	\[ |\mathbf{ x}|_\infty=\max_j|x_j|,\quad
	\|\A\|_F=[\operatorname{tr}(\A'\A)]^{1/2},
	\quad
	\|f\|_\infty=\sup_x|f(x)|,
	\quad
	\|X\|_s=(\mathbb E|X|^s)^{1/s},
	\]
	where $X$ is a random variable and $s\ge1$. 
We use the standard stochastic order notation: $O_p(1)$ denotes boundedness in probability, and $o_p(1)$ indicates convergence to zero in probability. $\Phi(x)$ and $\phi(x)$ denote the cumulative distribution function and probability density function of the standard normal distribution, respectively. 

	\section{Main Results}\label{mainresults}

\subsection{Edgeworth expansions for the spiked eigenvalues}
Consider a spiked covariance model with a fixed number of spikes, $r\ge 1$. Let the data matrix be
\[\X=(\mathbf{ x}_1,\dots,\mathbf{ x}_n)'\in \mathbb{R}^{n\times q},\] 
where $\mathbf{x}_i\in \mathbb{R}^{q},~i=1\dots,n$,  are independent and identically distributed (i.i.d.) random vectors with mean vector $0$ and covariance matrix $\bm{\Sigma}$. We further assume that each vector $\mathbf{x}_i$ is a linear transformation of a vector of i.i.d. real random variables, i.e.,
\[\X=\Z\bm{\Sigma}^{1/2},\]
where 
$\Z=(\mathbf{ z}_1,\dots,\mathbf{ z}_n)'=(Z_{ij})_{n\times q}$, and the entries $Z_{ij}$ are i.i.d. random variables satisfying $\mathbb{E}(Z_{ij})=0$ and $\mathrm{Var}(Z_{ij})=1$. Throughout the paper, $p$ denotes the number of unit population eigenvalues and
$q:=p+r$ 
denotes the total data dimension.
{The eigendecomposition of $\bm{\Sigma}$ is given by
	\begin{align*}
		\bm{\Sigma}=\V\text{diag}(l_1,\dots,l_r,1,\dots,1)\V'=(\V_1,\V_2)\text{diag}(l_1,\dots,l_r,1,\dots,1)(\V_1,\V_2)',
	\end{align*}
	where $\V$ is a $q\times q$ orthogonal matrix and $\infty>l_1>\dots>l_r>1$ represent the population spiked eigenvalues. In this structure, $\V_1$ is the $q\times r$ submatrix consisting of the first $r$ eigenvectors corresponding to the spikes, while $\V_2$ is its $q\times p$ orthogonal complement.}
Our main results are established under the following assumptions:
\setlength{\leftmargini}{93pt}
\begin{itemize}
	\item[\bf{Assumption (a)}] The random variables satisfy
	\[\mathbb{E}Z_{11}=0,\quad \mathbb{E}Z_{11}^2=1, \quad\Gamma=\mathbb{E}Z_{11}^3,\quad\beta=\mathbb{E}Z_{11}^4-3,\quad \Delta=\mathbb{E}Z_{11}^6<\infty.\]
	\item[\bf{Assumption (b)}] The dimensionality ratio converges, i.e., as $\min\{p,n\}\to\infty$,
	\[\gamma_n=p/n\to\gamma\in(0,\infty).\]
	\item[\bf{Assumption (c)}] The smallest spike is sufficiently large:  \[l_r>1+\sqrt \gamma.\]
	\item[\bf{Assumption (d)}]  The distribution of $Z_{11}$ has a nonzero absolutely continuous component.
\end{itemize}

\begin{remark}
	Assumption (a), which specifies a zero mean and unit variance, is a standard normalization. The finite sixth moment condition is crucial, as our first-order Edgeworth expansion involves the third-order cumulant of squared observations, which requires a finite sixth moment of $Z_{11}$. Assumption (b) is a standard condition in random matrix theory concerning the asymptotic dimensionality ratio. Assumption (c) is the standard phase transition condition for spiked eigenvalues.  Assumption (d) is the usual nonsingularity condition used in Edgeworth theory.  It yields Cram\'er's condition for the random vector $(Z_{11}^2-1,Z_{11})$ by the argument of \cite{hall2013bootstrap}. Notably, \cite{yang2018edgeworth} did not need to impose this assumption, as the known properties of the Gaussian distribution inherently satisfy this requirement.
\end{remark}

Let $\mathbf{S}=n^{-1}\X'\X$ be the sample covariance matrix and $\hat{\rho}_k$ be its $k$-th largest eigenvalue. A primary result of this paper is a first-order Edgeworth correction to the normal approximation for the distribution of the spiked eigenvalue statistic:
\[R_k=\frac{n^{1/2}(\hat{\rho}_k-\rho_{nk})}{\widetilde\sigma_{nk}},\quad (1\le k\le r)\]
under non-Gaussian observations. Define
$
\mathcal{V}_{k}^{(i)}:=\sum_{j=1}^{p+r}v_{jk}^i. 
$ The centering and scaling functions are defined as
\[\rho_{nk}=l_k+\frac{\gamma_nl_k}{(l_k-1)},\quad \widetilde\sigma_{nk}^2={\sigma_{nk}^{2}+4^{-1}\pi_{nk}\sigma_{nk}^{4}l_k^{-2}}, \quad\sigma_{nk}^2=2l_k^2(1-{\gamma_n}{(l_k-1)^{-2}}),\]
where $ \pi_{nk}=\beta\mathcal{V}_{k}^{(4)}$ and $\mathbf{v}_k=(v_{1k},\dots,v_{qk})'$ is the $k$-th column of the matrix $\V_1$. 

 Let 
	\[
	C_{m,n}(j_1,\ldots,j_m)
	:=
	\sum_{a=1}^{p+r}\prod_{s=1}^m v_{a j_s}, \qquad  3\le m \le 6,
	\]
	and define the finite-$n$ mixed-spike corrections
	\begin{align}
		\mathfrak a_{0k}^{\mathrm{mix}}
		&=
		\frac{6\beta^2}{l_k^2}
		\sum_{j\ne k}
		\frac{l_j-1}
		{(l_k-l_j)(l_k-1)}
		C_{4,n}^2(k,k,k,j),
		\label{eq:multi-mixed-third}\\
		\mathfrak a_{1k}^{\mathrm{mix}}
		&=
		\beta\sum_{j\ne k}
		\frac{l_j-1}
		{(l_k-l_j)(l_k-1)}
		C_{4,n}(k,k,j,j).
		\label{eq:multi-mixed-mean}
	\end{align}

Define some notation $\kappa_{2k}$, $\mathfrak a_{0k}$,
$\mathfrak a_{1k}$, $\mathfrak b_k$, $\mathfrak c_k$, and
$A(g_{kn})$ as follows. 
\begin{align*}
	&\kappa_{2k}=\frac{2(1-l_k)^2}{l_k^2((l_k-1)^2-\gamma_n)}+\frac{\pi_{nk}}{l_k^{2}},~~  A(g_{kn})=\frac{(l_k-1)^2}{l_k[(l_k-1)^2-\gamma_n]}\sum_{j\neq k}^r\frac{l_j}{l_k-l_j},\\
	&\kappa_{3k}=8(1-l_k^{-1})^3\big((l_k-1)^3+\gamma_n\big)\big((l_k-1)^2-\gamma_n\big)^{-3}+\frac{12\pi_{nk}l_k^{-1}(1-l_k^{-1})^2}{(l_k-1)^2-\gamma_n}\\
	&\quad\qquad+\Big\{(\Delta-15\beta-15)\mathcal{V}_{k}^{(6)}+10\Gamma^2\Big[(\mathcal{V}_{k}^{(3)})^2-\mathcal{V}_{k}^{(6)}\Big]\Big\}l_k^{-3},\\
	&\mu(g_{kn})=\frac{\gamma_n(l_k-1)^2+\beta\gamma_n[(l_k-1)^2-\gamma_n]}{[(l_k-1)^2-\gamma_n]^{2}(l_k-1)},~~	\mathfrak b_{k}
	:=\frac{8\Gamma^2\mathcal{V}_{k}^{(1)}\mathcal{V}_{k}^{(3)}}{\sqrt{n}l_k^2(l_k-1)},\\
	&{\mathfrak a}_{0k}
	=\kappa_{3k}+\frac{6\beta^2(\mathcal{V}_{k}^{(6)}-(\mathcal{V}_{k}^{(4)})^2)}{l_k^3(l_k-1)},~~ {\mathfrak a}_{1k}=\mu(g_{kn})+\frac{\beta(1-\mathcal{V}_{k}^{(4)})}{l_k(l_k-1)}
	+\frac{2\Gamma^2(\mathcal{V}_{k}^{(1)})^2}{nl_k(l_k-1)^2},\\
	&\mathfrak c_{k}=\frac{\pi_{nk}\gamma_nl_k}{(l_k-1)[(l_k-1)^2-\gamma_n]^{1/2}[2(l_k-1)^2+\pi_{nk}\{(l_k-1)^2-\gamma_n\}]^{1/2}}.
\end{align*}
Then we have the main theorem of the paper as follows.

\begin{theorem}[Multiple-spike Edgeworth expansion]
	\label{multi}
	Fix $k\in \{1,\ldots,r\}$. Under Assumptions (a)--(d), the first-order Edgeworth expansion for the distribution of $R_k$ is given by
	\begin{equation}
		\sup_{x\in\mathbb R}
		\Big|
		\mathbb P(R_k\le x)
		-\Phi(x)
		-\frac1{\sqrt n}p_{k}(x)\phi(x)
		\Big|
		=o(n^{-1/2}),
		\label{eq:main-multi-corrected}
	\end{equation}
	where
	\begin{align}
			p_{k}(x)
			={}&
			\frac{\mathfrak a_{0k}
				+\mathfrak a_{0k}^{\mathrm{mix}}}
			{6\kappa_{2k}^{3/2}}(1-x^2)
			-\frac{\mathfrak a_{1k}+A(g_{kn})+\mathfrak a_{1k}^{\mathrm{mix}}}
			{\kappa_{2k}^{1/2}}-\frac{\mathfrak b_k}{2\kappa_{2k}}x-\mathfrak c_kx^2.
			\label{eq:multi-edgeworth-polynomial}
	\end{align} 
\end{theorem}
\begin{remark}\label{rmk2.4}

		\cite{yang2018edgeworth} and \cite{yang2019edgeworth} derived
		Edgeworth corrections for spiked eigenvalues under Gaussianity. Theorem~\ref{multi} extends these results to non-Gaussian observations.
		The coefficients in $p_k(x)$ represent different components of the first-order Edgeworth correction. Specifically, $\kappa_{2k}$
		describes the effective variance of the leading fluctuation,
		$\mathfrak a_{0k}+\mathfrak a^{\rm mix}_{0k}$ gives the
		skewness-type correction, while
		$\mathfrak a_{1k}+A(g_{kn})+\mathfrak a^{\rm mix}_{1k}$ collects
		the first-order centering, bias, and multiple-spike effects.
		The terms $\mathfrak b_k$ and $\mathfrak c_k$ represent additional
		non-Gaussian corrections associated with the third moment and the
		kurtosis-dependent variance normalization, respectively.
		
		It is worth noting that the quantities $\mathcal V_k^{(s)}$ and $C_{4,n}(\cdot)$ describe how
		the geometry of the population spike eigenvectors enters the
		higher-order correction. In particular, if
		$\max_i |v_{ik}|\to0$, then
		$\mathcal V_k^{(3)}$, $\mathcal V_k^{(4)}$, $\mathcal V_k^{(6)}$
		and the mixed coordinatewise-overlap terms are asymptotically
		negligible. Delocalization alone, however, does not necessarily remove all eigenvector dependence, since the global quantity
		$\{\mathcal V_k^{(1)}\}^2/n$ may remain non-negligible. In contrast, for localized spike eigenvectors, the
		higher-order functionals and mixed overlaps may remain non-negligible
		and hence genuinely affect the first-order Edgeworth correction.
		We want to emphasize that orthogonality of two spike eigenvectors does not in general
		imply that their higher-order coordinatewise overlaps vanish.
		Under Gaussian observations, $\Gamma=0$, $\beta=0$, and $\Delta=15$,
		these explicitly non-Gaussian eigenvector-dependent corrections
		disappear.

\end{remark}

	\begin{remark}
		Edgeworth expansions are useful for statistical inference, since they provide higher-order refinements to the Gaussian approximation and can therefore improve finite-sample accuracy in estimation and testing problems (see \cite{petrov2000classical} and
		\cite{hall2013bootstrap} for example). In Section \ref{application}, we illustrate this statistical usefulness by developing Edgeworth-corrected confidence intervals for population spiked eigenvalues, building on the inference procedures of \cite{jiang2021generalized} and \cite{yang2019edgeworth}, and applying the resulting approximation to estimating the number of spikes.
		
		The same idea can potentially be applied to other statistics involving spiked eigenvalues. For example, \cite{li2020asymptotic} established joint asymptotic normality for extreme eigenvalues and the trace and studied the ratio statistic $
		{q\hat{\rho}_k}/{\tr(\mathbf S)}.$  An Edgeworth correction for this statistic would provide a natural higher-order refinement of their Gaussian approximation. To obtain such a result, however, one would need a joint higher-order expansion for the spiked eigenvalue and the trace, including the corresponding mixed cumulant terms. We leave this extension for future work.

\end{remark}

\begin{remark}
If all nonspiked population eigenvalues are heterogeneous, all variances, third cumulants, means of linear spectral statistics, and mixed spike correction coefficients in the theorem must be recalculated. Besides, the first-order corrections are typically not determined solely by the non-spiked eigenvalues. They may also depend on the orientation of the bulk eigenvectors relative to the coordinate basis of the independent standardized elements. Consequently, this requires additional structural assumptions and new calculations for the anisotropic cumulants. We leave this extension for future research.
\end{remark}

For the special case of a single spike ($r=1$), the largest sample eigenvalue corresponds directly to the spike and admits a simplified Edgeworth expansion. Let $\hat{\rho}$ be the largest eigenvalue of $\mathbf{S}$ and define
\[R_n=\frac{n^{1/2}(\hat{\rho}-\rho_{n1})}{\widetilde\sigma_{n1}},\]
where $\rho_{n1}$ and $\widetilde \sigma_{n1}$ are the centering and scaling parameters associated with the single population spike. 	

Then we have the following theorem.
\begin{theorem}[Single-spike Edgeworth expansion]\label{main}
	Under Assumptions (a)--(d) with $r=1$, the first-order Edgeworth expansion for the distribution of $R_n$ is given by
	\begin{align*}
		&\sup_x\bigg|\mathbb{P}(R_n\le x)-\Phi(x)-n^{-1/2}\bigg(\frac{{\mathfrak a}_{01}(1-x^2)}{6\kappa_{21}^{3/2}}-\frac{{\mathfrak a}_{11}}{\kappa_{21}^{1/2}}-\frac{\mathfrak b_{1}x}{2\kappa_{21}}-\mathfrak c_{1}x^2\bigg)\phi(x)\bigg|=o(n^{-1/2}),
	\end{align*}
	where
	\begin{align*}
		&\kappa_{21}=\frac{2(1-l_1)^2}{l_1^2((l_1-1)^2-\gamma_n)}+\frac{\pi_{n1}}{l_1^{2}},~~\mu(g_{1n})=\frac{\gamma_n(l_1-1)^2+\beta\gamma_n[(l_1-1)^2-\gamma_n]}{[(l_1-1)^2-\gamma_n]^{2}(l_1-1)},\\
		&\kappa_{31}=8(1-l_1^{-1})^3\big((l_1-1)^3+\gamma_n\big)\big((l_1-1)^2-\gamma_n\big)^{-3}+\frac{12\pi_{n1}l_1^{-1}(1-l_1^{-1})^2}{(l_1-1)^2-\gamma_n}\\
		&\quad\qquad+\Big\{(\Delta-15\beta-15)\mathcal{V}_{1}^{(6)}+10\Gamma^2\Big[(\mathcal{V}_{1}^{(3)})^2-\mathcal{V}_{1}^{(6)}\Big]\Big\}l_1^{-3},\\
		&{\mathfrak a}_{01}
		=\kappa_{31}+\frac{6\beta^2(\mathcal{V}_{1}^{(6)}-(\mathcal{V}_{1}^{(4)})^2)}{l_1^3(l_1-1)},~~ {\mathfrak a}_{11}=\mu(g_{1n})+\frac{\beta(1-\mathcal{V}_{1}^{(4)})}{l_1(l_1-1)}
		+\frac{2\Gamma^2(\mathcal{V}_{1}^{(1)})^2}{nl_1(l_1-1)^2},\\
		&\mathfrak c_{1}=\frac{\pi_{n1}\gamma_nl_1(l_1-1)^{-1}}{[(l_1-1)^2-\gamma_n]^{1/2}[2(l_1-1)^2+\pi_{n1}\{(l_1-1)^2-\gamma_n\}]^{1/2}},~~	\mathfrak b_{1}
		:=\frac{8\Gamma^2\mathcal{V}_{1}^{(1)}\mathcal{V}_{1}^{(3)}}{\sqrt{n}l_1^2(l_1-1)}.
	\end{align*}
\end{theorem}

	\subsection{Estimation of the moments of $Z_{11}$}

	The coefficient $p_k(x)$ in Theorem~\ref{multi} involves unknown
	moments of $Z_{11}$ as well as several eigenvector-dependent
	quantities. As discussed in Remark~\ref{rmk2.4}, many of the latter
	are asymptotically negligible and hence require no estimation in
	important regimes, such as the delocalized case. When they remain
	non-negligible, their estimation can be based on existing inference
	methods for spiked eigenvectors, including generalized-component and
	thresholding procedures under the corresponding regularity
	conditions. We therefore focus here on the additional
	non-Gaussian moment parameters $\Gamma^2$, $\beta$, and $\Delta$. 
	Recall the i.i.d.
	sample
	\[
	\mathbf x_i
	=
	\bm\Sigma^{1/2}\mathbf z_i,
	\qquad
	i=1,\ldots,n.
	\]
	\cite{ZhouC24S} proposed a consistent estimator of $\beta$ based on
	U-statistics, which requires the existence of the eighth moment of $Z_{11}$.   Motivated by this idea, we construct split-sample
	estimators of $\beta$, $\Gamma^2$ and $\Delta$ under a weaker moment condition. 
	
	For any index set $\mathscr I$, let
	$
	\mathbf X_{\mathscr I}
	=
	(\mathbf x_i':i\in\mathscr I)'
	\in\mathbb R^{n_{\mathscr I}\times q}
	$ and $\mathbf X_{\mathscr I}^{\circ r}$ denote the $r$-th Hadamard
	power of $\mathbf X_{\mathscr I}$, where
	$
	n_{\mathscr I}=|\mathscr I|
	$
	denotes the cardinality of the set $\mathscr I$.  For each integer $s\geq1$, define the vector of coordinatewise
	population moments by
	$$
	\boldsymbol m_s
	=
	\left(
	\mathbb E X_{11}^s,\ldots,
	\mathbb E X_{1q}^s
	\right)'.
	$$
	Correspondingly, define the coordinatewise empirical
	$s$-th moment vector by
	\[
	\widehat{\boldsymbol m}_{s\mathscr I}
	=\frac1{n_{\mathscr I}}
	\left(\mathbf X_{\mathscr I}^{\circ s}\right)'
	\mathbf1_{n_{\mathscr I}}
	\in\mathbb R^q,
	\]
	where $\1_n$ is the vector of  $n$ ones. Let $K_s(Z)$ be the $s$-th cumulant of a random variable $Z$ and $\bm\Sigma^{1/2}=(a_{jk})_{q\times q}$.

	We first estimate $\Gamma^2$. 
	Let the sample be split into two disjoint sets
	$
	\mathscr I_1$ and $\mathscr I_2,
	$
	whose sizes are both proportional to $n$. 
	For simplicity, write
	$
	n_s=|\mathscr I_s|,
	~
	s=1,2.
	$ 
	Let $\mathcal{D}_3
	=
	{\bf1}_q'
	\bm\Sigma^{\circ3}
	{\bf1}_q$ and $
	\widehat M_3
	=
	\frac{1}{n_1n_2}
	\sum_{i\in\mathscr I_1}
	\sum_{j\in\mathscr I_2}
	(\mathbf x_i'\mathbf x_j)^3
	$. Since $
	\mathbb E(\mathbf x_1'\mathbf x_2)^3
	=
	\Gamma^2\mathcal{D}_3$
	and $
	{\mathcal{D}_3}/{q}\to1$,
	the estimator of $\Gamma^2$ is then defined as
	\[
	\widehat{\Gamma^2}
	=
	{\widehat M_3}/{q}.
	\]

	We next construct a split-sample estimator of $\beta$. Since
	$\beta=K_4(Z_{11})$, the additivity and homogeneity properties of
	cumulants imply that
	\[
	\mathcal C_4
	=\sum_{j=1}^q K_4(X_{1j})
	=\beta\mathcal D_4,
	\qquad
	\mathcal D_4=\sum_{j,k=1}^q a_{jk}^4.
	\]
	Let the sample be split into three disjoint sets $\mathscr K_1,\mathscr K_2,\mathscr K_3$, whose sizes are all proportional to $n$. Define \[
	\widehat{	\mathcal C}_4
	=
	\mathbf1_q'
	\left[
	\widehat{\boldsymbol m}_{4\mathscr K_1}
	-
	3
	\widehat{\boldsymbol m}_{2\mathscr K_2}
	\circ
	\widehat{\boldsymbol m}_{2\mathscr K_3}
	\right].
	\]
	By the fact that ${\mathbb E}\widehat{\mathcal C}_4= 	\mathcal C_4$ and $
	{\mathcal{D}_4}/{q}\to1,$
	we construct the estimator of $\beta$ to be
	\[
	\widehat\beta={\widehat{\mathcal C}_4}/{q}.
	\]

	For the estimation of $\Delta$,  note that 
	$ K_6(Z_{11})=\Delta-15\beta-10\Gamma^2-15 $
	and $K_6(X_{1j})=\sum_{k=1}^qa_{jk}^6K_6(Z_{11})$,
	we define
	\[
	\mathcal C_6
	=\sum_{j=1}^q K_6(X_{1j})
	=K_6(Z_{11})\mathcal D_6,
	\qquad
	\mathcal D_6=\sum_{j,k=1}^q a_{jk}^6.
	\]
	Since $
	{\mathcal{D}_6}/{q}\to1, $
	to estimate $\Delta$, it remains to estimate $\mathcal{C}_6$. Let
	$
	\mathscr J_1,\ldots,\mathscr J_8 
	$ 
	be disjoint sets, whose sizes are all proportional to $n$. Define
	\[
	\begin{aligned}
		\widehat{\mathcal{C}}_6
		=
		\mathbf 1_q'
		\Big(
		\widehat{\boldsymbol m}_{6\mathscr J_1}
		-
		15
		\widehat{\boldsymbol m}_{4\mathscr J_2}
		\circ
		\widehat{\boldsymbol m}_{2\mathscr J_3}
		-
		10
		\widehat{\boldsymbol m}_{3\mathscr J_4}
		\circ
		\widehat{\boldsymbol m}_{3\mathscr J_5}
		+
		30
		\widehat{\boldsymbol m}_{2\mathscr J_6}
		\circ
		\widehat{\boldsymbol m}_{2\mathscr J_7}
		\circ
		\widehat{\boldsymbol m}_{2\mathscr J_8}
		\Big),
	\end{aligned}
	\]
	which is a split-sample estimator of $\mathcal{C}_6$. By the mutual independence
	of the eight blocks,
	$
	\mathbb E\widehat{\mathcal{C}}_6=\mathcal{C}_6.
	$
	Then, the estimator of $K_6(Z_{11})$ is defined as
	$
	\widehat K_6(Z)
	=
	{\widehat{\mathcal{C}}_6}/{q}.
	$
	Finally, the estimator of the sixth moment is \[
	\widehat\Delta
	=
	\widehat K_6(Z)
	+
	15\widehat\beta
	+
	10\widehat{\Gamma^2}
	+
	15.
	\]
	\begin{theorem}\label{thm:moment-estimators}
		Suppose that the spiked covariance model specified in Section~\ref{mainresults} holds and that Assumptions~(a)--(c) are satisfied.
		Then, as $\min\{p,n\}\to\infty$,
		\[
		\widehat{\Gamma^2}-\Gamma^2
		\overset{p}{\to}0,  ~~  \widehat\beta-\beta
		\overset{p}{\to}0
		\mbox{~~and~~}
		\widehat\Delta-\Delta
		\overset{p}{\to}0.
		\]
	\end{theorem}

	\begin{remark}
		The proof of Theorem \ref{thm:moment-estimators} is postponed 
		in the supplementary material. In particular, the assumed covariance structure ensures
		that
		$
		{\mathcal  D_s}/{q}\to1,~s=3,4,6.
		$ 
		For a general covariance
		matrix $\bm\Sigma$, the same split-sample construction remains
		applicable, although the corresponding quantities ${\mathcal  D_s}$
		must be replaced by consistent estimators.
		Under suitable nondegeneracy and growth conditions on these
		normalizing quantities, analogous consistency results can be obtained.
	\end{remark}
	\begin{remark}
		We emphasize that the Edgeworth theorems above are stated for the population moment quantities. The consistency result in this subsection does not, by itself, imply that replacing every occurrence of $\beta$, $\Gamma^2$, $\Delta$, or the spike-direction contractions by estimates preserves the $o(n^{-1/2})$ coverage error. Such a plug-in theorem would require rate conditions and a separate treatment of the induced studentization error. In small or moderate samples, higher-moment estimators may also introduce nonnegligible variability.
		Nevertheless, the above theorem has an independent interest beyond the
		implementation of the Edgeworth corrections considered in this paper:
		it provides consistent estimators of the latent third-, fourth-, and sixth-order
		moment functionals in a high-dimensional independent component model,
		which may also be useful in other high-dimensional inference problems.
\end{remark}

\section{Application}\label{application}
This section considers two statistical problems in high-dimensional
covariance analysis: (i) inference for spiked population eigenvalues
and (ii) estimation of the number of spikes. Both problems arise
naturally in high-dimensional principal component analysis.

\subsection{Inference for spiked eigenvalues}\label{secYang}
Many inferential procedures for individual population spikes presuppose
that the number of spikes is known, whereas in applications this number
often needs to be estimated from the data. Following
\cite{yang2019edgeworth}, we first consider conditional inference for
supercritical population spikes and then extend the Edgeworth-based
procedure to non-Gaussian observations.

Let $\hat{\rho}_k$ be the $k$th largest eigenvalue of $\mathbf S$ and
write $\mathbf l=(l_1,\ldots,l_r)$. Denote by $F_{kn}$ the finite-sample
distribution function of $\hat{\rho}_k$, with its dependence on the
model parameters suppressed from the notation, and let
$\bar F_{kn}=1-F_{kn}$. Define
\[
\theta_n=(1+\sqrt{\gamma_n})^2+n^{-1/3}\sqrt{\gamma_n}.
\]
If $F_{kn}$ is continuous, the probability integral transform gives
\[
u_{kn}(\hat{\rho}_k)
=
	{\bar F_{kn}(\hat{\rho}_k)}/
	{\bar F_{kn}(\theta_n)}
\]
and
\begin{equation}\label{eqYang}
	u_{kn}(\hat{\rho}_k)
	\,\big|\,\{\hat{\rho}_k>\theta_n\}
	\sim U(0,1).
\end{equation}

Replacing the unknown distribution by its normal approximation leads
to the $Z$-type pivot
\[
u_{kn}^Z(\hat{\rho}_k,l_k)
=
{\bar\Phi\!\left(z_{kn}(\hat{\rho}_k,l_k)\right)}/
{\bar\Phi\!\left(z_{kn}(\theta_n,l_k)\right)},
\]
where
\[
z_{kn}(y,l_k)
=
{\sqrt n(y-\rho_{nk})}/
{\widetilde\sigma_{nk}},
\]
and $\rho_{nk}$ and $\widetilde\sigma_{nk}$ are the centering and
scaling quantities defined in Section~\ref{mainresults}. Dependence of $\widetilde\sigma_{nk}$ on the remaining moment parameters is suppressed from the notation.

Theorem~\ref{multi} provides a first-order refinement of this normal
approximation. Define
\[
F_{kn}^{E}(x,\mathbf l)
=
\Phi(x)
+
n^{-1/2}p_k(x)\phi(x),
\qquad
\bar F_{kn}^{E}=1-F_{kn}^{E}.
\]
The corresponding
$E$-type pivot is
\[
u_{kn}^{E}(\hat{\rho}_k,\mathbf l)
=
{
	\bar F_{kn}^{E}
	\!\left(
	z_{kn}(\hat{\rho}_k,l_k),\mathbf l
	\right)
}/{
	\bar F_{kn}^{E}
	\!\left(
	z_{kn}(\theta_n,l_k),\mathbf l
	\right)
}.
\]

\begin{theorem}\label{Pivots}
	Suppose that Assumptions~$(a)$--$(d)$ hold. For each fixed
	$k\in\{1,\ldots,r\}$, evaluate the $Z$- and $E$-type pivots at the
	population moment parameters appearing in
	Theorem~\ref{multi}. Then, for every fixed
	$\alpha\in(0,1)$,
	\[
	\left|
	\mathbb P
	\left(
	u_{kn}^{Z}(\hat{\rho}_k,l_k)
	\leq\alpha
	\,\middle|\,
	\hat{\rho}_k>\theta_n
	\right)
	-\alpha
	\right|
	=
	O(n^{-1/2}),
	\]
	whereas
	\[
	\left|
	\mathbb P
	\left(
	u_{kn}^{E}(\hat{\rho}_k,\mathbf l)
	\leq\alpha
	\,\middle|\,
	\hat{\rho}_k>\theta_n
	\right)
	-\alpha
	\right|
	=
	o(n^{-1/2}).
	\]
\end{theorem}

\begin{remark}
	Theorem~\ref{Pivots} shows that the conditional calibration error of
	the $E$-type pivot is $o(n^{-1/2})$, whereas the normal approximation
	has calibration error of order $O(n^{-1/2})$. This provides the basis
	for Edgeworth-corrected conditional inference for population spiked
	eigenvalues after spike selection. 
		For practical implementation, a plug-in interval with nominal
		confidence level $1-\alpha$ can be obtained by solving
		\[
		u_{kn}^{E}
		\left(
		\hat{\rho}_k,\hat{\mathbf l}_k
		\right)
		=
		{\alpha}/{2}
		\quad\text{and}\quad
		u_{kn}^{E}
		\left(
		\hat{\rho}_k,\hat{\mathbf l}_k
		\right)
		=
		1-{\alpha}/{2}
		\]
		over
		$
		l_k>1+\sqrt{\gamma_n},
		$ 
		where
		$
		\hat{\mathbf l}_k
		=
		(
		\hat l_1,\ldots,\hat l_{k-1},
		l_k,
		\hat l_{k+1},\ldots,\hat l_r
		)
		$ 
		and, for $j\neq k$,
		\begin{align}\label{tildelj}
			\hat l_j
			=
			\frac{
				\hat\rho_j+1-\gamma_n
				+
				\sqrt{
					(\hat\rho_j+1-\gamma_n)^2
					-
					4\hat\rho_j
				}
			}{2}.
		\end{align}
		The remaining moment- and eigenvector-dependent moment quantities
		are replaced by suitable estimators.
		
		We emphasize that Theorem~\ref{Pivots} concerns the corresponding
		oracle pivots evaluated at the population moment parameters.
		Consistency of the moment parameter estimators alone does not imply
		that the fully plug-in interval retains an $o(n^{-1/2})$ coverage
		error. Establishing such a result would require appropriate
		convergence-rate conditions and a separate analysis of the plug-in
		and studentization errors.
	
\end{remark}

\subsection{Estimation of the number of spiked eigenvalues}
\label{Jiang}

The accurate identification of spiked eigenvalues is a fundamental problem
in high-dimensional statistical theory, with important implications for
principal component analysis and its applications. Extending the asymptotic
framework developed in \cite{jiang2021generalized}, we propose a sequential
screening procedure that incorporates the Edgeworth correction derived in
this paper. 
The procedure is described as follows.

	\textbf{Step 1. }
	Let
	$
	\hat{\rho}_1\geq \hat{\rho}_2\geq \cdots
	$
	denote the ordered eigenvalues of the sample covariance matrix. We first
	select the sample eigenvalues that exceed the theoretical upper edge of the
	Marchenko--Pastur distribution. Specifically, let
	\[
	m=\#\left\{j:\hat{\rho}_j>(1+\sqrt{\gamma_n})^2\right\}.
	\]
	For each $1\leq k\leq m$, we obtain a preliminary estimator
	$\hat l_k$ of the corresponding population spike by solving the inverse
	spike equation, as in \eqref{tildelj}. These preliminary estimates are used
	as plug-in quantities in the Edgeworth correction.
	
	\textbf{Step 2. }
	For $1\leq k\leq m$, define $
	\widetilde{\mathbf l}_k=(\hat l_1,\ldots,\hat l_k)'$.
	Let
	\[
	E_{\alpha,k}
	=
	z_\alpha
	-
	n^{-1/2}p_k\bigl(z_\alpha;\widetilde{\mathbf l}_k\bigr)
	\]
	be the Cornish--Fisher corrected $\alpha$-quantile (see
	\cite{hall2013bootstrap}), where $z_\alpha$
	is the $\alpha$-quantile of the standard normal distribution. Here
	$p_k\bigl(z_\alpha;\widetilde{\mathbf l}_k\bigr)$ denotes the Edgeworth
	correction polynomial with the unknown spike parameters in the term
	$A(g_{kn})$ replaced by their preliminary estimates
	$\widetilde{\mathbf l}_k$.
	
	Following the normalization used in \cite{jiang2021generalized}, we define
	the Edgeworth-corrected lower critical value for the $k$-th sample
	eigenvalue by
	\[
	L_k=\hat{\phi}_k+\widetilde{\sigma}_n(\hat l_k,\gamma_n)E_{\alpha,k} n^{-1/2},
	\]
	where $\hat{\phi}_k$ is the estimated centering obtained by the method of \cite{jiang2021generalized}. And
	\[\widetilde{\sigma}_n(\hat l_k,\gamma_n)=\hat l_k\sqrt{
		2\left(1-{\gamma_n}/{(\hat l_k-1)^2}\right)
		+
		\beta\mathcal{V}_k^{(4)}
		\left(1-{\gamma_n}/{(\hat l_k-1)^2}\right)^2
	}.\]
	The use of $E_{\alpha,k}$, instead of the standard normal quantile
	$z_\alpha$, incorporates the first-order Edgeworth correction and is
	intended to improve the finite-sample accuracy of the detection threshold.
	
	\textbf{Step 3. }
	Starting from $k=1$, we compare the $k$-th sample eigenvalue
	$\hat{\rho}_k$ with the Edgeworth-corrected lower critical value $L_k$.
	The $k$-th eigenvalue is classified as a spike if
	\[
	\hat{\rho}_k\geq L_k.
	\]
	The procedure stops at the first index for which this inequality fails. Let
	\[
	\mathscr K
	=
	\left\{
	1\leq k\leq m:\hat{\rho}_k<L_k
	\right\}.
	\]
	We then estimate the number of spiked eigenvalues by
	\[
	\hat r
	=
	\begin{cases}
		\min \mathscr K-1, & \mathscr K\neq\varnothing,\\
		m, & \mathscr K=\varnothing.
	\end{cases}
	\]
	Equivalently, $\hat r$ is the largest integer $q\leq m$ such that
	\[
	\hat{\rho}_k\geq L_k,\qquad k=1,\ldots,q.
	\]
	Thus, the proposed procedure retains the sequential nature of the method in
	\cite{jiang2021generalized}, while replacing the Gaussian critical value by
	an Edgeworth-corrected one.

\begin{remark}
	It is worth noting that the above procedure is mainly designed as a practical comparison method to demonstrate the effect of the Edgeworth correction on spike-number estimation. It should be viewed as a plug-in implementation rather than a fully optimized estimator. Several components of the procedure can be replaced by other high-dimensional methods. For example, the initial number $m$, the preliminary spike estimates $\hat {l} _ {k} $, and the centering and scaling estimators $\hat {\phi} _ {k} $ and $\hat {\sigma} _ {k} $ may be obtained from alternative consistent procedures for high-dimensional spiked covariance models. One may also choose a slightly larger initial candidate set to reduce the risk of missing weak spikes during the screening stage.
\end{remark}

\section{Simulation}\label{simulation}

\subsection{Finite-sample distribution of $R_k$}
This section presents a comprehensive simulation study evaluating the finite-sample performance of both the Gaussian approximation and our proposed first-order Edgeworth expansion for the distribution of spiked eigenvalues under non-Gaussian observations.  

We generate Monte Carlo samples from various distributions with mean 0 and variance 1 to evaluate the finite sample distributions of the spiked eigenvalues. To balance computational efficiency and the effectiveness of the Edgeworth expansion, we select a sample size of $n=200$. This size is large enough for the asymptotic theory to be relevant but small enough that the higher-order correction remains visible. Let $q$ be the total dimension of $\mathbf S$. For a model with $r$ spikes, the noise dimension is $p=q-r=20$, and all theoretical centering and scaling formulas use $\gamma_n=p/n$. For each configuration, we perform 10,000 Monte Carlo replications. Using these results, we plot the empirical density function of $\hat{\rho}_k$, along with the theoretical curves from the Gaussian and Edgeworth approximations. The experimental design is summarized in Table \ref{tab:settings}.

\begin{table}
	\caption{Experimental Settings Summary}
	\label{tab:settings}
	\centering 
	\setlength{\tabcolsep}{10pt} 
	\renewcommand{\arraystretch}{1.3} 
	\begin{tabular}{lll|>{\centering\arraybackslash}p{3cm}} 
		\toprule
		\textbf{Setting} & \textbf{Covariance Matrix} & \textbf{Spikes} & \textbf{Distribution of $Z_{11}$} \\
		\midrule
		1 & $\bm{\Sigma}=\operatorname{diag}(4,1,\dots,1)$ & 4 & \multirow{9}{=}{%
			\centering 
			$U(-\sqrt{3},\sqrt{3})$\\[2.5mm]
			$\frac{1}{\sqrt 2}(\chi^2(1)-1)$ \\[2.5mm]
			$Ga(1,1)-1$ \\[2.5mm]
			$2Ga(1,2)-1$\\ [2.5mm]
			$\sqrt{2}(Ga(2,2)-1)$\\[2.5mm]
			$\sqrt{3}(Ga(3,3)-1)$%
		} \\
		2 & $\bm{\Sigma}=\operatorname{diag}(6,4,1,\dots,1)$ & 6, 4 & \\
		3 & $\bm{\Sigma}=\operatorname{diag}(8,6,4,1,\dots,1)$ & 8, 6, 4 & \\
		\cmidrule(lr){1-3}
		4 & $\bm{\Sigma}=\mathbf{Q}_1'\operatorname{diag}(4,1,\dots,1)\mathbf{Q}_1$ & 4 & \\
		5 & $\bm{\Sigma}=\mathbf{Q}_1'\operatorname{diag}(6,4,1,\dots,1)\mathbf{Q}_1$ & 6, 4 & \\
		6 & $\bm{\Sigma}=\mathbf{Q}_1'\operatorname{diag}(8,6,4,1,\dots,1)\mathbf{Q}_1$ & 8, 6, 4 & \\
		\cmidrule(lr){1-3}
		7 & $\bm{\Sigma}=\mathbf{Q}_2'\operatorname{diag}(4,1,\dots,1)\mathbf{Q}_2$ & 4 & \\
		8 & $\bm{\Sigma}=\mathbf{Q}_2'\operatorname{diag}(6,4,1,\dots,1)\mathbf{Q}_2$ & 6, 4 & \\
		9 & $\bm{\Sigma}=\mathbf{Q}_2'\operatorname{diag}(8,6,4,1,\dots,1)\mathbf{Q}_2$ & 8, 6, 4 & \\
		\bottomrule
	\end{tabular}
	\par 
	\smallskip
	\footnotesize
	Here $\mathbf{Q}_1=\begin{pmatrix}\mathbf{O},&\bf 0\\
		\bf 0,&\mathbf{I}\end{pmatrix}$, and $\mathbf{O}$ is  a $3\times 3$ 
	Haar-distributed orthogonal matrix, and $\mathbf{Q}_2$ is a $q\times q$ 
	Haar-distributed orthogonal matrix.
\end{table}

The numerical results, illustrated in Figures \ref{fig:both_images1} and \ref{fig:both_images2}, demonstrate that the Edgeworth expansion provides a substantially more accurate approximation than the standard Gaussian approximations. For clarity, we present results for two representative cases here, with complete simulation results provided in Appendix~\ref{app:simulations}.

\begin{figure}
	\begin{subfigure}{0.33\textwidth}
		\centering
		\includegraphics[width=\linewidth]{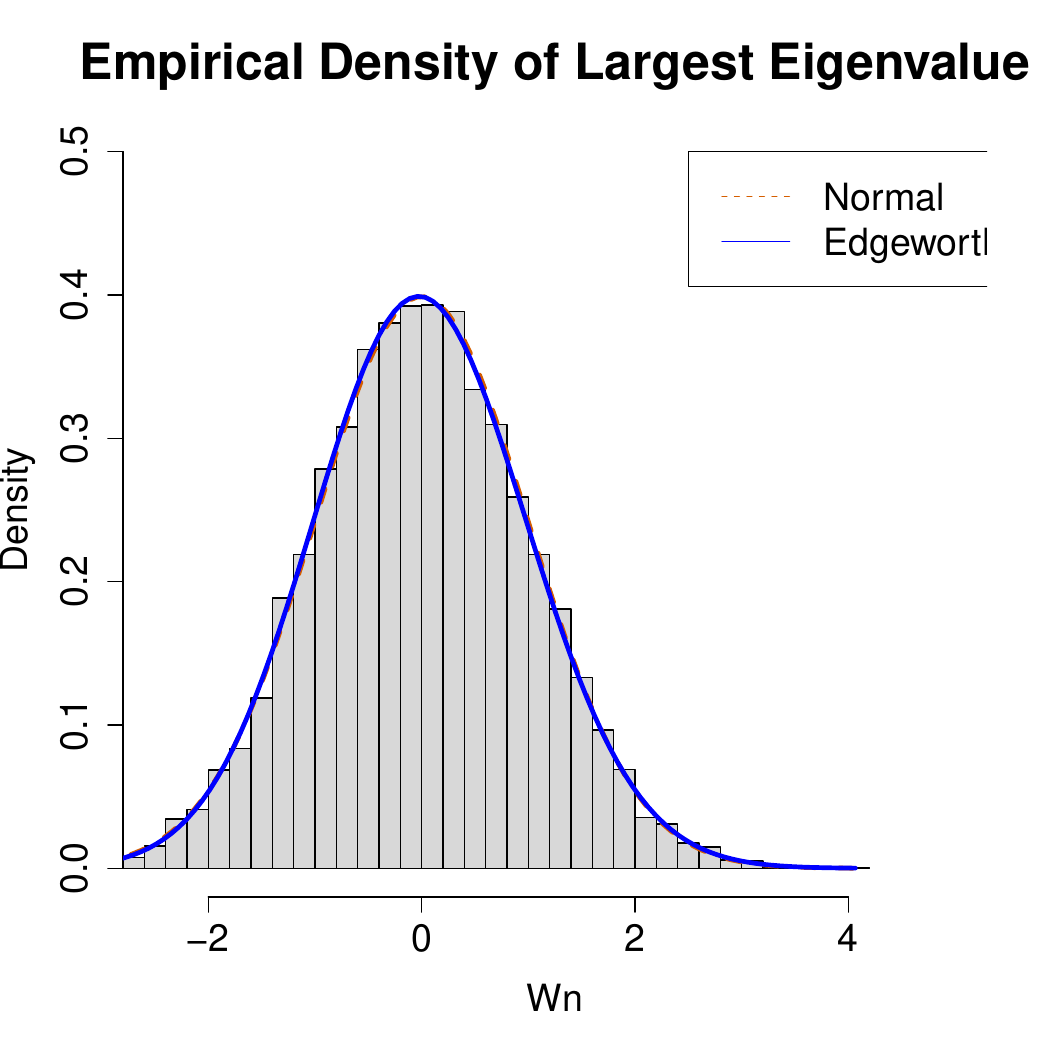}
		\caption{ $U(-\sqrt{3},\sqrt{3}),~ l=4$}
		\label{fig:image1}
	\end{subfigure}%
	\begin{subfigure}{0.33\textwidth}
		\centering
		\includegraphics[width=\linewidth]{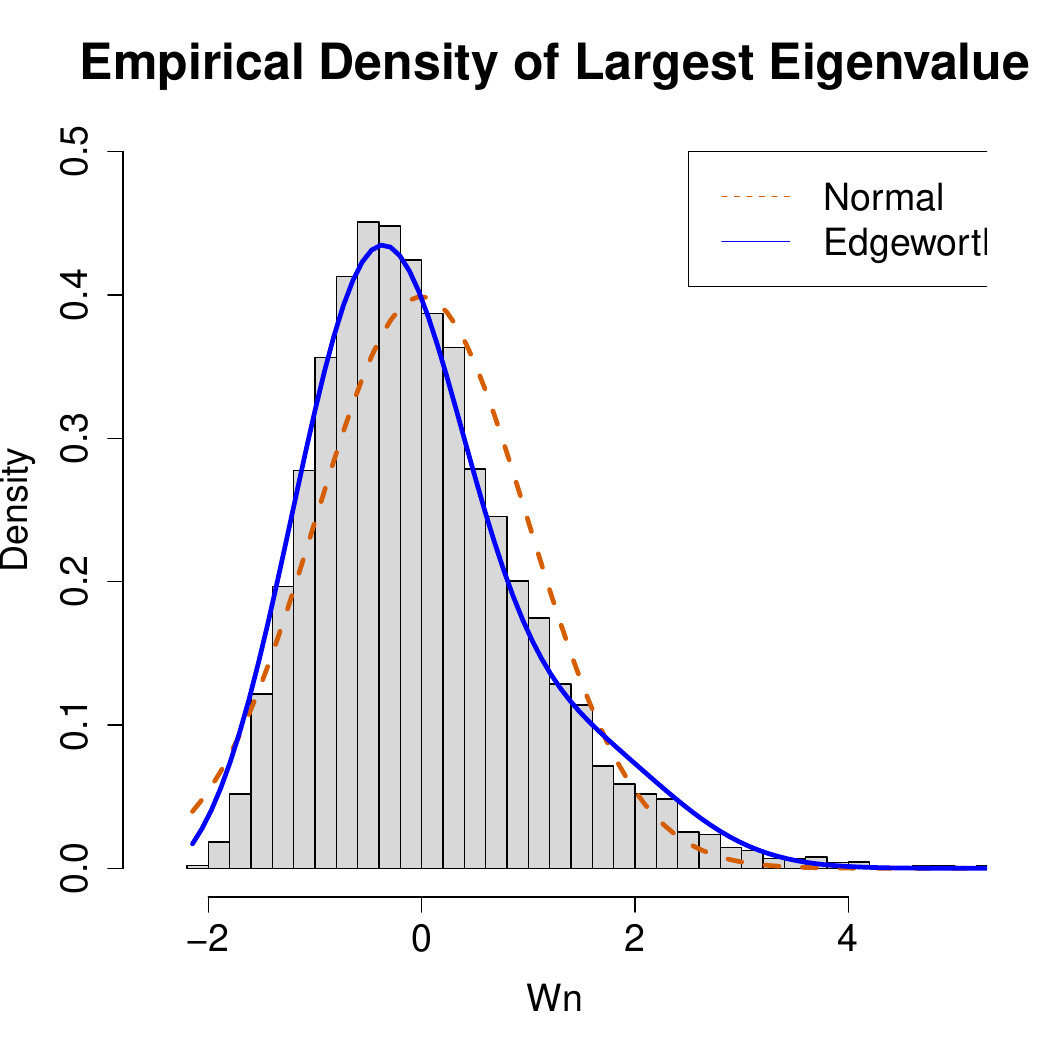}
		\caption{$\chi^2(1),~ l=4$}
		\label{fig:image2}
	\end{subfigure}
	\begin{subfigure}{0.33\textwidth}
		\centering
		\includegraphics[width=\linewidth]{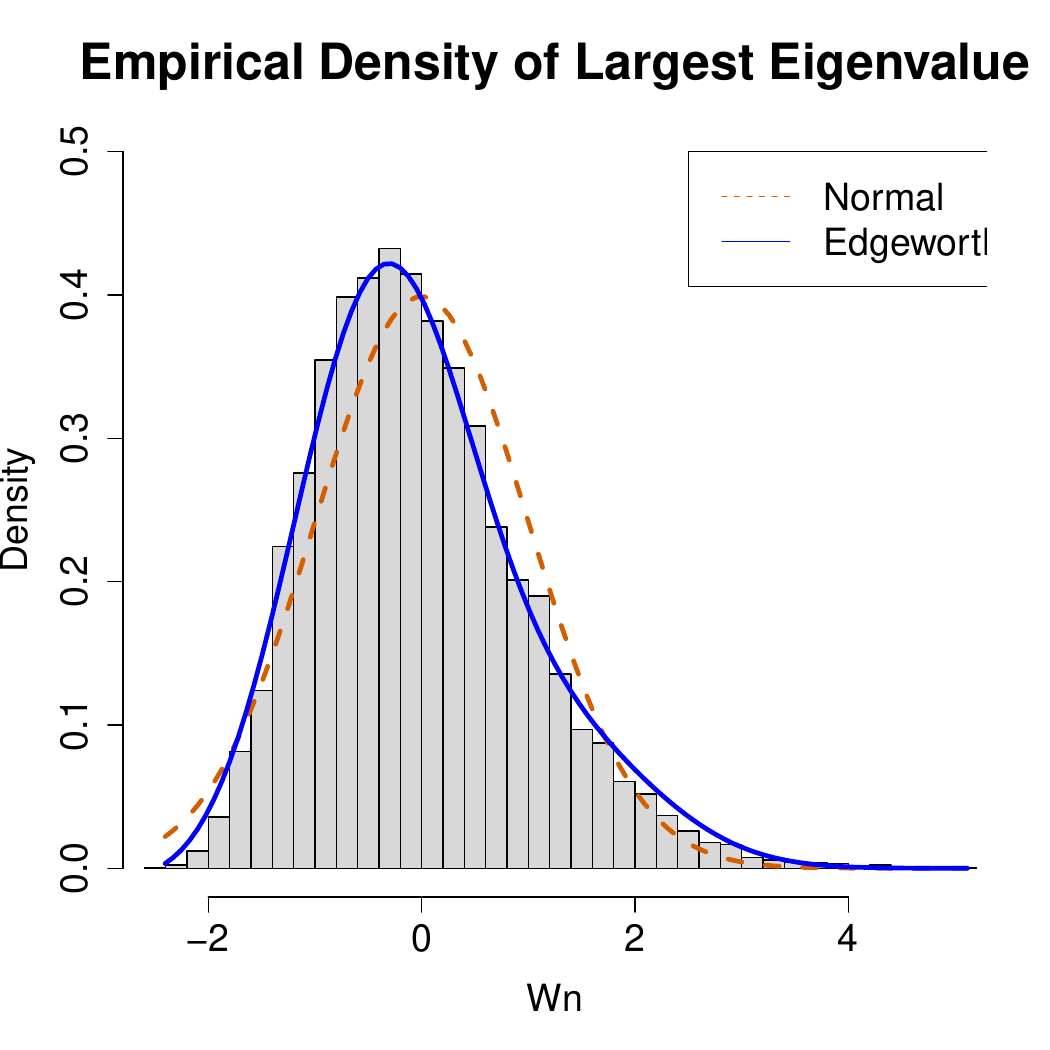}
		\caption{ $Ga(1,1),~ l=4$}
		\label{fig:image1-2}
	\end{subfigure}%

	\begin{subfigure}{0.33\textwidth}
		\centering
		\includegraphics[width=\linewidth]{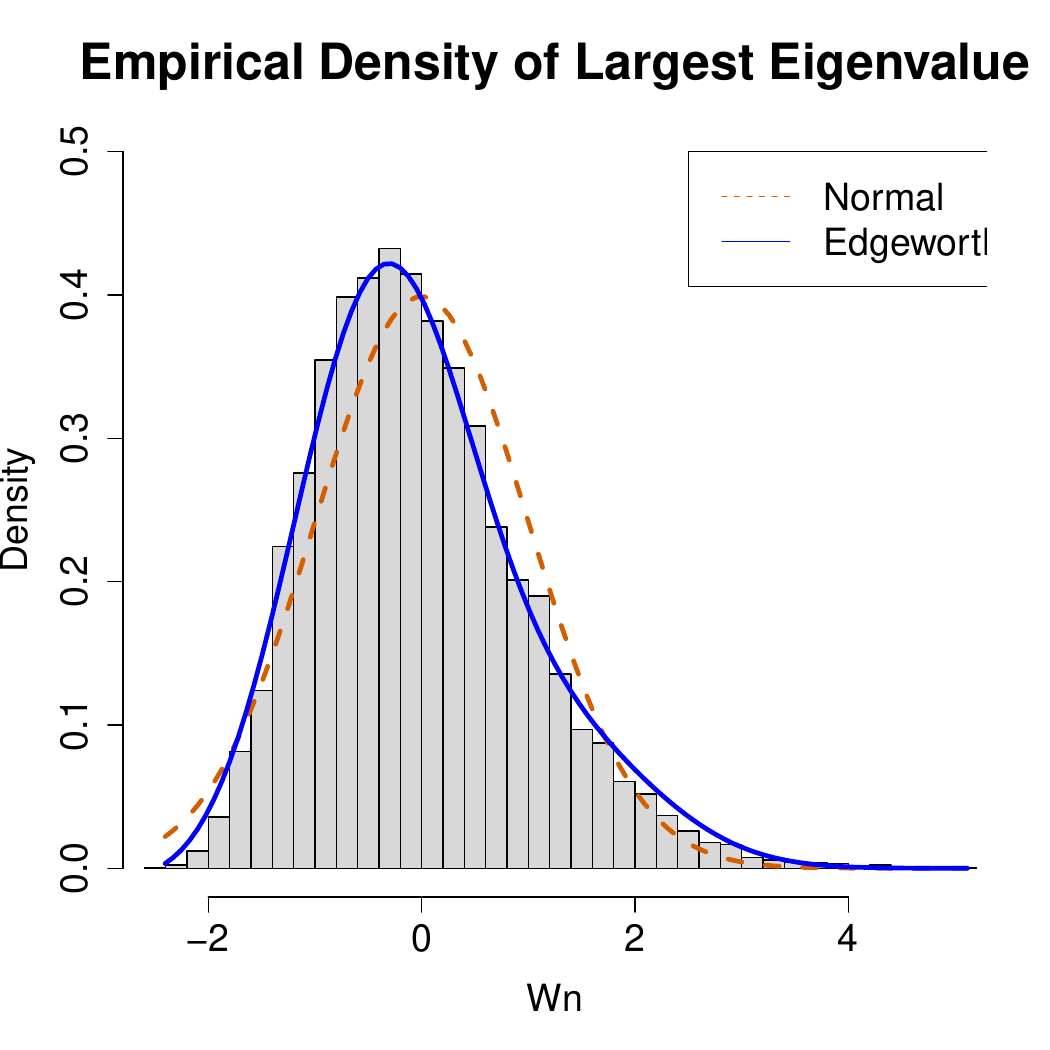}
		\caption{ $Ga(1,2),~ l=4$}
		\label{fig:image1-3}
	\end{subfigure}%
	\begin{subfigure}{0.33\textwidth}
		\centering
		\includegraphics[width=\linewidth]{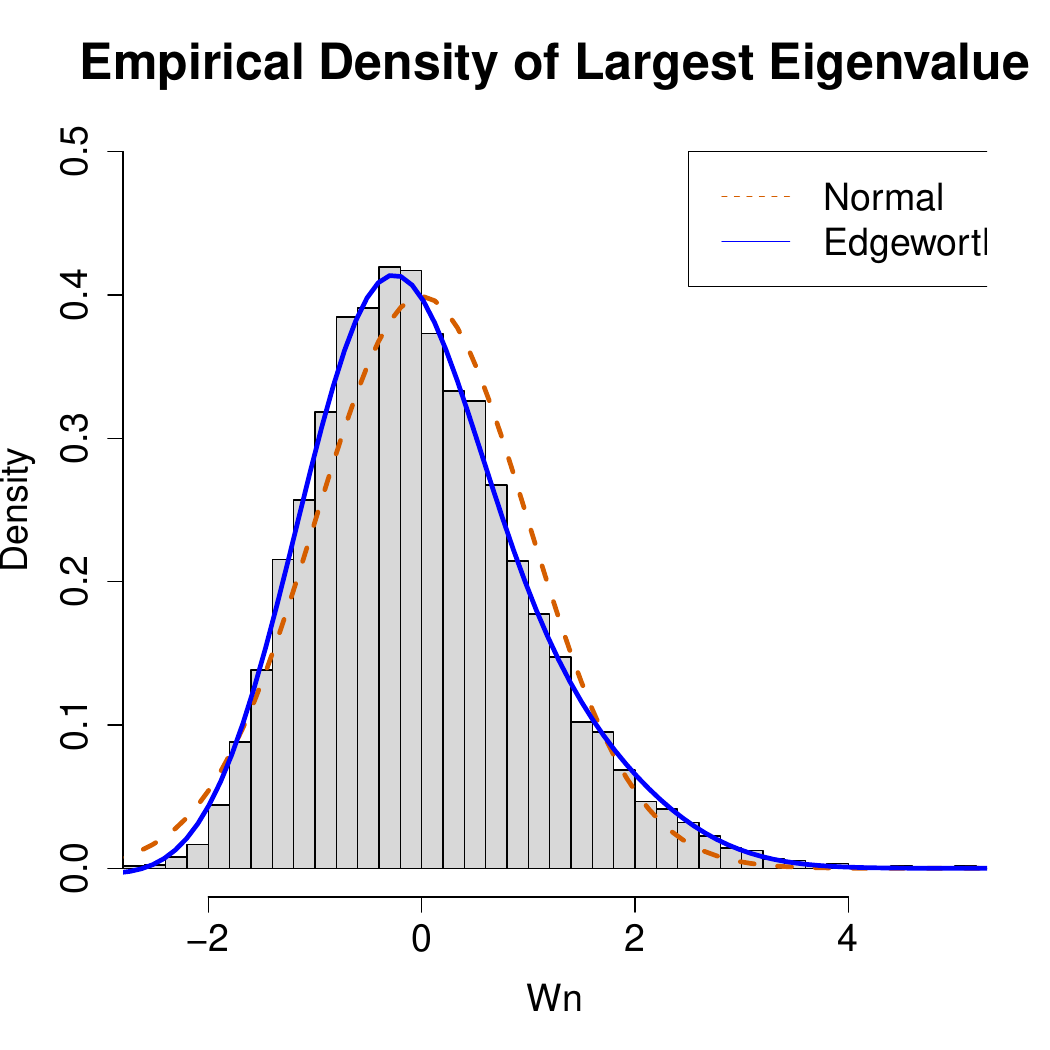}
		\caption{$Ga(2,2),~ l=4$}
		\label{fig:image2-2}
	\end{subfigure}
	\begin{subfigure}{0.33\textwidth}
		\centering
		\includegraphics[width=\linewidth]{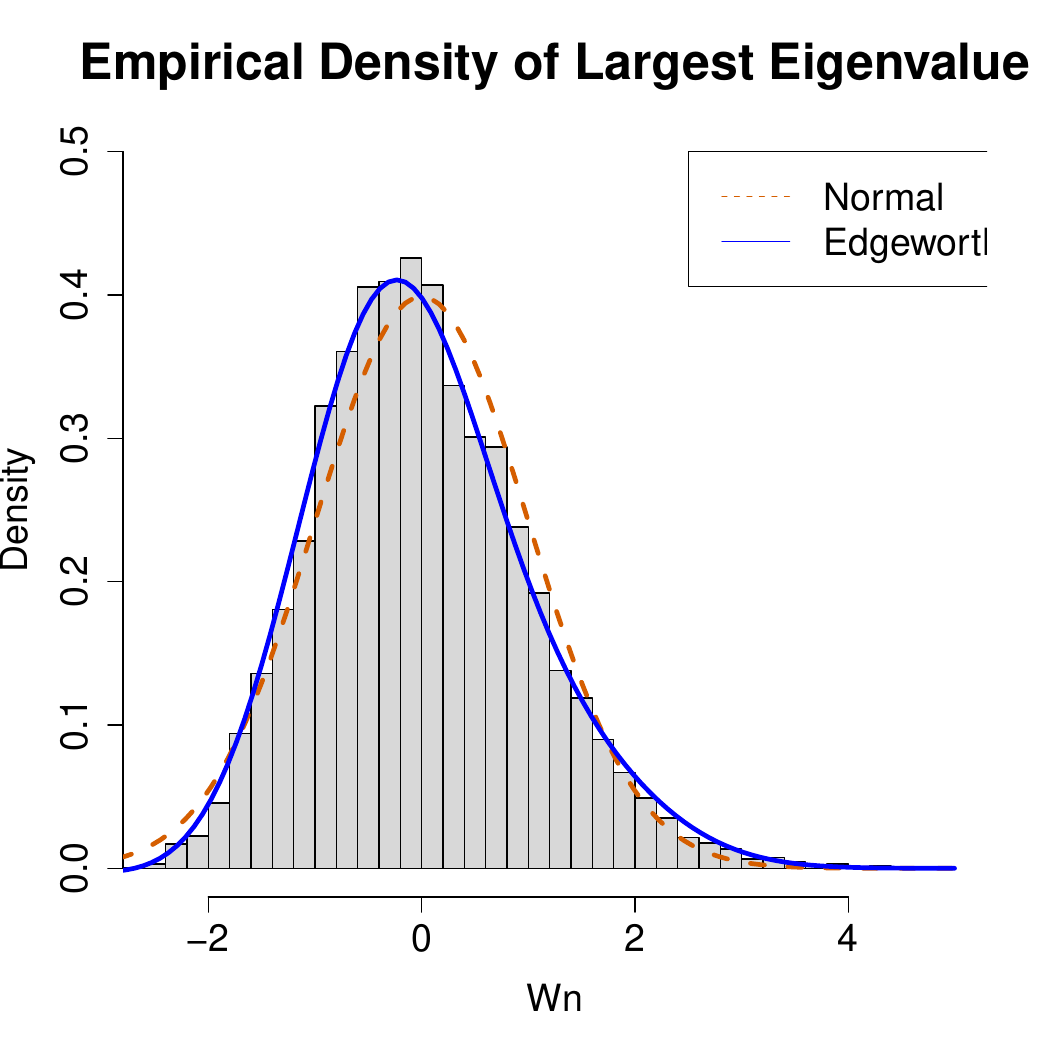}
		\caption{ $Ga(3,3),~ l=4$}
		\label{fig:image1-4}
	\end{subfigure}%
	\caption{ Edgeworth expansion  for different distributions under Setting 1}
	\label{fig:both_images1}
\end{figure}

\begin{figure}
	
	\begin{subfigure}{0.33\textwidth}
		\centering
		\includegraphics[width=\linewidth]{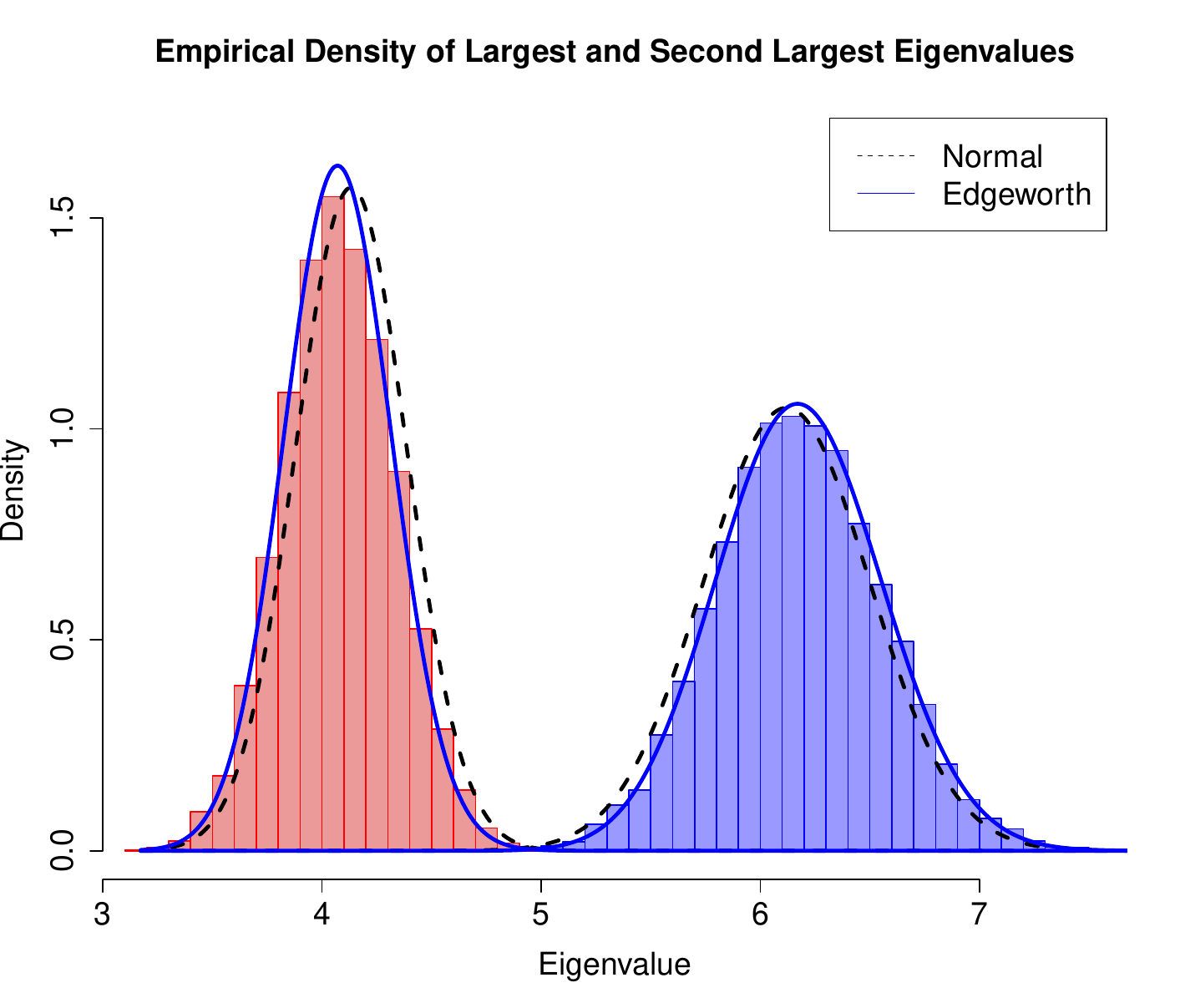}
		\caption{ $U(-\sqrt{3},\sqrt{3}),l_1=6,l_2=4$}
		\label{fig:image21}
	\end{subfigure}%
	\begin{subfigure}{0.33\textwidth}
		\centering
		\includegraphics[width=\linewidth]{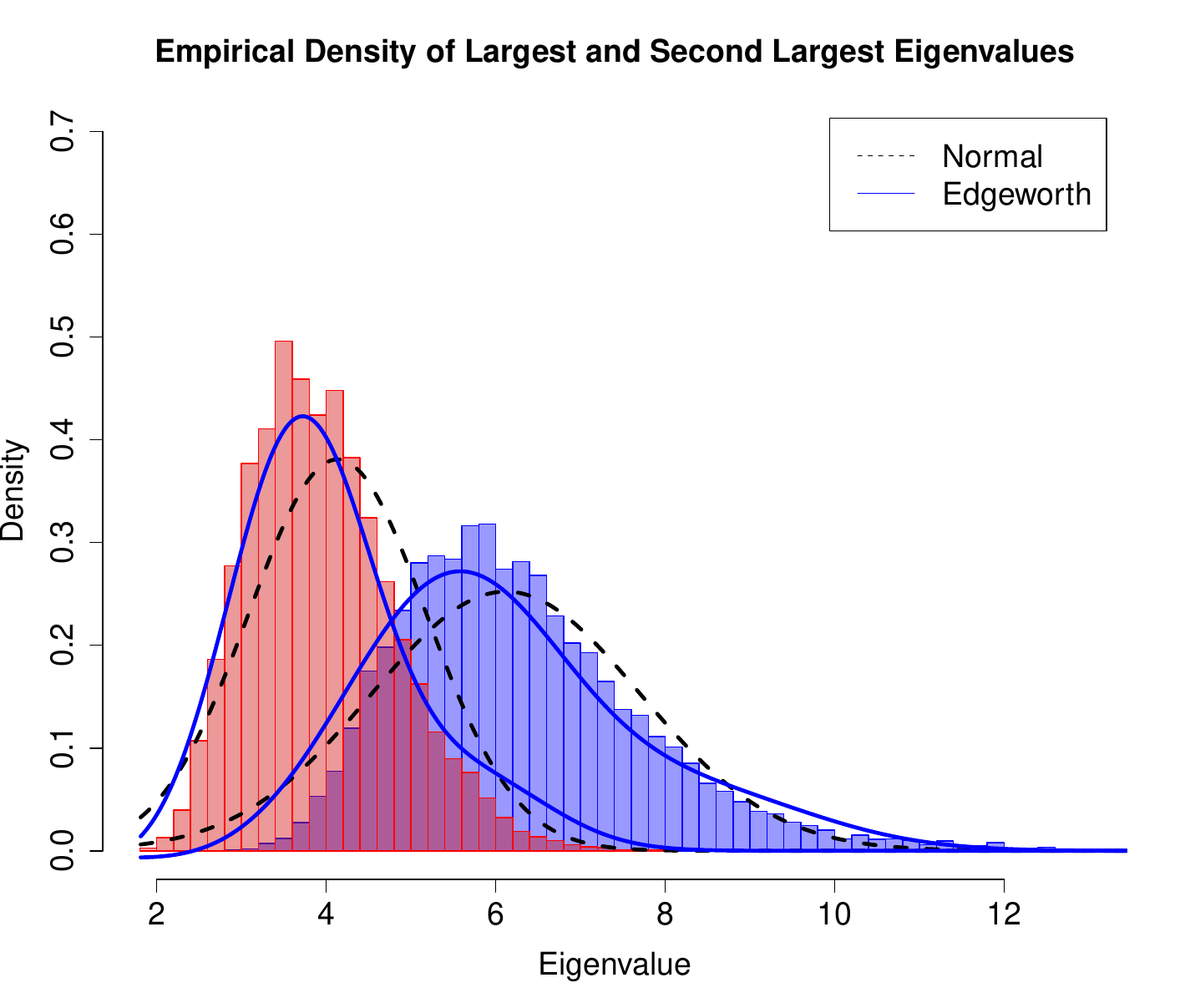}
		\caption{ $\chi^2(1),~ l_1=6,~l_2=4$}
		\label{fig:image18}
	\end{subfigure}%
	\begin{subfigure}{0.33\textwidth}
		\centering
		\includegraphics[width=\linewidth]{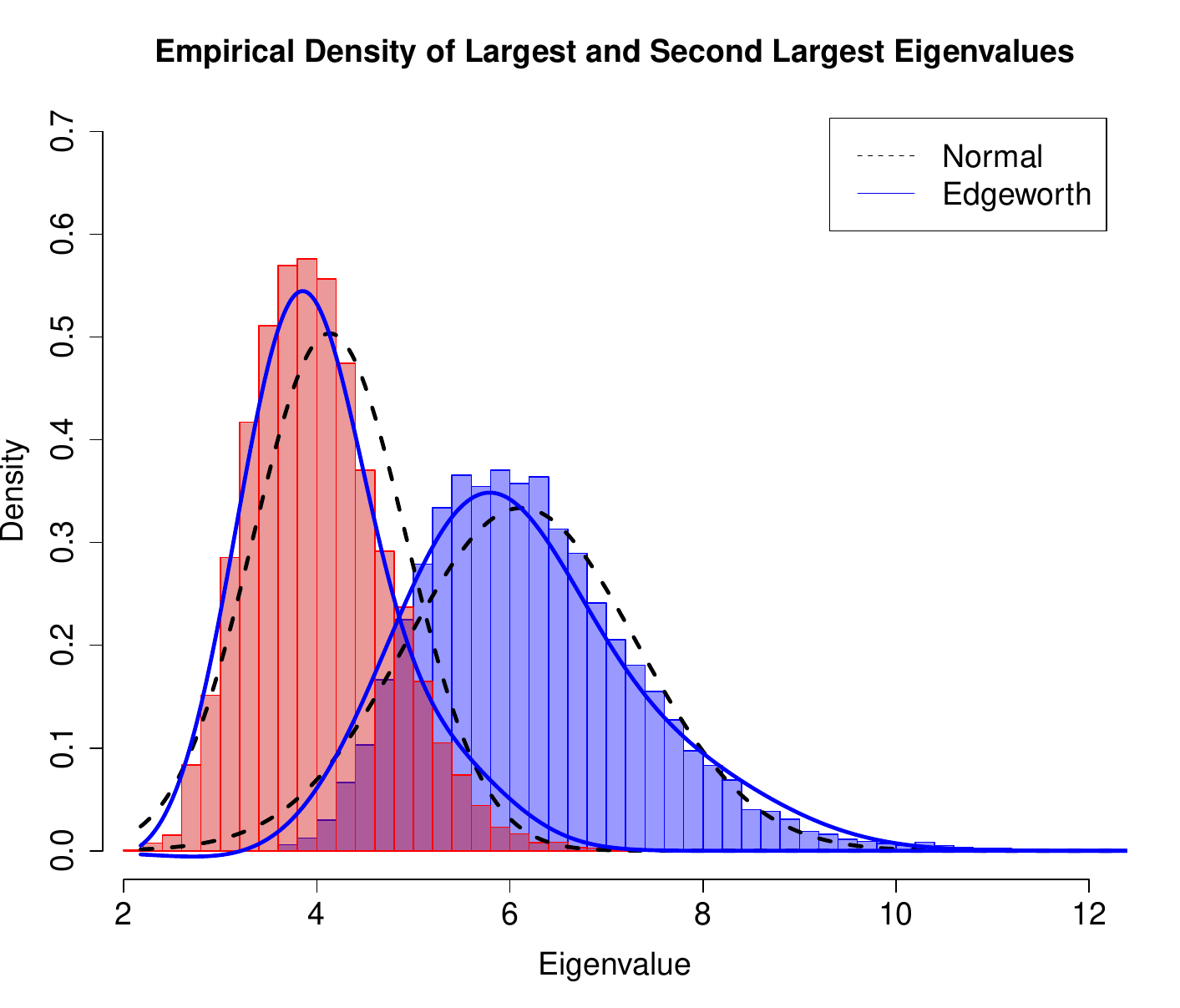}
		\caption{ $Ga(1,1),~ l_1=6,~l_2=4$}
		\label{fig:image21-2}
	\end{subfigure}%

	\begin{subfigure}{0.33\textwidth}
		\centering
		\includegraphics[width=\linewidth]{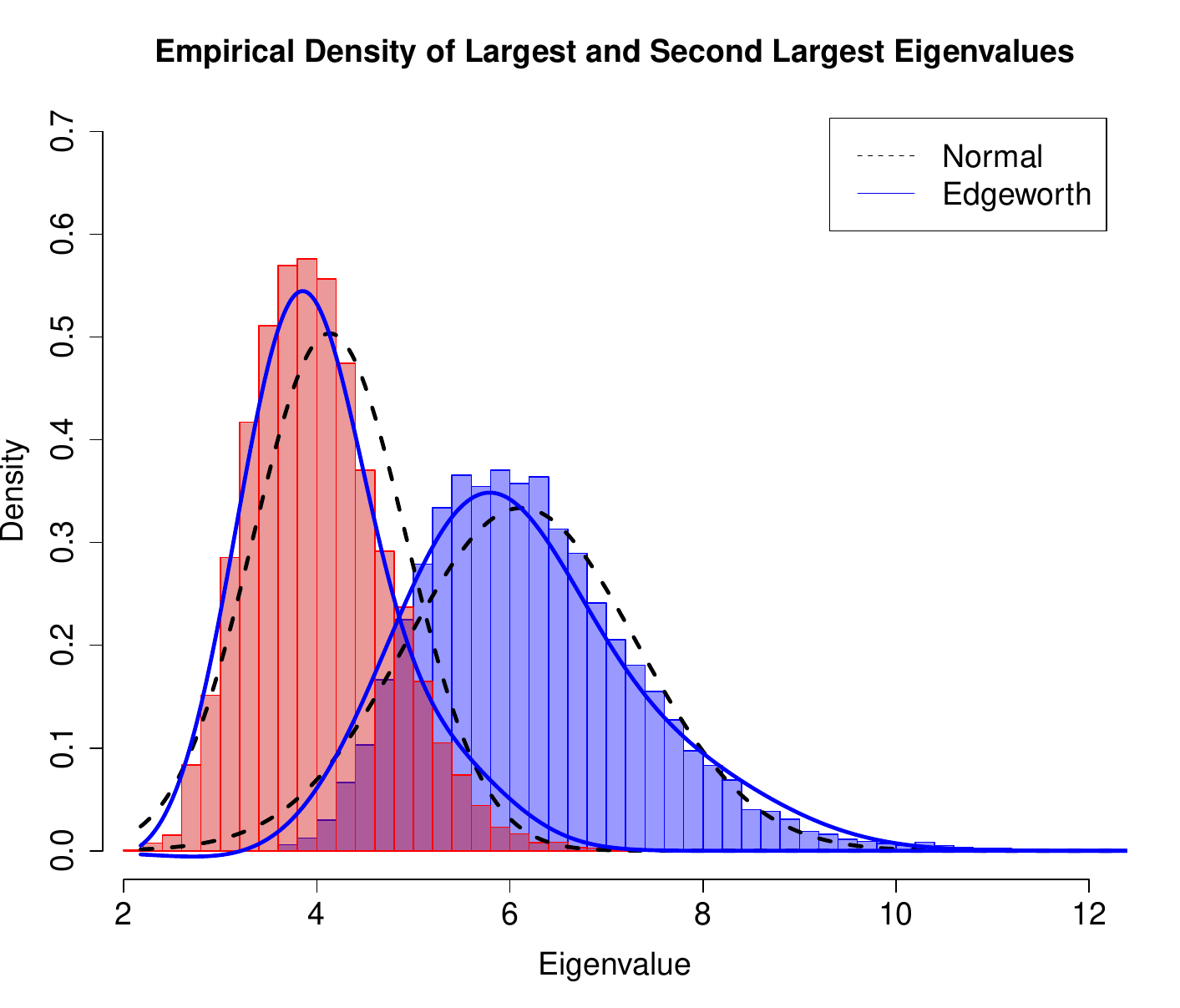}
		\caption{ $Ga(1,2),~ l_1=6,~l_2=4$}
		\label{fig:image18-2}
	\end{subfigure}%
	\begin{subfigure}{0.33\textwidth}
		\centering
		\includegraphics[width=\linewidth]{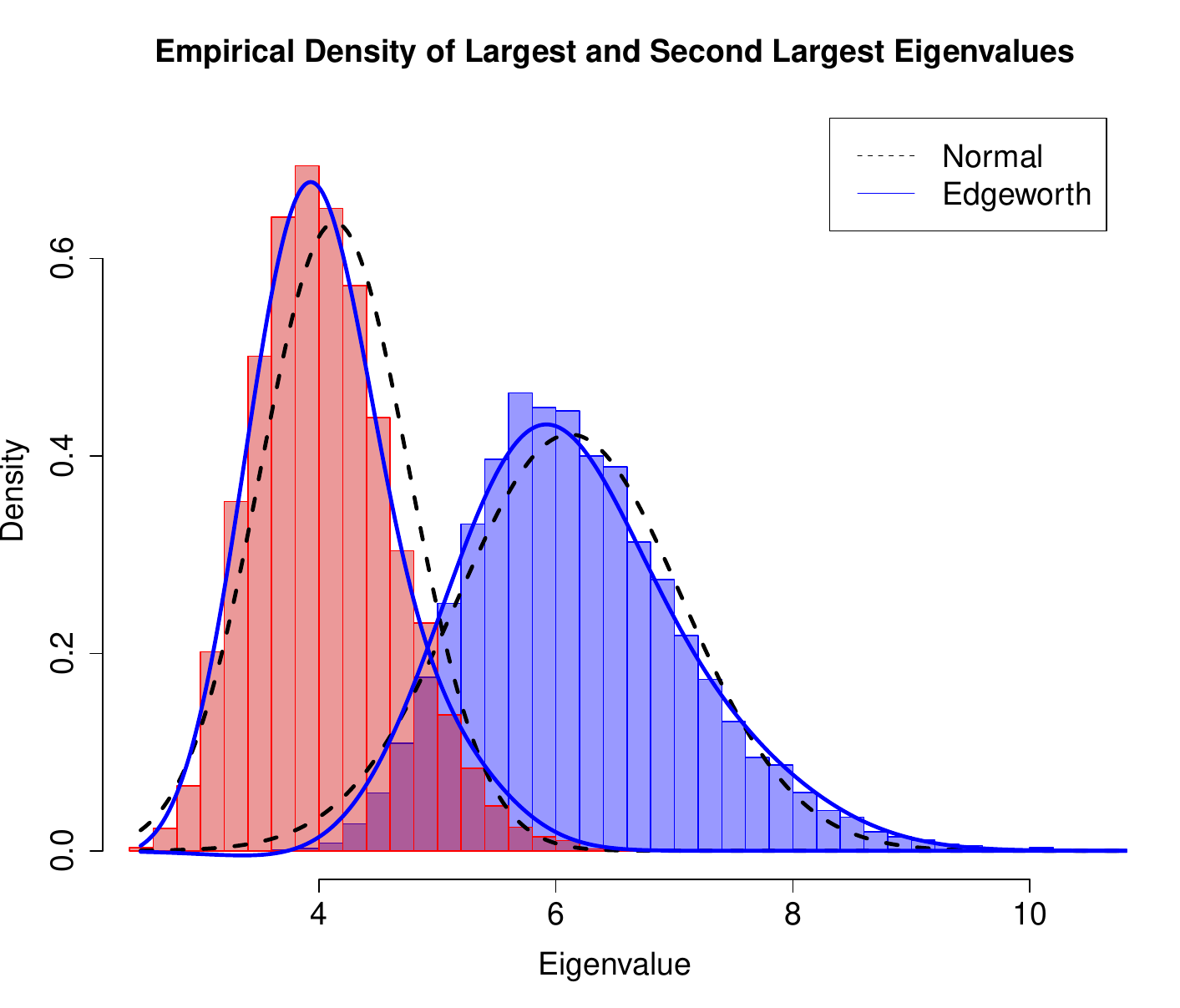}
		\caption{$Ga(2,2),~ l_1=6,~l_2=4$}
		\label{fig:image19}
	\end{subfigure}
	\begin{subfigure}{0.33\textwidth}
		\centering
		\includegraphics[width=\linewidth]{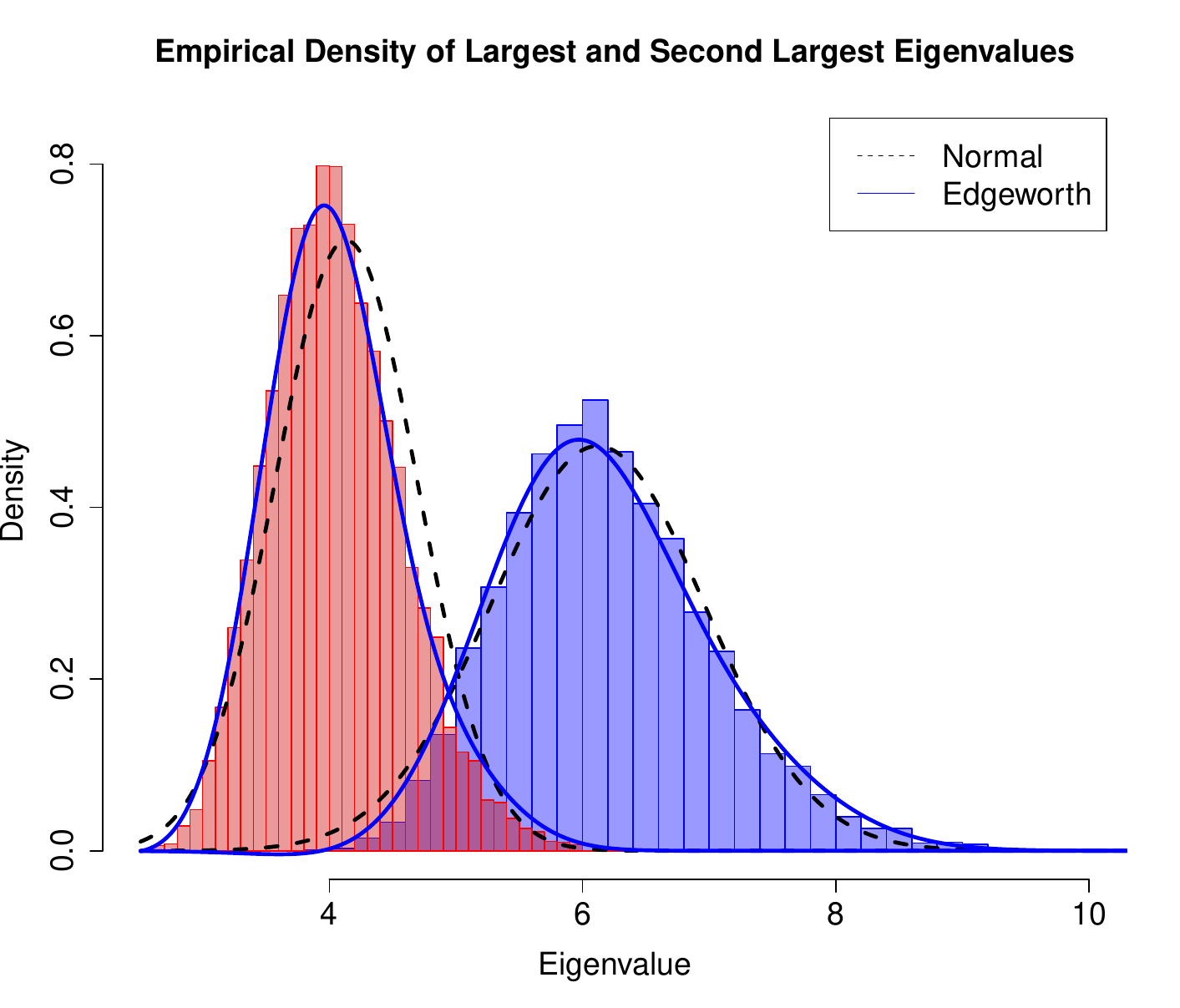}
		\caption{$Ga(3,3),~ l_1=6,~l_2=4$}
		\label{fig:image22}
	\end{subfigure}
	\caption{ Edgeworth expansion  for different distributions under Setting 2}
	\label{fig:both_images2}
\end{figure}

		\subsection{Estimation of the number of spikes}
	In this section, we apply the interval estimation frameworks established in Sections \ref{secYang} and \ref{Jiang} to estimate the number of spikes. We set $\alpha=0.1$ for the J\&B-E method, consistent with the approach of \cite{jiang2021generalized}. Furthermore, we evaluate their performance across three distributional scenarios:
	\setlength{\leftmargini}{33pt}
	\begin{itemize}
		\item Gamma-distributed observations: $2Ga(1,2)-1$,
		\item Uniformly-distributed observations: $U(-\sqrt 3, \sqrt 3)$,
		\item Gaussian-distributed observations: $N(0,1)$.
	\end{itemize}
	
	For each distribution, we consider three covariance structures:
	\begin{itemize}
		\item Diagonal matrices: $\bm{\Sigma}=\text{diag}(5,4.5,4,1,\dots,1)$,
		\item Block-diagonal matrices: $\bm{\Sigma}=\mathbf{Q}_1'\text{diag}(5,4.5,4,1,\dots,1)\mathbf{Q}_1$,
		\item General positive-definite matrices: $\bm{\Sigma}=\mathbf{Q}_2'\text{diag}(5,4.5,4,1,\dots,1)\mathbf{Q}_2$.
	\end{itemize}
	Let $p$ be the dimension of the sample covariance matrix $\mathbf{S}$. The dimensional ratio $\gamma_n=p/n$ varies in $\{0.1, 0.2, 0.3\}$ to assess performance under different growth regimes.

	\begin{table}
		\centering
		\caption{Estimation accuracy under $2Ga(1,2)-1$ of diagonal covariance matrix (All values in percentages)}
		\label{sphericcase1}
		\begin{tabular}{@{}c*{7}{@{\hspace{0.5cm}}c}@{}}
			\hline
			$(p,n)$ &J\&B-E&Y\&J-E&J\&B&Y\&J & P\&Y& BCF(AIC)&K\&N\\[6pt]
			(5,50) & 89.10 & \textbf{89.50} & 80.60 & \textbf{89.50} & 47.60 & 59.90 & 88.50 \\
			(10,100) & \textbf{89.00} & 87.60 & 78.00 & 87.60 & 66.00 & 39.10 & 88.60 \\
			(20,200) & \textbf{84.70} & 82.30 & 68.80 & 82.30 & 77.80 & 38.30 & 83.20 \\
			(40,400) & \textbf{90.10} & 84.30 & 68.80 & 84.30 & 85.20 & 51.30 & 82.80 \\
			(10,50) & 82.10 & \textbf{83.70} & 77.50 & \textbf{83.70} & 31.70 & 37.10 & 83.10 \\
			(20,100) & \textbf{86.30} & 84.00 & 72.70 & 84.00 & 61.30 & 43.80 & 84.00 \\
			(40,200) & \textbf{87.30} & 83.70 & 67.30 & 83.70 & 77.70 & 53.30 & 81.80 \\
			(80,400) & \textbf{95.00} & 89.70 & 66.90 & 89.70 & 84.40 & 80.70 & 83.90 \\
			(15,50) & 75.60 & \textbf{76.20} & 70.50 & \textbf{76.20} & 26.70 & 39.50 & 76.00 \\
			(30,100) & \textbf{86.80} & 86.50 & 73.10 & 86.50 & 56.20 & 51.90 & 84.50 \\
			(60,200) & \textbf{88.50} & 86.10 & 66.90 & 86.10 & 74.90 & 70.10 & 80.30 \\
			(120,400) & \textbf{97.50} & 94.40 & 70.30 & 94.40 & 83.20 & 94.90 & 85.80 \\
			\hline
		\end{tabular}
	\end{table}

	Tables \ref{sphericcase1} present numerical comparisons between our Edgeworth-corrected approaches (J\&B-E, Y\&J-E) and existing methods proposed by \cite{jiang2021generalized} (J\&B), \cite{passemier2012determining} (P\&Y),  \cite{bai2018consistency} (BCF(AIC)), \cite{kritchman2008determining} {(K\&N)} and \cite{yang2019edgeworth} (Y\&J). For each configuration, we conduct 1000 replications, and the reported accuracy is the percentage of replications in which the estimated number of spikes equals the true value $r=3$.   Complete simulation results for all cases are provided in Supplementary Material. 
	
	   Table~\ref{sphericcase1} and the corresponding supplementary tables show that J\&B-E outperforms J\&B in the Gamma setting. Furthermore, these accuracy gains generally increase as $n$ and $p$ grow at a fixed ratio. This indicates that the first order correction can effectively reduce the finite sample bias of the Gaussian approximation. Furthermore, we note that the suboptimal performance at $n=50$ is mainly due to the inaccurate estimation of the spike values. If we substitute the true values, J\&B-E performs the best among all methods evaluated in our study.
		
		Following Section \ref{secYang}, Y\&J and Y\&J-E use the same spike number estimate $ \widehat{r}=\#\{k:\widehat{\rho}_k>\theta_n\} $. The Gaussian and Edgeworth approximations are used only for post selection interval inference and do not modify this count. Consequently, the two procedures have identical selection accuracies.

		Finally, our simulation studies reveal distinct performance patterns across different distributional assumptions and dimensions. For Gamma distributed data, the J\&B-E method demonstrates superior performance compared to all competing approaches. For uniform and Gaussian data, our method yields no specific advantages, yet its performance remains highly competitive.

\section{Preliminary results and proof strategy}\label{limitingdistribution}
In this section, we first detail the proof strategy for the single-spike scenario ($r=1$); the generalization to multiple spikes ($r>1$) can be achieved through analogous analytical techniques. For notational simplicity, throughout this section we write
\[
l=l_1,\qquad
\rho_n=\rho_{n1},\qquad
\sigma_n=\sigma_{n1},\qquad
\widetilde\sigma_n=\widetilde\sigma_{n1},\qquad
g_n=g_{1n}.
\]
We develop the technical foundations in Sections \ref{yi}--\ref{san}, and Section \ref{strategy} outlines the proof of Theorem \ref{main}.

\subsection{Approximating the distribution of $R_n$}\label{yi}
Let $\hat{\rho}$ denote the largest eigenvalue in the single-spike case. Since the population spike is fixed and supercritical, $\rho_n-(1+\sqrt{\gamma_n})^2$ is bounded away from zero. 
 ByLemma~\ref{qtruncation} of the Supplementary Material, there exists
	an event $\mathcal E_n$ with
	$\mathbb P(\mathcal E_n^c)=o(n^{-1/2})$ on which both
	$\hat\rho$ and $\rho_n$ exceed
	$\lambda_{\max}(n^{-1}\V_2'\Z'\Z\V_2)$
	by a fixed positive amount.
	All calculations below are understood on $\mathcal E_n$.
	The block determinant formula therefore gives
\begin{align*}
	0&=|\hat{\rho} \I-\mathbf{S}|=\bigg|\hat{\rho}\I-n^{-1}\begin{pmatrix}l\V_1'\Z'\Z\V_1,& \sqrt l \V_1'\Z'\Z\V_2\\
		\sqrt l \V_2'\Z'\Z\V_1,&\V_2'\Z'\Z\V_2\end{pmatrix}\bigg|\\
	&=|\hat{\rho}\I_p-n^{-1}\V_2'\Z'\Z\V_2||\hat{\rho} -n^{-1}l\V_1'\Z'\Z\V_1-n^{-2} l \V_1'\Z'\Z{\V_2}(\hat{\rho} \I_{p}-n^{-1}{\V_2'}\Z'\Z{\V_2})^{-1}{\V_2'}\Z'\Z\V_1|.
\end{align*}
The first determinant is nonzero on $\mathcal E_n$, so the characteristic equation simplifies to
\begin{align}\label{determinant}
	0&=\Big|\frac{\hat{\rho}}{l}-n^{-1}\V_1'\Z'\Z\V_1-n^{-2}  \V_1'\Z'\Z{\V_2}(\hat{\rho} \I_{p}-n^{-1}{\V_2'}\Z'\Z{\V_2})^{-1}{\V_2'}\Z'\Z\V_1\Big|\\\nonumber
	&=\Big|\frac{\hat{\rho}}{l}-n^{-1}\V_1'\Z'\Big(\I+n^{-1}\Z{\V_2}(\hat{\rho}\I_{p}-n^{-1}{\V_2'}\Z'\Z{\V_2})^{-1}{\V_2'}\Z'\Big)\Z\V_1\Big|\\\nonumber
	&=\Big|\frac{\hat{\rho}}{l}-  \hat{\rho}n^{-1}\V_1'\Z'(\hat{\rho}\I_{n}-n^{-1}\Z{\V_2\V_2'}\Z')^{-1} \Z\V_1\Big|  .
\end{align}

Define
\[\Omega(\lambda,\Z)=\frac{\lambda}{\sqrt n}\Big(\mathrm{tr}\big[(\lambda\I_{n}-n^{-1}\Z{\V_2\V_2'}\Z')^{-1}\big]-\V_1'\Z'(\lambda \I_{n}-n^{-1}\Z{\V_2\V_2'}\Z')^{-1} \Z\V_1\Big).\]
Let $F_{\gamma_n}$ denote the Marchenko--Pastur law associated with the
$p\times p$ sample covariance matrix, and let
\[
\mathrm F_{\gamma_n}
:=
(1-\gamma_n)\delta_0+\gamma_nF_{\gamma_n}
\]
denote its companion law, associated with the $n\times n$ matrix
$n^{-1}\mathbf B\mathbf B'$. For a test function $f$, write
\[
\mathrm F_{\gamma_n}(f)
:=
\int f(\lambda)\,\mathrm F_{\gamma_n}(d\lambda).
\]

Let the entries of the $n\times (p+1)$ matrix $\Y=(\mathbf{y}_1,\dots,\mathbf{y}_n)'$ be an independent copy of $\Z$. Consider the eigendecomposition $n^{-1}\Y{\V_2\V_2'}\Y'=\U\bm{\Lambda}\U'$, where $\U$ is an $n\times n$ orthogonal matrix and $\bm{\Lambda}$ is the diagonal matrix containing all the eigenvalues {$\lambda_1,\dots, \lambda_n$} of $n^{-1}\Y{\V_2\V_2'}\Y'$. Let $\bm{\omega}=\U'\Z\V_1=(\omega_1,\dots, \omega_n)'$. Similarly, let $n^{-1}\Z{\V_2\V_2'}\Z'=\U_\omega\bm{\Lambda}_\omega\U_\omega'$ be the eigendecomposition for the non-Gaussian case, where $\U_\omega$ is an $n\times n$ orthogonal matrix and $\bm{\Lambda}_\omega$ is the diagonal matrix containing all the eigenvalues $\widetilde\lambda_1, \dots, \widetilde\lambda_n$ of $n^{-1}\Z{\V_2\V_2'}\Z'$. Let $\widetilde{\bm{\omega}}=\U_\omega'\Z\V_1=(\widetilde \omega_1,\dots,\widetilde \omega_n)'$ and $g_n:=g_n(\lambda)=(\rho_n-\lambda)^{-1}$. Define 
\begin{align*}
	&\widetilde S_n(f)=n^{-1/2}\Big(\V_1'\Z'\U_\omega f(\bm{\Lambda}_\omega)\U_\omega'\Z\V_1-\tr f(\bm{\Lambda}_{\omega} )\Big), ~{\mathrm{F}}_{\gamma_n}(f)=\int f( \lambda )\mathrm{F}_{\gamma_n}(d \lambda ),\\
	&S_n(f)=n^{-1/2}\Big(\V_1'\Z'\U f(\bm{\Lambda})\U'\Z\V_1-\tr f(\bm{\Lambda})\Big),~~ h_n(\lambda) = \frac{\mathrm{F}_{\gamma_n}(\lambda g_n^3)}{\mathrm{F}_{\gamma_n}(g_n^2)} -\lambda g_n(\lambda).
\end{align*}
Besides, for $\mathbf{B}\in\{\Z,\Y\}$, we define
$ \widetilde{G}_{n,\mathbf{B}}(f) = \operatorname{tr} f\left(n^{-1}\mathbf{B}\V_2\V_2'\mathbf{B}'\right) - n\mathrm{F}_{\gamma_n}(f)$.

\begin{theorem}\label{convergence}
	Under Assumptions $(a)-(d)$, the distribution of $R_n$ can be approximated as follows:
	\[\sup_x\bigg|\mathbb{P}(R_n\le x)-\mathbb{P}\bigg(\frac{-\Omega(\rho_n, \Z)+\rho_n n^{-1/2}\widetilde G_{n,\Z}(g_n)}{[\rho_n\mathrm{F}_{\gamma_n}(g_n^2)-n^{-1/2}\widetilde S_n(g_nh_n)]\widetilde\sigma_n}\le x\bigg)\bigg|=o(n^{-1/2}).\]
\end{theorem}

\begin{remark}
	The terms $\widetilde S_n(f)$ and $\widetilde G_n(f)$ are of order $O_p(1)$. A detailed justification is provided in the supplementary material.
\end{remark}

\subsection{Approximation under partial generalized four moment theorem (PG4MT)}\label{CLT}

Define $\Omega(\rho_n,\Z)$ and its counterpart $\Omega(\rho_n,\Y)$:
\begin{align*}
	\Omega(\rho_n,\Z)&=\frac{\rho_n}{\sqrt n}\Big(\mathrm{tr}\big[(\rho_n\I_{n}-n^{-1}\Z{\V_2\V_2'}\Z')^{-1}\big]-\V_1'\Z'(\rho_n \I_{n}-n^{-1}\Z{\V_2\V_2'}\Z')^{-1} \Z\V_1\Big),\\
	\Omega(\rho_n,\Y )&=\frac{\rho_n}{\sqrt n}\Big(\mathrm{tr}\big[(\rho_n\I_{n}-n^{-1}\Y{\V_2\V_2'}\Y')^{-1}\big]-\V_1'\Z'(\rho_n \I_{n}-n^{-1}\Y{\V_2\V_2'}\Y')^{-1} \Z\V_1\Big).
\end{align*}
Through a series of rigorous calculations, we derive
\begin{align*}
	\Omega(\rho_n,\Z)&=\frac{\rho_n}{\sqrt n}\Big(\mathrm{tr}\big[(\rho_n\I_{n}-n^{-1}\Z{\V_2\V_2'}\Z')^{-1}\big]-\V_1'\Z'(\rho_n \I_{n}-n^{-1}\Z{\V_2\V_2'}\Z')^{-1} \Z\V_1\Big)\\\nonumber
	&=\frac{\rho_n}{\sqrt n}\Big(\mathrm{tr}\big[\U_\omega(\rho_n\I-\bm{\Lambda}_\omega)^{-1}\U_\omega'\big]-\V_1'\Z'\U_\omega(\rho_n\I-\bm{\Lambda}_\omega )^{-1}\U_\omega'\Z\V_1\Big)\\\nonumber
	&=\frac{\rho_n}{\sqrt n}\Big(\mathrm{tr}\big[(\rho_n\I-\bm{\Lambda}_\omega)^{-1}\big]-\widetilde{\bm{\omega}}'(\rho_n\I-\bm{\Lambda}_\omega )^{-1}\widetilde{\bm{\omega}}\Big)=-\rho_n\widetilde S_n(g_n).
\end{align*}
Since $\rho_n\mathrm F_{\gamma_n}(g_n^2)$ is bounded away from
zero, the resolvent fluctuation bound in Lemma~\ref{resolventfluctuations} of the
Supplementary Material and the corresponding conditional
moment bound on the conditional quadratic form imply
\begin{align*}
	\mathbb P\big(
	\rho_n\mathrm F_{\gamma_n}(g_n^2)
	-n^{-1/2}\widetilde S_n(g_nh_n)\le 0
	\big)
	&=o(n^{-1/2}),\\
	\mathbb P\big(
	\rho_n\mathrm F_{\gamma_n}(g_n^2)
	-n^{-1/2}S_n(g_nh_n)\le 0
	\big)
	&=o(n^{-1/2}).
\end{align*}
Indeed, the first estimate follows from the moderate deviation
bound for $\widetilde S_n(g_nh_n)$, while the second follows from
the conditional quadratic-form moment estimate and Markov's
inequality. Consequently, cross-multiplication by either random
denominator changes the corresponding probability by at most
$o(n^{-1/2})$, uniformly in $x\in\mathbb R$.
Therefore, based on the relationship between $\Omega(\rho_n, \Z)$ and $\widetilde S_n(g_n)$, we derive
\begin{align*}
	&\quad\mathbb{P}\bigg(\frac{-\Omega(\rho_n, \Z)+\rho_n n^{-1/2}\widetilde G_{n,\Z}(g_n)}{[\rho_n\mathrm{F}_{\gamma_n}(g_n^2)-n^{-1/2}\widetilde S_n(g_nh_n)]\widetilde\sigma_n}\le x\bigg)\\
	&=\mathbb{P}\Big(\widetilde S_n(g_n)+n^{-1/2}\widetilde S_n(g_nh_n)\rho_n^{-1}\widetilde\sigma_nx\le\mathrm{F}_{\gamma_n}(g_n^2)\widetilde\sigma_nx-n^{-1/2}\widetilde G_{n,\Z}(g_n)\Big)+o(n^{-1/2}).
\end{align*}

The work of \cite{jiang2021partial} established that the limiting distributions of $\Omega(\lambda, \Z)$ and $\Omega(\lambda,\Y)$ agree at the CLT scale.  However, their distribution functions may differ by a term of order
	$n^{-1/2}$. Theorem~\ref{G4MT} below gives the explicit signed
	correction, with a uniform $o(n^{-1/2})$ remainder, where
	\begin{align*}
		&T_{n,\Z}(x)=\widetilde S_n(g_n)+n^{-1/2}\widetilde S_n(g_nh_n)\rho_n^{-1}\widetilde\sigma_nx+n^{-1/2}\widetilde G_{n,\Z}(g_n),\\
		&T_{n,\Y}(x)=S_n(g_n)+n^{-1/2}S_n(g_nh_n)\rho_n^{-1}\widetilde\sigma_nx+n^{-1/2}\widetilde G_{n,\Y}(g_n).
	\end{align*}

	For $\B\in\{\Z,\Y\}$, we define $a_n=\mathrm{F}_{\gamma_n}(g_n^2)\widetilde{\sigma}_n=\sqrt{\kappa_{21}}$ and $H_{n,\B}(x):=\mathbb{P}\{T_{n,\B}(x)\le a_n x\}$. Based on these definitions, we establish the following theorem.

\begin{theorem}[Corrected PG4MT for an arbitrary spike direction]
	\label{G4MT}
	Suppose that $\Z$ is a random matrix with i.i.d. entries and satisfying
		Assumptions $(a)$--$(d)$, and that $\Y$ is an independent copy of $\Z$. Then, for every deterministic unit spike direction
		$\V_1$,
		\begin{align}
			\sup_{x}\Bigg|
			H_{n,\Z}(x)-H_{n,\Y}(x)
			-\frac{\phi(x)}{\sqrt n}
			\left\{
			\frac{\mathfrak k_{1}}{6\kappa_{21}^{3/2}}(1-x^2)
			-\frac{\mathfrak b_{1}}{2\kappa_{21}}x
			-\frac{\mathfrak m_{1}}{\sqrt{\kappa_{21}}}
			\right\}
			\Bigg|=o(n^{-1/2}),
			\label{eq:final-cdf-comparison-corrected}
		\end{align}
		where
		\begin{align*}
			&\mathfrak m_{1}=\frac{\beta (1-\mathcal{V}_{1}^{(4)})}{l(l-1)}
			+\frac{2\Gamma^2n^{-1}(\mathcal{V}_{1}^{(1)})^2}{l(l-1)^2},
			\mathfrak b_{1}=\frac{8\Gamma^2n^{-1/2}\mathcal{V}_{1}^{(3)}\mathcal{V}_{1}^{(1)}}{l^2(l-1)},
			\mathfrak k_{1}=\frac{6\beta^2(\mathcal{V}_{1}^{(6)}-(\mathcal{V}_{1}^{(4)})^2)}{l^3(l-1)}.
	\end{align*}
\end{theorem}

\subsection{Higher-order approximation via Edgeworth correction}\label{san}

This subsection establishes the first-order Edgeworth expansion for the
ratio statistic involving the independent copy $\Y$. Additionally, this result is of independent interest, distinct from the developments in the preceding sections. Let $\tau_n=\rho_n^{-1}\widetilde\sigma_n x$ and $ c_n=1+n^{-1/2} \tau_n\mathrm{F}_{\gamma_n}(\lambda g_n^3)/\mathrm{F}_{\gamma_n}(g_n^2)$. Define:
\begin{align*}
	\Omega_s(&\Z\V_1,\Y)= S_n(g_n)+n^{-1/2} S_n(g_nh_n)\rho_n^{-1}\widetilde\sigma_nx= S_n\Big((1+n^{-1/2}\rho_n^{-1}\widetilde\sigma_nxh_n)g_n\Big)\\
	&=n^{-1/2}\sum_{i=1}^ng_n( \lambda _i)(\omega_i^2-1)+n^{-1}\sum_{i=1}^n\tau_nh_n( \lambda _i)g_n( \lambda _i)(\omega_i^2-1)\\
	&=n^{-1/2}c_n\sum_{i=1}^ng_n( \lambda _i)(\omega_i^2-1)-n^{-1}\sum_{i=1}^n\tau_n\lambda_i g^2_n( \lambda _i)(\omega_i^2-1),
\end{align*}

We first write $\Omega_s(\Z\V_1,\Y)$ as a conditional quadratic form and then compare it with an independent sum. 
We define $\widetilde{\mathbf{Z}}_1=\mathbf{Z}\mathbf{V}_1=(\widetilde{Z}_{11},\ldots,\widetilde{Z}_{n1})'$, and let
\begin{align*}
	\A &= (\rho_n\I-n^{-1}\Y\V_2\V_2'\Y')^{-1}=(a_{ij})_{n\times n},\\ 
	\M &= (n^{-1}\Y\V_2\V_2'\Y')(\rho_n\I-n^{-1}\Y\V_2\V_2'\Y')^{-2}=(m_{ij})_{n\times n}.
\end{align*}
Therefore, we can reformulate $\Omega_{s}(\Z\V_1,\Y)$ as a sum of random variables:
\begin{align*}
	&\quad\Omega_{s}(\Z\V_1,\Y)=-\frac{c_n}{\sqrt n}\Big(\mathrm{tr}(\A)-\V_1'\Z'\A\Z\V_1\Big)+\frac{\tau_n}{n}\Big(\mathrm{tr}(\M)-\V_1'\Z'\M\Z\V_1\Big)\\
	&=\frac{c_n}{\sqrt n}\Big(\sum_{i=1}^n a_{ii}(\widetilde Z_{i1}^2-1)+\sum_{i\neq j} a_{ij}\widetilde Z_{i1}\widetilde Z_{j1}\Big)-\frac{\tau_n}{n}\Big(\sum_{j=1}^n m_{jj}(\widetilde Z_{j1}^2-1)+\sum_{i\neq j} m_{ij}\widetilde Z_{i1}\widetilde Z_{j1}\Big).
\end{align*}

	Defining $
	\mathbf{C}_n(x)=c_n(x)\A-\frac{\tau_n(x)}{\sqrt{n}}\M
	=(c_{ij}(x))_{1\le i,j\le n}$,
	we express $\Omega_s(x)$ conditional on $\mathcal{F}_Y:=\sigma(\Y)$ as
	\begin{equation}\label{quadraticform}
		\Omega_s(x)
		=\frac{1}{\sqrt{n}}\left\{\widetilde{\mathbf{Z}}_1'\mathbf{C}_n(x)
		\widetilde{\mathbf{Z}}_1-\operatorname{tr}(\mathbf{C}_n(x))\right\}
		=\frac{1}{\sqrt{n}}\{D_n(x)+Q_n(x)\},
	\end{equation}
	where
	\[
	D_n(x)=\sum_{i=1}^n c_{ii}(x)(\widetilde{Z}_{i1}^2-1),
	\qquad
	Q_n(x)=2\sum_{1\le i<j\le n}c_{ij}(x)
	\widetilde{Z}_{i1}\widetilde{Z}_{j1}.
	\]
	Let $
	B_{2n}(x)=\operatorname{Var}\{D_n(x)+Q_n(x)\mid\mathcal F_Y\}$,
	and
	$B_{3n}(x)=\operatorname{cum}_3\{D_n(x)+Q_n(x)\mid\mathcal F_Y\}$.
	Define
	\[
	b_n(x)
	:=
	\Big[
	\frac4n
	\sum_{1\le i<j\le n}c_{ij}(x)^2
	\Big]^{1/2}, \qquad
	\gamma_n^*(x)
	=
	\frac{
		B_{3n}(x)
		-
		\widetilde\kappa_3
		\sum_{i=1}^n c_{ii}(x)^3
	}{
		n b_n(x)^3
	}.
	\]
	In the Supplementary Material, we construct
	$ \eta_{n1}(x),\ldots,\eta_{nn}(x)$.
	Given $\mathcal{F}_Y$, these variables are mutually independent and also independent of $\widetilde{Z}_{11},\ldots,\widetilde{Z}_{n1}$. Furthermore, these variables have uniformly bounded sixth
	moments and a common nondegenerate Gaussian component, such that
	\[
	\mathbb E\bigl(\eta_{ni}(x)\mid\mathcal F_Y\bigr)=0,
	\qquad
	\mathbb E\bigl(\eta_{ni}(x)^2\mid\mathcal F_Y\bigr)=1,
	\qquad
	\mathbb E\bigl(\eta_{ni}(x)^3\mid\mathcal F_Y\bigr)
	=
	\gamma_n^*(x).
	\]
	Set $Q_n^*(x)=b_n(x)\sum_{i=1}^n\eta_{ni}(x)$.
	By construction, $D_n(x)+Q_n(x)$ and $D_n(x)+Q_n^*(x)$ have exactly the
	same first three conditional cumulants. The subsequent analysis relies primarily on the following equation: 
	\begin{align}\label{step3}
		\mathbb E
		\sup_{|x|\le M_n}
		\sup_{z\in\mathbb R}
		\Big|
		\mathbb P\left\{
		\Omega_s(x)\le z \mid\mathcal F_Y
		\right\}-
		\mathbb P\left\{
		n^{-1/2}\big(D_n(x)+Q_n^*(x)\big)\le z \mid\mathcal F_Y
		\right\}
		\Big|
		=
		o(n^{-1/2}).
	\end{align}
	A detailed proof of this equation is provided in Section~\ref{sec:eq53-routeB}.

\begin{theorem}[Limiting distribution of $\Omega_s(\Z\V_1,\Y)$]\label{limiting}
	Assume that $\Y$ is an independent copy of $\Z$ and that
	Assumptions (a)--(d) hold. Then
	\begin{align*}
		&\sup_x\bigg|\mathbb{P}\bigg(\frac{-\Omega(\rho_n, \Y)+\rho_n n^{-1/2}\widetilde G_{n,\Y}(g_n)}{[\rho_n\mathrm{F}_{\gamma_n}(g_n^2)-n^{-1/2}S_n(g_nh_n)]\widetilde\sigma_n}\le x\bigg)\\
		&\qquad\qquad-\Phi(x)-n^{-1/2}\bigg\{
		\frac{\kappa_{31}}{6\kappa_{21}^{3/2}}(1-x^2)
		-\frac{\mu(g_n)}{\sqrt{\kappa_{21}}}
		-\mathfrak c_1 x^2
		\bigg\}\phi(x)\bigg|=o(n^{-1/2}).
	\end{align*}
\end{theorem}

\subsection{{Proof  strategy}}\label{strategy}
In \cite{yang2018edgeworth},   the observations $x_j$ follow a Gaussian distribution, an assumption instrumental in the derivation of independent and identically distributed variables $\z=(z_i)=\U'\Z_1$ (independent of the eigenvalues $\bm{\Lambda}$), which serve as input for the conditional Edgeworth expansion. Our proof strategy involves adapting the approach from \cite{yang2018edgeworth} for the Gaussian observations to the non-Gaussian observations. However, this adaptation is far from being straightforward. A key requirement is the generation of independent random variables to serve as inputs for the conditional Edgeworth expansion when dealing with non-Gaussian observations. This task forms the primary contribution of Theorem \ref{convergence} and Theorem \ref{G4MT}.

The derivation of the asymptotic distribution of eigenvalues typically begins with the characteristic equation, a methodology we also adopt. However, unlike \cite{bai2008central} and \cite{jiang2021generalized}, our objective is to achieve a more precise asymptotic distribution, which requires a more thorough derivation and the preservation of terms that they might have considered negligible. Through a series of rigorous derivations, we establish that the Kolmogorov distance between $R_n$ and 
\[\frac{1}{\big[\rho_n\mathrm{F}_{\gamma_n}(g_n^2)-n^{-1/2}\widetilde S_n(g_nh_n)\big]\widetilde\sigma_n}\Big(-\Omega(\rho_n, \Z)+\frac{\rho_n}{n^{1/2}}\widetilde G_{n,\Z}(g_n)\Big)\]
is of order $o(n^{-1/2})$. This constitutes the principal finding of our Theorem \ref{convergence}.

However, simply obtaining this statistic is insufficient because it involves the term $\U_\omega'\Z\V_1$, where $\U_\omega$ is the matrix of mutually orthogonal unit eigenvectors derived from $n^{-1}\Z{\V_2\V_2'}\Z'$. Under Gaussian assumptions, $\U_{\omega}$ is a Haar-distributed matrix independent of $\Z\V_1$, whereas this independence generally fails in the non-Gaussian case.  Theorem~\ref{G4MT} replaces the dependent background with the independent copy $\Y$ and identifies the non-vanishing change of order $n^{-1/2}$ in the distribution function.  After subtracting the explicit corrections for the mean, variance and third cumulant, the remaining error in Kolmogorov distance is $o(n^{-1/2})$.

Our objective now is to derive the Edgeworth-corrected limiting distribution of the statistic:
\[\frac{1}{\big[\rho_n\mathrm{F}_{\gamma_n}(g_n^2)-n^{-1/2} S_n(g_nh_n)\big]\widetilde\sigma_n}\Big(-\Omega(\rho_n, \Y)+\frac{\rho_n}{n^{1/2}}\widetilde G_{n,\Y}(g_n)\Big).\]
Conditionally on $\U,\bm{\Lambda}$, this statistic is compared with the
independent sum constructed above.  The Supplementary Material proves the
needed conditional expansion directly: the first three cumulants are matched,
the fixed Gaussian component supplies an integrable characteristic-function
envelope, and Fourier inversion gives a uniform $o(n^{-1/2})$ remainder.
Combining this conditional expansion with the eigenvalue reduction and the
corrected dependency replacement proves Theorem~\ref{main}.

\section{Conclusion and discussion}\label{future}

This paper establishes the first-order Edgeworth expansions for the distribution functions of spiked eigenvalues of sample covariance matrices under non-Gaussian observations. Our results resolve an open problem posed by \cite{yang2018edgeworth} and refine the asymptotic distributions derived in \cite{jiang2021generalized}.
Furthermore, we employ Edgeworth-corrected formulations to enhance the accuracy of both confidence interval estimation for spike values and the determination of the number of spikes. This higher-order approximation yields improved finite-sample performance compared to conventional Gaussian approximations.

A natural extension of this work would be to derive a second-order Edgeworth expansion. However, this remains a challenging open problem, as it would likely require a first-order approximation for the associated linear spectral statistic under non-Gaussian assumptions, which is not currently available.

	\section{Proofs of main results}\label{proof}
Lemma~\ref{qtruncation} establishes that truncating, centering, and standardizing the original arrays changes the distribution function comparisons by a negligible error of $o(n^{-1/2})$. Consequently, throughout the proofs in this appendix, $\Z$ and $\Y$ represent these transformed arrays. For distributional comparisons, Lemma~\ref{goodevent} allows us to modify the statistics outside the events
specified below, whose complements have probability
$o(n^{-1/2})$.

		\subsection{Proof of Theorem~\ref{multi}}	
	Define $\mathbf{l}=(l_1,\dots,l_r)$, and let $\mathbf{l}_k$ be the vector $\mathbf{l}$ with its $k$-th element removed. Let $l_{r+1}=\cdots=l_{p+r}=1$. Without loss of generality, we assume that the population covariance matrix has the following form:
	\begin{align*}
		\bm{\Sigma}=(\mathbf{v}_k,\V_k)\text{diag}(l_k,\mathbf{l}_{k},l_{r+1},\dots,l_{p+r})(\mathbf{v}_k,\V_k)',
	\end{align*}
	where $\V=(\mathbf{v}_k,\V_k)$ is a ${(p+r)\times (p+r)}$ orthogonal matrix and $\V_k\in\mathbb R^{(p+r)\times(p+r-1)}$ is the semi-orthogonal matrix obtained by removing $\mathbf v_k$ from $\V$. Define $\X=\Z\bm{\Sigma}^{1/2}$ and
	\[
	\D_{k}=
	\operatorname{diag}
	(\mathbf{l}_{k},
	l_{r+1},\ldots,l_{p+r})
	\in\mathbb R^{(p+r-1)\times(p+r-1)}.
	\]
	Let
	\[
	q_{n,k}=p+r-1,\qquad y_{n,k}=\frac{q_{n,k}}{n},\qquad
	H_{n,k}^{(-)}=\frac{p\delta_1+\sum_{j\ne k}\delta_{l_j}}{p+r-1}.
	\]
	Through the eigendecomposition
	\[n^{-1}\Z\V_{k}\D_k\V_{k}'\Z'=\U_\omega\bm{\Lambda}_\omega\U_\omega'=\U_\omega\text{diag}(\widetilde\lambda_1,\dots,\widetilde\lambda_n)\U_\omega',\qquad \widetilde\lambda_1\ge \dots\ge\widetilde\lambda_n.\]
	The remaining $r-1$ population spikes generate sample outliers at
	their corresponding deterministic locations. By
	Lemma~\ref{qtruncation}, there is an event $\mathcal G_n$ such that
	\[
	\mathbb P(\mathcal G_n^c)=o(n^{-1/2}),
	\]
	and, on $\mathcal G_n$, the bulk eigenvalues and the remaining sample
	spikes lie in disjoint fixed neighborhoods of their deterministic
	locations. The deterministic companion law for the $n$ eigenvalues
	of $n^{-1}\Z\V_k\D_k\V_k'\Z'$ is
	\[
	\mathrm F_{n,k}
	:=
	\mathrm F_{y_{n,k},H_{n,k}^{(-)}}
	=
	(1-y_{n,k})\delta_0
	+
	y_{n,k}F_{y_{n,k},H_{n,k}^{(-)}}.
	\]
	
	We next identify the deterministic contribution of the remaining
	spikes and their interaction with the $k$th spike direction. Let
	\[
	\bm\omega=\U_\omega'\Z\mathbf v_k,
	~~
	\mathrm F_n(f)
	=
	\frac1n\operatorname{tr}f(\bm\Lambda_\omega),
	~~
	g_{kn}(\lambda)=(\rho_{nk}-\lambda)^{-1},
	~~
	h_{kn}
	=
	\frac{\mathrm F_{\gamma_n}(\lambda g_{kn}^3)}
	{\mathrm F_{\gamma_n}(g_{kn}^2)}
	-\lambda g_{kn}.
	\]
	For every $f$ analytic on a neighborhood of the relevant spectra,
	define
	\[
	S_n(f)
	=
	\frac1{\sqrt n}
	\left\{
	\bm\omega'f(\bm\Lambda_\omega)\bm\omega
	-\operatorname{tr}f(\bm\Lambda_\omega)
	\right\},
	\qquad
	G_n(f)
	=
	n\left\{
	\mathrm F_n(f)-\mathrm F_{n,k}(f)
	\right\}.
	\]
	Then
	\[
	\frac1n\bm\omega'f(\bm\Lambda_\omega)\bm\omega
	=
	\mathrm F_{n,k}(f)
	+\frac1{\sqrt n}S_n(f)
	+\frac1nG_n(f).
	\]
	
	Let $\underline m_{\gamma_n}$ and $\underline m_{n,k}$ denote
	the Stieltjes transforms of $\mathrm F_{\gamma_n}$ and
	$\mathrm F_{n,k}$, respectively:
	\[
	\underline m_{\gamma_n}(z)
	=\int\frac{1}{\lambda-z}\,\mathrm F_{\gamma_n}(d\lambda),
	\qquad
	\underline m_{n,k}(z)
	=\int\frac{1}{\lambda-z}\,\mathrm F_{n,k}(d\lambda).
	\]
	Write
	\[
	z_{0,n}(m)=-\frac1m+\frac{\gamma_n}{1+m},\qquad
	Q_k(m)=\sum_{j\ne k}\frac{l_j}{1+l_jm}.
	\]
	The inverse Silverstein equation for $\mathrm F_{n,k}$ is exactly
	\[
	z=z_{0,n}(m)+\frac1nQ_k(m)
	=-\frac1m+\frac{\gamma_n}{1+m}
	+\frac1n\sum_{j\ne k}\frac{l_j}{1+l_jm},
	\]
	where $m=\underline m_{n,k}(z)$.	Consequently, uniformly on a fixed contour away from the poles,
	\[
	\underline m_{n,k}(z)
	=
	\underline m_{\gamma_n}(z)
	-
	\frac1n
	\frac{
		Q_k(\underline m_{\gamma_n}(z))
	}{
		z_{0,n}'(\underline m_{\gamma_n}(z))
	}
	+
	O(n^{-2}).
	\]
	
	For a counterclockwise contour $\Gamma_z$ enclosing the relevant spectrum and avoiding the singularities of $f$, we define
	\[
	A(f)
	=
	\frac1{2\pi i}
	\oint_{\Gamma_z}
	f(z)
	\frac{
		Q_k(\underline m_{\gamma_n}(z))
	}{
		z_{0,n}'(\underline m_{\gamma_n}(z))
	}\,dz.
	\]
	The preceding expansion gives
	\[
	\mathrm F_{n,k}(f)
	=
	\mathrm F_{\gamma_n}(f)
	+\frac1nA(f)
	+O(n^{-2}),
	\]
	and hence
	\[
	\frac1n\bm\omega'f(\bm\Lambda_\omega)\bm\omega
	=
	\mathrm F_{\gamma_n}(f)
	+\frac1{\sqrt n}S_n(f)
	+\frac1n\{G_n(f)+A(f)\}
	+O(n^{-2}).
	\]
	
	 Notice
		\[
		\V_k\D_k\V_k'
		=\I-\mathbf v_k\mathbf v_k'
		+\sum_{j\ne k}(l_j-1)\mathbf v_j\mathbf v_j'.
		\]
		For $\mathbf B\in\{\Z,\Y\}$, define $\mathbf A_{k,\mathbf B}^{(0)}
		=
		\frac1n
		\mathbf B
		(\mathbf I-\mathbf v_k\mathbf v_k')
		\mathbf B'$ and $\mathbf A_{k,\mathbf B}:=\frac1n\mathbf B\V_k\D_k\V_k'\mathbf B'$. Hence,
		\[
		\begin{aligned}
			\mathbf A_{k,\mathbf B}=\mathbf A_{k,\mathbf B}^{(0)}+
			\sum_{j\neq k}
			(l_j-1)
			\left(\frac{\mathbf B\mathbf v_j}{\sqrt n}\right)
			\left(\frac{\mathbf B\mathbf v_j}{\sqrt n}\right)'.
		\end{aligned}
		\]
		Let $\widetilde{\Z}_k=\Z\mathbf v_k$. Define
		\begin{align}
			T_{n,k,\mathbf B}(x)
			&=
			\frac1{\sqrt n}
			\Big[
			\widetilde{\Z}_k'
			g_{kn}(\mathbf A_{k,\mathbf B})
			\widetilde{\Z}_k
			-n\mathrm F_{n,k}(g_{kn})
			+A(g_{kn})
			\Big]
			\notag\\
			&\qquad+
			\frac{\rho_{nk}^{-1}\widetilde\sigma_{nk}x}{n}
			\Big[
			\widetilde{\Z}_k'
			(g_{kn}h_{kn})(\mathbf A_{k,\mathbf B})
			\widetilde{\Z}_k
			-
			\operatorname{tr}
			(g_{kn}h_{kn})(\mathbf A_{k,\mathbf B})
			\Big].
			\label{multidef}
		\end{align}
		For comparison, define $T_{n,k,\mathbf B}^{(0)}(x)$ by the same
		formula with
		\[
		\mathbf A_{k,\mathbf B},
		\qquad
		\mathrm F_{n,k}(g_{kn}),
		\qquad
		A(g_{kn})
		\]
		replaced respectively by
		\[
		\mathbf A_{k,\mathbf B}^{(0)},
		\qquad
		\mathrm F_{\gamma_n}(g_{kn}),
		\qquad
		0.
		\]
		Since
		\[
		\mathrm F_{n,k}(g_{kn})
		=
		\mathrm F_{\gamma_n}(g_{kn})
		+\frac1nA(g_{kn})
		+O(n^{-2}),
		\]
		we have
		\[
		\frac1{\sqrt n}
		\left\{
		-n\mathrm F_{n,k}(g_{kn})
		+A(g_{kn})
		\right\}
		=
		-\sqrt n\,\mathrm F_{\gamma_n}(g_{kn})
		+O(n^{-3/2}).
		\]
		Thus, the deterministic centering terms in
		$T_{n,k,\mathbf B}(x)$ and $T^{(0)}_{n,k,\mathbf B}(x)$
		differ by $O(n^{-3/2})$. Specifically, $A(g_{kn})$ captures the deterministic finite rank adjustment to the centering.
		
		We next compare the quadratic-form terms in
		$T_{n,k,\mathbf B}(x)$ and $T_{n,k,\mathbf B}^{(0)}(x)$.
		For $j\ne k$, define
		\[
		W_{kj,\mathbf B}
		=\widetilde{\mathbf Z}_k'
		g_{kn}(\mathbf A_{k,\mathbf B}^{(0)})
		\frac{\mathbf B\mathbf v_j}{\sqrt n}.
		\]
		Then, the expression $
		\widetilde{\mathbf Z}_k'
		\bigl[
		g_{kn}(\mathbf A_{k,\mathbf B})
		-g_{kn}(\mathbf A_{k,\mathbf B}^{(0)})
		\bigr]
		\widetilde{\mathbf Z}_k$
		constitutes a quadratic form in $(W_{kj,\mathbf B})_{j\ne k}$,
		characterized by the coefficient matrix
		\[
		\Big(
		\frac{\mathbf1_{\{i=j\}}}{l_j-1}
		-\frac1n\mathbf v_i'\mathbf B'
		g_{kn}(\mathbf A_{k,\mathbf B}^{(0)})
		\mathbf B\mathbf v_j
		\Big)_{i,j\ne k}^{-1}.
		\]
		We use the bounded extensions off $\mathcal G_n$
		from the proof of Lemma~\ref{qtruncation}.
		On $\mathcal G_n$, the companion-resolvent identity gives
		\[
		\begin{aligned}
			\frac1n\mathbf v_i'\mathbf B'
			g_{kn}(\mathbf A_{k,\mathbf B}^{(0)})
			\mathbf B\mathbf v_j=
			\rho_{nk}\mathbf v_i'\mathbf V_k
			\Big(
			\rho_{nk}\mathbf I
			-\frac1n\mathbf V_k'\mathbf B'\mathbf B\mathbf V_k
			\Big)^{-1}
			\mathbf V_k'\mathbf v_j
			-\mathbf1_{\{i=j\}}.
		\end{aligned}
		\]
		The projected-resolvent argument leading to
		\eqref{eq:short-resolvent-limits}, with spike direction
		$\mathbf v_k$ and spike strength $l_k$, therefore yields
		\[
		\begin{aligned}
			\frac1n\mathbf v_i'\mathbf B'
			g_{kn}(\mathbf A_{k,\mathbf B}^{(0)})
			\mathbf B\mathbf v_j=
			\Big(\frac{\rho_{nk}}{l_k-1+\gamma_n}-1\Big)
			\mathbf1_{\{i=j\}}+o_{L^2}(1)
			=\frac{\mathbf1_{\{i=j\}}}{l_k-1}+o_{L^2}(1).
		\end{aligned}
		\]
		Here the fixed dimension difference is absorbed into the
		error, since $r$ is fixed.
		The spectral bounds and bounded extensions make these
		bilinear forms uniformly bounded, so the convergence
		also holds in $L^3$. Thus, uniformly for
		$\mathbf B\in\{\Z,\Y\}$ and $i,j\ne k$,
		\[
		\Big\|
		\frac1n\mathbf v_i'\mathbf B'
		g_{kn}(\mathbf A_{k,\mathbf B}^{(0)})
		\mathbf B\mathbf v_j
		-\frac{\mathbf1_{\{i=j\}}}{l_k-1}
		\Big\|_3\to0.
		\]
		Since the spikes are distinct, the limiting coefficient
		matrix is diagonal, with $j$th diagonal entry
		\[
		\Big(\frac1{l_j-1}-\frac1{l_k-1}\Big)^{-1}
		=\frac{(l_j-1)(l_k-1)}{l_k-l_j}.
		\]
		Moreover, the sixth-moment bound gives
		$\mathbb E|W_{kj,\mathbf B}|^3\le C$.
		Under the same spectral bounds, H\"older's inequality
		therefore controls the coefficient errors and off-diagonal
		terms by $o_{L^1}(n^{-1/2})$ after normalization.
		
		The corresponding bounds for the resolvent derivatives
		and trace terms give an $O_{L^1}(|x|/n)$ contribution
		from the $x$-dependent difference.
		Combining this with the $O(n^{-3/2})$ centering difference
		yields, uniformly for $|x|\le n^{1/4}$,
		\begin{align}
			T_{n,k,\mathbf B}(x)-T_{n,k,\mathbf B}^{(0)}(x)
			&=\frac1{\sqrt n}\sum_{j\ne k}
			\frac{(l_j-1)(l_k-1)}{l_k-l_j}
			W_{kj,\mathbf B}^2+r_{n,k,\mathbf B}(x),
			\label{eq:multi-threshold-difference-short}
		\end{align}
		where
		$\sup_{|x|\le n^{1/4}}
		\mathbb E|r_{n,k,\mathbf B}(x)|=o(n^{-1/2})$.
		
		To evaluate the contribution of
		\eqref{eq:multi-threshold-difference-short}
		to the characteristic function, we apply the martingale
		central limit theorem to linear combinations of
		\[
		\left(
		T_{n,k,\mathbf B}^{(0)}(x),
		(W_{kj,\mathbf B})_{j\ne k}
		\right),
		\]
		obtaining their joint Gaussian approximation. We next compute their means and covariances.
		Uniformly for $\mathbf B\in\{\Z,\Y\}$ and $|x|\le n^{1/4}$,
		$\mathbf B\in\{\Z,\Y\}$ and $|x|\le n^{1/4}$,
		\[
		\mathbb ET_{n,k,\mathbf B}^{(0)}(x)=o(1),
		\qquad
		\operatorname{Var}\{T_{n,k,\mathbf B}^{(0)}(x)\}
		=
		\kappa_{2k}+o(1),
		\qquad
		\mathbb EW_{kj,\mathbf B}=o(1).
		\]
		The sixth-moment quadratic-form bound verifies the Lindeberg condition, while the
		uniform third-moment bound above permits multiplication by
		$W_{kj,\mathbf B}^2$.
		
		The joint Gaussian approximation and uniform integrability of
		$W_{kj,\mathbf B}^2$ give, uniformly on compact $t$-intervals,
		\begin{align}
			&\mathbb E\!\left[
			e^{itT_{n,k,\mathbf B}^{(0)}(x)/\sqrt{\kappa_{2k}}}
			W_{kj,\mathbf B}^2\right]\notag\\
			&\quad=e^{-t^2/2}\left[
			\operatorname{Var}(W_{kj,\mathbf B})
			-\frac{t^2}{\kappa_{2k}}
			\operatorname{Cov}\{T_{n,k,\mathbf B}^{(0)}(x),W_{kj,\mathbf B}\}^2
			\right]+o(1).
			\label{multiregress}
		\end{align}
		For each fixed $M>0$, let $\varepsilon_n(M)$ be the supremum
		of the absolute error in \eqref{multiregress}
		over $|t|\le M$. Then $\varepsilon_n(M)\to0$.
		Write
		\[
		d_{kj}=\frac{(l_j-1)(l_k-1)}{l_k-l_j},\qquad
		Q_{n,k,\mathbf B}=\sum_{j\ne k}d_{kj}W_{kj,\mathbf B}^2,
		\]
		and recall that $r_{n,k,\mathbf B}(x)$ is the remainder in
		\eqref{eq:multi-threshold-difference-short}.
		Choose $L_n\to\infty$ sufficiently slowly so that,
		in addition to the conditions required in the single-spike proof,
		the following hold:
		\[
		\sqrt n L_n\sup_{k,\mathbf B,x}\mathbb E|r_{n,k,\mathbf B}(x)|\to0,
		\qquad L_n\varepsilon_n(L_n)\to0,
		\qquad L_n=o(n^{1/6}).
		\]
		Such a choice is possible because
		$\sup_{k,\mathbf B,x}\mathbb E|r_{n,k,\mathbf B}(x)|=o(n^{-1/2})$
		and the preceding approximation is uniform on compact intervals.
		
		Define
		\[
		\psi_{k,\mathbf B,x}(t)
		=\mathbb E e^{itT_{n,k,\mathbf B}(x)/\sqrt{\kappa_{2k}}},\qquad
		\psi_{k,\mathbf B,x}^{(0)}(t)
		=\mathbb E e^{itT_{n,k,\mathbf B}^{(0)}(x)/\sqrt{\kappa_{2k}}}.
		\]
		Since $r$ is fixed, the preceding third-moment bound implies
		$\mathbb E|Q_{n,k,\mathbf B}|^{3/2}\le C$.
		Using $|e^{iu}-e^{iv}|\le|u-v|$ and
		$|e^{iu}-1-iu|\le C|u|^{3/2}$, we obtain
		\begin{align*}
			&\sup_{k,\mathbf B,x}\int_0^{L_n}\frac1t
			\left|\psi_{k,\mathbf B,x}(t)-\psi_{k,\mathbf B,x}^{(0)}(t)
			-\frac{it}{\sqrt{n\kappa_{2k}}}
			\mathbb E\!\left[e^{itT_{n,k,\mathbf B}^{(0)}(x)/\sqrt{\kappa_{2k}}}
			Q_{n,k,\mathbf B}\right]\right|dt\\
			&\qquad\le C L_n\sup_{k,\mathbf B,x}\mathbb E|r_{n,k,\mathbf B}(x)|
			+C n^{-3/4}L_n^{3/2}=o(n^{-1/2}).
		\end{align*}
		Substituting \eqref{multiregress} yields an error bounded by $C n^{-1/2}L_n\varepsilon_n(L_n)=o(n^{-1/2})$. This justifies the mixed correction within the Fourier integral.
		
		It remains to compare the variance and covariance terms in
		\eqref{multiregress} for $\mathbf B=\Z$ and $\mathbf B=\Y$.
		We first compute the coefficients of the relevant row moments
		at $x=0$.
		For row $i$, let $\mathbf B_{-i}$ and
		$\widetilde{\mathbf Z}_{k,-i}$ denote the arrays with that row
		removed, keeping the normalization $n$ unchanged, and set
		\[
		\begin{aligned}
			\mathbf H^{(i)}
			=\mathbf V_k\Big(
			\rho_{nk}\mathbf I-\frac{1}{n}
			\mathbf V_k'\mathbf B_{-i}'\mathbf B_{-i}\mathbf V_k
			\Big)^{-1}\mathbf V_k',~
			\mathbf a^{(i)}
			=\frac{1}{n}\mathbf H^{(i)}\mathbf B_{-i}'
			\widetilde{\mathbf Z}_{k,-i},
			~
			\eta_i=1-\frac{1}{n}\mathbf b_i'\mathbf H^{(i)}\mathbf b_i.
		\end{aligned}
		\]
		Write $\Delta_i X=X-X^{(-i)}$, where $X^{(-i)}$ is obtained
		by zeroing the $i$th score and background row, with $n$ and
		the centering constants unchanged.
		The rank-one inverse formula gives the exact increments
		\[
		\begin{aligned}
			\Delta_i W_{kj,\mathbf B}
			=\frac{(\mathbf b_i'\mathbf H^{(i)}\mathbf v_j)
				(\widetilde Z_{ik}+\mathbf b_i'\mathbf a^{(i)})}
			{\sqrt n\,\eta_i},\qquad
			\Delta_i T_{n,k,\mathbf B}^{(0)}(0)
			=\frac{(\widetilde Z_{ik}+\mathbf b_i'\mathbf a^{(i)})^2}
			{\sqrt n\,\rho_{nk}\eta_i}.
		\end{aligned}
		\]
		Put $m=(l_k-1+\gamma_n)^{-1}$ and $e=l_k/\rho_{nk}$.
		The preceding projected-resolvent argument is unchanged by
		one row deletion and gives
		\[
		\mathbf u'\mathbf H^{(i)}\mathbf w
		=m\mathbf u'(\mathbf I-\mathbf v_k\mathbf v_k')\mathbf w
		+o_{L^2}(1),
		\qquad
		\eta_i=e+o_{L^2}(1),
		\]
		uniformly for deterministic unit vectors $\mathbf u,\mathbf w$.
		Use the increments above in the row-martingale covariance
		expansion. The terms involving the deleted-score vectors
		in the mixed fourth-moment contribution are negligible by
		\eqref{eq:short-score-bounds} and \eqref{eq:short-mixed-bound}.
		The remaining fourth-moment terms have a factor $m/e$ from
		each $W$ increment and $1/(\rho_{nk}e)$ from each $T^{(0)}$
		increment. Thus their coefficients in the variance difference
		and covariance are, respectively,
		\[
		\frac{m^2}{e^2}=\frac1{(l_k-1)^2},
		\qquad
		\frac{m}{\rho_{nk}e^2}=\frac1{l_k(l_k-1)}.
		\]
		These are coefficients of the mixed fourth-moment terms,
		not approximations of the complete increments.
		
		For $\mathbf B=\Z$, coordinate independence and
		$\mathbf v_k'\mathbf v_j=0$ give the two required row moments:
		\begin{align}
			&\mathbb E\{(\widetilde Z_{1k}^2-1)
			(\widetilde Z_{1j}^2-1)\}
			=\beta C_{4,n}(k,k,j,j),
			\label{eq:multi-fourth-contraction-short}\\
			&\mathbb E\{(\widetilde Z_{1k}^2-1)
			\widetilde Z_{1k}\widetilde Z_{1j}\}
			=\beta C_{4,n}(k,k,k,j).
			\label{eq:multi-third-fourth-contraction-short}
		\end{align}
		For $\mathbf B=\Y$, the corresponding mixed row covariance
		vanishes because
		\[
		\begin{aligned}
			\mathbb E\bigl[(\widetilde Z_{1k}^2-1)
			\widetilde Z_{1k}(\mathbf y_1'\mathbf v_j)\bigr]=
			\mathbb E\bigl[(\widetilde Z_{1k}^2-1)
			\widetilde Z_{1k}\bigr]
			\mathbb E(\mathbf y_1'\mathbf v_j)=0.
		\end{aligned}
		\]
		Here independence is used only to factor this row expectation;
		the two complete statistics need not be independent.
		The second-moment contributions have the same leading limit
		in the two variance calculations and cancel on subtraction.
		Combining the row moments with the coefficients above and
		the row-moment remainder bounds gives the following limits.
		They hold uniformly for $|x|\le n^{1/4}$, since
		\[
		\|T_{n,k,\mathbf B}^{(0)}(x)
		-T_{n,k,\mathbf B}^{(0)}(0)\|_2
		\le C|x|/\sqrt n=o(1).
		\]
		Consequently,
		\begin{align}
			&\operatorname{Var}(W_{kj,\Z})
			-\operatorname{Var}(W_{kj,\Y})
			=\frac{\beta}{(l_k-1)^2}C_{4,n}(k,k,j,j)+o(1),
			\label{eq:multi-W-variance-short}\\
			&\operatorname{Cov}\{T_{n,k,\Z}^{(0)}(x),W_{kj,\Z}\}
			=\frac{\beta}{l_k(l_k-1)}C_{4,n}(k,k,k,j)+o(1),
			\label{eq:multi-TW-covariance-short}\\
			&\operatorname{Cov}\{T_{n,k,\Y}^{(0)}(x),W_{kj,\Y}\}
			=o(1).
		\end{align}

		The constant term in \eqref{multiregress}, together with
		\eqref{eq:multi-threshold-difference-short}, produces
		\begin{align}
			\mathfrak a_{1k}^{\mathrm{mix}}
			&=
			\beta
			\sum_{j\neq k}
			\frac{l_j-1}{(l_k-l_j)(l_k-1)}
			C_{4,n}(k,k,j,j).
			\label{eq:multi-mean-correction-calculation}
		\end{align}
		The covariance-square term is multiplied by
		$it(-t^2)=(it)^3$. In the standard third-cumulant normalization, it
		produces
		\begin{align}
			\mathfrak a_{0k}^{\mathrm{mix}}
			&=
			\frac{6\beta^2}{l_k^2}
			\sum_{j\neq k}
			\frac{l_j-1}{(l_k-l_j)(l_k-1)}
			\{C_{4,n}(k,k,k,j)\}^2.
			\label{eq:multi-third-correction-calculation}
		\end{align}
		These are the two additional mixed-spike corrections.

		Next, we verify that the bounds in
		\eqref{eq:T53-uniform-middle} and
		\eqref{eq:T53-uniform-high}
		remain valid for the complete multi-spike statistic. Fix $k$, put $v=\mathbf v_k$, $\mathbf{P}=\I-vv'$, $\U=(\mathbf v_j)_{j\ne k}$,
		and $\mathbf{P}_k=\V_k\D_k\V_k'$. Let
		\[
		\mathbf{H}_0=\mathbf{P}(\rho \I-\mathbf{P}\B'\B\mathbf{P}/n)^{-1}\mathbf{P},\qquad
		\mathbf{H}=\mathbf{P}_k^{1/2}(\rho \I-\mathbf{P}_k^{1/2}\B'\B\mathbf{P}_k^{1/2}/n)^{-1}\mathbf{P}_k^{1/2}.
		\]
		The finite-rank inverse identity gives
		\[
		\mathbf{H}=\mathbf{H}_0+\mathbf{H}_0 \U \mathbf{K} \U'\mathbf{H}_0,\qquad
		\mathbf{K}=\Big[
		\operatorname{diag}_{j\ne k}\!\Big\{\frac{l_j}{\rho(l_j-1)}\Big\}
		-\U'\mathbf{H}_0\U\Big]^{-1}.
		\]
		On $\mathcal G_n$ of Lemma~\ref{qtruncation},
		chosen jointly for the row-deleted arrays used below,
		$\mathbf H_0$, $\mathbf H$, $\mathbf K$ and their required
		spectral derivatives are uniformly bounded. Moreover,
		\[
		\mathbf Hv=0,\qquad
		\operatorname{rank}(\mathbf H-\mathbf H_0)\le r-1,\qquad
		|\operatorname{tr}(\mathbf H-\mathbf H_0)|\le C.
		\]
		
		We now describe the changes to the deterministic coefficients
		used in the proof of \eqref{eq:T53-uniform-middle}. At $\rho=\rho_{nk}$, write
		$m=(l_k-1+\gamma_n)^{-1}$ and set
		\[
		\mathbf{M}_k=m\mathbf{P}+\rho_{nk}m^2
		\sum_{j\ne k}d_{kj}\mathbf v_j\mathbf v_j'.
		\]
		For the arrays obtained by deleting one row, the estimates
		used in the proof of \eqref{eq:mf-diagonal-rate}, together
		with the identity above and the separation of the fixed
		spikes, give
		\[
		\mathbb E\|\U'\mathbf H_0\U-m\I\|^2
		+\mathbb E\left\|
		\mathbf K-\rho_{nk}\operatorname{diag}_{j\ne k}(d_{kj})
		\right\|^2
		\le Cn^{-1}.
		\]
		Combining these bounds with the corresponding estimates
		for $e_i'\mathbf H_0\mathbf v_j$ and for the diagonal
		entries and trace of $\mathbf H_0$, we obtain
		\[
		\max_i
		\left(\mathbb E\left|
		\mathbf H_{ii}-(\mathbf M_k)_{ii}
		\right|^2\right)^{1/2}
		\le Cn^{-1/2},\qquad
		\Big(\mathbb E\Big|
		1-\frac{\operatorname{tr}\mathbf H}{n}
		-\frac{l_k}{\rho_{nk}}
		\Big|^2\Big)^{1/2}
		\le Cn^{-1/2}.
		\]
		
		For $\xi=\Z v$, let $a_0=\mathbf H_0\B'\xi/n$ and
		$a=\mathbf H\B'\xi/n$. The identity above and
		\eqref{eq:short-score-bounds}, applied in the finitely many
		directions $\mathbf v_j$, yield
		\[
		a=a_0+\mathbf H_0\U\mathbf K\U'a_0,
		\qquad
		\bigl(\mathbb E|a-a_0|^3\bigr)^{1/3}
		\le Cn^{-1/2}.
		\]
		Thus the moment and derivative estimates used in
		\eqref{mffour}--\eqref{eq:mf-signed-integrable}
		remain valid, with $(\mathbf M_k)_{ii}$ replacing
		$m(1-v_i^2)$.
		The formulas obtained by conditioning on the remaining
		rows have the same form with $\mathbf H$ in place of
		$\mathbf H_0$. Their moment bounds follow from bounded
		coefficient matrices and independent rows, while
		$\mathbf Hv=0$ preserves the orthogonality cancellations
		and the lower bound on the norm of the single-row
		coefficient vector. 	
		
		It remains to check the characteristic-function bound used
		in the proof of \eqref{eq:T53-uniform-high}.
		Conditionally on the rows outside the selected block,
		the relevant linear combinations have squared coefficient
		norms in $[1,C]$. For each fixed $K>\epsilon$ and
		$\epsilon\sqrt n\le |t|\le K\sqrt n$, the quadratic
		coefficients in their normalized phases are uniformly
		bounded and bounded away from zero.
		Lemma~\ref{lem:projected-score-gap} therefore gives the
		same bound strictly below one for the characteristic
		functions of the corresponding squared linear combinations.
		The moment estimates above also provide the single-row
		approximation and interaction bounds used in that proof.
		
		With these estimates and structural properties verified,
		we can repeat the calculations proving
		\eqref{mffour}--\eqref{eq:mf-signed-integrable}.
		Taking $K_n=n^{1/7}$ and integrating as in the original
		proof gives \eqref{eq:T53-uniform-middle} for the complete
		multi-spike statistic.
		Likewise, the argument using the selected block of rows
		in the proof of \eqref{eq:T53-uniform-high} gives, for
		every fixed $K>\epsilon$ and $A>1$,
		\[
		\sup_{k,\mathbf B,x}
		\sup_{\epsilon\sqrt n\le |t|\le K\sqrt n}
		|\psi_{k,\mathbf B,x}(t)|
		\le C_{A,K}
		\left(n^{-A}+\frac{\log^2n}{n}\right)
		\]
		in the standardized truncated model.
		
		Choosing $B_n\to\infty$ sufficiently slowly and using
		the same integration and restoration steps yields
		\eqref{eq:T53-uniform-high}.
		The concentration argument in the same proof then gives
		\eqref{unconcentration}.
		All these bounds are uniform in $k$.

	We now apply the arguments of Theorems~\ref{convergence}, \ref{G4MT}, and~\ref{limiting} by combining the preceding expansion and bounds with the multiple spike corrections. Together with the identity
	\[
	\sqrt{\kappa_{2k}} = \mathrm{F}_{\gamma_n}(g_{kn}^{2})\widetilde{\sigma}_{nk},
	\]
	this yields the following three comparisons:
	\begin{align}
		\label{first}
		&\sup_x\left|\mathbb P(R_k\le x)-\mathbb P\left(T_{n,k,\Z}(x)\le \sqrt{\kappa_{2k}}\,x\right)\right|=o(n^{-1/2}),\\
		\label{second}
		&\sup_x\Bigg|\mathbb P\left(T_{n,k,\Z}(x)\le \sqrt{\kappa_{2k}}\,x\right)-
		\mathbb P\left(T_{n,k,\Y}(x)\le \sqrt{\kappa_{2k}}\,x\right)	-\frac{\phi(x)}{\sqrt n}\Big\{-\frac{\mathfrak b_kx}{2\kappa_{2k}}
		\nonumber\\
		&\qquad\qquad
		+\frac{\mathfrak a_{0k}-\kappa_{3k}+\mathfrak a_{0k}^{\mathrm{mix}}}{6\kappa_{2k}^{3/2}} (1-x^2)
		-\frac{\mathfrak a_{1k}-\mu(g_{kn})+\mathfrak a_{1k}^{\mathrm{mix}}}{
			\sqrt{\kappa_{2k}}}\Big\}\Bigg|=o(n^{-1/2}),\\
		\label{third}
		&\sup_x\Bigg|\mathbb P\left(T_{n,k,\Y}(x)\le \sqrt{\kappa_{2k}}\,x\right)
		-\Phi(x)\nonumber\\
		&\qquad\qquad-\frac{\phi(x)}{\sqrt n}\Big\{
		\frac{\kappa_{3k}}{6\kappa_{2k}^{3/2}}(1-x^2)-\frac{\mu(g_{kn})+A(g_{kn})}{\sqrt{\kappa_{2k}}}-\mathfrak c_kx^2\Big\}\Bigg|=o(n^{-1/2}).
	\end{align}

	For all sufficiently large $n$, the spikes are distinct and separated. Since
	\[
	z_{0,n}(-1/l_k)=\rho_{nk},\qquad
	\underline m_{\gamma_n}(\rho_{nk})=-\frac1{l_k},\qquad
	z_{0,n}'(-1/l_k)=\frac{l_k^2\{(l_k-1)^2-\gamma_n\}}{(l_k-1)^2},
	\]
	and
	\[
	\mathrm F_{\gamma_n}(g_{kn})
	=
	-\underline m_{\gamma_n}(\rho_{nk}),
	\]
	the preceding Stieltjes-transform expansion gives
	\begin{align*}
		A(g_{kn})=\frac{Q_k(-1/l_k)}{z_{0,n}'(-1/l_k)}=\frac{(l_k-1)^2}{l_k\{(l_k-1)^2-\gamma_n\}}
		\sum_{j\ne k}\frac{l_j}{l_k-l_j}.
	\end{align*}
	Equivalently, the new added-column correction equals the replacement-model correction in the original argument plus the missing dimension term:
	\[
	A(g_{kn})
	=\frac{l_k-1}{(l_k-1)^2-\gamma_n}
	\sum_{j\ne k}\frac{l_j-1}{l_k-l_j}
	+(r-1)\frac{l_k-1}{l_k\{(l_k-1)^2-\gamma_n\}}.
	\]
	
	Therefore, combining equations \eqref{first}, \eqref{second}, and
	\eqref{third} with the definitions of $\mathfrak a_{0k}$ and
	$\mathfrak a_{1k}$ in Theorem~\ref{multi}, we derive
	\begin{align*}
		&\sup_x\Bigg|
		\mathbb P(R_k\le x)-\Phi(x)
		-\frac{\phi(x)}{\sqrt n}
		\Bigg\{
		\frac{(\mathfrak a_{0k}+\mathfrak a_{0k}^{\mathrm{mix}})(1-x^2)}{6\kappa_{2k}^{3/2}}\\
		&\qquad\qquad
		-\frac{\mathfrak a_{1k}+A(g_{kn})+\mathfrak a_{1k}^{\mathrm{mix}}}
		{\sqrt{\kappa_{2k}}}
		-\frac{\mathfrak b_k x}{2\kappa_{2k}}
		-\mathfrak c_kx^2
		\Bigg\}
		\Bigg|=o(n^{-1/2}).
	\end{align*}
	Hence, the result holds.

		\subsection{Proof of Theorem~\ref{main}}
	
	Let
	\begin{align*}
		\mathcal{R}_{n,\Z}
		&=\frac{-\Omega(\rho_n,\Z)+\rho_n n^{-1/2}\widetilde G_{n,\Z}(g_n)}
		{[\rho_n\mathrm F_{\gamma_n}(g_n^2)-n^{-1/2}\widetilde S_n(g_nh_n)]\widetilde\sigma_n},\\
		\mathcal{R}_{n,\Y}
		&=\frac{-\Omega(\rho_n,\Y)+\rho_n n^{-1/2}\widetilde G_{n,\Y}(g_n)}
		{[\rho_n\mathrm F_{\gamma_n}(g_n^2)-n^{-1/2}S_n(g_nh_n)]\widetilde\sigma_n}.
	\end{align*}
Theorems~\ref{convergence}, \ref{G4MT}, and~\ref{limiting} give, uniformly in
	$x\in\mathbb R$,
	\begin{align*}
		&\mathbb P(R_n\le x)
		=\mathbb P(\mathcal{R}_{n,\Z}\le x)+o(n^{-1/2}),\\
		&\mathbb P(\mathcal{R}_{n,\Z}\le x)
		=\mathbb P(\mathcal{R}_{n,\Y}\le x)
		+\frac{\phi(x)}{\sqrt n}\left\{
		\frac{\mathfrak k_1(1-x^2)}{6\kappa_{21}^{3/2}}
		-\frac{\mathfrak b_1 x}{2\kappa_{21}}
		-\frac{\mathfrak m_1}{\sqrt{\kappa_{21}}}
		\right\}+o(n^{-1/2}),\\
		&\mathbb P(\mathcal{R}_{n,\Y}\le x)
		=\Phi(x)+\frac{\phi(x)}{\sqrt n}\left\{
		\frac{\kappa_{31}(1-x^2)}{6\kappa_{21}^{3/2}}
		-\frac{\mu(g_{1n})}{\sqrt{\kappa_{21}}}
		-\mathfrak c_1x^2
		\right\}+o(n^{-1/2}),
	\end{align*}
	Summing these three equations and using
	\[
	\mathfrak a_{01}=\kappa_{31}+\mathfrak k_1,
	\qquad
	\mathfrak a_{11}=\mu(g_{1n})+\mathfrak m_1,
	\]
	we obtain
	\[
	\sup_x\left|\mathbb P(R_n\le x)-\Phi(x)
	-\frac{\phi(x)}{\sqrt n}\left\{
	\frac{\mathfrak a_{01}}{6\kappa_{21}^{3/2}}(1-x^2)
	-\frac{\mathfrak a_{11}}{\sqrt{\kappa_{21}}}
	-\frac{\mathfrak b_1}{2\kappa_{21}}x
	-\mathfrak c_1x^2
	\right\}\right|=o(n^{-1/2}).
	\]
	Hence, the result holds.

		\subsection{Proof of Theorem~\ref{thm:moment-estimators}}
		Throughout the proof, $C>0$ denotes a constant that may vary
		from line to line. Let $\mathbf A=\bm\Sigma^{1/2}$, so that
		$\mathbf x_i=\mathbf A\mathbf z_i$.

		Consider first $\widehat{\Gamma^2}$.
		Rosenthal's inequality, together with $\mathbb E|Z_{11}|^6<\infty$,
		gives
		\begin{equation}\label{eq:moment-linear-bound}
			\mathbb E|\mathbf b'\mathbf z_1|^s
			\le C\|\mathbf b\|_2^s,
			\qquad 2\le s\le6,
		\end{equation}
		for every deterministic vector $\mathbf b$.
		Let $h(\mathbf x_1,\mathbf x_2)=(\mathbf x_1'\mathbf x_2)^3$.
		Since $\mathbf x_1'\mathbf x_2
		=\mathbf z_1'\bm\Sigma\mathbf z_2$,
		independence and centering of the coordinates of $\mathbf z_1$
		and $\mathbf z_2$ imply
		\[
		\mathbb Eh=\Gamma^2\mathcal D_3,
		\qquad
		\mathbb Eh^2
		\le C\mathbb E\|\bm\Sigma\mathbf z_1\|_2^6
		\le Cq^3.
		\]
		Moreover,
		\[
		g(\mathbf z_1)
		:=\mathbb E\{h(\mathbf x_1,\mathbf x_2)\mid\mathbf z_1\}
		=\Gamma\sum_{k=1}^q(\bm\Sigma\mathbf z_1)_k^3.
		\]
		The row norms of $\bm\Sigma$ are bounded, thus
		\[
		\mathbb E|g(\mathbf z_1)|^2
		\le \Gamma^2q\sum_{k=1}^q
		\mathbb E|(\bm\Sigma\mathbf z_1)_k|^6
		\le Cq^2.
		\]
		Two terms in the split-sample average $\widehat M_3$ are independent
		unless they share a sample index. If they share exactly one index,
		their covariance is $\operatorname{Var}(g(\mathbf z_1))$.
		It follows that
		\[
		\operatorname{Var}(\widehat M_3)
		\le C\left(
		\frac{q^2}{n_1}
		+\frac{q^2}{n_2}
		+\frac{q^3}{n_1n_2}
		\right).
		\]
		Since $q,n_1,n_2$ and $n$ are of the same order,  the variance of $\widehat M_3/q$
		tends to zero. Hence
		\begin{equation}\label{eq:Gamma-consistency-new}
			\widehat{\Gamma^2}-\Gamma^2
			=\frac{\widehat M_3-\mathbb E\widehat M_3}{q}
			+\Gamma^2\left(\frac{\mathcal D_3}{q}-1\right)
			\overset{p}{\to}0.
		\end{equation}
		
		For the fourth and sixth cumulants, it is enough to control the
		coordinatewise empirical moments in $L^1$ after averaging over the
		coordinates. Let $\mathbf B=\mathbf A-\mathbf I_q$ and $\mathbf b_j'$ be the $j$th row of $\mathbf B$.
		For $s\in\{2,3,4,6\}$, the inequality
		$|u^s-v^s|\le s|u-v|(|u|+|v|)^{s-1}$,
		Holder's inequality and \eqref{eq:moment-linear-bound} give
		\[
		\mathbb E|X_{1j}^s-Z_{1j}^s|
		\le C\|\mathbf b_j\|_2.
		\]
		We used $X_{1j}-Z_{1j}=\mathbf b_j'\mathbf z_1$ and the
		uniformly bounded row norms of $\mathbf A$.
		Therefore,
		\begin{equation}\label{eq:moment-coordinate-comparison}
			\frac1q\sum_{j=1}^q\mathbb E|X_{1j}^s-Z_{1j}^s|
			\le \frac{C\|\mathbf B\|_{\mathrm F}}{\sqrt q}
			=O(q^{-1/2}).
		\end{equation}
		
		For a sample block $\mathscr I$, write
		$\widehat m_{s\mathscr I,j}
		=n_{\mathscr I}^{-1}\sum_{i\in\mathscr I}X_{ij}^s$,
		$m_{s,j}=\mathbb EX_{1j}^s$, and $\mu_s=\mathbb EZ_{11}^s$.
		Adding and subtracting the empirical and population moments of
		$Z_{ij}$, and using \eqref{eq:moment-coordinate-comparison}, gives
		\begin{align}
			\frac1q\sum_{j=1}^q
			\mathbb E|\widehat m_{s\mathscr I,j}-m_{s,j}|
			&\le Cq^{-1/2}
			+\mathbb E\left|
			\frac1{n_{\mathscr I}}\sum_{i\in\mathscr I}Z_{i1}^s-\mu_s
			\right|
			\to0.
			\label{eq:moment-averaged-L1}
		\end{align}
		The last expectation tends to zero by the $L^1$ law of large numbers.
		In particular, truncation of $Z_{i1}^s$ at absolute level $M$ bounds
		it by
		\[
		\frac{M}{\sqrt{n_{\mathscr I}}}
		+2\mathbb E\bigl[
		|Z_{11}|^s\mathbf1\{|Z_{11}|^s>M\}
		\bigr],
		\]
		which tends to zero by first letting $n\to\infty$ and then
		$M\to\infty$. Notice that no moments above order six are used.
		
		Now take disjoint blocks $\mathscr I_1,\ldots,\mathscr I_t$ and
		integers $s_1,\ldots,s_t\in\{2,3,4,6\}$, with $t\le3$.
		The corresponding empirical moments are independent across the
		blocks. Since
		$\sup_{n,j}\mathbb E|X_{1j}|^s<\infty$ for $s\le6$,
		a telescoping expansion of the product gives
		\[
		\mathbb E\left|
		\prod_{\ell=1}^t\widehat m_{s_\ell\mathscr I_\ell,j}
		-\prod_{\ell=1}^t m_{s_\ell,j}
		\right|
		\le C\sum_{\ell=1}^t
		\mathbb E\left|
		\widehat m_{s_\ell\mathscr I_\ell,j}-m_{s_\ell,j}
		\right|.
		\]
		Averaging this bound over $j$ and applying
		\eqref{eq:moment-averaged-L1}, we obtain
		\begin{equation}\label{eq:moment-products-L1}
			\frac1q\sum_{j=1}^q\left(
			\prod_{\ell=1}^t\widehat m_{s_\ell\mathscr I_\ell,j}
			-\prod_{\ell=1}^t m_{s_\ell,j}
			\right)
			\overset{L^1}{\to}0.
		\end{equation}
		Apply this result to the exponent tuples
		$(4)$, $(2,2)$, $(6)$, $(4,2)$, $(3,3)$ and $(2,2,2)$.
		These are exactly the terms in the definitions of
		$\widehat{\mathcal C}_4$ and $\widehat{\mathcal C}_6$.
		Thus
		\[
		\frac{\widehat{\mathcal C}_s-\mathcal C_s}{q}
		\overset{L^1}{\to}0,
		\qquad s=4,6.
		\]
		
		Additivity and homogeneity of cumulants yield
		\[
		\mathcal C_4=\beta\mathcal D_4,
		\qquad
		\mathcal C_6=K_6(Z_{11})\mathcal D_6.
		\]
		Together with ${\max\{\mathcal D_4,\mathcal D_6\}}/{q}\to1$, this gives
		\begin{align*}
			\widehat\beta-\beta
			&=\frac{\widehat{\mathcal C}_4-\mathcal C_4}{q}
			+\beta\left(\frac{\mathcal D_4}{q}-1\right)
			\overset{p}{\to}0,
		\end{align*}
		and
		\begin{align*}
			\widehat K_6(Z)-K_6(Z_{11})
			&=\frac{\widehat{\mathcal C}_6-\mathcal C_6}{q}
			+K_6(Z_{11})\left(\frac{\mathcal D_6}{q}-1\right)
			\overset{p}{\to}0.
		\end{align*}
		Finally, $K_6(Z_{11})=\Delta-15\beta-10\Gamma^2-15$, and hence
		\[
		\widehat\Delta-\Delta
		=\widehat K_6(Z)-K_6(Z_{11})
		+15(\widehat\beta-\beta)
		+10(\widehat{\Gamma^2}-\Gamma^2)
		\overset{p}{\to}0.
		\]
		We then complete the proof of Theorem~\ref{thm:moment-estimators}.

	\subsection{Proof of Theorem~\ref{Pivots}   }

Since $l_k>1+\sqrt{\gamma}$, we have
\[
\rho_{nk}-(1+\sqrt{\gamma_n})^2
=
\frac{(l_k-1-\sqrt{\gamma_n})^2}{l_k-1},
\]
which is bounded away from zero for all sufficiently large $n$.
Because
\[
\theta_n=(1+\sqrt{\gamma_n})^2
+n^{-1/3}\sqrt{\gamma_n},
\]
it follows that
\[
z_{kn}(\theta_n,l_k)\to-\infty.
\]
Therefore, Theorem~\ref{multi} implies
\[
\bar F_{kn}(\theta_n,\mathbf l)=1+o(1),
\quad
\bar\Phi\bigl(z_{kn}(\theta_n,l_k)\bigr)=1+o(1),\quad \bar F^E_{kn}
\bigl(z_{kn}(\theta_n,l_k),\mathbf l\bigr)
=1+o(1).
\]
In particular, all three denominators appearing in the pivots are
bounded away from zero.

By Theorem~\ref{multi},
\begin{align}
	&\sup_{x\in\mathbb R}
	\left|
	F_{kn}(x,\mathbf l)
	-\Phi\bigl(z_{kn}(x,l_k)\bigr)
	\right|
	=O(n^{-1/2}),
	\label{eq:pivot-normal-comparison}\\
	&\sup_{x\in\mathbb R}
	\left|
	F_{kn}(x,\mathbf l)
	-F^E_{kn}\bigl(z_{kn}(x,l_k),\mathbf l\bigr)
	\right|
	=o(n^{-1/2}).
	\label{eq:pivot-edgeworth-comparison}
\end{align}

For every fixed $\alpha\in(0,1)$, the elementary quantile-comparison
inequality gives
\begin{align*}
	\left|
	\mathbb P\left(
	u_{kn}^z(\hat\rho_k,l_k)\leq\alpha
	\,\middle|\,\hat\rho_k>\theta_n
	\right)-\alpha
	\right|\leq
	\frac{2}{\bar F_{kn}(\theta_n,\mathbf l)}
	\sup_{x\in\mathbb R}
	\left|
	F_{kn}(x,\mathbf l)
	-\Phi\bigl(z_{kn}(x,l_k)\bigr)
	\right|.
\end{align*}
Combining this inequality with
\eqref{eq:pivot-normal-comparison} yields
\[
\mathbb P\left(
u_{kn}^z(\hat\rho_k,l_k)\leq\alpha
\,\middle|\,\hat\rho_k>\theta_n
\right)-\alpha
=
O(n^{-1/2}).
\]

Similarly,
\begin{align*}
	\left|
	\mathbb P\left(
	u_{kn}^E(\hat\rho_k,\mathbf l)\leq\alpha
	\,\middle|\,\hat\rho_k>\theta_n
	\right)-\alpha
	\right|\leq
	\frac{2}{\bar F_{kn}(\theta_n,\mathbf l)}
	\sup_{x\in\mathbb R}
	\left|
	F_{kn}(x,\mathbf l)
	-F^E_{kn}\bigl(z_{kn}(x,l_k),\mathbf l\bigr)
	\right|.
\end{align*}
For fixed $\alpha$, the relevant standardized quantile remains in a
fixed compact interval. On this interval,
\[
\frac{d}{dx}F^E_{kn}(x,\mathbf l)
=
\phi(x)
\left[
1+n^{-1/2}\{p_k'(x)-xp_k(x)\}
\right]
>0
\]
for all sufficiently large $n$, so the preceding quantile comparison
is valid. Using \eqref{eq:pivot-edgeworth-comparison}, we obtain
\[
\mathbb P\left(
u_{kn}^E(\hat\rho_k,\mathbf l)\leq\alpha
\,\middle|\,\hat\rho_k>\theta_n
\right)-\alpha
=
o(n^{-1/2}).
\]
This completes the proof.

	\subsection{Proof of Theorem~\ref{convergence}}
We work on the event $\mathcal E_n$ of
Lemma~\ref{qtruncation} throughout the following
identities and resolvent expansions. According to the definition of $\Omega(\lambda, \Z)$, equation \eqref{determinant} can be decomposed as follows:
\begin{align}\label{4.1}
	0=\Big|\frac{\hat \rho}{l}+\frac{1}{\sqrt n}\Omega(\hat \rho, \Z)-\frac {\hat \rho}{n} \mathrm{tr}\big[(\hat \rho\I_n-n^{-1}\Z\V_2\V_2'\Z')^{-1}\big] \Big|.
\end{align}
Continuing the calculation on equation \eqref{determinant}, we obtain:
\begin{align*}
	0&=\Big|\frac{\hat \rho}{l}-  \hat \rho n^{-1}\V_1'\Z'(\hat \rho\I_{n}-n^{-1}\Z\V_2\V_2'\Z')^{-1} \Z\V_1\Big|\\
	&=\Big|\frac{\hat \rho}{l}-  \hat \rho n^{-1}\V_1'\Z'(\rho_n\I-n^{-1}\Z\V_2\V_2'\Z')^{-1} \Z\V_1+(\hat \rho-\rho_n)\hat \rho n^{-1}\V_1'\Z'\\
	&\qquad\times(\hat \rho\I-n^{-1}\Z\V_2\V_2'\Z')^{-1}(\rho_n\I-n^{-1}\Z\V_2\V_2'\Z')^{-1} \Z\V_1\Big|\\
	&=\Big|\frac{\hat \rho}{l}-\rho_nn^{-1}\V_1'\Z'(\rho_n\I-n^{-1}\Z\V_2\V_2'\Z')^{-1} \Z\V_1\\
	&\qquad-(\hat \rho-\rho_n)n^{-1}\V_1'\Z'(\rho_n\I-n^{-1}\Z\V_2\V_2'\Z')^{-1} \Z\V_1\\
	&\qquad+(\hat \rho-\rho_n)\hat \rho n^{-1}\V_1'\Z'(\hat \rho\I-n^{-1}\Z\V_2\V_2'\Z')^{-1}(\rho_n\I-n^{-1}\Z\V_2\V_2'\Z')^{-1} \Z\V_1\Big|\\
	&=\Big|\frac{\hat \rho}{l}-\rho_nn^{-1}\V_1'\Z'(\rho_n\I-n^{-1}\Z\V_2\V_2'\Z')^{-1} \Z\V_1+(\hat \rho-\rho_n)n^{-1}\V_1'\Z'(n^{-1}\Z\V_2\V_2'\Z')\\
	&\qquad\times(\hat \rho\I-n^{-1}\Z\V_2\V_2'\Z')^{-1}(\rho_n\I-n^{-1}\Z\V_2\V_2'\Z')^{-1} \Z\V_1\Big|.
\end{align*}
Therefore, we derive the following result:
\begin{align}\label{com1}
	&\qquad\hat \rho-\rho_n\\\nonumber
	&=\frac{l\rho_n\big(n^{-1}\V_1'\Z'(\rho_n\I-n^{-1}\Z\V_2\V_2'\Z')^{-1} \Z\V_1-l^{-1}\big)}{1+ln^{-1}\V_1'\Z'(n^{-1}\Z\V_2\V_2'\Z') (\hat \rho\I-n^{-1}\Z\V_2\V_2'\Z')^{-1}(\rho_n\I-n^{-1}\Z\V_2\V_2'\Z')^{-1} \Z\V_1}.
\end{align}

The companion empirical spectral distribution $\mathrm{F}_n$ of $n^{-1}\Z\V_2\V_2'\Z'$ converges to the companion M-P law 
\[
\mathrm F_{\gamma}=(1-\gamma)\delta_0+\gamma F_{\gamma}.
\]
Then, we obtain the following result:
\begin{align}\label{com2}
	n^{-1}\V_1'\Z'(\rho_n\I-n^{-1}\Z\V_2\V_2'\Z')^{-1} \Z\V_1=\mathrm{F}_{\gamma_n}(g_n)+n^{-1/2}\widetilde S_n(g_n)+n^{-1}\widetilde G_{n,\Z}(g_n).
\end{align}
Besides,
\begin{align}\label{com3}
	&\quad~ 1+ln^{-1}\V_1'\Z'(n^{-1}\Z\V_2\V_2'\Z') (\hat \rho\I-n^{-1}\Z\V_2\V_2'\Z')^{-1}(\rho_n\I-n^{-1}\Z\V_2\V_2'\Z')^{-1} \Z\V_1\\\nonumber
	&=1+ln^{-1}\V_1'\Z'(n^{-1}\Z\V_2\V_2'\Z')(\rho_n\I-n^{-1}\Z\V_2\V_2'\Z')^{-2} \Z\V_1+ln^{-1}\V_1'\Z'(n^{-1}\Z)\\\nonumber
	&\qquad \times \V_2\V_2'\Z'(\rho_n-\hat \rho)(\hat \rho\I-n^{-1}\Z\V_2\V_2'\Z')^{-1}(\rho_n\I-n^{-1}\Z\V_2\V_2'\Z')^{-2} \Z\V_1\\\nonumber
	&=1+l\Big(\mathrm{F}_{\gamma_n}(\lambda g_n^2)+n^{-1/2}\widetilde S_n(\lambda g_n^2)+n^{-1}\widetilde G_{n,\Z}(\lambda g_n^2)\Big)+ln^{-1}\V_1'\Z'(n^{-1}\Z\V_2\V_2'\Z') \\\nonumber
	&\qquad\times(\rho_n-\hat \rho)(\hat \rho\I-n^{-1}\Z\V_2\V_2'\Z')^{-1}(\rho_n\I-n^{-1}\Z\V_2\V_2'\Z')^{-2} \Z\V_1\\\nonumber
	&=l\rho_n\mathrm{F}_{\gamma_n}(g_n^2)+l\Big(n^{-1/2}\widetilde S_n(\lambda g_n^2)+n^{-1}\widetilde G_{n,\Z}(\lambda g_n^2)\Big)+ln^{-1}\V_1'\Z'(n^{-1}\Z\V_2\V_2'\Z') \\\nonumber
	&\qquad\times(\rho_n-\hat \rho)(\hat \rho\I-n^{-1}\Z\V_2\V_2'\Z')^{-1}(\rho_n\I-n^{-1}\Z\V_2\V_2'\Z')^{-2} \Z\V_1.
\end{align}  
Therefore, combining Lemma~\ref{resolventfluctuations}
with equations \eqref{com1}--\eqref{com3}, we derive
\begin{gather}\label{xihua}
	n^{1/2}(\hat \rho-\rho_n)=\frac{\widetilde S_n(g_n)}{\mathrm{F}_{\gamma_n}(g_n^2)}+O_p(n^{-1/2}).
\end{gather}
The functional $\mathrm{F}_{\gamma_n}(g_n^j)(j=1,2)$ encodes key transformations of the Mar\v{c}enko-Pastur distribution, with the special cases
\[\mathrm{F}_{\gamma_n}(g_n)=l^{-1},\qquad\mathrm{F}_{\gamma_n}(g_n^2)=2\sigma_n^{-2}.\] The detailed derivation of these relations appear in \cite{yang2018edgeworth}.
Note that
\begin{align*}
	\Omega(\hat \rho,\Z)&=\frac{\hat \rho}{\sqrt n}\Big(\mathrm{tr}\big[(\hat \rho\I-n^{-1}\Z\V_2\V_2'\Z')^{-1}\big]-\V_1'\Z'(\hat \rho\I-n^{-1}\Z\V_2\V_2'\Z')^{-1}\Z\V_1\Big)\\
	&=\frac{\hat \rho-\rho_n}{\sqrt n}\Big(\mathrm{tr}\big[(\hat \rho\I-n^{-1}\Z\V_2\V_2'\Z')^{-1}\big]-\V_1'\Z'(\hat \rho\I-n^{-1}\Z\V_2\V_2'\Z')^{-1}\Z\V_1\Big)\\
	&\quad+\frac{\rho_n}{\sqrt n}\Big(\mathrm{tr}\big[(\hat \rho\I-n^{-1}\Z\V_2\V_2'\Z')^{-1}-(\rho_n\I-n^{-1}\Z\V_2\V_2'\Z')^{-1}\big]-\V_1'\Z'\\
	&\quad\times(\hat \rho\I-n^{-1}\Z\V_2\V_2'\Z')^{-1}\Z\V_1+\V_1'\Z'(\rho_n\I-n^{-1}\Z\V_2\V_2'\Z')^{-1}\Z\V_1\Big)+\Omega(\rho_n,\Z)\\
	&=\frac{\hat \rho-\rho_n}{\sqrt n}\Big(\mathrm{tr}\big[(\hat \rho\I-n^{-1}\Z\V_2\V_2'\Z')^{-1}\big]-\V_1'\Z'(\hat \rho\I-n^{-1}\Z\V_2\V_2'\Z')^{-1}\Z\V_1\Big)\\
	&\quad+\frac{\rho_n(\rho_n-\hat \rho)}{\sqrt n}\Big(\mathrm{tr}\big[(\hat \rho\I-n^{-1}\Z\V_2\V_2'\Z')^{-1}(\rho_n\I-n^{-1}\Z\V_2\V_2'\Z')^{-1}\big]\\
	&\quad-\V_1'\Z'\big[(\hat \rho\I-n^{-1}\Z\V_2\V_2'\Z')^{-1}(\rho_n\I-n^{-1}\Z\V_2\V_2'\Z')^{-1}\big]\Z\V_1\Big)+\Omega(\rho_n,\Z)\\
	&=\frac{\sqrt n(\hat \rho-\rho_n)}{n}\Big(\V_1'\Z'(\hat \rho\I-n^{-1}\Z\V_2\V_2'\Z')^{-1}(\rho_n\I-n^{-1}\Z\V_2\V_2'\Z')^{-1}\\
	&\quad\times(n^{-1}\Z\V_2\V_2'\Z')\Z\V_1-\mathrm{tr}\big[(\hat \rho\I-n^{-1}\Z\V_2\V_2'\Z')^{-1}(\rho_n\I-n^{-1}\Z\V_2\V_2'\Z')^{-1}\\
	&\quad\times(n^{-1}\Z\V_2\V_2'\Z')\big]\Big)+\Omega(\rho_n,\Z)\\
	&=\frac{\sqrt n(\hat \rho-\rho_n)}{n}\Big(\V_1'\Z'(\rho_n\I-n^{-1}\Z\V_2\V_2'\Z')^{-2}(n^{-1}\Z\V_2\V_2'\Z')\Z\V_1\\
	&\quad-\mathrm{tr}\big[(\rho_n\I-n^{-1}\Z\V_2\V_2'\Z')^{-2}(n^{-1}\Z\V_2\V_2'\Z')\big]\Big)+\Omega(\rho_n,\Z)+o_p(n^{-1/2}).
\end{align*}
The final equality follows from the matrix identity $\mathbf{A}^{-1}-\mathbf{B}^{-1}=\mathbf{A}^{-1}(\mathbf{B}-\mathbf{A})\mathbf{B}^{-1}$ and
Lemma~\ref{resolventfluctuations}. Therefore, we derive
\[\Omega(\hat \rho, \Z)=\Omega(\rho_n, \Z)+\sqrt n(\hat \rho-\rho_n)n^{-1/2}\widetilde S_n(\lambda g_n^2)+o_p\big(n^{-1/2}\big).\]
Hence, equation \eqref{4.1} can be expressed as 
\begin{align}\label{4.2}
	0&=\Big|\frac{\rho_n}{l}+\frac{1}{\sqrt n}\Omega(\rho_n, \Z)+(\hat \rho-\rho_n)n^{-1/2}\widetilde S_n(\lambda g_n^2)\\\nonumber
	&\qquad -\frac {\rho_n}{n} \mathrm{tr}\big[(\rho_n\I-n^{-1}\Z\V_2\V_2'\Z')^{-1}\big]+B_1(\hat \rho)+B_2(\hat \rho)+o_p\big(n^{-1}\big)\Big|,
\end{align}
where $B_1(\hat \rho)=(\hat \rho-\rho_n)/l$ and
\begin{align*}
	B_2(\hat \rho)=\frac {\rho_n}{n}  \mathrm{tr}\big[(\rho_n\I-n^{-1}\Z\V_2\V_2'\Z')^{-1}\big] -\frac {\hat \rho}{n}  \mathrm{tr}\big[(\hat \rho\I-n^{-1}\Z\V_2\V_2'\Z')^{-1}\big]. 
\end{align*}
Let $\delta_n=\sqrt n(\hat \rho/\rho_n-1)$. It then follows immediately that $\delta_n$ is of order $O_p(1)$. Utilizing the fundamental matrix identity $\mathbf{A}^{-1}-\mathbf{B}^{-1}=\mathbf{A}^{-1}(\mathbf{B}-\mathbf{A})\mathbf{B}^{-1}$  and the order of $\delta_n$, we derive:
\begin{align*}
	B_2(\hat \rho)&=\frac {\rho_n}{n}  \mathrm{tr}\big[(\rho_n\I-n^{-1}\Z\V_2\V_2'\Z')^{-1}\big] -\frac {\hat \rho}{n}  \mathrm{tr}\big[(\hat \rho\I-n^{-1}\Z\V_2\V_2'\Z')^{-1}\big]\\
	&=\frac {\rho_n}{n}\Big(  \mathrm{tr}(\rho_n\I-n^{-1}\Z\V_2\V_2'\Z')^{-1} - (1+n^{-1/2}\delta_n)\mathrm{tr}(\hat \rho\I-n^{-1}\Z\V_2\V_2'\Z')^{-1}\Big)\\
	&=\frac {\rho_n}{n}\Big(  \mathrm{tr}(\rho_n\I-n^{-1}\Z\V_2\V_2'\Z')^{-1} - \mathrm{tr}(\hat \rho\I-n^{-1}\Z\V_2\V_2'\Z')^{-1}\Big)\\
	&\qquad\qquad\quad-\frac {\rho_n}{n}n^{-1/2}\delta_n \mathrm{tr}(\hat \rho\I-n^{-1}\Z\V_2\V_2'\Z')^{-1}\\
	&=\frac {\rho_n^2\delta_n}{n^{3/2}}  \mathrm{tr}\Big((\rho_n\I-n^{-1}\Z\V_2\V_2'\Z')^{-1} [(1+n^{-1/2}\delta_n)\rho_n\I-n^{-1}\Z\V_2\V_2'\Z']^{-1}\Big)\\
	&\qquad\qquad\quad-\frac {\rho_n}{n}n^{-1/2}\delta_n \mathrm{tr}(\hat \rho\I-n^{-1}\Z\V_2\V_2'\Z')^{-1}\\
	&=\frac {\delta_n}{n^{1/2}}  \Big(\frac{\rho_n^2}{n}\mathrm{tr}(\rho_n\I-n^{-1}\Z\V_2\V_2'\Z')^{-2}-\frac{\rho_n^3\delta_n}{n^{3/2}}\mathrm{tr}(\rho_n\I-n^{-1}\Z\V_2\V_2'\Z')^{-3} -\frac {\rho_n}{n}\\
	&\qquad\times \mathrm{tr}(\rho_n\I-n^{-1}\Z\V_2\V_2'\Z')^{-1}+\frac{\rho_n^2\delta_n}{n^{3/2}}\mathrm{tr}(\rho_n\I-n^{-1}\Z\V_2\V_2'\Z')^{-2}\Big)+o_p(n^{-1})\\
	&=\frac {\delta_n}{n^{1/2}}\Big(\rho_n^2\mathrm{F}_{\gamma_n}(g_n^2)-\rho_n\mathrm{F}_{\gamma_n}(g_n)\Big)-\frac{\delta_n^2}{n}\Big(\rho_n^3\mathrm{F}_{\gamma_n}(g_n^3)-\rho_n^2\mathrm{F}_{\gamma_n}(g_n^2)\Big)+o_p(n^{-1}).
\end{align*}
Recall
\[\widetilde G_{n,\Z}(g_n)=\mathrm{tr}\big[(\rho_n\I-n^{-1}\Z\V_2\V_2'\Z')^{-1}\big]-n\int g_n(\lambda)\mathrm{F}_{\gamma_n}(d\lambda).\]
Substituting the transformation \eqref{xihua} into \eqref{4.2} while accounting for the order of $\delta_n$, we derive:
\begin{align*}
	0&=\bigg|\frac{\rho_n}{l}+\frac{1}{\sqrt n}\Omega(\rho_n, \Z)+(\hat \rho-\rho_n)n^{-1/2}\widetilde S_n(\lambda g_n^2)-\frac {\rho_n}{n}  \mathrm{tr}\big[(\rho_n\I-n^{-1}\Z\V_2\V_2'\Z')^{-1}\big] \\
	&\qquad+B_1(\hat \rho)+B_2(\hat \rho)+o_p\big(n^{-1}\big)\bigg|\\
	&=\bigg|\frac {\delta_n}{n^{1/2}}\Big(\frac{\rho_n}{l}+\rho_n^2\mathrm{F}_{\gamma_n}(g_n^2)-\rho_n\mathrm{F}_{\gamma_n}(g_n)\Big)-\frac{\delta_n^2}{n}\Big(\rho_n^3\mathrm{F}_{\gamma_n}(g_n^3)-\rho_n^2\mathrm{F}_{\gamma_n}(g_n^2)\Big)\\
	&\qquad-\frac{\rho_n}{n}\widetilde G_{n,\Z}(g_n)+\frac{1}{ n^{1/2}}\Omega(\rho_n, \Z)+\frac{\delta_n\rho_n}{n}\widetilde S_n(\lambda g_n^2)+o_p(n^{-1})\bigg|\\
	&=\bigg|\delta_n\rho_n\Big(\frac{1}{l}+\rho_n\mathrm{F}_{\gamma_n}(g_n^2)-\mathrm{F}_{\gamma_n}(g_n)\Big)-\frac{\delta_n^2\rho_n^2}{n^{1/2}}\Big(\rho_n\mathrm{F}_{\gamma_n}(g_n^3)-\mathrm{F}_{\gamma_n}(g_n^2)\Big)+\Omega(\rho_n, \Z)\\
	&\qquad+\frac{\delta_n\rho_n}{n^{1/2}}\widetilde S_n(\lambda g_n^2)-\frac{\rho_n}{n^{1/2}}\widetilde G_{n,\Z}(g_n)+o_p(n^{-1/2})\bigg|\\
	&=\bigg|\delta_n\rho_n\Big(\rho_n\mathrm{F}_{\gamma_n}(g_n^2)-n^{-1/2}\frac{\widetilde S_n(g_n)\rho_n\mathrm{F}_{\gamma_n}(g_n^3)}{\mathrm{F}_{\gamma_n}(g_n^2)}+n^{-1/2}\widetilde S_n(g_n)+n^{-1/2}\widetilde S_n(\lambda g_n^2)\Big)\\
	&\qquad+\Omega(\rho_n, \Z)-\frac{\rho_n}{n^{1/2}}\widetilde G_{n,\Z}(g_n)+o_p(n^{-1/2})\bigg|.
\end{align*}

 	Let $
	\mathcal E_{1n}
	=\mathcal E_n\cap\mathcal G_n\cap
	\left\{
	\sum_{j=1}^{3}
	\left(
	|\widetilde S_n(g_n^j)|
	+|\widetilde G_{n,\Z}(g_n^j)|
	\right)
	+\sqrt n|\hat\rho-\rho_n|
	\le n^{1/8}
	\right\}$.
	On $\mathcal E_{1n}$, the denominator below is bounded
	away from zero, and the preceding expansion gives
	\begin{align}
		\sqrt n(\widehat\rho-\rho_n)=
		\frac{
			-\Omega(\rho_n,\Z)
			+\rho_n n^{-1/2}\widetilde G_{n,\Z}(g_n)
		}{
			\rho_n\mathrm F_{\gamma_n}(g_n^2)
			-n^{-1/2}\widetilde S_n(g_nh_n)
		}
		+\mathcal R_{1n}=:\widetilde\sigma_n\widetilde R_n
		+\mathcal R_{1n}, 
		\nonumber
	\end{align}
	where
	\begin{align*}
		|\mathcal R_{1n}|
		\le
		\frac Cn
		\Big(
		1
		&+
		|\widetilde S_n(g_n)|
		+
		|\widetilde S_n(\lambda g_n^2)|
		+
		|\widetilde S_n(g_nh_n)|+
		|\widetilde G_{n,\Z}(g_n)|
		+
		\sqrt n|\widehat\rho-\rho_n|
		\Big)^3.
	\end{align*}

	Lemmas~\ref{resolventfluctuations} and~\ref{qtruncation}
	give $\mathbb P(\mathcal E_{1n}^c)=o(n^{-1/2})$.
	Thus, on $\mathcal E_{1n}$,
	\[
	|\mathcal R_{1n}|
	\le
	\frac Cn
	\Big[
	1+
	\sum_{j=1}^{3}
	\left(
	|\widetilde S_n(g_n^j)|
	+
	|\widetilde G_{n,\Z}(g_n^j)|
	\right)
	+
	\sqrt n|\widehat\rho-\rho_n|
	\Big]^3
	\le Cn^{-5/8}.
	\]

	Theorem~\ref{G4MT} and~\ref{limiting} guarantee a uniform Edgeworth expansion of the first order for $\widetilde{R}_n$. Applying Lemma~\ref{delta} by setting $U_n=\widetilde{R}_n$, $J_n=R_n-\widetilde{R}_n$, and $a_n=2Cn^{-5/8}$, we directly obtain
	\[
	\sup_x
	\left|
	\mathbb P(R_n\le x)-\mathbb P(\widetilde R_n\le x)
	\right|
	\le
	\mathbb P(\mathcal E_{1n}^c)
	+\sup_x\mathbb P(|\widetilde R_n-x|\le a_n)=o(n^{-1/2}).
	\]
	This proves Theorem~\ref{convergence}.

\subsection{Proof of Theorem~\ref{G4MT}  }

\subsubsection{Preliminary calculations}
Assume that $\Y=(\mathbf{y}_1,\dots, \mathbf{y}_n)'$ and $\Z=(\z_1,\dots,\z_{n})'$ are two independent random matrices satisfying Theorem~\ref{G4MT}. Recall the previous definitions of $\Omega(\rho_n,\Z)$ and $\Omega(\rho_n,\Y)$,
\begin{align*}
	\Omega(\rho_n,\Z)&=\frac{\rho_n}{\sqrt n}\Big(\mathrm{tr}\big[(\rho_n\I-n^{-1}\Z\V_2\V_2'\Z')^{-1}\big]-\widetilde\Z_1'(\rho_n\I-n^{-1}\Z\V_2\V_2'\Z')^{-1}\widetilde\Z_1\Big),\\
	\Omega(\rho_n,\Y)&=\frac{\rho_n}{\sqrt n}\Big(\mathrm{tr}\big[(\rho_n\I-n^{-1}\Y\V_2\V_2'\Y')^{-1}\big]-\widetilde\Z_1'(\rho_n\I-n^{-1}\Y\V_2\V_2'\Y')^{-1}\widetilde\Z_1\Big).
\end{align*}

Next, consider the eigendecomposition $n^{-1}\Y\V_2\V_2'\Y'=\U\bm{\Lambda}\U'$, where $\U$ is an $n\times n$ orthogonal matrix and $\bm{\Lambda}$ is the diagonal matrix containing all the eigenvalues of $n^{-1}\Y\V_2\V_2'\Y'$. Recalling that $\bm{\omega}=\U'\widetilde\Z_1$, we then obtain:
\begin{align*}
	\Omega(\rho_n,\Y)&=\frac{\rho_n}{\sqrt n}\Big(\mathrm{tr}\big[(\rho_n\I-n^{-1}\Y\V_2\V_2'\Y')^{-1}\big]-\widetilde\Z_1'(\rho_n\I-n^{-1}\Y\V_2\V_2'\Y')^{-1}\widetilde\Z_1\Big)\\\nonumber
	&=\frac{\rho_n}{\sqrt n}\Big(\mathrm{tr}\big[(\rho_n\I-\U\bm{\Lambda} \U')^{-1}\big]-\widetilde\Z_1'(\rho_n\I-\U\bm{\Lambda} \U')^{-1}\widetilde\Z_1\Big)\\\nonumber
	&=\frac{\rho_n}{\sqrt n}\Big(\mathrm{tr}\big[(\rho_n\I-\bm{\Lambda})^{-1}\big]-\bm{\omega}'(\rho_n\I-\bm{\Lambda} )^{-1}\bm{\omega}\Big)=-\rho_nS_n(g_n),
\end{align*}
where
\[g_n(\lambda)=(\rho_n-\lambda)^{-1}, \quad S_n(f)=n^{-1/2}\sum_{i=1}^nf(\lambda_i)(\omega_i^2-1).\]
In general, although $\widetilde{\Z}_1$ is independent of $\Y$, the transformed
variables $\omega_1,\ldots,\omega_n$ are not conditionally independent
given $\Y$ unless $\widetilde{\Z}_1$ is Gaussian. Then, we rewrite $\Omega_s(\widetilde\Z_1,\Z)$ and $\Omega_s(\widetilde\Z_1,\Y)$ as follows:
\begin{align*}
	\Omega_s(\widetilde\Z_1,\Z)&=-\frac{c_n}{\sqrt n}\Big(\mathrm{tr}\big[(\rho_n\I-n^{-1}\Z\V_2\V_2'\Z')^{-1}\big]-\widetilde\Z_1'(\rho_n\I-n^{-1}\Z\V_2\V_2'\Z')^{-1}\widetilde\Z_1\Big)\\
	&\quad+\frac{1}{n}\tau_n\Big(\mathrm{tr}\big[(n^{-1}\Z\V_2\V_2'\Z')(\rho_n\I-n^{-1}\Z\V_2\V_2'\Z')^{-2}\big]\\
	&\quad -\widetilde\Z_1'(n^{-1}\Z\V_2\V_2'\Z')(\rho_n\I-n^{-1}\Z\V_2\V_2'\Z')^{-2}\widetilde\Z_1\Big)\\
	&:=\Omega_{s_1}(\widetilde\Z_1,\Z)+\Omega_{s_{21}}(\widetilde\Z_1,\Z)+\Omega_{s_{22}}(\widetilde\Z_1,\Z),\\
	\Omega_{s}(\widetilde\Z_1,\Y)&=-\frac{c_n}{\sqrt n}\Big(\mathrm{tr}\big[(\rho_n\I-n^{-1}\Y\V_2\V_2'\Y')^{-1}\big]-\widetilde\Z_1'(\rho_n\I-n^{-1}\Y\V_2\V_2'\Y')^{-1}\widetilde\Z_1\Big)\\
	&\quad+\frac{1}{n}\tau_n\Big(\mathrm{tr}\big[(n^{-1}\Y\V_2\V_2'\Y')(\rho_n\I-n^{-1}\Y\V_2\V_2'\Y')^{-2}\big]\\
	&\quad -\widetilde\Z_1'(n^{-1}\Y\V_2\V_2'\Y')(\rho_n\I-n^{-1}\Y\V_2\V_2'\Y')^{-2}\widetilde\Z_1\Big)\\
	&:=\Omega_{s_1}(\widetilde\Z_1,\Y)+\Omega_{s_{21}}(\widetilde\Z_1,\Y)+\Omega_{s_{22}}(\widetilde\Z_1,\Y),
\end{align*}
where $c_n=1+n^{-1/2}\tau_n\mathrm{F}_{\gamma_n}(\lambda g_n^3)/\mathrm{F}_{\gamma_n}(g_n^2)$.
The following identities provide the complete row increments. The required distributional comparison at order $o(n^{-1/2})$ is established below. Define the following notations:
\begin{align*}
	&\Z_{k0}=(\z_1,\dots,\z_{k-1}, 0, \z_{k+1},\dots, \z_{n})',~~\Z_{0k}=(\z_1,\dots,\z_{k-1}, \z_{k+1},\dots, \z_{n})',\\
	&\Y_{k0}=(\z_1,\dots,\z_{k-1}, 0, \y_{k+1},\dots, \y_n)',~~\Y_{0k}=(\z_1,\dots,\z_{k-1}, \y_{k+1},\dots, \y_n)',\\
	& \Y_k=(\z_1,\dots,\z_{k-1},\y_{k},\dots, \y_n)',\\
	&\theta_k=\rho_n-\frac 1n\z_k'\V_2\V_2'\z_k-\frac {1}{n^2} \z_k'\V_2\V_2' \Y_{0k}'\Big(\rho_n\I_{n-1}-\frac 1n\Y_{0k}\V_2\V_2'\Y_{0k}'\Big)^{-1}\Y_{0k}\V_2\V_2'\z_k,\\
	&\quad=\rho_n\bigg(1-\frac 1n\z_k'\V_2\Big(\rho_n\I_{p}-\frac 1n\V_2'\Y_{0k}'\Y_{0k}\V_2\Big)^{-1}\V_2'\z_k\bigg):=\rho_n\eta_k,\\
	&\bar\eta_k=1-\frac 1n \mathrm{tr}\Big[\Big(\rho_n\I_{p}-\frac 1n\V_2'\Y_{0k}'\Y_{0k}\V_2\Big)^{-1}\Big],\\
	&\s_k= \V_2\V_2'\Y_{0k}'\Big(\rho_n\I_{n-1}-\frac 1n\Y_{0k}\V_2\V_2'\Y_{0k}'\Big)^{-1}\Z_{0k}\V_1,\\
	&\epsilon_k=\frac 1n \Big\{\z_k'\V_2\Big(\rho_n\I_{p}-\frac 1n\V_2'\Y_{0k}'\Y_{0k}\V_2\Big)^{-1}\V_2'\z_k\\
	&\qquad\qquad-\mathrm{tr}\Big[\Big(\rho_n\I_{p}-\frac 1n\V_2'\Y_{0k}'\Y_{0k}\V_2\Big)^{-1}\Big]\Big\},\\
	&\m_k=\V_2\V_2'\Y_{0k}'\Big(\rho_n\I_{n-1}-\frac 1n\Y_{0k}\V_2\V_2'\Y_{0k}'\Big)^{-1},\\ &\q_k=\V_2\V_2'\Y_{0k}'\Big(\rho_n\I_{n-1}-\frac 1n\Y_{0k}\V_2\V_2'\Y_{0k}'\Big)^{-1}\Y_{0k}\V_2\V_2',\\
	&\e_k=\V_2\V_2'\Y_{0k}'\Big(\rho_n\I_{n-1}-\frac 1n\Y_{0k}\V_2\V_2'\Y_{0k}'\Big)^{-2},\\ &\p_k=\V_2\V_2'\Y_{0k}'\Big(\rho_n\I_{n-1}-\frac 1n\Y_{0k}\V_2\V_2'\Y_{0k}'\Big)^{-2}\Y_{0k}\V_2\V_2'.
\end{align*}

For the insertion of the background component generated by $\y_k$, define
\begin{align*}
	\epsilon_{ky}
	&:=
	\frac{1}{n}
	\bigg\{
	\y_k'\V_2
	\left(
	\rho_n\I_p
	-\frac{1}{n}\V_2'\Y_{0k}'\Y_{0k}\V_2
	\right)^{-1}
	\V_2'\y_k
	\nonumber\\
	&\hspace{3.5cm}
	-\operatorname{tr}
	\left(
	\rho_n\I_p
	-\frac{1}{n}\V_2'\Y_{0k}'\Y_{0k}\V_2
	\right)^{-1}
	\bigg\},\\
	\eta_{ky}
	&:=\bar{\eta}_k-\epsilon_{ky},\qquad
	\theta_{ky}:=\rho_n\eta_{ky}.
\end{align*}
Next, through straightforward calculations, we obtain the following result:
\begin{align*}
	&\quad\Omega_{s_1}( \widetilde{\Z}_1,\Y_{k+1})-\Omega_{s_1}(\Z_{k0}\V_1,\Y_{k0})\\
	&=\frac{-c_n }{\theta_k\sqrt n}\bigg(1-\V_1'\z_k\z_k'\V_1-\frac 1n\V_1'\Z_{0k}'\Big(\rho_n\I_{n-1}-\frac 1n\Y_{0k}\V_2\V_2'\Y_{0k}'\Big)^{-1}\Y_{0k}\V_2\V_2'\z_k\z_k'\V_1\\
	&\quad+\frac {1}{n^2}\z_k'\V_2\V_2'\Y_{0k}'\Big(\rho_n\I_{n-1}-\frac 1n\Y_{0k}\V_2\V_2'\Y_{0k}'\Big)^{-2}\Y_{0k}\V_2\V_2'\z_k\\
	&\quad-\frac 1n\V'_1\z_k\z_k'\V_2\V_2'\Y_{0k}'\Big(\rho_n\I_{n-1}-\frac 1n\Y_{0k}\V_2\V_2'\Y_{0k}'\Big)^{-1}\Z_{0k}\V_1\\
	&\quad-\frac{1}{n^2}\V_1'\Z_{0k}'\Big(\rho_n\I_{n-1}-\frac 1n\Y_{0k}\V_2\V_2'\Y_{0k}'\Big)^{-1} \Y_{0k}\V_2\V_2'\z_k\z_k'\V_2\V_2'\Y_{0k}'\\
	&\quad\times\Big(\rho_n\I_{n-1}-\frac 1n\Y_{0k}\V_2\V_2'\Y_{0k}'\Big)^{-1}\Z_{0k}\V_1\bigg)+	\frac{c_n}{\rho_n\sqrt n}=	\frac{c_n}{\rho_n\sqrt n}+\sum_{i=1}^5\sum_{j=0}^2\alpha_{kij},
\end{align*}
where the expressions for $\alpha_{kij}$ are provided in the Appendix \ref{appendix}.	Similarly, for the $y_k$-background insertion, we have
\begin{align*}
	&\Omega_{s_1}(\widetilde{\Z}_1,\Y_k)
	-
	\Omega_{s_1}(\Z_{k0}\V_1,\Y_{k0})
	=\frac{c_n}{\rho_n\sqrt n}+
	\sum_{i=1}^{5}\sum_{j=0}^{2}\alpha_{kijy}.
\end{align*}
Here, the subscript $y$ indicates that the background component
$\z_k'\V_2$ is replaced by $\y_k'\V_2$, whereas the spike component
$\z_k'\V_1$ is kept unchanged. Using the same approach for the two correction terms, we obtain
\begin{align*}
	&\Omega_{s_{21}}(\widetilde{\Z}_1,\Y_{k+1})
	-
	\Omega_{s_{21}}(\Z_{k0}\V_1,\Y_{k0})
	=
	\sum_{i=1}^{12}\sum_{j=0}^{2}\beta_{kij},\\
	&\Omega_{s_{21}}(\widetilde{\Z}_1,\Y_k)
	-
	\Omega_{s_{21}}(\Z_{k0}\V_1,\Y_{k0})
	=
	\sum_{i=1}^{12}\sum_{j=0}^{2}\beta_{kijy},\\
	&\Omega_{s_{22}}(\widetilde{\Z}_1,\Y_{k+1})
	-
	\Omega_{s_{22}}(\Z_{k0}\V_1,\Y_{k0})
	=
	-\sum_{i=1}^{25}\sum_{j=0}^{2}\Xi_{kij},\\
	&\Omega_{s_{22}}(\widetilde{\Z}_1,\Y_k)
	-
	\Omega_{s_{22}}(\Z_{k0}\V_1,\Y_{k0})
	=
	-\sum_{i=1}^{25}\sum_{j=0}^{2}\Xi_{kijy}.
\end{align*}
The explicit expressions for
$\alpha_{kij}$, $\alpha_{kijy}$,
$\beta_{kij}$, $\beta_{kijy}$,
$\Xi_{kij}$, and $\Xi_{kijy}$
are provided in Appendix~\ref{appendix}. Subtracting the two expansions and then taking
expectations, we obtain:
\begin{align}\label{eq:original-row-yz}
	&\qquad\mathbb{E}\Big(\Omega_{s}( \widetilde{\Z}_1,\Y_{k+1})-\Omega_{s}( \widetilde{\Z}_1,\Y_{k})\Big)\\\nonumber
	&=\mathbb{E}\sum_{i=1}^5\sum_{j=0}^2(\alpha_{kij}-\alpha_{kijy})+\mathbb{E}\sum_{i=1}^{12}\sum_{j=0}^2(\beta_{kij}-\beta_{kijy})-\mathbb{E}\sum_{i=1}^{25}\sum_{j=0}^2(\Xi_{kij}-\Xi_{kijy}).
\end{align}

\begin{remark}
	\label{rem:explicit-terms}
	The $\alpha$, $\beta$ and $\Xi$ expansions, together with the
	trace increments, will be used below to compute the first three
	signed conditional moments of the complete row increments.
\end{remark}

\subsubsection{Completion of the proof}
\label{sec:T53-quantitative-completion}

	By the definitions of $S_n$, $\widetilde S_n$, and
	$\widetilde G_{n,\mathbf B}$, the statistic $T_{n,\mathbf B}(x)$ can be written as
	\begin{align}\label{TnBdef}
		T_{n,\mathbf B}(x)
		&:=
		\frac{1}{\sqrt n}
		\left[
		\widetilde{\mathbf Z}_1'
		g_n\left(
		\frac1n\mathbf B\mathbf V_2\mathbf V_2'\mathbf B'
		\right)
		\widetilde{\mathbf Z}_1
		-
		n\mathrm F_{\gamma_n}(g_n)
		\right]+
		\frac{\rho_n^{-1}\widetilde\sigma_nx}{n}\\\nonumber
		&\qquad~~ \times
		\left\{
		\widetilde{\mathbf Z}_1'
		(g_nh_n)\left(
		\frac1n\mathbf B\mathbf V_2\mathbf V_2'\mathbf B'
		\right)
		\widetilde{\mathbf Z}_1
		-
		\operatorname{tr}
		(g_nh_n)\left(
		\frac1n\mathbf B\mathbf V_2\mathbf V_2'\mathbf B'
		\right)
		\right\},
	\end{align}		
	and let
	\[
	\psi_{n,\B,x}(t)
	=
	\mathbb E\exp\left\{
	\frac{itT_{n,\B}(x)}{\sqrt{\kappa_{21}}}
	\right\}.
	\]
	For each $k$, define $T_{k0}(x)$ by the right-hand side of
	\eqref{TnBdef}, with
	$\mathbf B=\mathbf Y_{k0}$ and
	$\widetilde{\mathbf Z}_1=\mathbf Z_{k0}\mathbf V_1$, and set
	\[
	\Delta_{kz}(x)
	=
	T_{n,\mathbf Y_{k+1}}(x)-T_{k0}(x),
	\qquad
	\Delta_{ky}(x)
	=
	T_{n,\mathbf Y_k}(x)-T_{k0}(x).
	\]
	
	To relate these complete increments to the explicit row
	expansions obtained above, write
	\[
	T_{n,\mathbf B}(x)
	=
	\Omega_s(\widetilde{\mathbf Z}_1,\mathbf B)
	+
	L_{n,\mathbf B}, \qquad
	L_{n,\mathbf B}
	=
	\frac1{\sqrt n}
	\left\{
	\operatorname{tr}g_n\left(
	\frac1n\mathbf B\mathbf V_2\mathbf V_2'\mathbf B'
	\right)
	-
	n\mathrm F_{\gamma_n}(g_n)
	\right\}.
	\]
	Accordingly,
	\[
	\Delta_{kb}(x)
	=
	\Delta_{kb}^{\Omega}(x)
	+
	\Delta_{kb}^{\operatorname{tr}},
	\qquad b\in\{z,y\}.
	\]
	The preceding expansions give
	\[
	\begin{aligned}
		\Delta_{kz}(x)
		=\frac{c_n}{\rho_n\sqrt n}
		+\sum_{i=1}^{5}\sum_{j=0}^{2}\alpha_{kij}
		+\sum_{i=1}^{12}\sum_{j=0}^{2}\beta_{kij}-\sum_{i=1}^{25}\sum_{j=0}^{2}\Xi_{kij}
		+\Delta_{kz}^{\operatorname{tr}},
	\end{aligned}
	\]
	with the analogous identity at the $y$-endpoint obtained by
	using the corresponding $y$-subscripted terms.
	The remaining trace increments are
	\begin{align*}
		\Delta_{kz}^{\operatorname{tr}}
		&=
		\frac1{\sqrt n}
		\left\{
		\operatorname{tr}g_n\left(
		\frac1n\mathbf Y_{k+1}\mathbf V_2\mathbf V_2'
		\mathbf Y_{k+1}'
		\right)
		-
		\operatorname{tr}g_n\left(
		\frac1n\mathbf Y_{k0}\mathbf V_2\mathbf V_2'
		\mathbf Y_{k0}'
		\right)
		\right\},\\
		\Delta_{ky}^{\operatorname{tr}}
		&=
		\frac1{\sqrt n}
		\left\{
		\operatorname{tr}g_n\left(
		\frac1n\mathbf Y_k\mathbf V_2\mathbf V_2'
		\mathbf Y_k'
		\right)
		-
		\operatorname{tr}g_n\left(
		\frac1n\mathbf Y_{k0}\mathbf V_2\mathbf V_2'
		\mathbf Y_{k0}'
		\right)
		\right\}.
	\end{align*}
	Conditional on the deleted arrays, $\Delta_{kz}^{\operatorname{tr}}$ and $\Delta_{ky}^{\operatorname{tr}}$ share the same distribution because $\mathbf{z}_k$ and $\mathbf{y}_k$ are identically distributed and independent of the deleted arrays. Consequently, the pure trace terms cancel in the signed comparison between $z_k$ and $y_k$. We include the mixed terms involving $\Delta_{kb}^{\operatorname{tr}}$ in the remainder estimates below, which amount to $o(n^{-1/2})$ after summation.
	Hence,
	\begin{equation}
		\psi_{n,\Z,x}(t)-\psi_{n,\Y,x}(t)
		=
		\sum_{k=1}^n
		\mathbb E\left[
		e^{itT_{k0}(x)/\sqrt{\kappa_{21}}}
		\left\{
		e^{it\Delta_{kz}(x)/\sqrt{\kappa_{21}}}
		-
		e^{it\Delta_{ky}(x)/\sqrt{\kappa_{21}}}
		\right\}
		\right].
		\label{telescopeclosed}
	\end{equation}
	
	\begin{lemma}
		\label{unifreplace}
		There exist sequences
		$$ \varepsilon_n\to 0, \qquad L_n\to\infty, \qquad B_n\to\infty, \qquad L_n=o(\sqrt{n}), $$
		and events $\mathcal{E}_{2n}$ such that
		\begin{equation}
			\mathbb{P}(\mathcal{E}_{2n}^c)=o(n^{-1/2}).
			\label{badevent}
		\end{equation}
		On $\mathcal{E}_{2n}$, all resolvents and the finitely many derivatives in $\rho_n$ appearing in the row expansion have uniformly bounded operator norms. Moreover, uniformly for $|x|\le n^{1/4}$, the following assertions hold.
		
		First, for $|t|\le L_n$,
		\begin{align}\label{uniformlow}
			&\qquad\psi_{n,\Z,x}(t)-\psi_{n,\Y,x}(t)\\\nonumber
			&=\frac{e^{-t^2/2}}{\sqrt n}\Bigg\{
			\frac{it}{\sqrt{\kappa_{21}}}
			\bigl[\mathfrak a_{11}-\mu(g_{1n})\bigr]
			+\frac{(it)^2}{2\kappa_{21}}\mathfrak b_1+\frac{(it)^3}{6\kappa_{21}^{3/2}}
			\bigl[\mathfrak a_{01}-\kappa_{31}\bigr]\Bigg\}
			+\frac{|t|}{\sqrt n}\,o(1).
		\end{align}
		
		Second,
		\begin{equation}
			\int_{L_n<|t|\le\epsilon\sqrt n}
			\frac{|\psi_{n,\Z,x}(t)-\psi_{n,\Y,x}(t)|}{|t|}\,dt
			=o(n^{-1/2})
			\label{eq:T53-uniform-middle}
		\end{equation}
		for some fixed $\epsilon>0$, and
		\begin{align}
			\int_{\epsilon\sqrt n<|t|\le B_n\sqrt n}
			\frac{|\psi_{n,\Z,x}(t)|+|\psi_{n,\Y,x}(t)|}{|t|}\,dt
			&=o(n^{-1/2}),
			\label{eq:T53-uniform-high}\\
			\sup_y
			\mathbb P\left\{
			y<
			\frac{T_{n,\Y}(x)}{\sqrt{\kappa_{21}}}
			\le
			y+(B_n\sqrt n)^{-1}
			\right\}
			&=o(n^{-1/2}).
			\label{unconcentration}
		\end{align}
	\end{lemma}
	
	\begin{proof}
		Let $\mathcal G_n$ be the event of Lemma~\ref{qtruncation}
		for all prescribed arrays and independent-copy versions.
		For every deterministic block or finite active-row set $I$
		used below, let $\mathbb E_I$ integrate those rows and define
		\[
		\mathcal E_{2n}
		=\mathcal G_n\cap\bigcap_I
		\left\{\mathbb E_I\mathbf1_{\mathcal G_n^c}
		\le e^{-n^{1/2-\alpha}}\right\}.
		\]
		There are polynomially many such sets and arrays.
		Markov's inequality gives
		$\mathbb P(\mathcal E_{2n}^c)=O_D(n^{-D})$ for every fixed $D>0$,
		and hence \eqref{badevent}.
		Use bounded inverse extensions, and localize each deleted
		coefficient using only the corresponding row-deleted array,
		without changing it on $\mathcal G_n$.
		The resulting expressions are polynomially bounded in $n$.
		Thus localization errors vanishing on $\mathcal G_n$ have
		conditional $L^p$ norms $O_D(n^{-D})$ for every fixed $p,D$,
		at the outside-$I$ configurations admitted by $\mathcal E_{2n}$.
		All subsequent estimates are uniform for $|x|\le n^{1/4}$.
		
		\medskip
		\noindent
		\textbf{Proof of \eqref{uniformlow}.}
		We use the third-order Taylor expansion with remainder
		\[
		r_4(u)
		:=e^{iu}-1-iu-\frac{(iu)^2}{2}-\frac{(iu)^3}{6},
		\qquad u\in\mathbb R.
		\]
		Substituting
		$u=t\Delta_{kb}(x)/\sqrt{\kappa_{21}}$, $b\in\{z,y\}$,
		into \eqref{telescopeclosed}, we define the
		Taylor remainder contribution from $\mathcal E_{2n}$ by
		\[
		\begin{aligned}
			\mathcal R_{4n}(t,x)
			:=\sum_{k=1}^n\mathbb E\bigg[
			\mathbf1_{\mathcal E_{2n}}
			e^{itT_{k0}(x)/\sqrt{\kappa_{21}}}\Big\{
			r_4\Big(\frac{t\Delta_{kz}(x)}{\sqrt{\kappa_{21}}}\Big)
			-r_4\Big(\frac{t\Delta_{ky}(x)}{\sqrt{\kappa_{21}}}\Big)
			\Big\}
			\bigg].
		\end{aligned}
		\]
		For $b\in\{z,y\}$, write $\mathbf b_k=\mathbf z_k$ or
		$\mathbf y_k$, respectively. Applying the block matrix inversion formula and differentiating with respect to $\rho$, we obtain the following at $\rho=\rho_n$, uniformly for $|x|\le n^{1/4}$:
		\[
		\sqrt n\,|\Delta_{kb}(x)|
		\le C\Big\{
		1+|\mathbf z_k'\mathbf V_1|^2
		+\Big|\frac{\mathbf s_k'\mathbf b_k}{n}\Big|^2
		+\Big|\frac{\mathbf b_k'\mathbf e_k
			\mathbf Z_{0k}\mathbf V_1}{n}\Big|^2
		\Big\},\qquad\text{on }\mathcal E_{2n}.
		\]
		Indeed, the increment of the uncentered resolvent quadratic form
		at the $z$-endpoint is exactly
		\[
		\frac1{\theta_k}
		\Big(\mathbf z_k'\mathbf V_1
		+\frac{\mathbf s_k'\mathbf z_k}{n}\Big)^2.
		\]
		Differentiation introduces
		$n^{-1}\mathbf z_k'\mathbf e_k\mathbf Z_{0k}\mathbf V_1$.
		All denominator factors, their required derivatives, and the trace
		increments are bounded on $\mathcal E_{2n}$. Together with
		$|c_n|+|\tau_n|/\sqrt n\le C$, this proves the displayed bound;
		the $y$-endpoint is treated in the same way.

		The coefficient vectors depend only on the deleted arrays and satisfy
		\[
		\left|\frac{\mathbf s_k}{n}\right|
		+\left|\frac{\mathbf e_k\mathbf Z_{0k}\mathbf V_1}{n}\right|
		\le C\frac{|\mathbf Z_{0k}\mathbf V_1|}{\sqrt n}.
		\]
		We next show that the family $\{H_{nkb}^3\}$ is uniformly integrable, where
		\[
		H_{nkb} := \mathbf{1}_{\mathcal E_{2n}} \sup_{|x|\le n^{1/4}} \frac{\sqrt{n}\,|\Delta_{kb}(x)|}{\sqrt{\kappa_{21}}}.
		\]
		For the moment estimates, zeroing the coefficient vectors whenever the spectral bounds omitting row $k$ fail ensures they remain independent of $\mathbf{z}_k$ and $\mathbf{y}_k$. Fixed-level truncation and Rosenthal's inequality give
		uniform integrability of the sixth powers of
		bounded-norm linear forms. Jensen's inequality similarly guarantees uniform integrability for $(n^{-1/2}|\mathbf{Z}_{0k}\mathbf{V}_1|)^6$. Combined with prior bounds, this controls large coefficient norms. For bounded norms, conditioning on the remaining rows and applying the linear combination result suffices. Because $\kappa_{21}$ is bounded away from zero, these controls collectively establish the uniform integrability of $\{H_{nkb}^3\}_{n,k,b}$.

		For every real $u$, Taylor's expansion yields $|r_4(u)|\le C|u|^3\min\{|u|,1\}$. Owing to the uniform integrability of $\{H_{nkb}^3\}_{n,k,b}$, we deduce
		\[
		\delta_n := \max_{1\le k\le n,\,b\in\{z,y\}} \mathbb{E}\left[H_{nkb}^3 \min\{n^{-1/4}H_{nkb},1\}\right] \to 0.
		\]
		Consequently, for $|t|\le n^{1/4}$, we obtain
		\begin{align*}
			\sup_{|x|\le n^{1/4}}|\mathcal R_{4n}(t,x)|
			\le\sum_{k=1}^n\sum_{b\in\{z,y\}}
			\sup_{|x|\le n^{1/4}}
			\mathbb E\left[\mathbf 1_{\mathcal E_{2n}}
			\left|r_4\left(\frac{t\Delta_{kb}(x)}
			{\sqrt{\kappa_{21}}}\right)\right|\right]\le\frac{C\delta_n}{\sqrt n}|t|^3.
		\end{align*}
		Choose $\varepsilon_n\to0$ such that $\varepsilon_n\ge C\delta_n$. We then select $L_n\to\infty$ sufficiently slowly to ensure $L_n\le n^{1/4}$ and $\delta_nL_n^3\to0$. Then,
		\[
		\sup_{|x|\le n^{1/4}}
		\int_{0<|t|\le L_n}
		\frac{|\mathcal R_{4n}(t,x)|}{|t|}\,dt
		\le\frac{C\delta_nL_n^3}{\sqrt n}
		=o(n^{-1/2}).
		\]
		The preceding pointwise estimate shows that the Taylor
		remainder can be included in the error term of
		\eqref{uniformlow}. To complete the proof, we next evaluate
		the first-, second-, and third-order terms and control
		the remaining errors.

		Define
		\[
		\mathbf H_k=\mathbf V_2
		\left(\rho_n\mathbf I_p-n^{-1}\mathbf V_2'\mathbf Y_{0k}'
		\mathbf Y_{0k}\mathbf V_2\right)^{-1}\mathbf V_2',
		\qquad
		\mathbf a_k=\mathbf s_k/n,
		\qquad
		\mathbf r=\mathbf V_2\mathbf V_2'\big((\mathbf V_1)_i^3\big)_i.
		\]
		Then
		\[
		\mathbf H_k\mathbf V_1=\mathbf0,\qquad
		\mathbf a_k'\mathbf V_1=0,\qquad
		|\mathbf r|^2
		=\mathcal V_1^{(6)}-(\mathcal V_1^{(4)})^2.
		\]
		All estimates are uniform in $k$ and under the prescribed
		row deletions.
		
		At $x=0$, $c_n=1$ and the $\beta$- and $\Xi$-terms vanish.
		All $\alpha$-terms in the next two displays are evaluated at $x=0$.
		The trace increment cancels the background part of the original expansion:
		\[
		\frac{1}{\rho_n\sqrt n}
		+\sum_{j=0}^2\alpha_{k1j}
		+\Delta_{kz}^{\operatorname{tr}}
		=\frac{1}{\rho_n\sqrt n\,\eta_k},
		\]
		and the same identity holds at the $y$-endpoint. Hence
		\begin{equation}\label{eq:short-complete-increment}
			\begin{aligned}
				\Delta_{kz}(0)
				=\frac{1}{\rho_n\sqrt n\,\eta_k}
				+\sum_{i=2}^5\sum_{j=0}^2\alpha_{kij},\qquad
				\Delta_{ky}(0)
				=\frac{1}{\rho_n\sqrt n\,\eta_{ky}}
				+\sum_{i=2}^5\sum_{j=0}^2\alpha_{kijy}.
			\end{aligned}
		\end{equation}
		Write $S_k=\mathbf z_k'\mathbf V_1$ and
		$\epsilon_{kb}=n^{-1}(\mathbf b_k'\mathbf H_k\mathbf b_k
		-\operatorname{tr}\mathbf H_k)$, for $b\in\{z,y\}$;
		$\mathbb E_k$ averages over the active rows only.

		We first record the coefficient bounds used below. For deterministic
		unit vectors $\mathbf u$, row deletion gives
		\begin{equation}
			|\mathbb E\mathbf u'\mathbf a_k|\le Cn^{-1},\qquad
			\mathbb E|\mathbf u'\mathbf a_k|^q\le Cn^{-q/2}\ (q=2,3),\qquad
			\mathbb E|\mathbf a_k|^6\le C.
			\label{eq:short-score-bounds}
		\end{equation}
		For a background matrix $\mathbf B$, define
		\[
		\mathbf H
		=\mathbf V_2
		\left(
		\rho_n\mathbf I_p
		-\frac1n\mathbf V_2'\mathbf B'\mathbf B\mathbf V_2
		\right)^{-1}\mathbf V_2'.
		\]
		Let $\mathbf H^{(j)}$ for the same matrix with row $j$
		of $\mathbf B$ deleted, keeping $n$ unchanged.
		Indeed, for background rows $\mathbf x_j=\mathbf V_2\mathbf V_2'\mathbf b_j$ and
		scores $\xi_j=\mathbf z_j'\mathbf V_1$, write
		$\mathbf a=\mathbf H n^{-1}\sum_j\mathbf x_j\xi_j$.
		Deleting row $j$, with
		$\eta_j=1-n^{-1}\mathbf x_j'\mathbf H^{(j)}\mathbf x_j$, yields
		\[
		\mathbf a-\mathbf a^{(j)}=
		\frac{\mathbf H^{(j)}\mathbf x_j
			(\xi_j+\mathbf x_j'\mathbf a^{(j)})}{n\eta_j},
		\qquad
		\mathbf H-\mathbf H^{(j)}=
		\frac{\mathbf H^{(j)}\mathbf x_j\mathbf x_j'\mathbf H^{(j)}}{n\eta_j}.
		\]
		The scalar increments have $q$th moments $O(n^{-q})$, $q=2,3$;
		the row martingale inequalities give the centered bounds in
		\eqref{eq:short-score-bounds}.
		For the mean, expand the denominator in
		$n^{-1}\sum_j\mathbb E[\xi_j\mathbf u'\mathbf H^{(j)}
		\mathbf x_j/\eta_j]$.
		The constant term vanishes since $\mathbb E(\xi_j\mathbf x_j)=0$.
		The linear term is bounded by
		\[
		Cn^{-1}\Big|\sum_i (\mathbf V_1)_i(\mathbf H^{(j)}\mathbf u)_i(\mathbf H^{(j)})_{ii}\Big|
		\le Cn^{-1},
		\]
		and the remainder has the same order by sixth moments.
		Also $|\mathbf a|\le C|\boldsymbol\xi|/\sqrt n$.
		The resulting sixth powers are uniformly integrable, and
		$\sum_i\mathbb E|(\mathbf a_k)_i|^3=O(n^{-1/2})$ implies
		$\max_i|(\mathbf a_k)_i|\to0$ in probability.
		
		For two arrays sharing any subset of their rows and using independent
		copies elsewhere, the same identities give
		\begin{equation}
			\mathbb E|\mathbf u'\mathbf H\widetilde{\mathbf a}|^2\le Cn^{-1}.
			\label{eq:short-mixed-bound}
		\end{equation}
		To see this, delete the active block from both arrays. The scalar
		increments again have second moments $O(n^{-2})$, giving variance
		$O(n^{-1})$. In the mean, substitute
		$\widetilde{\mathbf a}=n^{-1}\widetilde{\mathbf H}
		\sum_j\widetilde{\mathbf x}_j\widetilde\xi_j$ and replace the
		normalized quadratic forms in the deletion formulas by their traces.
		The leading terms vanish by $\mathbb E(\xi_j\mathbf x_j)=0$;
		the quadratic-form variance bound gives an $O(n^{-1/2})$ error.
		This proves \eqref{eq:short-mixed-bound}.

		Let $m_n=(l-1+\gamma_n)^{-1}$, $e_n=l/\rho_n$, and
		$\mathbf G_k=(\rho_n\mathbf I-n^{-1}\mathbf Y_{0k}'\mathbf Y_{0k})^{-1}$.
		Theorem~2.5, (2.15), and Remark~2.7 of \cite{alex2014isotropic},
		after translating the normalization, give
		\[
		\|\mathbf u'\mathbf G_k\mathbf w-m_n\mathbf u'\mathbf w\|_2=o(1)
		\]
		uniformly in $k$ and for deterministic vectors of bounded norm.
		Under the finite sixth-moment assumption, apply the result first
		to entries truncated at a fixed level, recentered and standardized;
		then remove this truncation using \cite{bai1993limit}, the resolvent
		identity, and the bounded spectral extensions. Fixed changes in
		the numbers of rows and columns do not affect the limit.
		On $\mathcal E_{2n}$,
		\[
		\mathbf H_k=\mathbf G_k-
		\frac{\mathbf G_k\mathbf V_1\mathbf V_1'\mathbf G_k}
		{\mathbf V_1'\mathbf G_k\mathbf V_1}.
		\]
		The denominator tends in mean square to $m_n$, bounded away from zero.
		Applying the preceding bilinear estimate, then taking diagonal entries
		and their normalized trace, gives
		\begin{equation}\label{eq:short-resolvent-limits}
			\bar\eta_k=e_n+o(1),\quad
			(\mathbf H_k)_{ii}=m_n\left(1-(\mathbf V_1)_i^2\right)+o(1),\quad
			\mathbf u'\mathbf H_k\mathbf w
			=m_n\mathbf u'\mathbf V_2\mathbf V_2'\mathbf w+o(1).
		\end{equation}
		Here $1-\gamma_nm_n=e_n$, and all errors tend to zero in mean square,
		uniformly in $k,i$ and for deterministic vectors of bounded norm.

		Throughout this calculation, all $\alpha$-terms are evaluated at $x=0$.
		Write $\Delta_{kb}=\Delta_{kb}(0)$ and
		$\mathbf d_k=((\mathbf H_k)_{ii})_i$.
		Expand the first three powers of the sums in
		\eqref{eq:short-complete-increment}, retaining denominator
		orders through $3-r$ in the $r$th power, $r=1,2,3$.
		At order $n^{-3/2}$, the first signed conditional moment retains
		\[
		\mathbb E_k\bigl[
		\alpha_{k51}-\alpha_{k51y}+\alpha_{k52}-\alpha_{k52y}
		\bigr],
		\]
		and the second retains
		\[
		2\mathbb E_k\bigl[
		\alpha_{k50}(\alpha_{k20}+\alpha_{k30}+\alpha_{k51})
		-\alpha_{k50y}(\alpha_{k20y}+\alpha_{k30y}+\alpha_{k51y})
		\bigr].
		\]
		The discarded contributions to these signed moments are
		$o(n^{-3/2})$ in $L^1$ by the estimates below.
		These terms are evaluated using the following row-moment identities:
		\begin{align*}
			&\mathbb E_k[(S_k^2-1)\epsilon_{kz}]
			=\frac\beta n\sum_i (\mathbf V_1)_i^2(\mathbf H_k)_{ii},\quad
			\mathbb E_k[(S_k^2-1)\epsilon_{kz}^2]
			=\frac{2\Gamma^2}{n^2}(\mathbf V_1'\mathbf d_k)^2+O(n^{-3/2}),\\
			&\mathbb E_k[S_k^4\epsilon_{kz}]
			=\frac{4\Gamma^2}{n}\mathcal V_1^{(3)}\mathbf V_1'\mathbf d_k+O(n^{-1}),\\
			&\mathbb E_k[(S_k+\mathbf a_k'\mathbf z_k)^4
			-(S_k+\mathbf a_k'\mathbf y_k)^4]
			=\beta\Big(4\mathbf r'\mathbf a_k
			+6\sum_i (\mathbf V_1)_i^2(\mathbf a_k)_i^2+4\sum_i (\mathbf V_1)_i(\mathbf a_k)_i^3\Big).
		\end{align*}
		By \eqref{eq:short-score-bounds}, each of the last two sums
		has expected absolute value at most $Cn^{-1}$.
		Equations \eqref{eq:short-score-bounds}--\eqref{eq:short-resolvent-limits}
		also give
		\[
		\mathbb E\Big|\sum_i (\mathbf V_1)_i(\mathbf a_k)_i(\mathbf H_k)_{ii}\Big|=o(1),
		\qquad \mathbb E|\mathbf a_k'\mathbf d_k|=o(\sqrt n),
		\]
		which control the other cross terms.
		In the sixth-moment difference, the paired terms cancel and the
		remaining terms vanish by $\max_i|(\mathbf a_k)_i|\to0$ and uniform integrability.
		Finally, with $B_{kb}=S_k+\mathbf a_k'\mathbf b_k$, the normalized
		denominator remainders for the first two moments are bounded by
		\[
		C|B_{kb}|^{2j}|\sqrt n\epsilon_{kb}|^{3-j}
		\min(|\epsilon_{kb}|,1),\qquad j=1,2.
		\]
		Their expectations tend to zero by H\"older's inequality and the
		uniform integrability of $|B_{kb}|^6$ and
		$|\sqrt n\epsilon_{kb}|^3$. 
		
		Combining these identities with the preceding remainder
		estimates yields
		\begin{align}\label{eq:short-increment-moments}
			&\mathbb E\bigg|
			\mathbb E_k(\Delta_{kz}-\Delta_{ky})
			-\frac{\beta\sum_i(\mathbf V_1)_i^2(\mathbf H_k)_{ii}}
			{n^{3/2}\rho_n\bar\eta_k^2}
			-\frac{2\Gamma^2(\mathbf V_1'\mathbf d_k)^2}
			{n^{5/2}\rho_n\bar\eta_k^3}
			\bigg|
			=o(n^{-3/2}),\\\nonumber
			&\mathbb E\bigg|
			\mathbb E_k(\Delta_{kz}^2-\Delta_{ky}^2)
			-\frac{4\beta\mathbf r'\mathbf a_k}{nl^2}
			-\frac{8\Gamma^2\mathcal V_1^{(3)}\mathbf V_1'\mathbf d_k}
			{n^2\rho_n^2\bar\eta_k^3}
			\bigg|
			=o(n^{-3/2}),\\\nonumber
			&\mathbb E\Big|
			\mathbb E_k(\Delta_{kz}^3-\Delta_{ky}^3)
			\Big|
			=o(n^{-3/2}).
		\end{align}
		
		For each fixed $K>0$, a second-order Taylor expansion using
		\eqref{eq:short-increment-moments} and
		$\sum_{k,b}\mathbb E|\Delta_{kb}|^3=O(n^{-1/2})$
		bounds the difference between the characteristic functions
		for each array arising in the successive row replacements
		and for the independent background by $C_Kn^{-1/2}$,
		uniformly for $|t|\le K$.
		The constant $C_K$ may depend on $K$, but not on $n$
		or the row replacement index.
		For the latter, the quadratic-form martingale CLT applies: the summed
		third absolute moments are $O(n^{-1/2})$ and the conditional variance
		is 
		\[2\mathrm F_{\gamma_n}(g_n^2)+
		\beta\mathcal V_1^{(4)}\mathrm F_{\gamma_n}(g_n)^2+o(1)=\kappa_{21}+o(1).\]
		Hence 
		\[\mathbb E e^{itT_{k0}(0)/\sqrt{\kappa_{21}}}=e^{-t^2/2}+o(1).\]
		
		We next evaluate
		\[
		\mathbb E\left[
		e^{itT_{k0}(0)/\sqrt{\kappa_{21}}}
		\sqrt n\,\mathbf r'\mathbf a_k
		\right].
		\]
		Order the $n-1$ independent row blocks entering
		$\mathbf a_k$ and $T_{k0}(0)$.
		Using a common collection of independent copies,
		let $\mathbf a_k^{(j)}$ denote the coefficient obtained
		by replacing only the $j$th block, and let
		$T_{k0}^{[j]}(0)$ denote the statistic obtained by replacing
		the first $j$ blocks, with $T_{k0}^{[0]}(0)=T_{k0}(0)$.
		Summing the successive row-replacement differences
		and interchanging each block with its independent copy gives
		\[
		\begin{aligned}
			&\mathbb E\left[
			e^{itT_{k0}(0)/\sqrt{\kappa_{21}}}
			\sqrt n\,\mathbf r'\mathbf a_k
			\right]
			-\sqrt n\,\mathbb E(\mathbf r'\mathbf a_k)
			\,\mathbb E e^{itT_{k0}(0)/\sqrt{\kappa_{21}}}\\
			&=\frac{\sqrt n}{2}\sum_{j=1}^{n-1}
			\mathbb E\left[
			\mathbf r'(\mathbf a_k-\mathbf a_k^{(j)})
			\left\{
			e^{itT_{k0}^{[j-1]}(0)/\sqrt{\kappa_{21}}}
			-e^{itT_{k0}^{[j]}(0)/\sqrt{\kappa_{21}}}
			\right\}
			\right].
		\end{aligned}
		\]
		After deleting the active row from both arrays, write their coefficient
		vectors as $\mathbf a,\widetilde{\mathbf a}$ and put
		$\mathbf c=\mathbf H\mathbf r$.
		The two active increments are
		\[
		\frac{(\mathbf c'\mathbf x)(\xi+\mathbf a'\mathbf x)}{\sqrt n\eta},
		\qquad
		\frac{(\xi+\widetilde{\mathbf a}'\mathbf x)^2}
		{\sqrt n\rho_n\widetilde\eta}.
		\]
		Expand the exponential to first order and replace both denominators by
		$e_n$. The summed errors are $o(1)$ by the sixth-moment bounds.
		At a dependent row, the remaining active-row covariance satisfies
		\[
		\Big\|
		2(\mathbf c'\widetilde{\mathbf a})(1+\mathbf a'\widetilde{\mathbf a})
		+\beta\sum_i c_i((\mathbf V_1)_i+a_i)((\mathbf V_1)_i+\widetilde a_i)^2
		-\beta m_n|\mathbf r|^2
		\Big\|_1
		\to0.
		\]
		Here \eqref{eq:short-mixed-bound} removes the first term, and
		\eqref{eq:short-score-bounds} removes all cumulant terms involving
		$a_i$ or $\widetilde a_i$.
		At an independent row,
		\[
		\Big\|\beta\sum_i c_i a_i\widetilde a_i^2\Big\|_1
		\to0.
		\]
		There are $k-1$ dependent rows; using
		$m_n/(\rho_ne_n^2)=1/[l(l-1)]$ and $\mathbb EH=O(n^{-1/2})$ gives
		\begin{equation}
			\mathbb E\left[e^{itT_{k0}(0)/\sqrt{\kappa_{21}}}
			\sqrt n\,\mathbf r'\mathbf a_k\right]
			=\frac{it\,e^{-t^2/2}}{\sqrt{\kappa_{21}}}
			\frac{k-1}{n}\frac{\beta|\mathbf r|^2}{l(l-1)}+o(1).
			\label{eq:short-weighted-score}
		\end{equation}

		We now restore $c_n$ and the original $\beta$- and $\Xi$-terms
		to treat $|x|\le n^{1/4}$.
		Define $g_\rho(\lambda)=(\rho-\lambda)^{-1}$ and
		$\mathcal D=1+\mathrm F_{\gamma_n}(\lambda g_\rho^3)/
		\mathrm F_{\gamma_n}(g_\rho^2)+\rho\partial_\rho$.
		Since $g_\rho h_\rho=\mathcal Dg_\rho$, their combined contribution
		can be evaluated by differentiating the original row expansion,
		giving, at $\rho=\rho_n$,
		\[
		\begin{aligned}
			\Delta_{kb}(x)
			&=\Delta_{kb}(0)+\frac{\tau_n}{\sqrt n}\mathcal D\Delta_{kb}(0)
			-\frac{\tau_n}{n}\mathcal D J_{kb},\\
			J_{kb}
			&=\partial_\rho\log(\bar\eta_k-\epsilon_{kb})
			=\frac{n^{-1}\mathbf b_k'\mathbf H_k^2\mathbf b_k}
			{\bar\eta_k-\epsilon_{kb}}.
		\end{aligned}
		\]
		Conditional on the remaining rows, the terms involving
		only trace increments have the same expectation for
		$b=z$ and $b=y$ and therefore cancel in the conditional
		moment differences.
		The second and third absolute moments of the centered
		trace factors are bounded by $Cn^{-1}$ and $Cn^{-3/2}$,
		respectively.
		For the remaining products, we apply the same row-moment
		identities to linear combinations and extract the mixed
		terms. We then differentiate the resulting identities
		with respect to the spectral parameters, initially
		taken to be distinct.
		Combining these estimates, we obtain the following bound
		for the changes in the first three signed conditional
		moments after summing over the row replacement index:
		\[
		C\left(|\tau_n|/n+\tau_n^2/n^{3/2}+|\tau_n|^3/n^2\right) = o(n^{-1/2}).
		\]
		Furthermore, we have $\|T_{k0}(x)-T_{k0}(0)\|_2\le C|\tau_n|/\sqrt{n}=o(1)$. Applying the Cauchy Schwarz inequality alongside \eqref{eq:short-score-bounds} thereby preserves \eqref{eq:short-weighted-score} when the exterior phase incorporates $x$.
		
		Combining \eqref{eq:short-increment-moments} and
		\eqref{eq:short-weighted-score} in the Taylor expansion of
		\eqref{telescopeclosed}, and using
		$n^{-1}\sum_k(k-1)/n=1/2+O(n^{-1})$, gives
		\begin{align*}
			\psi_{n,\mathbf Z,x}(t)-\psi_{n,\mathbf Y,x}(t)
			={}&\frac{e^{-t^2/2}}{\sqrt n}
			\left\{\frac{it}{\sqrt{\kappa_{21}}}\mathfrak m_1
			+\frac{(it)^2}{2\kappa_{21}}\mathfrak b_1
			+\frac{(it)^3}{6\kappa_{21}^{3/2}}\mathfrak k_1\right\}
			+\frac{|t|}{\sqrt n}o(1),
		\end{align*}
		where
		\begin{align*}
			\mathfrak m_1&=\frac{\beta(1-\mathcal V_1^{(4)})}{l(l-1)}
			+\frac{2\Gamma^2(\mathcal V_1^{(1)})^2}{nl(l-1)^2}
			=\mathfrak a_{11}-\mu(g_{1n}),\\
			\mathfrak b_1&=\frac{8\Gamma^2\mathcal V_1^{(1)}\mathcal V_1^{(3)}}
			{\sqrt n\,l^2(l-1)},\qquad
			\mathfrak k_1=\frac{6\beta^2|\mathbf r|^2}{l^3(l-1)}
			=\mathfrak a_{01}-\kappa_{31}.
		\end{align*}
		For the first two coefficients, replacing
		$\mathbf V_1'\mathbf d_k$ by $m_n\mathcal V_1^{(1)}$ is valid at the
		displayed order since $|\mathcal V_1^{(3)}|\le1$.
		The error is uniform in the stated $x$-range and on bounded
		$t$-intervals. Choose $L_n\to\infty$ sufficiently slowly that
		it is bounded by $\varepsilon_n|t|/\sqrt n$ for $|t|\le L_n$, with
		$\varepsilon_nL_n^7\to0$, $\delta_nL_n^3\to0$, and
		$L_n^9=o(\sqrt n)$.
		Its integral divided by $|t|$ is then at most
		$2\varepsilon_nL_n/\sqrt n=o(n^{-1/2})$, proving \eqref{uniformlow}.

		\medskip
		\noindent\textbf{Proof of \eqref{eq:T53-uniform-middle}.}
		Since the integrand in \eqref{eq:T53-uniform-middle}
		is even, it suffices to consider $t>0$.
		The proof employs two distinct approaches. In the regime below $K_n=n^{1/7}$, we evaluate the difference between the two endpoints before applying absolute values. Conversely, in the regime above $K_n$, we establish bounds for their respective characteristic functions separately.
		
		We first establish \eqref{mffour} and
		\eqref{mfthree}.
		We extend the preceding definitions of $\mathbf H_k$
		and $\mathbf a_k$ to a general array $\B$ and a variable
		spectral parameter $\rho$, defining
		\[
		W=\frac{T_{n,B}(x)}{\sqrt{\kappa_{21}}},~~~
		\mathbf H(\rho)
		=\mathbf V_2
		\big(
		\rho\mathbf I_p
		-\frac1n\mathbf V_2'\B'\B\mathbf V_2
		\big)^{-1}\mathbf V_2',~~~
		a(\rho)=\frac1n \mathbf H(\rho)\B'\xi,~~~
		\xi=\Z\mathbf V_1.
		\]
		Spectral arguments are evaluated at $\rho_n$ unless specified.
		We shall prove, for $L_n\le t\le\epsilon\sqrt n$,
		\begin{align}
			|\mathbb Ee^{itW}|&\le C\frac{\log^4(2+t)}{t^4}+Cn^{-3},
			\label{mffour}\\
			|\mathbb E(Fe^{itW})|&\le C\frac{\log^3(2+t)}{t^3}+Cn^{-3},
			\qquad F=\sqrt n\,u'a,\quad |u|\le C,
			\label{mfthree}
		\end{align}
		where $u$ is deterministic. Furthermore, the second estimate also holds for the
		finitely many spectral derivatives of $a$ used below.
		
		Take $I=\{1,\ldots,m\}$, where
		$m=\lfloor C_0n\log(2+t)/t^2\rfloor$, and put
		$\theta=m/n$ and $\|X\|_{p,I}=(\mathbb E_I|X|^p)^{1/p}$.
		Fixing the rows outside $I$ leaves its rows independent.
		Since $L_n\to\infty$, these blocks have $m=o(n)$ uniformly
		over the stated range and have at most $n$ possible sizes.

		For each nonempty $S\subseteq I$, define $
		W_S=\mathbb E_{I\setminus S}W-\mathbb E_IW
		-\sum_{\varnothing\ne T\subsetneq S}W_T$,
		where $\mathbb E_{I\setminus S}$ averages only over
		the rows in $I\setminus S$. Taking $S=I$ and using $\mathbb E_{\varnothing}W=W$, we obtain
		\[
		W=\mathbb E_IW+
		\sum_{\varnothing\ne S\subseteq I}W_S
		=\mu+L+U+V+R_5,
		\]
		where
		\[
		\mu=\mathbb E_IW,\quad L=\sum_i h_i,\quad h_i=W_{\{i\}},
		\quad U=\sum_{|S|=2}W_S,\quad V=\sum_{|S|=3,4}W_S.
		\]
		We shall prove the following bounds uniformly for
		$|x|\le n^{1/4}$:
		\begin{gather}
			\|U\|_{4,I}\le C\theta,\qquad
			\|V\|_{3,I}\le C\theta^{3/2},\qquad
			\|R_5\|_{2,I}\le C\theta^{5/2},
			\label{eq:mf-projection-bounds}\\
			\|h_i\|_{3,I}\le Cn^{-1/2},\qquad
			\operatorname{Var}_Ih_i\ge c/n,\qquad
			|\mathbb E_Ie^{ith_i}|\le e^{-ct^2/n}.
			\label{eq:mf-single-row}
		\end{gather}
		
		We first prove \eqref{eq:mf-projection-bounds}.
		For $T\subseteq I$, let $W^{(T)}$ be $W$ with $b_i$ and $\xi_i$
		set to zero for $i\in T$, all other rows and $n$ unchanged;
		in particular, $W^{(\varnothing)}=W$. For nonempty $S\subseteq I$, define
		\[
		\Delta_S^0W:=\sum_{T\subseteq S}(-1)^{|T|}W^{(T)}.
		\]
		Every term with $T\ne\varnothing$ is removed by centering in a row
		of $T$, so
		\[
		W_S=\Big[\prod_{i\in S}(1-\mathbb E_i)\Big]
		\mathbb E_{I\setminus S}\Delta_S^0W.
		\]
		Consequently, $\|W_S\|_{p,I}\le2^{|S|}\|\Delta_S^0W\|_{p,I}$
		for $p\ge1$, and it suffices to bound the deletion differences.
		
		For $\mathbf H$ formed from the rows outside
		$S$, set
		\[
		w=\frac1n\mathbf H\mathbf B_{S^c}'\xi_{S^c},
		\qquad r_i=\xi_i+b_i'w,
		\qquad r_S=(r_i)_{i\in S},
		\]
		and define the diagonal matrix $\mathbf D$ and the matrix $\mathbf E$ by
		\[
		\mathbf D_{ii}=1-\frac1n b_i'\mathbf H b_i,\qquad
		\mathbf E_{ij}=\frac1n b_i'\mathbf H b_j\quad(i\ne j),
		\qquad \mathbf E_{ii}=0.
		\]
		We first establish the moment bounds at $x=0$.
		The block matrix inversion formula gives
		\[
		\Delta_S^0W
		=\frac{1}{\rho_n\sqrt{\kappa_{21}}\sqrt n}
		\Delta_S^0\{r_S'(\mathbf D-\mathbf E)^{-1}r_S\}.
		\]
		
		Expand $(\mathbf D-\mathbf E)^{-1}$ in powers of $\mathbf D^{-1}\mathbf E$.
		The alternating difference cancels every term that omits
		a row in $S$. In each remaining term, successive integration
		over the rows gives a factor $n^{-1/2}$ for each of the
		$|S|-1$ factors needed to connect all these rows.
		Together with the factor $n^{-1/2}$ outside the quadratic
		form, this gives
		\[
		\|\Delta_S^0W\|_{p,I}\le Cn^{-|S|/2},
		\qquad 1\le |S|\le5,\quad p=2,3.
		\]
		These integrations use moments of order at most $2p\le6$.
		Localize the deleted coefficients so that $\|\mathbf H\|+|w|\le C$.
		Conditional on the rows outside $S$, Bernstein's inequality
		applied first to the squared original row norm and then to
		$b_i'\mathbf H b_j$ gives, uniformly for $2\le |S|\le5$,
		\[
		\mathbb E_I\mathbf1_{\{
			\max_{i\ne j,\ i,j\in S}|\mathbf E_{ij}|
			>n^{-(1/2-\alpha)/2}\}}
		\le C e^{-c n^{(1/2-\alpha)/2}}.
		\]
		For the bounded inverse extensions, the quadratic-form
		remainder $R_M$ after $M$ terms therefore satisfies
		\[
		\|R_M\|_{p,I}
		\le Cn^{-1/2-M(1/2-\alpha)/2}+O_D(n^{-D}),
		\qquad p=2,3.
		\]
		Indeed, on $\mathcal G_n$ and the displayed small-entry set,
		the inverse remainder is bounded by $C\|\mathbf E\|^M$.
		Elsewhere it is polynomially bounded, so the exponential
		tail and the opening conditional localization bound apply.
		Taking fixed $M$ sufficiently large gives the required
		$O(n^{-3})$ remainder.
		
		For two rows, the exact inverse formula yields
		\[
		|\Delta_{\{i,j\}}^0W|
		\le Cn^{-1/2}
		\bigl\{|r_ir_j\mathbf E_{ij}|+(r_i^2+r_j^2)\mathbf E_{ij}^2\bigr\}.
		\]
		Expanding the fourth moments and using
		$1+\mathbb E|Z_{11}|^8\le Cn^{2\alpha}\le Cn$
		therefore gives
		$\|\Delta_{\{i,j\}}^0W\|_{4,I}\le C/n$.
		
		Applying the preceding row-moment estimates to the
		$x$-dependent quadratic and trace corrections in
		\eqref{TnBdef}, using $g_\rho h_\rho=\mathcal Dg_\rho$,
		$\partial_\rho\mathbf H=-\mathbf H^2$, and the additional
		coefficient factor $|\tau_n|/\sqrt n$, gives
		\[
		\|\Delta_S^0W\|_{p,I}
		\le C\Big(1+\frac{|\tau_n|}{\sqrt n}\Big)n^{-|S|/2},
		\qquad 
		\|\Delta_{\{i,j\}}^0W\|_{4,I}
		\le \frac{C}{n}
		\Big(1+\frac{|\tau_n|}{\sqrt n}\Big).
		\]
		Since $\tau_n=\rho_n^{-1}\widetilde{\sigma}_n x$, we derive $
		\sup_{|x|\le n^{1/4}}\frac{|\tau_n|}{\sqrt{n}} \le Cn^{-1/4}$.
		This ensures that the moment bounds retain their stated orders uniformly across $|x|\le n^{1/4}$.
		
		Successive martingale moment inequalities over
		the independent rows then give
		\[
		\|U\|_{4,I}
		\le C(m^2n^{-2})^{1/2}=C\theta,
		\qquad
		\|V\|_{3,I}
		\le C\sum_{q=3}^4(m^qn^{-q})^{1/2}
		\le C\theta^{3/2}.
		\]
		Moreover, orthogonality of the $W_S$ implies
		\[
		\|R_5\|_{2,I}^2
		\le \sum_{\substack{S\subseteq I\\|S|=5}}
		\|\Delta_S^0W\|_{2,I}^2
		\le C\binom m5 n^{-5}
		\le C\theta^5.
		\]
		Indeed, centering in each of five specified rows retains
		only projections involving all five rows, so every
		projection in $R_5$ is counted at least once.
		This proves \eqref{eq:mf-projection-bounds}.
		
		We next prove \eqref{eq:mf-single-row}.
		The one-row deletion estimate immediately gives
		$\|h_i\|_{3,I}\le Cn^{-1/2}$.
		At $x=0$, averaging the two-row formula gives
		\[
		\|\mathbb E_j\Delta_{\{i,j\}}^0W\|_{2,I}
		\le Cn^{-3/2}.
		\]
		For the one-row approximation below, take $S=I$ in the
		definitions of $\mathbf H,w,r_i$ and $D_{ii}$.
		Returning to $|x|\le n^{1/4}$, we sum over the other
		$m-1$ rows, replace the one-row denominator by its
		conditional mean, and include the $x$-dependent correction.
		This gives, uniformly over the stated range,
		\[
		\Big\|
		h_i-
		\frac{r_i^2-\mathbb E_Ir_i^2}
		{\rho_n\sqrt{\kappa_{21}}\sqrt n\,\mathbb E_ID_{ii}}
		\Big\|_{2,I}
		\le \frac{C\theta}{\sqrt n}
		+o(n^{-1/2})+\frac{C|\tau_n|}{n}
		=o(n^{-1/2}).
		\]
		Here $\mathbb E_ID_{ii}$ is bounded and bounded away from
		zero. Since $\mathbf H\mathbf V_1=0$, the coefficient vector
		of $r_i$ has squared norm $1+|w|^2\ge1$.
		The fourth-moment formula and the continuous-component
		assumption thus give $\operatorname{Var}_I(r_i^2)\ge c$.
		The preceding approximation consequently implies
		$\operatorname{Var}_Ih_i\ge c/n$.
		
		Finally, applying the expansion to $h_i$ gives
		\[
		|\mathbb E_Ie^{ith_i}|^2=\mathbb E_I\cos\big(t(h_i-\widetilde h_i)\big)
		\le 1-\frac{ct^2}{n}
		+\frac{C|t|^3}{n^{3/2}}.
		\]
		Choosing $\epsilon>0$ sufficiently small proves \eqref{eq:mf-single-row} for $|t|\le\epsilon\sqrt n$.

		The two-row formula and sixth moments, followed by conditional
		centering using independent copies and Minkowski's inequality, give
		\[
		\left\|\mathbb E_j|W_{\{i,j\}}h_j|\right\|_{3,i}
		\le Cn^{-3/2}.
		\]
		Expand $e^{it(U+V)}$, retaining
		\[
		1+itU+\frac{(itU)^2}{2}+\frac{(itU)^3}{6}
		+itV+(it)^2UV.
		\]
		After expanding the sums, each product involves at most six rows.
		For a row occurring in only one centered factor, replace
		$e^{ith_i}$ by $e^{ith_i}-1$; apply the preceding bound
		to the resulting terms in $U^3$ and H\"older's inequality
		with exponents $3,3,3$ to the remaining products.
		Summing these bounds and applying
		Lemma~\ref{integration} and
		\eqref{eq:mf-projection-bounds}, we obtain
		\[
		|\mathbb E_Ie^{itW}|
		\le C(1+t)^C e^{-c(m-6)t^2/n}
		+C\left(t\theta^{5/2}+t^2\theta^3
		+t^3\theta^{7/2}+t^4\theta^4\right).
		\]
		Taking $C_0$ sufficiently large proves \eqref{mffour}.
		
		For $F$, the identity $
		a_L=w+n^{-1}\mathbf H\mathbf B_L'
		(\mathbf D_L-\mathbf E_L)^{-1}r_L$
		yields conditional components $F=\nu+G+Q+S$ involving
		$0,1,2,\ge3$ rows, respectively, with
		$\|G\|_3\le C\sqrt\theta$, $\|Q\|_3\le C\theta$,
		and $\|S\|_2\le C\theta^{3/2}$.
		The individual one- and two-row terms have $L^3$ bounds
		$Cn^{-1/2}$ and $Cn^{-1}$, respectively.
		Retaining $G,Q,itGU,itQU$ in $(G+Q)e^{itW}$,
		centering as above, and applying
		Lemma~\ref{integration}(i) with $d=4$ gives
		\[
		|\mathbb E_I[(F-\nu)e^{itW}]|
		\le C(1+t)^C e^{-c(m-4)t^2/n}
		+C\left(\theta^{3/2}+t\theta^2+t^2\theta^{5/2}\right).
		\]
		The preceding conditional characteristic-function bound
		and $\mathbb E|\nu|\le C$ from \eqref{eq:short-score-bounds},
		followed by averaging over the rows outside $I$,
		prove \eqref{mfthree}.
		
		We next bound $|\psi_{n,\mathbf Z,x}(t)-\psi_{n,\mathbf Y,x}(t)|$
		for $L_n\le t\le K_n$.
		First take $x=0$ and condition on all rows except $k$.
		Using the preceding notation, set $
		\lambda=\frac{t}{\rho_n\sqrt{\kappa_{21}}\sqrt n}$.
		Summing the $\alpha$-terms in \eqref{eq:short-complete-increment} gives
		\[
		e^{it\Delta_{kb}(0)/\sqrt{\kappa_{21}}}
		=
		e^{i\lambda B_{kb}^2/(\bar\eta_k-\epsilon_{kb})},
		\qquad b\in\{z,y\}.
		\]
		
		For $0\le s\le1$, define
		\[
		R_s=\mathbf V_1'\mathbf z_k
		+\sqrt{s}\,\mathbf a_k'\mathbf z_k
		+\sqrt{1-s}\,\mathbf a_k'\mathbf y_k,
		\]
		so that $R_0=B_{ky}$ and $R_1=B_{kz}$.
		For $f(r)=e^{iur^2}-1-iur^2+u^2r^4/2$,
		coordinatewise Taylor expansion uses
		$|f'''(r)|\le C|u|^3|r|^3$, while the second-order terms
		cancel by orthogonality.
		The fourth-moment identity and
		\eqref{eq:short-score-bounds} then give
		\begin{equation}
			\Big\|
			\mathbb E_k\{e^{i\lambda B_{kz}^2/\bar\eta_k}
			-e^{i\lambda B_{ky}^2/\bar\eta_k}\}
			+\frac{2\beta\lambda^2}{\bar\eta_k^2}\mathbf r'\mathbf a_k
			\Big\|_1
			\le C\Big(\frac{\lambda^2}{n}
			+\frac{|\lambda|^3}{\sqrt n}\Big).
			\label{eq:mf-rotation-refined}
		\end{equation}
		
		To include the random denominators, write
		$r=c'X$, $Q=X'\mathbf HX-\operatorname{tr}\mathbf H$,
		and $\delta=Q/n$.
		For bounded $|c|,\|\mathbf H\|$, moment expansion gives
		\[
		\mathbb E(r^4Q^2)\le Cn,\qquad
		\mathbb EQ^4\le Cn^2,\qquad
		\mathbb E(r^2|\delta|^3)\le Cn^{-3/2},
		\]
		using the truncated eighth-moment bound
		$\mathbb E|X_1|^8\le Cn$.
		Expanding the denominator to second order and the resulting
		exponential correction to first order yields
		\begin{align}
			\mathbb Ee^{i\lambda r^2/(e-\delta)}
			=\mathbb Ee^{i\lambda r^2/e}
			+\frac{i\lambda}{e^2n}
			\mathbb E(r^2Qe^{i\lambda r^2/e})
			+\frac{i\lambda}{e^3n^2}\mathbb E(r^2Q^2)+O\left(\frac{|\lambda|}{n^{3/2}}
			+\frac{\lambda^2}{n}\right),
			\label{eq:mf-denominator-refined}
		\end{align}
		uniformly for $e$ bounded away from zero, with exceptional
		denominators handled by the bounded extensions.
		
		The required moment identities, with
		$d=(\mathbf H_{jj})_j$, are
		\begin{align*}
			&\mathbb E(r^2Q)
			=2c'\mathbf Hc+\beta\sum_jc_j^2\mathbf H_{jj},\\
			&\mathbb E(r^2Q^2)
			=|c|^2\{2\operatorname{tr}\mathbf H^2+\beta|d|^2\}
			+2\Gamma^2(c'd)^2+O(\sqrt n),\\
			&\mathbb E(r^2Qe^{iur^2})
			=\mathbb E(r^2Q)
			+4iu\Gamma^2(c'd)\sum_jc_j^3
			+O(|u|+u^2\sqrt n).
		\end{align*}
		The first two follow from fourth and sixth moments.
		For the last, center diagonal coordinates once and
		off-diagonal coordinates twice, using
		\[
		|(r^2e^{iur^2}-r^2-iur^4)''|\le Cu^2r^4,
		\qquad
		\sum_{i,j}|\mathbf H_{ij}c_ic_j|\le C\sqrt n.
		\]
		At the independent endpoint, augment $X$ by the independent
		spike row and extend $\mathbf H$ by zero rows and columns.
		
		We first refine the diagonal approximation to obtain
		\begin{equation}\label{eq:mf-diagonal-rate}
			\max_j
			\|(\mathbf{H}_k)_{jj}-m_n(1-(\mathbf{V}_1)_j^2)\|_2
			\le Cn^{-1/2},
			\qquad
			\|\bar{\eta}_k-e_n\|_2\le Cn^{-1/2}.
		\end{equation}
		For deterministic unit vectors $\mathbf u,\mathbf w$,
		the rank-one inverse formula and the row-deletion
		martingale argument give the variance bound below.
		Averaging the conditional denominators and using
		stability of the Marchenko--Pastur equation gives
		the mean approximation:
		\[
		\operatorname{Var}(\mathbf{u}'\mathbf{G}_k\mathbf{w})
		\le Cn^{-1},
		\qquad
		\|\mathbb{E}\mathbf{G}_k-m_n\mathbf{I}\|
		\le Cn^{-1/2}.
		\]
		Substitution into the preceding identity for $\mathbf H_k$
		proves \eqref{eq:mf-diagonal-rate}.
		The same estimates in a fixed complex neighbourhood,
		together with Cauchy's formula, give the required
		spectral derivative bounds.
		
		Since $B_{kz}$ and $B_{ky}$ have equal conditional variances,
		the Gaussian trace terms cancel.
		By \eqref{eq:mf-diagonal-rate}, the remaining coefficients
		\[
		\sum_j(\mathbf V_1)_j^2(\mathbf H_k)_{jj},
		\qquad
		\frac{(\mathbf V_1'\mathbf d_k)^2}{n},
		\qquad
		\frac{(\mathbf V_1'\mathbf d_k)
			\sum_j(\mathbf V_1)_j^3}{\sqrt n}
		\]
		have bounded deterministic approximations with
		$L^1$ errors at most $Cn^{-1/2}$.
		Combining these approximations with
		\[
		\mathbb E
		|\mathbf V_1'\operatorname{diag}(\mathbf H_k)\mathbf a_k|
		\le Cn^{-1/2},
		\qquad
		\mathbb E|\mathbf d_k'\mathbf a_k|\le C,
		\]
		bounds the total remainder by $Ct^3/n+Cn^{-3}$ for $t\ge1$.
		Thus, writing
		\[
		F_k=\sqrt n\,\mathbf r'\mathbf a_k,
		\qquad
		W_{k0}=\frac{T_{k0}(0)}{\sqrt{\kappa_{21}}},
		\]
		the leading terms in \eqref{telescopeclosed}
		are bounded-coefficient multiples of
		\[
		\frac{t}{\sqrt n}\frac1n\sum_k\mathbb Ee^{itW_{k0}},
		\qquad
		\frac{t^2}{\sqrt n}\frac1n\sum_k\mathbb Ee^{itW_{k0}},
		\qquad
		\frac{t^2}{\sqrt n}\frac1n\sum_k
		\mathbb E(F_ke^{itW_{k0}}).
		\]
		
		For $|x|\le n^{1/4}$, the correction identity used in
		\eqref{uniformlow} satisfies
		$|\tau_n|/\sqrt n\le Cn^{-1/4}$.
		Applying the same moment identities to linear combinations
		controls the mixed terms with the remainder in
		\eqref{eq:mf-denominator-refined}.
		The Gaussian terms still cancel because spectral
		differentiation preserves $\mathbf H_k\mathbf V_1=0$.
		In the fourth-moment difference, terms with three spike
		factors give $F_k$ and its spectral derivatives;
		those with one or two spike factors have $L^1$ bounds
		$Cn^{-1}$ by \eqref{eq:short-score-bounds}.
		The pure background and pure spike terms with denominator
		$\bar\eta_k$ cancel, and \eqref{eq:mf-diagonal-rate}
		controls the remaining diagonal coefficients.
		Consequently, the same reduction and remainder bound hold
		uniformly for $|x|\le n^{1/4}$, with
		$W_{k0}=T_{k0}(x)/\sqrt{\kappa_{21}}$
		and the corresponding spectral derivatives of $F_k$.
		
		Applying \eqref{mffour} alongside \eqref{mfthree} and combining these bounds with the preceding remainder estimates yields
		\begin{equation}
			\sup_{|x|\le n^{1/4}}|\psi_{n,Z,x}(t)-\psi_{n,Y,x}(t)|
			\le \frac{C\log^3(2+t)}{\sqrt n\,t}
			+\frac{Ct^3}{n}+Cn^{-3},
			\qquad L_n\le t\le K_n.
			\label{eq:mf-signed-integrable}
		\end{equation}

		Since $L_n^9=o(\sqrt n)$, we have $L_n<K_n$.
		Use \eqref{eq:mf-signed-integrable} below $K_n$ and
		\eqref{mffour} separately at both endpoints above it:
		\begin{align*}
			&\sup_{|x|\le n^{1/4}}
			\int_{L_n<|t|\le\epsilon\sqrt n}
			\frac{|\psi_{n,Z,x}(t)-\psi_{n,Y,x}(t)|}{|t|}\,dt\\
			&\quad\le \frac C{\sqrt n}
			\int_{L_n}^{\infty}\frac{\log^3(2+t)}{t^2}\,dt
			+C\frac{K_n^3}{n}
			+C\frac{\log^4(2+K_n)}{K_n^4}+Cn^{-3}\log n
			=o(n^{-1/2}),
		\end{align*}
		because $L_n\to\infty$ and $K_n^3/n=n^{-4/7}$. This proves \eqref{eq:T53-uniform-middle}.
		
		Finally, the proof of \eqref{mffour} only requires
		$C_0\log(2+t)/t^2$ to be sufficiently small.
		Thus, the estimate remains valid for
		$t_0\le |t|\le\epsilon\sqrt n$, where $t_0$ is a sufficiently
		large fixed constant.
		Integrating this estimate over that range and using
		$|\psi_{n,\mathbf Y,x}(t)|\le1$ for $|t|<t_0$, we obtain
		\begin{equation}\label{epsilonsqrtn}
			\sup_{|x|\le n^{1/4}}
			\int_{|t|\le\epsilon\sqrt n}
			|\psi_{n,\mathbf Y,x}(t)|\,dt=O(1).
		\end{equation}

		\medskip
		\noindent\textbf{Proof of \eqref{eq:T53-uniform-high}
			and \eqref{unconcentration}.}
		Fix $K>\epsilon$ and let
		$I=\{1,\ldots,m\}$, where $m=\lceil C_1\log n\rceil$.
		The conditional decomposition above gives
		\[
		W=\mu+L+U+R_3,\qquad
		L=\sum_{i\in I}h_i,\qquad
		\|U\|_{2,I}\le \frac{Cm}{n},\qquad
		\|R_3\|_{2,I}\le C(m/n)^{3/2}.
		\]
		By the preceding single-row calculation,
		\[
		\left\|
		h_i-\frac{r_i^2-\mathbb E_Ir_i^2}
		{\rho_n\sqrt{\kappa_{21}}\,e\sqrt n}
		\right\|_{1,I}
		\le C\left(
		\frac{m}{n^{3/2}}+\frac1n+\frac{|\tau_n|}{n}
		\right),
		\qquad
		e=1-\frac{\operatorname{tr}H}{n}.
		\]
		
		At either endpoint, $r_i$ is a linear projection of
		independent standardized entries, with squared coefficient
		norm $1+|w|^2\in[1,C]$, and $e$ is bounded away from zero.
		For $\epsilon\sqrt n<|t|\le K\sqrt n$, normalizing this
		projection gives a quadratic coefficient whose absolute
		value is bounded above and away from zero.
		Lemma~\ref{lem:projected-score-gap}, applied with zero
		linear coefficient, therefore yields some
		$q_0=q_0(\epsilon,K)<1$ such that
		\[
		\left|
		\mathbb E_I\exp\left\{
		\frac{it(r_i^2-\mathbb E_Ir_i^2)}
		{\rho_n\sqrt{\kappa_{21}}\,e\sqrt n}
		\right\}
		\right|\le q_0.
		\]
		Combining this bound with the $L^1$ estimate yields, for some fixed $q\in(q_0,1)$ uniformly over $i\in I$ and $\epsilon\sqrt{n}<|t|\le K\sqrt{n}$,
		\[
		|\mathbb{E}_Ie^{ith_i}|
		\le q_0+C_K\left(
		\frac{m}{n}+\frac{1}{\sqrt{n}}+\frac{|\tau_n|}{\sqrt{n}}
		\right)
		\le q<1.
		\]
		
		Applying the expansion of the first order, Lemma~\ref{integration}, and the Cauchy Schwarz inequality yields the following bound, potentially after increasing $q$ strictly within the interval $(0,1)$:
		\begin{align*}
			|\mathbb{E}_Ie^{itW}|
			&\le Cq^m(1+|t|m/n)
			+C|t|(m/n)^{3/2}+Ct^2(m/n)^2\\
			&\le C_Kq^m+C_K\frac{m^{3/2}+m^2}{n}.
		\end{align*}
		Choosing $C_1$ sufficiently large gives, for any fixed $A>1$,
		the bound $C_{A,K}\{n^{-A}+\log^2n/n\}$ for the truncated
		endpoints. Restoring truncation, centering, and scaling as above,
		put
		\[
		r_n=n^{2-6\alpha}+n^{1-4\alpha}+n^{3/2-5\alpha}
		=o(n^{-1/2}).
		\]
		We obtain
		\[
		\sup_{|x|\le n^{1/4}}
		\sup_{\epsilon\sqrt n<|t|\le K\sqrt n}
		\{|\psi_{n,\mathbf Z,x}(t)|+|\psi_{n,\mathbf Y,x}(t)|\}
		\le C_{A,K}\left\{n^{-A}+\frac{\log^2n}{n}+r_n\right\}
		=o(n^{-1/2}).
		\]
		The preceding estimate holds for every fixed $K>0$, although its
		constant may depend on $K$. We therefore cannot directly replace
		$K$ by an arbitrary divergent sequence. Instead, applying the
		estimate successively with $K=1,2,\ldots$, we choose a
		piecewise-constant sequence $B_n\to\infty$ sufficiently slowly
		so that
		\[
		\int_{\epsilon\sqrt n<|t|\le B_n\sqrt n}
		\frac{
			|\psi_{n,\mathbf Z,x}(t)|
			+
			|\psi_{n,\mathbf Y,x}(t)|
		}{|t|}\,dt
		=
		o(n^{-1/2}).
		\]

		By combining Esseen's concentration inequality and equation \eqref{epsilonsqrtn}, we obtain
		\[
		\sup_y
		\mathbb P\left\{
		y<
		\frac{T_{n,\mathbf Y}(x)}{\sqrt{\kappa_{21}}}
		\le
		y+\frac{1}{B_n\sqrt n}
		\right\}
		\le
		\frac{C}{B_n\sqrt n}
		\int_{-B_n\sqrt n}^{B_n\sqrt n}
		|\psi_{n,\mathbf Y,x}(t)|\,dt=o(n^{-1/2}).
		\]
		
		For $\epsilon\sqrt n<|t|\le B_n\sqrt n$, we have $
		\frac{1}{B_n\sqrt n}\le\frac1{|t|}$.
		Therefore,
		\[
		\frac{1}{B_n\sqrt n}
		\int_{\epsilon\sqrt n<|t|\le B_n\sqrt n}
		|\psi_{n,\mathbf Y,x}(t)|\,dt
		\le
		\int_{\epsilon\sqrt n<|t|\le B_n\sqrt n}
		\frac{|\psi_{n,\mathbf Y,x}(t)|}{|t|}\,dt
		=o(n^{-1/2})
		\]
		by \eqref{eq:T53-uniform-high}. This proves
		\eqref{unconcentration} and completes the proof.
	\end{proof}

	We now finish the theorem.  The Fourier--Stieltjes transform of
	\[
	\frac{\phi(x)}{\sqrt n}\left\{
	\frac{\mathfrak a_{01}-\kappa_{31}}
	{6\kappa_{21}^{3/2}}(1-x^2)
	-\frac{\mathfrak b_1}{2\kappa_{21}}x
	-\frac{\mathfrak a_{11}-\mu(g_{1n})}
	{\sqrt{\kappa_{21}}}
	\right\}
	\]
	is the leading term in \eqref{uniformlow}.  Apply the signed Esseen
	inequality with cutoff
	\[
	T_n=B_n\sqrt n.
	\]
	Combining Lemma~\ref{unifreplace} with \eqref{unconcentration} yields the following expansion, which holds uniformly for $|x|\le n^{1/4}$:
	\begin{align*}
		H_{n,\Z}(x)-H_{n,\Y}(x)=\frac{\phi(x)}{\sqrt{n}}\left\{
		\frac{\mathfrak{a}_{01}-\kappa_{31}}
		{6\kappa_{21}^{3/2}}(1-x^2)
		-\frac{\mathfrak{b}_1}{2\kappa_{21}}x
		-\frac{\mathfrak{a}_{11}-\mu(g_{1n})}
		{\sqrt{\kappa_{21}}}
		\right\}+o(n^{-1/2}).
	\end{align*}
	Finally, the fixed spectral gap and the sixth-moment bound give
	\[
	\sup_{\B\in\{\Z,\Y\}}
	\mathbb E|T_{n,\B}(0)|^3
	\le C,\qquad
	\sup_{\B\in\{\Z,\Y\}}
	\mathbb E|\partial_xT_{n,\B}(x)|^3
	\le Cn^{-3/2}.
	\]
	Indeed,
	\[
	\partial_xT_{n,\mathbf B}(x)
	=\frac{\widetilde\sigma_n}{\rho_n\sqrt n}
	\begin{cases}
		\widetilde S_n(g_nh_n), & \mathbf B=\mathbf Z,\\
		S_n(g_nh_n), & \mathbf B=\mathbf Y.
	\end{cases}
	\]
	and the third moment of the centered resolvent quadratic form is
	uniformly bounded.
	
	Since
	\[
	T_{n,\B}(x)
	=
	T_{n,\B}(0)+x\partial_xT_{n,\B}(x),
	\]
	Markov's inequality yields, uniformly for $|x|>n^{1/4}$,
	\[
	\mathbb P\{|T_{n,\B}(x)|>c|x|\}
	\le
	C|x|^{-3}+Cn^{-3/2}
	=o(n^{-1/2}).
	\]
	Since
	\[
	a_n=\sqrt{\kappa_{21}}
	=
	\mathrm F_{\gamma_n}(g_n^2)\widetilde\sigma_n
	\]
	is bounded away from zero, choose a constant
	$0<c<\inf_{n\ge n_0}a_n$. Then, for $x>n^{1/4}$,
	\[
	1-H_{n,\mathbf B}(x)
	=
	\mathbb P\{T_{n,\mathbf B}(x)>a_nx\}
	\le
	\mathbb P\{|T_{n,\mathbf B}(x)|>c|x|\},
	\]
	whereas, for $x<-n^{1/4}$,
	\[
	H_{n,\mathbf B}(x)
	=
	\mathbb P\{T_{n,\mathbf B}(x)\le a_nx\}
	\le
	\mathbb P\{|T_{n,\mathbf B}(x)|>c|x|\}.
	\]
	Consequently, both tails are $o(n^{-1/2})$ uniformly in
	$\mathbf B\in\{\Z,\Y\}$.
	The Gaussian correction is exponentially smaller.  Hence the
	comparison extends uniformly to $x\in\mathbb R$, which proves Theorem~\ref{G4MT}.

	\subsection{Proof of Equation~\eqref{step3}}
	\label{sec:eq53-routeB}
	
	Let $
	\mathbb P_Y(\cdot)
	=
	\mathbb P(\cdot\mid\mathcal F_Y)$
	and $
	\mathbb E_Y(\cdot)
	=
	\mathbb E(\cdot\mid\mathcal F_Y)$.
	To distinguish the
	coefficient parameter from the distributional threshold, use $x$ in
	\[\mathbf C_n(x),D_n(x),Q_n(x),\] and use $z$ for the threshold.  Let
	$M_n=K\sqrt{\log n}$, where $K>0$ is fixed sufficiently large.  Since
	\[
	\Omega_s(x)=n^{-1/2}\{D_n(x)+Q_n(x)\},
	\]
	it is enough to prove
	\begin{align}	\label{eq:eq53-target}
		&\mathbb E\sup_{|x|\le M_n}\sup_{z\in\mathbb R}
		\Big|
		\mathbb P_Y\left\{n^{-1/2}(D_n(x)+Q_n(x))\le z\right\}
		\\\nonumber
		&\qquad\qquad\qquad\qquad-\mathbb P_Y\left\{n^{-1/2}(D_n(x)+Q_n^*(x))\le z\right\}
		\Big|=o(n^{-1/2}),
	\end{align}

	\begin{lemma}
		\label{lem:eq53-matched-comparison}
		Suppose that there exist events $\mathcal E_{3n}\in\mathcal F_Y$ such
		that $\mathbb P(\mathcal E_{3n}^c)=o(n^{-1/2})$.
		On $\mathcal E_{3n}$, uniformly for $|x|\le M_n$, assume that $
		cn\le B_{2n}(x)\le Cn$ and, with
		\[
		\widetilde{\mathbf C}_n(x)
		=
		\frac{\mathbf C_n(x)}{\sqrt{B_{2n}(x)}}
		=
		\bigl(\widetilde c_{ij}(x)\bigr),
		\]
		that
		\begin{equation}
			\|\widetilde{\mathbf C}_n(x)\|\le Cn^{-1/2},\quad
			\max_i\sum_j\widetilde c_{ij}(x)^2\le Cn^{-1},\quad
			\min\Big\{
			\sum_i\widetilde c_{ii}(x)^2,\,
			\sum_{i\ne j}\widetilde c_{ij}(x)^2
			\Big\}\ge c.
			\label{eq:eq53-normalized-class}
		\end{equation}
		
		Suppose that $
		\widetilde Z_{i1}$ have uniformly integrable sixth
		moments and that the bound in
		Lemma~\ref{lem:projected-score-gap} applies uniformly to them.
		Conditionally on $\mathcal F_Y$, let
		\[
		Q_n^*(x)=b_n(x)\sum_{i=1}^n\eta_{ni}(x),
		\]
		where the variables $\eta_{ni}(x)$ are mutually independent,
		independent of $X_1,\ldots,X_n$, and satisfy
		\[
		\mathbb E_Y\eta_{ni}(x)=0,\qquad
		\mathbb E_Y\eta_{ni}(x)^2=1,\qquad
		\sup_{n,i,\,|u|\le M_n}
		\mathbb E_Y|\eta_{ni}(x)|^6\le C.
		\]
		Assume that the Gaussian part of $Q_n^*(x)$ has conditional variance
		at least $cB_{2n}(x)$ and that
		\begin{equation}
			\label{eq:eq53-three-cumulants}
			\operatorname{cum}_{j,Y}\{D_n(x)+Q_n(x)\}
			=
			\operatorname{cum}_{j,Y}\{D_n(x)+Q_n^*(x)\},
			\qquad j=1,2,3.
		\end{equation}
		Then, with
		\[
		W_n(x)=\frac{D_n(x)+Q_n(x)}{\sqrt{B_{2n}(x)}},
		\qquad
		W_n^*(x)=\frac{D_n(x)+Q_n^*(x)}{\sqrt{B_{2n}(x)}},
		\]
		we have
		\begin{equation}
			\label{eq:eq53-standardized-comparison}
			\mathbb E
			\sup_{|x|\le M_n}
			\sup_{z\in\mathbb R}
			\left|
			\mathbb P_Y\{W_n(x)\le z\}
			-
			\mathbb P_Y\{W_n^*(x)\le z\}
			\right|
			=
			o(n^{-1/2}).
		\end{equation}
	\end{lemma}
	
	\begin{proof}
		Fixing $x$ and conditioning on $\mathcal{F}_Y$, we define
		$$ W_n = \frac{D_n(x)+Q_n(x)}{\sqrt{B_{2n}(x)}}, \qquad W_n^* = \frac{D_n(x)+Q_n^*(x)}{\sqrt{B_{2n}(x)}}, $$
		and let $\varphi_n$ and $\varphi_n^*$ denote their characteristic functions. By \eqref{eq:eq53-three-cumulants}, both variables have mean zero and variance one, and they share the exact same third cumulant. Besides, we derive
		\[
		\lambda_{3n}(x)
		=
		\frac{B_{3n}(x)}{B_{2n}(x)^{3/2}}
		=
		O(n^{-1/2}).
		\]
		
		\emph{\rm{Step 1}.}
		Recall $
		W_n
		=
		\sum_{i=1}^n
		\widetilde c_{ii}\bigl(\widetilde Z_{i1}^{\,2}-1\bigr)
		+
		2\sum_{1\le i<j\le n}
		\widetilde c_{ij}\widetilde Z_{i1}\widetilde Z_{j1}$.
		For a fixed $i$, define
		\[
		U_i
		=
		\widetilde c_{ii}\bigl(\widetilde Z_{i1}^{\,2}-1\bigr)
		+
		2\sum_{j\ne i}
		\widetilde c_{ij}\widetilde Z_{i1}\widetilde Z_{j1}.
		\]
		The off-diagonal term
		$\widetilde c_{ij}\widetilde Z_{i1}\widetilde Z_{j1}$ occurs in both
		$U_i$ and $U_j$. To avoid counting it twice, define
		\[
		A_i
		=
		\widetilde c_{ii}\bigl(\widetilde Z_{i1}^{\,2}-1\bigr)
		+
		\sum_{j\ne i}
		\widetilde c_{ij}\widetilde Z_{i1}\widetilde Z_{j1}.
		\]
		Then $
		W_n=\sum_{i=1}^n A_i$.
		For each $i$, write
		$
		\mathbb E_i(\,\cdot\,)
		:=
		\mathbb E\left(
		\,\cdot\,
		\mathrel{\big|}
		\sigma\bigl(
		\mathcal F_Y,\,
		\widetilde Z_{j1}:j\ne i
		\bigr)
		\right)$.
		Thus, we have
		\[
		\mathbb E_iA_i=0.
		\]
		Furthermore, since $W_n-U_i$ does not depend on $\widetilde{Z}_{i1}$, we conclude that
		\[
		\mathbb E_i\left[
		A_i e^{it(W_n-U_i)}
		\right]
		=
		e^{it(W_n-U_i)}\mathbb E_iA_i
		=
		0.
		\]
		It follows that
		\[
		\begin{aligned}
			\varphi_n'(t)
			=
			i\mathbb E_Y\left[W_ne^{itW_n}\right]  
			=
			i\sum_{i=1}^n
			\mathbb E_Y\left[A_ie^{itW_n}\right]  
			=
			i\sum_{i=1}^n
			\mathbb E_Y\left[
			A_i e^{it(W_n-U_i)}
			\bigl(e^{itU_i}-1\bigr)
			\right].
		\end{aligned}
		\]

		Using $\left|e^{iu}-1-iu+\frac{u^2}{2}\right|\le Cu^2\min\{|u|,1\}$,
		we rewrite the preceding identity as
		\begin{align}\label{phix}
			\varphi_n'(t)
			={}&-t\sum_i\mathbb E_Y
			[A_iU_i e^{it(W_n-U_i)}]
			-\frac{it^2}{2}\sum_i\mathbb E_Y
			[A_iU_i^2 e^{it(W_n-U_i)}]
			+R_{P,n}(t,x),
		\end{align}
		where
		\[
		R_{P,n}(t,x)
		:=i\sum_i\mathbb E_Y\!\Big[
		A_i e^{it(W_n-U_i)}
		\Big(e^{itU_i}-1-itU_i+\frac{t^2U_i^2}{2}\Big)\Big].
		\]
		To bound this remainder, set
		\[
		\sigma_i=\Big(\sum_j\widetilde c_{ij}^2\Big)^{1/2},\qquad
		V_i=\sigma_i^{-1}\sum_{j\ne i}\widetilde c_{ij}\widetilde Z_{j1},\qquad
		G_i=|\widetilde Z_{i1}^2-1|+2|\widetilde Z_{i1}V_i|.
		\]
		When $\sigma_i=0$, set $V_i=0$; that row contributes nothing.
		Then
		\[
		|A_i|\le\sigma_iG_i,\quad |U_i|\le\sigma_iG_i,\quad
		\max_i\sigma_i\le Cn^{-1/2},\quad
		\sum_i\sigma_i^3\le
		\left(\max_i\sigma_i\right)\sum_{i,j}\widetilde c_{ij}^2
		\le Cn^{-1/2}.
		\]
		Independence and the sixth-moment bound give
		$\mathbb E_Y|V_i|^6\le C$ and
		$\mathbb E_Y|\widetilde Z_{i1}V_i|^6\le C$.
		Together with uniform integrability of $|\widetilde Z_{i1}^2-1|^3$,
		this shows that the family $\{G_i^3\}$ is uniformly integrable.
		Consequently,
		\[
		\omega(\eta):=\sup\mathbb E_Y
		[G_i^3\min\{\eta G_i,1\}]
		\to0\qquad(\eta\downarrow0),
		\]
		where the supremum is over $n,i$, $|x|\le M_n$, and all
		conditioning values and coefficient arrays allowed on $\mathcal E_{3n}$.
		Hence
		\begin{align}
			|R_{P,n}(t,x)|
			\le Ct^2\sum_i\mathbb E_Y
			[|A_i||U_i|^2\min\{|tU_i|,1\}]\le \frac{Ct^2}{\sqrt n}\,
			\omega\!\left(\frac{C|t|}{\sqrt n}\right).
			\label{eq:eq53-peano-sixth-moment}
		\end{align}
		In particular, for every fixed $T<\infty$,
		\[
		\sup_{|x|\le M_n}\sup_{|t|\le T}|R_{P,n}(t,x)|
		=o(n^{-1/2}),
		\]
		uniformly on $\mathcal E_{3n}$.

		Since $W_n-U_i$ is fixed under $\mathbb E_i$ and
		$\mathbb E_iA_i=0$, we have, for $k=1,2$,
		\[
		\mathbb E_Y\left[
		A_i(W_n-U_i)^k
		\right]
		=
		0.
		\]
		Therefore, we derive
		\[
		\begin{aligned}
			&\sum_i\mathbb E_Y(A_iU_i)
			=
			\sum_i\mathbb E_Y(A_iW_n)
			=
			\mathbb E_YW_n^2
			=
			\operatorname{Var}_Y(W_n)
			=
			1,\\
			&\sum_i
			\mathbb E_Y\left[
			A_i\{2(W_n-U_i)U_i+U_i^2\}
			\right]=
			\sum_i\mathbb E_Y(A_iW_n^2)
			=
			\mathbb E_YW_n^3
			=
			\operatorname{cum}_{3,Y}(W_n).
		\end{aligned}
		\]
		
		Although the preceding Taylor expansion \eqref{phix} is exact, its first and second terms do not yet properly factor into the variance and third cumulant contributions. Consequently, we isolate the two principal terms
		$$ -t\varphi_n(t) \qquad\text{and}\qquad -\frac{i}{2}\lambda_{3n}(x)t^2\varphi_n(t), $$
		and collect all remaining terms into the residual $r_n(t,x)$. This residual incorporates the errors arising from the factorization of the variance and third cumulant terms, alongside the Peano remainder, leading to the equivalent expression
		$$ \begin{aligned}
			r_n(t,x) &= -t\sum_i \mathbb{E}_Y\left[ A_iU_i e^{it(W_n-U_i)} \right] + t\varphi_n(t)- \frac{it^2}{2} \sum_i \mathbb{E}_Y\left[ A_iU_i^2 e^{it(W_n-U_i)} \right]  \\
			&\qquad\quad + \frac{it^2}{2}\lambda_{3n}(x)\varphi_n(t)+ t^3\sum_i \mathbb{E}_Y\left[ A_iU_i^3e^{it(W_n-U_i)} \int_0^1 \frac{(1-s)^2}{2}e^{istU_i}\,ds \right].
		\end{aligned} $$

		We next work conditionally on $\mathcal F_Y$, with all estimates
		uniform on $\mathcal E_{3n}$ and for $|x|\le M_n$.
		Put $G=\sum_iA_iU_i$, so that $\mathbb E_YG=1$.
		For each nonempty subset $S\subseteq\{1,\ldots,n\}$, define
		\[
		g_S=\sum_{T\subseteq S}(-1)^{|S|-|T|}
		\mathbb E_Y(G\mid \widetilde Z_{j1},\ j\in T),
		\]
		where the conditional expectation for $T=\varnothing$
		is understood as $\mathbb E_YG$.
		For each $i$, set $B_i=\sum_S\frac{g_S}{|S|}$,
		where the sum runs over all subsets $S$ containing $i$
		and having at most three elements, and $|S|$ denotes
		the number of elements in $S$.
		Since each monomial of $G$ involves at most three coordinates,
		\[
		\sum_iB_i=G-1,
		\qquad
		\mathbb E_iB_i=0.
		\]
		Define $\mathcal T_n=\sum_iB_iU_i-\frac12\sum_iA_iU_i^2$. The preceding cumulant identities yield
		\[
		\mathbb E_Y\mathcal T_n
		=\mathbb E_Y(W_nG)
		-\frac12\sum_i\mathbb E_Y(A_iU_i^2)
		=\frac{\lambda_{3n}}2.
		\]
		Applying the coordinate-deletion identity to $A_i$ and $B_i$
		and using this equality, we obtain
		\begin{equation}\label{eq:a3-joint-cancellation}
			r_n=-it^2\mathbb E_Y
			\bigl[(\mathcal T_n-\mathbb E_Y\mathcal T_n)e^{itW_n}\bigr]
			+R_0-tR_1,
		\end{equation}
		where, writing
		$\Delta_i(s)=e^{it(W_n-sU_i)}-e^{itW_n}$,
		\[
		\begin{aligned}
			R_0=it^2\sum_i\int_0^1(1-s)\mathbb E_Y[A_iU_i^2\Delta_i(s)]\,ds,\qquad
			R_1=it\sum_i\int_0^1\mathbb E_Y[B_iU_i\Delta_i(s)]\,ds.
		\end{aligned}
		\]
		
		We record the moment estimates needed for these remainders.
		Write $d_i=\widetilde c_{ii}$, $K=\widetilde{\mathbf C}_n-\operatorname{diag}(d)$,
		$\ell_i=\sum_jK_{ij}\widetilde Z_{j1}$, $b_i=\sum_jK_{ij}^2$, and
		$a_i=\sum_jK_{ij}d_j\mathbb E_Y\widetilde Z_{j1}^3$.
		Direct centering gives
		\[
		B_i=d_i^2Q_i+\widehat B_i,\qquad
		Q_i=(\widetilde Z_{i1}^2-1)^2-\mathbb E_Y(\widetilde Z_{i1}^2-1)^2.
		\]
		Every coordinate in $\widehat B_i$ has power at most three; its squared
		coefficient norm, summed over the finitely many coordinate-power patterns,
		is at most $Cw_i^2$, where
		$w_i=(a_i^2+n^{-2})^{1/2}$ and $\sum_iw_i^2\le C/n$.
		Truncating and recentering the single-coordinate factors shows that
		$|\widehat B_i/w_i|^2$ is uniformly integrable: its $L^2$ truncation error
		is bounded by sixth-moment tails, while each fixed truncation has bounded
		fourth moment. The preceding Peano estimate also gives uniform
		integrability of $|U_i/\sigma_i|^2$ and $|Q_i|G_i$; the latter involves
		only $|\widetilde Z_{i1}|^6$ and $|\widetilde Z_{i1}|^5|V_i|$, with $\widetilde Z_{i1}$ independent of $V_i$.
		Consequently, enlarging the modulus $\omega$ if necessary,
		\begin{align}
			&\sum_i\mathbb E_Y(|A_iU_i^2|+|B_iU_i|)
			\le C\sum_i(\sigma_i^3+d_i^2\sigma_i+w_i\sigma_i)
			\le Cn^{-1/2},\label{eq:a3-weight-moments}\\
			&\sum_i\mathbb E_Y[|B_iU_i|\min\{|tU_i|,1\}]
			\le Cn^{-1/2}\omega(C|t|/\sqrt n).\notag
		\end{align}
		Together with the Peano bound for $A_iU_i^2$, this yields
		$|R_0|+|tR_1|\le Ct^2n^{-1/2}\omega(C|t|/\sqrt n)$.
		
		We next prove that
		$\|\mathcal T_n-\mathbb E_Y\mathcal T_n\|_1=o(n^{-1/2})$.
		For $R>1$, define $\mathcal T_n^{(R)}$ by replacing each
		factor $\widetilde Z_{i1}^r$, $4\le r\le6$, in the
		monomial expansion of $\mathcal T_n$ by
		\[
		\widetilde Z_{i1}^r
		\mathbf1_{\{|\widetilde Z_{i1}|\le R\}},
		\]
		leaving all lower powers unchanged.
		Since each monomial has total degree at most six,
		at most one factor in each monomial is altered.
		Put
		\[
		\varepsilon_R
		=\sup_{n,i}\mathbb E_Y
		\left[
		|\widetilde Z_{i1}|^6
		\mathbf1_{\{|\widetilde Z_{i1}|>R\}}
		\right]\longrightarrow0.
		\]
		Expanding the affected terms and using independence gives
		\[
		\|\mathcal T_n-\mathcal T_n^{(R)}\|_1
		\le C\varepsilon_R
		\sum_i\left(
		|d_i|^3+d_i^2\sqrt{b_i}+|d_i|b_i
		\right)
		\le C\varepsilon_R n^{-1/2}.
		\]
		
		For fixed $R$, we bound the variance of
		$\mathcal T_n^{(R)}$ by centering its single-coordinate
		factors. Each such factor has bounded second moment:
		powers above three are truncated, and the remaining
		powers require only sixth moments.
		Products supported on different coordinate sets
		are orthogonal.
		
		By \eqref{eq:eq53-normalized-class}, the coefficient
		vectors and matrices obtained after centering have
		Euclidean and Frobenius norms bounded by $Cn^{-1}$,
		respectively.
		For the remaining coefficient arrays, the sum of
		squared coefficients is bounded by
		\[
		\begin{aligned}
			C\left(|d|_\infty^2+\max_i b_i\right)^2
			\left(|d|^2+\|K\|_F^2\right)
			+C\sum_{i,j}|(K^2)_{ij}|^3
			+C\left(\max_i b_i\right)\|K^2\|_F^2
			=O(n^{-2}).
		\end{aligned}
		\]
		The moment factors introduced by centering are bounded;
		excluding coincident indices produces only terms of
		the same types. Orthogonality therefore yields
		\[
		\|\mathcal T_n^{(R)}
		-\mathbb E_Y\mathcal T_n^{(R)}\|_2
		\le C_Rn^{-1}.
		\]
		
		Combining the two estimates, we obtain
		\[
		\begin{aligned}
			\|\mathcal T_n-\mathbb E_Y\mathcal T_n\|_1
			\le
			2\|\mathcal T_n-\mathcal T_n^{(R)}\|_1
			+\|\mathcal T_n^{(R)}
			-\mathbb E_Y\mathcal T_n^{(R)}\|_2\le C\varepsilon_Rn^{-1/2}+C_Rn^{-1}.
		\end{aligned}
		\]
		Multiplying by $\sqrt n$, letting $n\to\infty$,
		and then letting $R\to\infty$ proves the desired bound.
		
		Equation~\eqref{eq:a3-joint-cancellation} now yields, for some $\eta_n\to0$,
		\[
		|r_n(t,x)|\le \frac{Ct^2}{\sqrt n}
		\{\eta_n+\omega(C|t|/\sqrt n)\}.
		\]
		
		Write $W_n^*=\sum_{\nu=1}^{2n}\zeta_\nu$, where
		\[
		\zeta_i=d_i(\widetilde Z_{i1}^2-1),\qquad
		\zeta_{n+i}=\frac{b_n(x)\eta_{ni}}{\sqrt{B_{2n}(x)}}.
		\]
		The summands have uniformly integrable normalized third absolute moments,
		$\sum_\nu\mathbb E_Y|\zeta_\nu|^3\le C/\sqrt n$, and
		$\max_\nu\mathbb E_Y\zeta_\nu^2\le C/n$.
		The same Peano expansion, with independent factors and matched cumulants,
		therefore gives
		\[
		r_n^*=(\varphi_n^*)'+t\varphi_n^*
		+\frac i2\lambda_{3n}t^2\varphi_n^*,\qquad
		|r_n^*|\le \frac{Ct^2}{\sqrt n}\omega(C|t|/\sqrt n)
		+C|t|^3/n+Ct^4/n^{3/2}.
		\]
		Thus, for $\delta_n=C\{\eta_n+\omega(Cn^{-1/4})\}\to0$,
		\begin{equation}\label{eq:a3-low-legacy-envelope}
			|r_n(t,x)|+|r_n^*(t,x)|
			\le \frac{\delta_n}{\sqrt n}(1+|t|^8)
			+\frac Cn(1+|t|^{12}),\qquad |t|\le n^{1/4}.
		\end{equation}
		The subsequent choice $L_n^{11}=o(\sqrt n)$ ensures $L_n\le n^{1/4}$,
		so the low-frequency integration below applies unchanged.

		The coefficient $\lambda_{3n}(x)$ is the same because $W_n$ and
		$W_n^*$ have the same first three conditional cumulants. Subtracting the two derivative identities gives
		\begin{align}\label{dt}
			\frac{d}{dt}
			\left\{
			\varphi_n(t)-\varphi_n^*(t)
			\right\}+
			\left\{
			t+\frac{i}{2}\lambda_{3n}(x)t^2
			\right\}
			\left\{
			\varphi_n(t)-\varphi_n^*(t)
			\right\}
			=
			r_n(t,x)-r_n^*(t,x).
		\end{align}
		Since
		\[
		\varphi_n(0)-\varphi_n^*(0)=0,
		\]
		the equation \eqref{dt} yields
		\[
		|\varphi_n(t)-\varphi_n^*(t)|
		\le
		\int_0^{|t|}
		e^{-(t^2-s^2)/2}
		\left\{
		|r_n(s,x)|+|r_n^*(s,x)|
		\right\}\,ds.
		\]
		Substituting the residual bound into the preceding inequality gives,
		for $0<t\le1$,
		\[
		|\varphi_n(t)-\varphi_n^*(t)|
		\le
		Ct\left(
		\frac{\delta_n}{\sqrt n}+\frac1n
		\right).
		\]
		For $t\ge1$ and every fixed $m\ge1$, we have
		\[
		\int_0^t
		e^{-(t^2-s^2)/2}(1+s^m)\,ds
		\le
		C(1+t^{m-1}).
		\]
		Consequently,
		\[
		\int_0^{L_n}
		\frac{|\varphi_n(t)-\varphi_n^*(t)|}{t}\,dt
		\le
		\frac{C\delta_n}{\sqrt n}(1+L_n^7)
		+
		\frac{C}{n}(1+L_n^{11}).
		\]
		Choose $L_n\to\infty$ sufficiently slowly that $
		\delta_nL_n^7\to0$ and $
		L_n^{11}=o(\sqrt n)$.
		Then the right-hand side is $o(n^{-1/2})$. Using $
		|\varphi_n(-t)-\varphi_n^*(-t)|
		=
		|\varphi_n(t)-\varphi_n^*(t)|$
		proves
		\[
		\sup_{|x|\le M_n}
		\int_{0<|t|\le L_n}
		\frac{|\varphi_n(t)-\varphi_n^*(t)|}{|t|}\,dt
		=
		o(n^{-1/2}).
		\]

		\emph{\rm Step 2.}	Multiplying equation \eqref{dt} by $
		\exp\left\{
		\frac{t^2}{2}
		+
		\frac{i}{6}\lambda_{3n}(x)t^3
		\right\} $
		and integrating from $L_n$ to $t$ gives
		\begin{align*}
			|\varphi_n(t)-\varphi_n^*(t)|\le
			e^{-(t^2-L_n^2)/2}
			|\varphi_n(L_n)-\varphi_n^*(L_n)|+
			\int_{L_n}^t
			e^{-(t^2-s^2)/2}
			|r_n(s,x)-r_n^*(s,x)|\,ds.
		\end{align*}
		The first step gives us
		\begin{equation}
			\sup_{|x|\le M_n}
			|\varphi_n(L_n)-\varphi_n^*(L_n)|
			\le
			C\left\{
			\frac{\delta_n}{\sqrt{n}}(1+L_n^7)
			+
			\frac{1}{n}(1+L_n^{11})
			\right\}
			=
			o(n^{-1/2}).
			\label{eq:eq53-middle-boundary}
		\end{equation}
		It therefore remains to estimate
		$r_n(t,x)-r_n^*(t,x)$.
		By analogous reasoning, we only need to examine $L_n<t\le\epsilon\sqrt{n}$.

		According to Assumptions $(a)$ and $(d)$, there exist constants $s_0,c>0$ such that
		\begin{equation}
			\sup_{u\in\mathbb R}
			\left|
			\mathbb E
			\exp\left\{
			i\left[
			s(\widetilde Z_{i1}^{\,2}-1)
			+
			u\widetilde Z_{i1}
			\right]
			\right\}
			\right|
			\le
			e^{-cs^2},
			\qquad |s|\le s_0.
			\label{eq:eq53-local-score-middle}
		\end{equation}
		
		By \eqref{eq:eq53-normalized-class}, there are constants
		$c_0,C_0,c_1>0$ such that
		\[
		J_0=\{i:c_0n^{-1/2}\le|\widetilde c_{ii}|\le C_0n^{-1/2}\},
		\qquad |J_0|\ge c_1n.
		\]

		For $L_n<|t|\le\epsilon\sqrt n$, put
		\[
		\theta_t=\frac{A\log(2+|t|)}{t^2},\qquad
		m_t=\lfloor n\theta_t\rfloor,\qquad
		D(t)=\frac{\log^4(2+|t|)}{t^4}.
		\]
		For every fixed $i$, subset averaging gives four disjoint sets
		$I_1,\ldots,I_4\subset J_0\setminus\{i\}$, each of size $m_t$, such that
		\[
		\sum_{\substack{j,k\in I_\ell\\j\ne k}}
		\widetilde c_{jk}^2\le C\theta_t^2,\qquad 1\le\ell\le4.
		\]
		Indeed, the average internal sum is $O(m_t^2/n^2)$ by
		\eqref{eq:eq53-normalized-class}, and removing fewer than four such
		sets leaves at least $cn$ available indices because $m_t=o(n)$.
		These choices depend only on $\widetilde{\mathbf C}_n,t,i$.
		
		For any such set $I$, condition on the exterior coordinates and write
		\[
		e^{itW_n}=e^{itC_I}\Big(\prod_{j\in I}f_j\Big)e^{itV_I},\qquad
		V_I=2\sum_{\substack{j<k\\j,k\in I}}
		\widetilde c_{jk}\widetilde Z_{j1}\widetilde Z_{k1},
		\]
		where $C_I$ depends only on the exterior coordinates and
		\[
		f_j=\exp\Big\{it\Big[
		\widetilde c_{jj}(\widetilde Z_{j1}^2-1)
		+2\widetilde Z_{j1}\sum_{k\notin I}\widetilde c_{jk}\widetilde Z_{k1}\Big]\Big\}.
		\]
		By \eqref{eq:eq53-local-score-middle} and the definition of $J_0$,
		$|\mathbb E_If_j|\le e^{-ct^2/n}$ uniformly in the exterior coordinates.
		Averaging over bipartitions of $I$ and applying the sixth-moment
		bound for linear forms to each part gives
		$\|V_I\|_{6,I}\le C\theta_t$.
		Each term of $V_I^p$ involves at most $2p$ coordinates.
		Applying Lemma~\ref{integration} and expanding $e^{itV_I}$ through degree three gives
		\begin{equation}\label{eq:a3-middle-block}
			\left|\mathbb E_Ie^{itW_n}\right|
			\le Ce^{-cm_tt^2/n}+Ct^4\theta_t^4\le CD(t).
		\end{equation}
		For $i\notin I$, the same bound holds for $W_n-sU_i$, uniformly in
		$0\le s\le1$, since only the exterior constant and linear coefficients
		change.
		
		Each term of $\mathcal Z_i\in\{A_iU_i^2,B_iU_i\}$ involves at most
		three coordinates other than $i$.
		Applying Lemma~\ref{integration}
		conditionally on $\mathcal F_Y$ and $\widetilde Z_{i1}$, and then averaging over $\widetilde Z_{i1}$,
		gives
		\[
		\left|\mathbb E_Y[\mathcal Z_i e^{it(W_n-sU_i)}]\right|
		\le CD(t)\mathbb E_Y|\mathcal Z_i|,\qquad 0\le s\le1.
		\]
		
		Apply this to the two integral remainders and the two sums in
		$\mathcal T_n$. By \eqref{eq:a3-weight-moments},
		their total absolute weight is $O(n^{-1/2})$.
		Also $\mathbb E_Y\mathcal T_n=\lambda_{3n}/2=O(n^{-1/2})$, whose
		contribution is controlled by the unweighted block estimate.
		Hence \eqref{eq:a3-joint-cancellation} gives
		$|r_n(t,x)|\le Ct^2D(t)/\sqrt n$.
		
		For $W_n^*=\sum_\nu\zeta_\nu$, the same identities use
		$A_\nu=U_\nu=\zeta_\nu$ and
		$B_\nu=\zeta_\nu^2-\mathbb E_Y\zeta_\nu^2$.
		Their total absolute weight is bounded by
		$C\sum_\nu\mathbb E_Y|\zeta_\nu|^3\le C/\sqrt n$.
		A block of unchanged diagonal summands, excluding the active coordinate,
		has characteristic-function modulus at most $e^{-cm_tt^2/n}\le CD(t)$.
		Thus the same bound holds for $r_n^*$, and
		\begin{equation}\label{eq:eq53-middle-residual-bound}
			|r_n(t,x)-r_n^*(t,x)|
			\le\frac C{\sqrt n}\frac{\log^4(2+t)}{t^2},
			\qquad L_n<t\le\epsilon\sqrt n.
		\end{equation}
		Consequently,
		\begin{equation}\label{eq:eq53-middle-residual-integral}
			\sup_{|x|\le M_n}\int_{L_n}^{\epsilon\sqrt n}
			\frac{|r_n(t,x)-r_n^*(t,x)|}{1+t^2}\,dt
			\le\frac C{\sqrt n}\int_{L_n}^\infty
			\frac{\log^4(2+t)}{t^4}\,dt=o(n^{-1/2}).
		\end{equation}

		Since
		\[
		\int_s^\infty
		e^{-(t^2-s^2)/2}\frac{dt}{t}
		\le
		\frac{C}{1+s^2},
		\qquad s\ge1,
		\]
		Fubini's theorem gives
		\[
		\begin{aligned}
			\int_{L_n}^{\epsilon\sqrt n}
			\frac{|\varphi_n(t)-\varphi_n^*(t)|}{t}\,dt\le
			\frac{C}{L_n^2}
			|\varphi_n(L_n)-\varphi_n^*(L_n)|
			+
			C\int_{L_n}^{\epsilon\sqrt n}
			\frac{|r_n(s,x)-r_n^*(s,x)|}{1+s^2}\,ds.
		\end{aligned}
		\]
		By \eqref{eq:eq53-middle-boundary} and
		\eqref{eq:eq53-middle-residual-integral}, the right-hand side is
		$o(n^{-1/2})$, uniformly for $|x|\le M_n$. Using the conjugacy
		relation for negative $t$, we finally obtain
		\[
		\sup_{|x|\le M_n}
		\int_{L_n<|t|\le\epsilon\sqrt n}
		\frac{|\varphi_n(t)-\varphi_n^*(t)|}{|t|}\,dt
		=
		o(n^{-1/2}).
		\]

		\emph{\rm{Step 3}.}
		Fix $B>\epsilon$ and let $m=\lceil K\log n\rceil$.
		Subset averaging gives $J\subset J_0$, $|J|=m$, with
		\begin{equation}\label{eq:eq53-high-internal-energy}
			\sum_{\substack{i,j\in J\\i\ne j}}\widetilde c_{ij}^2
			\le Cm^2/n^2.
		\end{equation}
		Use the preceding conditional decomposition with $I=J$.
		Since $\epsilon c_0\le |t\widetilde c_{ii}|\le BC_0$,
		Lemma~\ref{lem:projected-score-gap} bounds each one-coordinate
		factor strictly below one, uniformly in the exterior coordinates.
		Lemma~\ref{integration} and $|e^{iu}-1-iu|\le u^2/2$ therefore give, for some
		$q=q_{\epsilon,B}<1$,
		\[
		|\varphi_n(t)|
		\le Cq^m(1+|t|\|V_J\|_{2,J})
		+Ct^2\|V_J\|_{2,J}^2.
		\]
		Here $\|V_J\|_{2,J}^2
		=4\sum_{i<j,\ i,j\in J}\widetilde c_{ij}^2\le Cm^2/n^2$.
		Thus, choosing $K$ sufficiently large for any fixed $R>1$ yields
		\begin{equation}\label{eq:eq53-high-cf}
			\sup_{\epsilon\sqrt n\le|t|\le B\sqrt n}|\varphi_n(t)|
			\le C_{R,\epsilon,B}n^{-R}
			+C_{\epsilon,B}\frac{\log^2n}{n}.
		\end{equation}
		The independent Gaussian component of $W_n^*$ has variance
		bounded away from zero, so $|\varphi_n^*(t)|\le e^{-ct^2}$.
		Consequently, for fixed $B$,
		\[
		\sup_{|x|\le M_n}
		\int_{\epsilon\sqrt n<|t|\le B\sqrt n}
		\frac{|\varphi_n(t)-\varphi_n^*(t)|}{|t|}\,dt
		=o(n^{-1/2}).
		\]

		Since the preceding conclusion holds for every fixed $B$, choose a piecewise constant sequence $B_n\to\infty$ that increases sufficiently slowly such that, by setting
		$$ T_n=B_n\sqrt{n}, $$
		we have
		\begin{equation}
			\sup_{|x|\le M_n}
			\int_{\epsilon\sqrt{n}<|t|\le T_n}
			\frac{|\varphi_n(t)-\varphi_n^*(t)|}{|t|}\,dt
			=
			o(n^{-1/2}).
			\label{eq:eq53-high-frequency}
		\end{equation}
		
		Finally, the Gaussian component of $W_n^*$ implies that its
		conditional density is bounded uniformly in $x$ and in the
		conditioning value. Hence
		\[
		\sup_z
		\mathbb P_Y\{z<W_n^*\le z+h\}
		\le Ch.
		\]
		The Esseen smoothing inequality gives
		\[
		\begin{aligned}
			\sup_z
			\left|
			\mathbb P_Y(W_n\le z)
			-
			\mathbb P_Y(W_n^*\le z)
			\right|\le
			C\int_{0<|t|\le T_n}
			\frac{|\varphi_n(t)-\varphi_n^*(t)|}{|t|}\,dt
			+
			CT_n^{-1}.
		\end{aligned}
		\]
		Steps 1--3 and Esseen's inequality give a uniform
		$o(n^{-1/2})$ bound on $\mathcal E_{3n}$.
		Averaging over $\mathcal F_Y$ and bounding the conditional
		distributional distance by one on the complement proves
		\eqref{eq:eq53-standardized-comparison}.
	\end{proof}

	We first verify the matrix assumptions for the actual coefficient array.
	Put
	\[
	P=\mathbf V_2\mathbf V_2',\qquad
	S_Y=n^{-1}\mathbf YP\mathbf Y'=(s_{ij}),\qquad
	A=(\rho_n I-S_Y)^{-1}.
	\]
	Here $P$ is an orthogonal projection of rank $p$.
	
	Let
	\[
	V_n=\frac1n\sum_{i\ne j}s_{ij}^2
	=\frac1{n^3}\sum_{i\ne j}T_{ij},\qquad
	T_{ij}=(\mathbf y_i'P\mathbf y_j)^2.
	\]
	Independence and bounded fourth moments give
	\[
	\mathbb ET_{12}=p,\qquad
	\mathbb ET_{12}^2\le Cp^2,\qquad
	\operatorname{Cov}(T_{12},T_{13})
	=\operatorname{Var}(\mathbf y_1'P\mathbf y_1)\le Cp.
	\]
	Terms on disjoint pairs of rows are independent. Hence
	\[
	\mathbb EV_n=(1-n^{-1})\gamma_n,\qquad
	\operatorname{Var}(V_n)
	\le C\Big(\frac{p^2}{n^4}+\frac p{n^3}\Big)
	=O(n^{-2}).
	\]
	Since $\gamma_n\to\gamma>0$, Chebyshev's inequality yields
	$\mathbb P(V_n<\gamma_n/2)=O(n^{-2})$.
	By Lemma~\ref{qtruncation} and the preceding bound,
	for some fixed $a>0$ we may take
	\[
	\mathcal E_{3n}
	=\left\{
	\lambda_{\max}(S_Y)\le\rho_n-a,\quad
	V_n\ge\gamma_n/2
	\right\}\in\mathcal F_Y,
	\qquad
	\mathbb P(\mathcal E_{3n}^c)=o(n^{-1/2}).
	\]
	
	Write $A_{\mathrm d}=\operatorname{diag}(a_{11},\ldots,a_{nn})$
	and let $\operatorname{off}(L)$ denote $L$ with its diagonal removed.
	On $\mathcal E_{3n}$,
	$\rho_n^{-1}I\preceq A\preceq a^{-1}I$.
	The inverse identity gives the deterministic estimate
	\[
	\begin{aligned}
		\|\operatorname{off}(S_Y)\|_F
		=\|\operatorname{off}(A^{-1}-A_{\mathrm d}^{-1})\|_F\le\|A^{-1}\|\,\|A_{\mathrm d}^{-1}\|\,
		\|A-A_{\mathrm d}\|_F
		\le\rho_n^2\|\operatorname{off}(A)\|_F .
	\end{aligned}
	\]
	Thus, on $\mathcal E_{3n}$,
	\[
	\min_i a_{ii}\ge c,\qquad
	n^{-1}\|\operatorname{off}(A)\|_F^2
	\ge\rho_n^{-4}V_n\ge c.
	\]
	
	Recall that $\mathbf M=S_YA^2$ and
	$\mathbf C_n(x)=c_n(x)A-n^{-1/2}\tau_n(x)\mathbf M$.
	The spectral gap and the definitions of $c_n,\tau_n$ imply
	\[
	\sup_{|x|\le M_n}\|\mathbf C_n(x)-A\|
	\le C M_n/\sqrt n=o(1).
	\]
	Using $\|\operatorname{off}(L)\|_F\le\sqrt n\|L\|$, we obtain,
	uniformly for $|x|\le M_n$ on $\mathcal E_{3n}$,
	\[
	\|\mathbf C_n(x)\|\le C,\qquad
	\sum_i c_{ii}(x)^2\ge cn,\qquad
	\sum_{i\ne j}c_{ij}(x)^2\ge cn.
	\]
	The conditional variance formula is
	\[
	B_{2n}(x)=\widetilde\kappa_2\sum_i c_{ii}(x)^2
	+2\sum_{i\ne j}c_{ij}(x)^2,
	\qquad
	\widetilde\kappa_2=\operatorname{Var}(X_1^2)\le C.
	\]
	Consequently $cn\le B_{2n}(x)\le Cn$ and
	$\widetilde{\mathbf C}_n(x)=\mathbf C_n(x)/\sqrt{B_{2n}(x)}$ satisfies
	\eqref{eq:eq53-normalized-class}. In particular,
	\[
	c\le b_n(x)^2=\frac2n\sum_{i\ne j}c_{ij}(x)^2\le C.
	\]
	The third-cumulant formula, bounded sixth moments, and
	$\|\mathbf C_n(x)\|\le C$ also give $|B_{3n}(x)|\le Cn$.
	Since $\sum_i|c_{ii}(x)|^3\le Cn$, this proves
	$|\gamma_n^*(x)|\le C$ as well.

	We next verify the remaining conditions on $Q_n^*(x)$.
	Recall that
	\[
	b_n(x)^2=\frac4n\sum_{i<j}c_{ij}(x)^2,
	\qquad
	\gamma_n^*(x)=
	\frac{B_{3n}(x)-\widetilde\kappa_3\sum_i c_{ii}(x)^3}
	{n b_n(x)^3}.
	\]
	Hence, on $\mathcal E_{3n}$, $b_n(x)$ is bounded above and away
	from zero and $|\gamma_n^*(x)|\le C$. 
	Fix $\varepsilon=1/2$. Let $
	(\xi_i)_{i\ge1}$ and
	$(G_i)_{i\ge1}$
	be mutually independent sequences of i.i.d. $U(0,1)$ and
	$N(0,1)$ random variables, respectively. Assume that both sequences
	are independent of $\mathcal F_Y$ and $(\widetilde Z_{i1})_{i\ge1}$. For $|x|\le M_n$, put
	\[
	p_n(x)
	=
	\frac12
	\bigg\{
	1-
	\frac{\gamma_n^*(x)}
	{\sqrt{\{\gamma_n^*(x)\}^2+4(1-\varepsilon)^3}}
	\bigg\}
	\]
	and define
	\[
	\eta_{ni}(x)
	=
	\sqrt{1-\varepsilon}\,
	\frac{
		\mathbf 1_{\{\xi_i\le p_n(x)\}}-p_n(x)
	}{
		\sqrt{p_n(x)\{1-p_n(x)\}}
	}
	+
	\sqrt{\varepsilon}\,G_i,
	\qquad 1\le i\le n.
	\]
	On $\mathcal E_{3n}^c$, set $p_n(x)=1/2$. On $\mathcal E_{3n}$, the values of $p_n(x)$ are bounded away
	from zero and one. Hence
	\[
	\sup_{n,i,\,|x|\le M_n}
	\mathbb E_Y|\eta_{ni}(x)|^6
	<\infty.
	\]
	Finally, set $Q_n^*(x)=b_n(x)\sum_{i=1}^n\eta_{ni}(x)$. The first three conditional cumulants then match exactly.  Indeed,
	\begin{align*}
		&\mathbb E_Y\{D_n+Q_n\}=\mathbb E_Y\{D_n+Q_n^*\}=0,\qquad
		\operatorname{Var}_YQ_n
		=4\sum_{i<j}c_{ij}^2
		=nb_n^2
		=\operatorname{Var}_YQ_n^*,\\
		&\operatorname{Cov}_Y(D_n,Q_n)=0,\qquad 	
		\operatorname{cum}_{3,Y}(D_n+Q_n^*)
		=\widetilde\kappa_3\sum_i c_{ii}^3
		+nb_n^3\gamma_n^*
		=B_{3n}.
	\end{align*}
	Moreover, the Gaussian part of
	$(D_n+Q_n^*)/\sqrt{B_{2n}}$ has variance
	\[
	\frac{\varepsilon n b_n(x)^2}{B_{2n}(x)}\ge c>0,
	\]
	and the third-moment Lindeberg condition follows from the finite sixth moment
	of $\widetilde Z_{i1}$, the uniformly bounded sixth moments of $\eta_{ni}$, and the
	uniform boundedness of $c_{ii}(x)$ and $b_n(x)$.  Thus all assumptions of
	Lemma~\ref{lem:eq53-matched-comparison} hold.
	
	Applying Lemma~\ref{lem:eq53-matched-comparison}
	gives \eqref{eq:eq53-target}.  In particular, it remains valid when
	$z$ is replaced by the $\mathcal F_Y$-measurable random threshold appearing
	in the proof of the limiting-distribution theorem.  This proves
	Equation~\eqref{step3}.

\subsection{ Proof of Theorem \ref{limiting}}
In this section, we represent the proposed statistics as sums of independent random variables, which allows us to apply established Edgeworth expansion theory to derive the corresponding expansion of the distribution function of our statistic. Then, we derive the following result:
\begin{align*}
	&\qquad
	\mathbb{P}\bigg(
	\frac{
		-\Omega(\rho_n,\Y)
		+\rho_n n^{-1/2}\widetilde G_{n,\Y}(g_n)
	}{
		[\rho_n\mathrm F_{\gamma_n}(g_n^2)
		-n^{-1/2}S_n(g_nh_n)]
		\widetilde\sigma_n
	}
	\le x
	\bigg)
	\\
	&=
	\mathbb{E}\bigg[
	\mathbb{P}\bigg(
	S_n(g_n)
	+n^{-1/2}S_n(g_nh_n)
	\rho_n^{-1}\widetilde\sigma_nx
	\le y_n^0
	\,\bigg|\,
	\mathcal F_Y
	\bigg)
	\bigg]
	+o(n^{-1/2})
	\\
	&=
	\mathbb{E}\bigg[
	\mathbb{P}\bigg(
	n^{-1/2}(D_n+Q_n^*)
	\le y_n^0
	\,\bigg|\,
	\mathcal F_Y
	\bigg)
	\bigg]
	+o(n^{-1/2}).
\end{align*}
where $y_n^0=\mathrm{F}_{\gamma_n}(g_n^2)\widetilde\sigma_nx-n^{-1/2}\widetilde G_{n,\Y}(g_n)$.

Next, we express the following statistic as a sum of independent random variables,
\begin{align*}
	D_n+Q_n^* :=\sum_{i=1}^nX_{ni},~~ X_{ni}=c_{ii}(\widetilde Z_{i1}^2-1)+b_n\eta_{ni},~~ c_{ii}=a_{ii}c_n-n^{-1/2}\rho_n^{-1}\widetilde\sigma_nx m_{ii}.
\end{align*}

\textbf{Moment analysis}. Note that $\widetilde{Z}_{11}$ is a random variable with mean $0$ and variance $1$. We define
$ \widetilde{\kappa}_2 = \mathbb{E}(\widetilde{Z}_{11}^2-1)^2$ and $\widetilde{\kappa}_3 = \mathbb{E}(\widetilde{Z}_{11}^2-1)^3 $.
 Furthermore, we set $\bar{V}_n=\frac{1}{n}\sum_{i=1}^n \mathbb{E}(X_{ni}^2\mid \mathcal{F}_Y)=\widetilde{\kappa}_2n^{-1} \sum_{i=1}^n c_{ii}^2 +b_n^2$ and $d_{ni}=\big(1+n^{-1/2}\tau_nh_n(\lambda_i)\big)g_n(\lambda_i)$.
Finally, let
$$ y_n=\bar{V}_n^{-1/2}y_n^0 =\bar{V}_n^{-1/2}\left\{ \mathrm{F}_{\gamma_n}(g_n^2)\widetilde{\sigma}_nx -n^{-1/2}\widetilde{G}_{n,\Y}(g_n)\right\}. $$
Observe that
\begin{align*}
	\sum_{i=1}^nc_{ii}=\sum_{i=1}^n\big(1+n^{-1/2}\tau_nh_n(\lambda_i)\big)g_n(\lambda_i),\quad b_n^2=\frac{2}{n}\Big(\sum_{i=1}^nd_{ni}^2-\sum_{i=1}^nc_{ii}^2\Big).
\end{align*}

 Moreover, the proof of \eqref{step3} constructs events
	$\mathcal G_n\in\mathcal F_Y$ such that
	$\mathbb P(\mathcal G_n^c)=o(n^{-1/2})$ and all the following estimates hold
	uniformly for $x\in\mathcal M_n:=\{|x|\le M_n\}$, where
	$M_n=K\sqrt{\log n}=o(\sqrt n)$ for a fixed $K>1$. Hence, on
$\mathcal G_n$ we derive:
\begin{align*}
	&|c_{ii}|=|a_{ii}c_n-n^{-1/2}\rho_n^{-1}\widetilde\sigma_nx m_{ii}|\le |a_{ii}c_n|+n^{-1/2}|\rho_n^{-1}\widetilde\sigma_nx m_{ii}|\le C,\\
	&|d_{ni}|=|(1+n^{-1/2}\tau_nh_n(\lambda_i))g_n(\lambda_i)|\le |g_n(\lambda_i)|+n^{-1/2}|\rho_n^{-1}\widetilde\sigma_nx h_n(\lambda_i)g_n(\lambda_i)|\le C\\
	&n^{-1}\bar V_n^{-3/2}\sum_{i=1}^n\mathbb{E}\Big(|X_{ni}|^3|\mathcal F_Y\Big)
	\le n^{-1}\bar V_n^{-3/2}\sum_{i=1}^n\Big[4|c_{ii}|^3\mathbb{E}\Big(|(\widetilde Z_{i1}^2-1)|^3\Big)+4|b_n|^3\mathbb{E}|\eta_{ni}|^3\Big]\\
	&\qquad\qquad\qquad\qquad\qquad\qquad\le C\mathbb{E}\Big(|\widetilde Z_{11}^2-1|^3+O(1)\Big),
\end{align*}
hence R1 holds.   For any $\tau\in(0,1/2)$, using the above bound and the integrability of
	$|\widetilde Z_{11}^2-1|^3$, we obtain:
	\begin{align*}
		&\qquad n^{-1}\bar V_{n}^{-3/2}\sum_{i=1}^n\mathbb{E}\Big(I\{\bar V_n^{-1/2}|X_{ni}|>n^{\tau}\}|X_{ni}|^3\mid\mathcal F_Y\Big)\\
		&\le C\mathbb{E}|\widetilde Z_{11}^2-1|^3I\{|\widetilde Z_{11}^2-1|>C^{-1} n^{\tau}\}+\sup_{n,i}\mathbb{E}|\eta_{ni}|^3I\{|\eta_{ni}|>C^{-1} n^{\tau}\}=o(1),
\end{align*}
hence R2 holds.	 Because $\eta_{ni}$ contains a fixed, non-degenerate Gaussian component, we can derive:
\[n^{1/2}\int_{\{|t|>\epsilon\}}|t|^{-1}\Big|\mathbb{E}\Big[\exp\Big(it\bar V_n^{-1/2}\sum_{i=1}^n X_{ni}\big)\Big|\,
\mathcal F_Y\Big]\Big|dt\le n^{1/2}\int_{\{|t|>\epsilon\}}|t|^{-1}e^{-Cnt^2} dt =o(1),\]
for every fixed $\epsilon>0$, hence R3 holds. 
Let $\bar\kappa_3
=n^{-1}\widetilde\kappa_3\sum_i c_{ii}^3+b_n^3\gamma_n^*$.
Applying Lemma~\ref{EE condition} conditionally on $\mathcal{F}_Y$, utilizing the preceding uniform bounds, and subsequently taking expectations yields the following result:
\begin{align}\label{Mbound}
	&\mathbb E\Big[
	\sup_{x\in\mathcal M_n}
	\Big|
	\mathbb P\left(
	n^{-1/2}(D_n+Q_n^*)\le y_n^0
	\,\middle|\,\mathcal F_Y
	\right)\\\nonumber
	&\qquad\qquad\qquad
	-\Phi(y_n)
	-\frac{1}{6}n^{-1/2}\bar V_n^{-3/2}\bar\kappa_3
	(1-y_n^2)\phi(y_n)
	\Big|
	\Big]
	=o(n^{-1/2}).
\end{align}

Before presenting our main theorem, we require the following auxiliary lemma.

\begin{lemma}\label{PO} If there exists a polynomial $p_n(t)=\sum_{i=0}^ke_{ni}t^i$ and $\mathbb{E}|e_{ni}|=O(n^{-\alpha})$ for some $\alpha>0$, then $p_n$ is called $PO(n^{-\alpha})$. With this definition,
	\begin{align}
		\label{6.2} \bar V_n(x)-\kappa_{21}-2n^{-1/2}\kappa_{21}\mathfrak c_1 x&=PO(n^{-1}),\\
		\label{6.3} \bar\kappa_{3}-\kappa_{31}&=PO(n^{-1/2}).
	\end{align}
\end{lemma}

\begin{proof}Recall $
	\mathrm F_n(f)
	=
	\frac1n\sum_{i=1}^nf(\lambda_i)$ and $
	\widetilde G_{n,\Y}(f)
	=
	n\left\{
	\mathrm F_n(f)-\mathrm F_{\gamma_n}(f)
	\right\}$.
	The proofs of equations \eqref{6.2} and \eqref{6.3} proceed as follows:
		\begin{align*}
			\bar V_n-\kappa_{21}&=\widetilde{\kappa}_2n^{-1}
			\sum_{i=1}^n c_{ii}^2
			+b_n^2-\kappa_{21}=\widetilde{\kappa}_2n^{-1}
			\sum_{i=1}^n c_{ii}^2
			+b_n^2-2\mathrm{F}_{\gamma_n}(g_n^2)-(\widetilde \kappa_2-2)\mathrm{F}^2_{\gamma_n}(g_n)\\
			&=n^{-1}\widetilde\kappa_{2}\sum_{i=1}^nc_{ii}^2-(\widetilde\kappa_{2}-2)\mathrm{F}^2_{\gamma_n}(g_n)+\frac{2}{n}\Big(\sum_{i=1}^nd_{ni}^2-\sum_{i=1}^nc_{ii}^2\Big)-2\mathrm{F}_{\gamma_n}(g_n^2)\\
			&= 2\Big(n^{-1}\sum_{i=1}^nd_{ni}^2-\mathrm{F}_{\gamma_n}(g_n^2)\Big)+(\widetilde\kappa_{2}-2)\Big(n^{-1}\sum_{i=1}^nc_{ii}^2-\mathrm{F}^2_{\gamma_n}(g_n)\Big)\\
			&=2n^{-1}\Big(\widetilde G_{n,\Y}(g_n^2)+\mathrm{F}_n(h_n^2g_n^2)\rho_n^{-2}(\widetilde{\sigma}_n)^2x^2+2n^{1/2}\mathrm{F}_n(h_ng^2_n)\rho_n^{-1}(\widetilde{\sigma}_n)x\Big)\\
			&\qquad+(\widetilde\kappa_{2}-2)\Big[\Big(n^{-1}\sum_{i=1}^nc_{ii}\Big)^2-\mathrm{F}^2_{\gamma_n}(g_n)+n^{-1}\sum_{i=1}^n\Big(c_{ii}-n^{-1}\sum_{i=1}^nc_{ii}\Big)^2\Big]\\
			&=2n^{-1}\Big(\widetilde G_{n,\Y}(g_n^2)+\mathrm{F}_n(h_n^2g_n^2)\rho_n^{-2}(\widetilde{\sigma}_n)^2x^2+2n^{-1/2}\widetilde G_{n,\Y}(g_n^2h_n)\rho_n^{-1}(\widetilde{\sigma}_n)x\Big)\\
			&\qquad+(\widetilde\kappa_{2}-2)\Big[\Big(n^{-1}\sum_{i=1}^nc_{ii}\Big)^2-\mathrm{F}^2_{\gamma_n}(g_n)+n^{-1}\sum_{i=1}^n\Big(c_{ii}-n^{-1}\sum_{i=1}^nc_{ii}\Big)^2\Big].
	\end{align*}
	The penultimate equality is obtained through the identity $\mathrm{F}_n(h_ng^2_n)-n^{-1}\widetilde{G}_{n,\Y}(g_n^2h_n)=0$. Besides,
	by the definitions of $c_{ii}$ and $c_n$, we have
	\[
	\Big(
	\frac1n\sum_{i=1}^n c_{ii}
	\Big)^2
	-\mathrm F_{\gamma_n}^2(g_n)
	=
	2n^{-1/2}\rho_n^{-1}\widetilde\sigma_n
	\mathrm F_{\gamma_n}(g_n)
	\mathrm F_{\gamma_n}(g_nh_n)x
	+PO(n^{-1}).
	\]
	Hence, we derive
	\[
	(\widetilde\kappa_2-2)
	\Big[
	\Big(
	\frac1n\sum_{i=1}^n c_{ii}
	\Big)^2
	-\mathrm F_{\gamma_n}^2(g_n)
	\Big]
	=
	2n^{-1/2}\kappa_{21}\mathfrak c_1x
	+PO(n^{-1}).
	\]
	Consequently,
	\[
	\bar V_n(x)-\kappa_{21}
	-2n^{-1/2}\kappa_{21}\mathfrak c_1x
	=
	PO(n^{-1}).
	\]

	Define $\mathcal R_{3n}=8n^{-1}\sum_{i<j}c_{ij}^3+12n^{-1}\sum_{i<j}c_{ii}c_{ij}c_{jj}$. It can be readily verified that $\mathcal R_{3n}=PO(n^{-1/2})$. Define $\bar c_1=n^{-1}\sum_{i=1}^nc_{ii}$ and $\bar c_2=n^{-1}\sum_{i=1}^n(\mathbf C_n(x)^2)_{ii}$.
	Applying analogous arguments yields:
			\begin{align*}
			&\bar\kappa_{3}-\kappa_{31}=n^{-1}\widetilde\kappa_{3}\sum_{i=1}^nc_{ii}^3+b_n^3\gamma_n^*-\kappa_{31}\\
			&=\frac8n\sum_{i=1}^nd_{ni}^3+
			\frac{12(\widetilde\kappa_2-2)}{n}
			\sum_{i=1}^n
			c_{ii}(\mathbf C_n^2)_{ii}+
			\frac{\widetilde\kappa_3-12\widetilde\kappa_2+16}{n}
			\sum_{i=1}^nc_{ii}^3
			+(\mathbb{E}\widetilde Z^3_{11})^2\mathcal R_{3n}\\
			&\quad-8\mathrm{F}_{\gamma_n}(g_n^3)-12(\widetilde\kappa_2-2)\mathrm{F}_{\gamma_n}(g_n)\mathrm{F}_{\gamma_n}(g_n^2)-(\widetilde\kappa_3-12\widetilde\kappa_2+16)\mathrm{F}_{\gamma_n}^3(g_n)\\
			&=8n^{-1}\sum_{i=1}^nd_{ni}^3-8\mathrm{F}_{\gamma_n}(g_n^3)+12(\widetilde\kappa_2-2)\Big(n^{-1}	\sum_{i=1}^n
			c_{ii}(\mathbf C_n^2)_{ii}-\mathrm{F}_{\gamma_n}(g_n)\mathrm{F}_{\gamma_n}(g_n^2)\Big)\\
			&\quad +(\widetilde\kappa_3-12\widetilde\kappa_2+16)\Big(n^{-1}\sum_{i=1}^nc_{ii}^3-\mathrm{F}_{\gamma_n}^3(g_n)\Big)	+(\mathbb{E}\widetilde Z^3_{11})^2\mathcal R_{3n}\\
			&=8\Big(n^{-1}\sum_{i=1}^n(1+n^{-1/2}h_n\rho_n^{-1}\widetilde{\sigma}_nx)^3g^3_n(\lambda_i)-\mathrm{F}_{\gamma_n}(g_n^3)\Big)+12(\widetilde\kappa_2-2)\Big( \bar c_1\bar c_2\\
			&\quad-\mathrm{F}_{\gamma_n}(g_n)\mathrm{F}_{\gamma_n}(g_n^2)+n^{-1}\sum_{i=1}^n(c_{ii}-\bar c_1)((\mathbf C_n^2)_{ii}-\bar c_2)\Big)+(\widetilde\kappa_3-12\widetilde\kappa_2+16)\Big(\bar c_1^3\\
			&\quad-\mathrm{F}^3_{\gamma_n}(g_n)+3n^{-1}\bar c_1\sum_{i=1}^n(c_{ii}-\bar c_1)^2+n^{-1}\sum_{i=1}^n(c_{ii}-\bar c_1)^3\Big)+PO(n^{-1/2})\\
			&=8n^{-1/2}\Big(n^{-1/2}\widetilde G_{n,\Y}(g_n^3)+n^{-1}\rho_n^{-3}(\widetilde{\sigma}_n)^3x^3\mathrm{F}_n(h_n^3g_n^3)+n^{-1/2}3\rho_n^{-2}(\widetilde{\sigma}_n)^2\\
			&\qquad\times x^2\mathrm{F}_n(g_n^3h_n^2)+3\rho_n^{-1}\widetilde{\sigma}_nx\mathrm{F}_n(g_n^3h_n)\Big)+PO(n^{-1/2})=PO(n^{-1/2}).
	\end{align*}
	Hence, the result holds.
\end{proof}

\begin{lemma}\label{bound}
	Building on the definitions provided above, assume
	\begin{align*}
		y_n=2\widetilde{\sigma}_n\bar V_n^{-1/2}x/\sigma_n^2-n^{-1/2}\bar V_n^{-1/2}\widetilde G_{n,\Y}(g_n).
	\end{align*}
	then we can get
	\begin{align}
		\label{converge1}&\mathbb{E}\Big[\Phi(y_n)-\Phi(x)+n^{-1/2}\phi(x)\Big(\kappa_{21}^{-1/2}\mu(g_n)+\mathfrak c_1 x^2\Big)\Big]=o(n^{-1/2}),\\
		\label{converge2}&\mathbb{E}\Big(\bar V_n^{-3/2}\bar\kappa_{3}(1-y_n^2)\phi(y_n)-\kappa_{21}^{-3/2}\kappa_{31}\big(1-x^2\big)\phi(x)\Big)=O(n^{-1/2}).
	\end{align}
\end{lemma}

\begin{proof} Observe that $\widetilde{\sigma}_n, \sigma_n, \mathbb{E}\widetilde Z_{i1}^4$ and $\mathbb{E}\widetilde Z_{i1}^2$ are uniformly bounded for all $x\in\mathbb R$. Therefore, we derive:
	\begin{align*}
		&\mathbb{E}\big[\Phi(2\widetilde{\sigma}_n\bar V_n^{-1/2}x/\sigma_n^2-n^{-1/2}\bar V_n^{-1/2}\widetilde G_{n,\Y}(g_n))-\Phi(2\widetilde{\sigma}_n\bar V_n^{-1/2}x/\sigma_n^2)\\
		&\qquad\qquad+n^{-1/2}\bar V_n^{-1/2} \widetilde G_{n,\Y}(g_n)\phi(2\widetilde{\sigma}_n\bar V_n^{-1/2}x/\sigma_n^2)\big]=O\Big(n^{-1}\mathbb{E}\big(\widetilde G^2_{n,\Y}(g_n)\big)\Big),\\
		&\mathbb{E}\big[n^{-1/2}\bar V_n^{-1/2}\widetilde G_{n,\Y}(g_n)-n^{-1/2}\kappa_{21}^{-1/2}\widetilde G_{n,\Y}(g_n)\big]\phi(x)=o(n^{-1/2}).
	\end{align*}
	Next, we consider two cases to proceed with equation \eqref{converge1}: (Case 1) $|x|>M_n$. For all sufficiently large $n$, $\widetilde\kappa_2$ is bounded
	above and bounded away from zero. Hence, on $\mathcal G_n$,
	\[
	\frac{c}{n}\operatorname{tr}\mathbf C_n(x)^2
	\le \bar V_n(x)\le
	\frac{C}{n}\operatorname{tr}\mathbf C_n(x)^2.
	\]
	Writing $t_n(x)=n^{-1/2}\rho_n^{-1}\widetilde\sigma_nx$, we have
	\[
	\frac1n\operatorname{tr}\mathbf C_n(x)^2
	=
	\mathrm F_n(g_n^2)+2t_n(x)\mathrm F_n(g_n^2h_n)
	+t_n(x)^2\mathrm F_n(g_n^2h_n^2).
	\]
	Since $\mathrm F_{\gamma_n}(g_n^2h_n)=0$, while
	$\mathrm F_{\gamma_n}(g_n^2)$ and
	$\mathrm F_{\gamma_n}(g_n^2h_n^2)$ are bounded away from zero,
	the corresponding concentration bounds on $\mathcal G_n$ yield
	\[
	c\left(1+\frac{x^2}{n}\right)
	\le \bar V_n(x)\le
	C\left(1+\frac{x^2}{n}\right),
	\qquad x\in\mathbb R.
	\]
	In this case, we derive:
	\[(2\widetilde{\sigma}_n\bar V_n^{-1/2}x/\sigma_n^2)^2\ge\frac{cx^2}{1+x^2/n}\ge c \min(x^2,n).\]
	Hence, we derive:
	\begin{align*}
		|\Phi(2\widetilde{\sigma}_n\bar V_n^{-1/2}x/\sigma_n^2)-\Phi(x)|\le Ce^{-c\min(x^2,n)}.
	\end{align*}
	By \eqref{eq:LSS-correct-mean-moment} in the Supplementary Material,
	\[
	\mathbb E\widetilde G_{n,\Y}(g_n)=\mu(g_n)+o(1),
	\qquad
	\mathbb E|\widetilde G_{n,\Y}(g_n)|^2=O(1).
	\]
	Therefore, we derive:
	\begin{align*}
		&\quad\mathbb{E}\Big[\Phi(y_n)-\Phi( x)+n^{-1/2}\kappa_{21}^{-1/2}\mu(g_n)\phi(x)\Big]\\
		&=\mathbb{E}\Big[\Big(\Phi(2\widetilde{\sigma}_n\bar V_n^{-1/2}x/\sigma_n^2-n^{-1/2}\bar V_n^{-1/2}\widetilde G_{n,\Y}(g_n))-\Phi(2\widetilde{\sigma}_n\bar V_n^{-1/2}x/\sigma_n^2)+n^{-1/2}\bar V_n^{-1/2}\\
		&\quad~~\times \widetilde G_{n,\Y}(g_n)\phi(2\widetilde{\sigma}_n\bar V_n^{-1/2}x/\sigma_n^2)\Big)-n^{-1/2}\bar V_n^{-1/2}\widetilde G_{n,\Y}(g_n)\Big(\phi(2\widetilde{\sigma}_n\bar V_n^{-1/2}x/\sigma_n^2)-\phi(x)\Big)\\
		&\quad~~-\Big(\Phi( x) -\Phi(2\widetilde{\sigma}_n\bar V_n^{-1/2}x/\sigma_n^2)\Big)-\Big(n^{-1/2}\bar V_n^{-1/2}\widetilde G_{n,\Y}(g_n)-n^{-1/2}\kappa_{21}^{-1/2}\mu(g_n)\Big)\phi(x)\Big] \\
		&=o(n^{-1/2}).
	\end{align*}
	
	Since $M_n=K\sqrt{\log n}$ with $K$ sufficiently large, $\sup_{|x|>M_n}x^2\phi(x)=o(n^{-1/2})$.
	Therefore the term
	$n^{-1/2}\mathfrak c_1x^2\phi(x)$
	is also negligible uniformly in this region.	
	
	(Case 2) $|x|\le M_n$. Put $\bar V_{n,0}(x)=\kappa_{21}(1+2n^{-1/2}\mathfrak c_1x)$.
	Since $M_n=o(\sqrt n)$, the deterministic quantity
	$\bar V_{n,0}(x)$ is bounded away from zero uniformly over
	$|x|\le M_n$ for all sufficiently large $n$. Then by \eqref{6.2}, and because the $PO(n^{-1})$ remainder in that
	equation has degree at most two in $x$,
	\[
	\mathbb E
	\left|
	\bar V_n(x)-\bar V_{n,0}(x)
	\right|
	\le
	\frac{C}{n}(1+x^2).
	\]
	Under the localization convention, both $\bar V_n(x)$ and
	$\bar V_{n,0}(x)$ are bounded away from zero. Moreover, for
	$v\ge c$ and $a\in\mathbb R$,
	\[
	\Big|
	\frac{\partial}{\partial v}
	\Phi\left(\frac{a}{\sqrt v}\right)
	\Big|
	=
	\frac{|a/\sqrt v|\phi(a/\sqrt v)}{2v}
	\le C.
	\]
	Consequently,
	\begin{align*}
		\mathbb E\Phi(y_n)
		&=
		\mathbb E\Phi\Big(
		\frac{
			\sqrt{\kappa_{21}}x
			-
			n^{-1/2}\widetilde G_{n,\Y}(g_n)}
		{\sqrt{\bar V_{n,0}(x)}}
		\Big)
		+
		O\Big\{\frac{1+x^2}{n}\Big\}
		+
		o(n^{-1/2}).
	\end{align*}
	
	Set $z_{n,0}(x)=\frac{\sqrt{\kappa_{21}}x}{\sqrt{\bar V_{n,0}(x)}}$. Since $\mathbb E |\widetilde G_{n,\Y}(g_n)|^2=O(1)$,
	a second-order Taylor expansion only with respect to the LSS shift
	gives
	\begin{align*}
		\mathbb E\Phi(y_n)
		={}
		\Phi\bigl(z_{n,0}(x)\bigr)
		-
		\frac{\mathbb E\widetilde G_{n,\Y}(g_n)}
		{\sqrt{n\bar V_{n,0}(x)}}
		\phi\bigl(z_{n,0}(x)\bigr)
		+
		O(n^{-1})+
		O\left\{\frac{1+x^2}{n}\right\}
		+
		o(n^{-1/2}).
	\end{align*}
	Furthermore, $z_{n,0}(x)=x-n^{-1/2}\mathfrak c_1x^2+O(n^{-1}|x|^3)$.
	Using $
	\mathbb E\widetilde G_{n,\Y}(g_n)
	=
	\mu(g_n)+o(1)$,
	we therefore obtain
	\[
	\mathbb E\Phi(y_n)
	=
	\Phi(x)
	-
	\frac{\phi(x)}{\sqrt n}
	\left\{
	\kappa_{21}^{-1/2}\mu(g_n)
	+
	\mathfrak c_1x^2
	\right\}
	+
	O\left\{
	\frac{1+M_n^4}{n}
	\right\}
	+
	o(n^{-1/2}).
	\]
	Since $n^{-1}(1+M_n^4)=o(n^{-1/2})$. This proves \eqref{converge1} in Case 2.
	
	Therefore, combining these cases yields the desired result in equation \eqref{converge1}. Define $s_n=\kappa_{21}^{1/2}\bar V_n(x)^{-1/2}$. Next, we decompose the right-hand side (RHS) of equation \eqref{converge2} into the following equation and proceed  with the proof:
	\begin{align}
		\label{6.4}&\mathbb{E}\Big[\bar V_n^{-3/2}\bar\kappa_{3}\Big((1-y_n^2)\phi(y_n)-(1-s_n^2x^2)\phi(s_nx)\Big)\Big]=O(n^{-1/2}),\\
		\label{6.5}&\mathbb{E}\Big[\bar V_n^{-3/2}\bar\kappa_{3}\Big((1-s_n^2x^2)\phi(s_nx)-(1-x^2)\phi(x)\Big)\Big]=O(n^{-1/2}),\\
		\label{6.6}&\mathbb{E}\Big(\bar V_n^{-3/2}\bar\kappa_{3}(1-x^2)\phi(x)-\kappa_{21}^{-3/2}\kappa_{31}(1-x^2)\phi(x)\Big)=O(n^{-1/2}).
	\end{align}
	
	The proofs for equations \eqref{6.4}, \eqref{6.5} and \eqref{6.6} are straightforword when $x\in[-\sqrt n,\sqrt n]^c$. Therefore, we will focus on the case where $x\in[-\sqrt n,\sqrt n]$. For equation \eqref{6.6}, based on results \eqref{6.2} and \eqref{6.3} of Lemma \ref{PO}, we derive:
	\begin{align*}
		&\quad~\mathbb{E}\Big(\bar V_n^{-3/2}\bar\kappa_{3}(1-x^2)\phi(x)-\kappa_{21}^{-3/2}\kappa_{31}(1-x^2)\phi(x)\Big)\\
		&=\mathbb{E}\Big(\bar V_n^{-3/2}\bar\kappa_{3}(1-x^2)\phi(x)-\bar V_n^{-3/2}\kappa_{31}(1-x^2)\phi(x)\\
		&\qquad\qquad+\bar V_n^{-3/2}\kappa_{31}(1-x^2)\phi(x)-\kappa_{21}^{-3/2}\kappa_{31}(1-x^2)\phi(x)\Big)\\
		&=\mathbb{E}\Big(\bar V_n^{-3/2}(\bar\kappa_{3}-\kappa_{31})(1-x^2)\phi(x)+(\bar V_n^{-3/2}-\kappa_{21}^{-3/2})\kappa_{31}(1-x^2)\phi(x)\Big)=O(n^{-1/2}).
	\end{align*}
	For equation \eqref{6.4}, we derive
	\begin{align*}
		&\quad\Big|\mathbb{E}\Big[\bar V_n^{-3/2}\bar\kappa_{3}\Big((1-y_n^2)\phi(y_n)-(1-s_n^2x^2)\phi(s_nx)\Big)\Big]\Big|\le \frac{C}{\sqrt n}\mathbb{E}\Big(\bar V_n^{-2}|\bar\kappa_{3}|\big|\widetilde G_{n,\Y}(g_n)\big|\Big)\\
		&=O(n^{-1/2}).
	\end{align*}
	For equation \eqref{6.5}, notice that
	\begin{align*}
		|s_n-1|\le C|\bar V_n-\kappa_{21}|,\qquad \mathbb{E}|\bar V_n-\kappa_{21}|\le c\Big(n^{-1/2}|x|+n^{-1}(1+x^2)\Big).
	\end{align*} 
	Then, by the mean value theorem, we obtain:
	\begin{align*}
		&\quad\mathbb{E}\Big[\bar V_n^{-3/2}\bar\kappa_{3}\Big((1-s_n^2x^2)\phi(s_nx)-(1-x^2)\phi(x)\Big)\Big]\le \mathbb{E}\Big(C|\bar V_n-\kappa_{21}|(x^4+x^2)e^{-cx^2}\Big)\\
		&=O(n^{-1/2}).
	\end{align*}
	
	The last two estimates follow from Lemma~\ref{PO} and the preceding
	uniform two-sided bound for $\bar V_n(x)$. Therefore, combining \eqref{6.4}--\eqref{6.6}, we obtain:
	\begin{align*}
		&\mathbb{E}\Big(\bar V_n^{-3/2}\bar\kappa_{3}(1-y_n^2)\phi(y_n)-\kappa_{21}^{-3/2}\kappa_{31}(1-x^2)\phi(x)\Big)\\
		&=\mathbb{E}\Big[\bar V_n^{-3/2}\bar\kappa_{3}\Big((1-y_n^2)\phi(y_n)-(1-\kappa_{21}\bar V_n(x)^{-1}x^2)\phi(\kappa_{21}^{1/2}\bar V_n(x)^{-1/2}x)\Big)\Big]\\
		&\qquad +\mathbb{E}\Big[\bar V_n^{-3/2}\bar\kappa_{3}\Big((1-\kappa_{21}\bar V_n(x)^{-1}x^2)\phi(\kappa_{21}^{1/2}\bar V_n(x)^{-1/2}x)-(1-x^2)\phi(x)\Big)\Big]\\
		&\qquad+\mathbb{E}\Big[\Big(\bar V_n^{-3/2}\bar\kappa_{3}-\kappa_{21}^{-3/2}\kappa_{31}\Big)(1-x^2)\phi(x)\Big]=O(n^{-1/2}).
	\end{align*}
	Hence, the result holds.
\end{proof}

Based on the results of Lemma \ref{PO} and Lemma \ref{bound}, we establish the following conclusions:
\begin{align*}
	&\mathbb{E}\Big[\Phi(y_n)-\Phi(x)+n^{-1/2}\phi(x)\Big(\kappa_{21}^{-1/2}\mu(g_n)+\mathfrak c_1x^2\Big)\Big]=o(n^{-1/2}),\\
	&\mathbb{E}\Big[\bar V_n^{-3/2}\bar\kappa_{3}(1-y_n^2)\phi(y_n)-\kappa_{21}^{-3/2}\kappa_{31}(1-x^2)\phi(x)\Big]=O(n^{-1/2})
\end{align*}
uniformly for $|x|\le M_n$. Hence, it can be demonstrated that
\begin{align*}
	\mathbb{E}&\Big[\Phi (y_n)+\frac 16n^{-1/2}\bar V_n^{-3/2}\bar\kappa_{3}(1-y_n^2)\phi(y_n)\Big]\\
	&=\Phi(x)+n^{-1/2}\bigg(\frac 16\kappa_{21}^{-3/2}\kappa_{31}(1-x^2)-\kappa_{21}^{-1/2}\mu(g_n)-\mathfrak c_1x^2\bigg)\phi(x)+o(n^{-1/2})
\end{align*}
uniformly for $|x|\le M_n$.
Therefore, based on the preceding results, we obtain
\begin{align}\label{uniformMn}
	&\qquad\mathbb{E}\bigg[\mathbb{P}\bigg( \frac{-\Omega(\rho_n, \Y)+\rho_n n^{-1/2}\widetilde G_{n,\Y}(g_n)}{[\rho_n\mathrm{F}_{\gamma_n}(g_n^2)-n^{-1/2}S_n(g_nh_n)]\widetilde{\sigma}_n}\le x\bigg|\U,\bm{\Lambda}\bigg)\bigg]\\\nonumber
	&=\Phi(x)+n^{-1/2}\bigg(\frac 16\kappa_{21}^{-3/2}\kappa_{31}(1-x^2)-\kappa_{21}^{-1/2}\mu(g_n)-\mathfrak c_1x^2\bigg)\phi(x)+o(n^{-1/2})
\end{align}
uniformly for $|x|\le M_n$. 	

Applying Lemma~\ref{tailbound} of the Supplementary Material to the unconditional distribution function in \eqref{uniformMn}, which is equivalent to assuming a trivial conditioning $\sigma$ field, yields
\begin{align*}
	&\sup_{x}\bigg|\mathbb{P}\bigg( \frac{-\Omega(\rho_n, \Y)+\rho_n n^{-1/2}\widetilde G_{n,\Y}(g_n)}{[\rho_n\mathrm{F}_{\gamma_n}(g_n^2)-n^{-1/2}S_n(g_nh_n)]\widetilde{\sigma}_n}\le x\bigg)\\\nonumber
	&\qquad\quad-\Phi(x)-n^{-1/2}\bigg(\frac 16\kappa_{21}^{-3/2}\kappa_{31}(1-x^2)-\kappa_{21}^{-1/2}\mu(g_n)-\mathfrak c_1x^2\bigg)\phi(x)\bigg|=o(n^{-1/2}).
\end{align*}
Hence, the conclusion holds.

	\appendix
	\section{Theoretical supplements} \label{appendix}

\subsection{CLT of linear spectral statistics}\label{mu}
In this section, we present a special  case of the CLT for linear spectral statistics involving non-Gaussian elements. \cite{bai2004clt} established the CLT for linear spectral statistics under Gaussian-like moment conditions, while \cite{wang2013sphericity} and \cite{zheng2015substitution} derived new CLT results specifically for non-Gaussian elements.

For $\mathbf X\in\{\mathbf Z,\mathbf Y\}$, let
$e=(1,0,\ldots,0)'$ and define
\[
\mathbf S_{\mathbf V_1}
=\frac1n\mathbf X(\I-\mathbf V_1\mathbf V_1')\mathbf X',
\qquad
\mathbf S_e=\frac1n\mathbf X(\I-ee')\mathbf X',
\qquad
\mathbf{Q}=\left(\rho_n \I-\frac1n\mathbf X'\mathbf X\right)^{-1}.
\]
Here $\mathbf S_{\mathbf V_1}$ is the original background matrix,
whereas $\mathbf S_e$ is the sample covariance matrix obtained
by deleting the first column of $\mathbf X$.
The eigenvector spectral-distribution convergence in
\cite{bai2007asymptotics} and the Sherman--Morrison identity give
\[
\operatorname{tr}(\rho_n \I-\mathbf S_{\mathbf V_1})^{-1}
-\operatorname{tr}(\rho_n \I-\mathbf S_e)^{-1}
=
\frac{e'\mathbf{Q}^2e}{e'\mathbf{Q}e}
-\frac{\mathbf V_1'\mathbf{Q}^2\mathbf V_1}
{\mathbf V_1'\mathbf{Q}\mathbf V_1}
=o_p(1).
\]
The preceding identity and Lemma~\ref{qtruncation} give
\[
\mathbb E\left|
\operatorname{tr}g_n(\mathbf S_{\mathbf V_1})
-\operatorname{tr}g_n(\mathbf S_e)
\right|^2=o(1).
\]
Since subtracting $n\mathrm F_{\gamma_n}(g_n)$ from both
traces leaves their difference unchanged, the limiting
distribution given by Theorem~A.1 of
\cite{wang2013sphericity} for $\mathbf S_e$ also holds for
$\mathbf S_{\mathbf V_1}$. Consequently,
\begin{align*}
	I_2&=\frac{1}{2\pi i}\oint_{|\xi|=1}g(|1+\sqrt\gamma_n\xi|^2)\frac{1}{\xi^3}d\xi=\frac{1}{2\pi i}\oint_{|\xi|=1}\frac{1}{\xi^2(\rho_n\xi-\xi-\sqrt\gamma_n-\sqrt\gamma_n\xi^2-\gamma_n\xi)}d\xi\\
	&=-\frac{1}{\sqrt{\gamma_n}}\frac{1}{2\pi i}\oint_{|\xi|=1}\frac{1}{(\xi-\frac{\sqrt \gamma_n}{l-1})(\xi-\frac{l-1}{\sqrt \gamma_n})\xi^2}d\xi=\frac{\gamma_n}{[(l-1)^2-\gamma_n](l-1)}.
\end{align*}
Hence, the corresponding expression for $\mu(g_n)$ is given by
\begin{align*}
	\mu(g_n)&=\bigg(\frac{g_n(a(\gamma_n))+g_n(b(\gamma_n))}{4}-\frac{1}{2\pi}\int_{a(\gamma_n)}^{b(\gamma_n)}\frac{g_n(x)}{\sqrt{4\gamma_n-(x-1-\gamma_n)^2}}dx\bigg)+\beta I_2\\
	&=\frac{\gamma_n(l-1)}{[(l-1)^2-\gamma_n]^{2}}+\beta I_2=\frac{\gamma_n(l-1)}{[(l-1)^2-\gamma_n]^{2}}+\frac{\beta\gamma_n}{[(l-1)^2-\gamma_n](l-1)}\\
	&=\frac{\gamma_n(l-1)^2+\beta\gamma_n[(l-1)^2-\gamma_n]}{[(l-1)^2-\gamma_n]^{2}(l-1)}
\end{align*}
for sufficiently large $n$.

	\subsection{Lemmas}
We introduce several auxiliary lemmas that are instrumental in deriving limiting distributions and formulating Edgeworth expansion expressions.

\begin{lemma}\label{resolventfluctuations}
	Under Assumptions (a)--(c), for
	$f\in\{g_n,g_n^2,g_n^3\}$,
	\begin{equation}\label{SG}
		\widetilde S_n(f)=O_p(1),\qquad
		\widetilde G_n(f)=O_p(1).
	\end{equation}
	For every fixed $a>0$,
	\begin{equation}\label{eq:dependent-resolvent-moderate}
		\mathbb P\left(
		\max\left\{
		\sqrt n\,|\widehat\rho-\rho_n|,
		\max_f\bigl(|\widetilde S_n(f)|+|\widetilde G_n(f)|\bigr)
		\right\}>an^{1/8}
		\right)=o(n^{-1/2}).
	\end{equation}
	Both conclusions extend to bounded deterministic linear
	combinations of these functions.
\end{lemma}

\begin{proof}
	Use the standardized arrays and the common event $\mathcal G_n$
	of Lemma~\ref{qtruncation} for the deletions and hybrids below,
	retaining the notation. Localize each statistic using only
	its own background spectral gap.
	Choose $s>4$ sufficiently close to $4$ that
	\[
	1-s/2+\alpha(2s-6)<0,
	\qquad
	\mathbb E|Z_{11}|^{2s}\le Cn^{\alpha(2s-6)}.
	\]
	
	Write $L_n$ for either $\widetilde S_n(f)$ or
	$\widetilde G_n(f)$. Put $\mathbf x=\mathbf Z\mathbf V_1$ and
	$\mathbf B=\mathbf Z\mathbf V_2$, with rows $x_i$ and
	$\mathbf b_i'$. For $g_\rho(\lambda)=(\rho-\lambda)^{-1}$, let
	\[
	R_i=(\rho I-n^{-1}\mathbf B_{-i}\mathbf B_{-i}')^{-1},
	\qquad
	u_i=n^{-1}\mathbf B_{-i}\mathbf b_i,
	\qquad
	s_i=\rho-n^{-1}\|\mathbf b_i\|^2-u_i'R_i u_i.
	\]
	Writing $\Delta_iL=L-L^{(i)}$, with normalization and
	centering unchanged, block inversion gives on $\mathcal G_n$
	\[
	\begin{aligned}
		\Delta_i\widetilde G_n(g_\rho)
		=\frac{1+u_i'R_i^2u_i}{s_i},\qquad
		\Delta_i\widetilde S_n(g_\rho)
		=\frac{(x_i+u_i'R_i\mathbf x_{-i})^2
			-1-u_i'R_i^2u_i}{\sqrt n\,s_i}.
	\end{aligned}
	\]
	The cases $g_n^2,g_n^3$ follow by spectral differentiation.
	Apply fixed deleted-array norm cutoffs and bounded
	Lipschitz coefficient extensions to these identities,
	and denote the resulting increments by $U_i$.
	Their discrepancies from $\Delta_iL_n$ vanish on $\mathcal G_n$
	and are polynomially bounded, so every fixed absolute
	moment is $O_D(n^{-D})$.
	
	Let $\mathbb E_{(i)}$ integrate row $i$.
	The trace increment is bounded Lipschitz in normalized row
	quadratic forms; the score increment is $n^{-1/2}$ times
	a quadratic polynomial in row linear forms.
	The row-moment inequalities therefore give
	\[
	\mathbb E_{(i)}|U_i-\mathbb E_{(i)}U_i|^2\le Cn^{-1},
	\qquad
	\mathbb E_{(i)}|U_i-\mathbb E_{(i)}U_i|^s
	\le Cn^{-s/2+\alpha(2s-6)}.
	\]
	The row-martingale moment inequality consequently yields
	\[
	\mathbb E|L_n-\mathbb EL_n|^s
	\le C\{1+n^{1-s/2+\alpha(2s-6)}\}\le C.
	\]
	
	The LSS CLT and rank-one comparison in Section~\ref{mu}
	give $\widetilde G_n(f)=O_p(1)$, which, together with the
	centered moment bound, implies
	$|\mathbb E\widetilde G_n(f)|\le C$.
	For the score mean, leave $\mathbf x$ unchanged and replace
	the background rows successively by those of
	$\mathbf Y\mathbf V_2$.
	The score-dependent increment terms have the form
	\[
	n^{-1/2}\Big\{
	x_i^2A(q_i)+x_i\sum_h(\mathbf b_i'w_{ih})B_h(q_i)
	\Big\},
	\]
	with bounded Lipschitz coefficient functions, bounded
	deleted vectors, and normalized row quadratic forms $q_i$.
	Freezing the coefficients at $\mathbb E_{(i)}q_i$ makes
	the endpoint expectations agree, since
	$\mathbb Ex_i^2=1$ and $\mathbb E(x_i\mathbf b_i)=0$.
	Background-only terms also agree in expectation.
	As $\mathbb E_{(i)}|q_i-\mathbb E_{(i)}q_i|^2\le C/n$,
	Cauchy--Schwarz bounds each replacement error by $C/n$.
	The independent-background endpoint has mean zero, hence
	$|\mathbb E\widetilde S_n(f)|\le C$.
	
	Thus $\mathbb E|L_n|^s\le C$, proving \eqref{SG}, and
	Markov's inequality gives
	\[
	\mathbb P\{|L_n|>an^{1/8}\}
	\le C_a n^{-s/8}=o(n^{-1/2}).
	\]
	On $\mathcal G_n\cap\mathcal E_n$,
	equations \eqref{com1}--\eqref{com2} give
	\[
	\sqrt n\,|\widehat\rho-\rho_n|
	\le C\left\{
	|\widetilde S_n(g_n)|
	+n^{-1/2}|\widetilde G_n(g_n)|
	\right\}.
	\]
	Lemma~\ref{qtruncation} and a finite union bound give
	both assertions for the original arrays.
	The extension to bounded deterministic linear combinations
	follows by linearity.
	
	For $f=g_n$, the preceding moment bound implies uniform
	integrability. The LSS limit in Section~\ref{mu} therefore gives
	\begin{equation}\label{eq:LSS-correct-mean-moment}
		\mathbb E\widetilde G_n(g_n)=\mu(g_n)+o(1),
		\qquad
		\mathbb E|\widetilde G_n(g_n)|^2=O(1).
	\end{equation}
\end{proof}

\begin{lemma}
	\label{goodevent}
	Let $\mathcal E_n$ be an event such that $
	\mathbb P(\mathcal E_n^c)=o(n^{-1/2})$.
	Suppose that two random quantities $X_n$ and $\widetilde X_n$
	satisfy $X_n=\widetilde X_n$ on $\mathcal E_n$.
	Then
	\[
	\sup_{z\in\mathbb R}
	\left|
	\mathbb P(X_n\le z)-\mathbb P(\widetilde X_n\le z)
	\right|
	\le \mathbb P(\mathcal E_n^c)
	=o(n^{-1/2}),
	\]
	and, for every $t\in\mathbb R$,
	\[
	\left|
	\mathbb E e^{itX_n}
	-
	\mathbb E e^{it\widetilde X_n}
	\right|
	\le 2\mathbb P(\mathcal E_n^c)
	=o(n^{-1/2}).
	\]
	More generally, if $f$ is bounded and measurable, then
	\[
	\left|\mathbb Ef(X_n)-\mathbb Ef(\widetilde X_n)\right|
	\le
	2\|f\|_\infty\,\mathbb P(\mathcal E_n^c)
	=o(n^{-1/2}).
	\]
\end{lemma}

\begin{proof}
	Since $X_n=\widetilde X_n$ on $\mathcal E_n$, the indicators
	$\mathbf 1_{\{X_n\le z\}}$ and
	$\mathbf 1_{\{\widetilde X_n\le z\}}$ can differ only on
	$\mathcal E_n^c$. Hence
	\[
	\left|
	\mathbb P(X_n\le z)-\mathbb P(\widetilde X_n\le z)
	\right|
	\le \mathbb P(\mathcal E_n^c),
	\]
	uniformly in $z$. The characteristic-function and bounded-function
	bounds follow in the same way from
	\[
	|f(X_n)-f(\widetilde X_n)|
	\le
	2\|f\|_\infty\mathbf 1_{\mathcal E_n^c}.
	\]
\end{proof}

\begin{lemma}\label{lem:projected-score-gap}
	Let $Z_1,Z_2,\ldots$ be independent copies of a centered,
	unit-variance random variable $Z_{11}$ with a finite sixth
	moment and a nonzero absolutely continuous component.
	Put $X_{\mathbf v}=\sum_{j=1}^qv_jZ_j$.
	For every $0<\epsilon<A$, there exists
	$\eta_{\epsilon,A}\in(0,1)$ such that
	\begin{equation}\label{eq:projected-score-gap-strong}
		\sup_{\substack{q\ge1,\ |\mathbf v|=1\\
				\epsilon\le|s|\le A,\ u\in\mathbb R}}
		\left|\mathbb E
		e^{i\{s(X_{\mathbf v}^2-1)+uX_{\mathbf v}\}}\right|
		\le\eta_{\epsilon,A}.
	\end{equation}
	The bound also holds uniformly for the truncated, centered
	and standardized entry laws of Section~\ref{trunca},
	for all sufficiently large $n$.
\end{lemma}

\begin{proof}
	Let $\phi_X(s,u)=\mathbb E e^{i\{s(X^2-1)+uX\}}$.
	Each fixed projection $X_{\mathbf v}$ has a nonzero
	absolutely continuous component,
	\cite{hall2013bootstrap} gives Cram\'er's condition for
	$(X_{\mathbf v}^2-1,X_{\mathbf v})$.
	We must show that the bound can be chosen independently
	of $q$ and $\mathbf v$.
	
	Fix $0<\epsilon<A$. Suppose that
	\eqref{eq:projected-score-gap-strong} does not hold.
	For each $\ell\ge2$, taking $\eta=1-\ell^{-1}$,
	we can therefore choose $X_\ell=X_{\mathbf v_\ell}$
	and parameters $\epsilon\le|s_\ell|\le A$,
	$u_\ell\in\mathbb R$ such that
	\[
	1\ge|\phi_{X_\ell}(s_\ell,u_\ell)|
	>1-\ell^{-1}\longrightarrow1.
	\]
	Put $d_\ell=\max_j|v_{\ell,j}|$.
	We distinguish three cases:
	
	\emph{(i) $u_\ell$ is bounded.}
	Choose $(s_\ell,u_\ell)\to(s,u)$, where $|s|\ge\epsilon$.
	After ordering the coefficients, diagonal selection and
	a characteristic-function expansion of the remaining
	small terms give $X_\ell\Rightarrow X$.
	The limit is a variance-one sum of independent scaled
	copies of $Z_{11}$ and an independent Gaussian term,
	so it has a nonzero absolutely continuous component.
	Consequently,
	\[
	|\phi_{X_\ell}(s_\ell,u_\ell)|
	\longrightarrow|\phi_X(s,u)|<1.
	\]
	
	\emph{(ii) $|u_\ell|\to\infty$ and $\liminf_\ell d_\ell>0$.}
	Condition on all entries except one with coefficient
	of magnitude $d_\ell$.
	The sum of the other terms is $O_p(1)$, so the resulting
	linear coefficient diverges in absolute value in probability,
	while the quadratic coefficient remains bounded.
	Cram\'er's condition for $(Z_{11}^2-1,Z_{11})$ therefore gives
	\[
	\limsup_\ell|\phi_{X_\ell}(s_\ell,u_\ell)|<1.
	\]
	
	\emph{(iii) $|u_\ell|\to\infty$ and $d_\ell\to0$.}
	Choose a subsum $T_\ell$ of $X_\ell$ by adding terms
	until its variance first exceeds $u_\ell^{-2}$.
	Then
	\[
	u_\ell^{-2}\le\mathbb ET_\ell^2
	\le u_\ell^{-2}+d_\ell^2\longrightarrow0.
	\]
	The entry bound
	$|\mathbb E e^{itZ_{11}}|\le e^{-c\min(t^2,1)}$
	and independence give
	$|\mathbb E e^{iu_\ell T_\ell}|\le e^{-c}$,
	since $u_\ell^2\mathbb ET_\ell^2\ge1$.
	As $T_\ell$ is independent of $X_\ell-T_\ell$,
	the inequality $|e^{ix}-1|\le|x|$ yields
	\[
	|\phi_{X_\ell}(s_\ell,u_\ell)|
	\le |\mathbb E e^{iu_\ell T_\ell}|
	+A\Big(\mathbb ET_\ell^2
	+2\sqrt{\mathbb ET_\ell^2}\Big)
	\le e^{-c}+o(1).
	\]

	Thus every possible case yields
	\[
	\limsup_{\ell\to\infty}
	|\phi_{X_\ell}(s_\ell,u_\ell)|<1,
	\]
	contradicting the choice of the sequence.
	This proves \eqref{eq:projected-score-gap-strong}.
	
	For the standardized truncated entries,
	$Z_n\to Z_{11}$ in $L^6$, so the subsequence argument
	and the required moment bounds remain valid uniformly.
	Moreover,
	\[
	|\phi_{Z_n}(s,u)|
	\le
	\left|\phi_{Z_{11}}\!\left(
	\frac{s}{s_n^2},
	\frac{u}{s_n}-\frac{2s\mu_n}{s_n^2}
	\right)\right|
	+2\mathbb P(|Z_{11}|>n^\alpha).
	\]
	Since $(\mu_n,s_n)\to(0,1)$, the entry Cram\'er bound
	is also uniform for sufficiently large $n$.
	The same argument proves the truncated version.
\end{proof}

\begin{lemma}\label{delta}
	Suppose that $U_n$ admits the first-order Edgeworth expansion
	\[\mathbb{P}(U_n\le x)=\Phi(x)+n^{-1/2}p_1(x)\phi(x)+o(n^{-1/2})\]
	uniformly in $x\in\mathbb R$. Assume that the random variables $J_n$ satisfy 
	\[\mathbb{P}(|J_n|\ge n^{-1/2}\epsilon_n)=o(n^{-1/2})\]
	for a non-random sequence $\{\epsilon_n\}$ converging to $0$. Then, we have
	\[\mathbb{P}(U_n+J_n\le x)-\mathbb{P}(U_n\le x)=o(n^{-1/2})\]
 uniformly in $x \in R$.
\end{lemma}

\begin{proof}
	Note that
	\begin{align*}
		|\mathbb{P}(U_n+J_n\le x)&-\mathbb{P}(U_n\le x)|\le \mathbb{P}(|J_n|>n^{-1/2}\epsilon_n)+\mathbb{P}(|U_n-x|\le n^{-1/2}\epsilon_n).
	\end{align*}
	Therefore, it suffices to show that
	\[\mathbb{P}(|U_n-x|\le n^{-1/2}\epsilon_n)=o(n^{-1/2}).\]  
	
	Notice
	\begin{align}
		\label{P1}~~&|\Phi(x+n^{-1/2}\epsilon_n)-\Phi(x)|\le  n^{-1/2}\epsilon_n \sup_x|\phi(x)|=o(n^{-1/2}),\\
		\label{P2}~~
		&|p_1(x+n^{-1/2}\epsilon_n)\phi(x+n^{-1/2}\epsilon_n)
		-p_1(x)\phi(x)| \le n^{-1/2}\epsilon_n
		\sup_{u\in\mathbb R}\left|\{p_1(u)\phi(u)\}'\right|
		=o(n^{-1/2}).
	\end{align}
	By the uniform convergence of $U_n$ and equations \eqref{P1} and \eqref{P2}, we derive:
	\begin{align*}
		\mathbb{P}(|U_n-x|\le n^{-1/2}\epsilon_n)&=\mathbb{P}(x-n^{-1/2}\epsilon_n\le U_n\le x+n^{-1/2}\epsilon_n)\\
		&\le \mathbb{P}(x-n^{-1/2}\epsilon_n\le U_n\le x)+\mathbb{P}(x < U_n\le x+n^{-1/2}\epsilon_n)\\
		&=o(n^{-1/2}).
	\end{align*}
	Hence, the result holds.
\end{proof}

\begin{lemma}[Corollary 4 of \cite{yang2018edgeworth}]\label{EE condition} Let $\{X_{ni}:1\le i\le n\}$ be a row-wise independent,
	centered triangular array. Define
	\[
	\bar V_n=\frac1n\sum_{i=1}^n\mathbb E X_{ni}^2>0,
	\qquad
	\bar\chi_{3,n}
	=\frac{1}{n\bar V_n^{3/2}}
	\sum_{i=1}^n\mathbb E X_{ni}^3.
	\]
	For fixed $k=3$ and $l=0$, and for $(X_{ni})_{n\in N,i\in\{1,\dots,n\}}$ satisfying the following regularity conditions:
	\begin{itemize}
		\item[R1] The normalized third absolute moments are uniformly bounded:
		\[
		\sup_{n\ge n_0}
		\frac1{n\bar V_n^{3/2}}
		\sum_{i=1}^n\mathbb E|X_{ni}|^3<\infty.
		\]
		\item[R2] For some $\tau\in (0, 1/2)$,
		\[n^{-1}\bar V_{n}^{-3/2}\sum_{i=1}^n\mathbb{E}\Big(I\{\bar V_n^{-1/2}|X_{ni}|>n^{\tau}\}|X_{ni}|^k\Big)=o(1).\]
		\item[R3] A generalized Cram\'{e}r's condition
		\[n^{1/2}\int_{|t|>\epsilon}|t|^{-1}\Big|\mathbb{E}\Big[\exp\Big(it\bar V_n^{-1/2}\sum_{i=1}^n X_{ni}\Big)\Big]\Big|dt=o(1),\]
		then
		\begin{align*}
			\mathbb{P}\Big(n^{-1/2}\bar V_n^{-1/2}\sum_{i=1}^nX_{ni}\le x\Big)=\Phi(x)+\frac 16n^{-1/2}\bar\chi_{3,n}(1-x^2)\phi(x)+o(n^{-1/2}),
		\end{align*}
		uniformly in $x\in\mathbb R$.
	\end{itemize}
\end{lemma}

\begin{lemma}
		\label{tailbound}
		Let $T_n$ be a real-valued random variable and define
		\[
		F_n(x)
		=
		\mathbb P(T_n\le x\mid\mathcal F_Y),
		\qquad
		G_n(x)
		=
		\Phi(x)+\frac1{\sqrt n}p_{1,n}(x)\phi(x),
		\]
		where $
		p_{1,n}(x)=\sum_{\nu=0}^m a_{\nu n}x^\nu$,
		$m$ is fixed, and the coefficients $a_{\nu n}$ are
		$\mathcal F_Y$-measurable. Let
		\[
		C_n=\sum_{\nu=0}^m|a_{\nu n}|,
		\qquad
		M_n=K\sqrt{\log n},
		\qquad
		\Delta_n^{(c)}
		=
		\sup_{|x|\le M_n}|F_n(x)-G_n(x)|,
		\]
		where $K>1$. If
		\[
		\mathbb EC_n=O(1),
		\qquad
		\mathbb E\Delta_n^{(c)}=o(n^{-1/2}),
		\]
		then
		\[
		\mathbb E\sup_{x\in\mathbb R}|F_n(x)-G_n(x)|
		=
		o(n^{-1/2}).
		\]
		
		The same conclusion holds for families $F_{n,x}$ and $G_{n,x}$:
		if
		\[
		\mathbb E\sup_{|x|\le M_n}C_n(x)=O(1)
		\]
		and
		\[
		\mathbb E
		\sup_{|x|\le M_n}
		\sup_{|z|\le M_n}
		|F_{n,x}(z)-G_{n,x}(z)|
		=o(n^{-1/2}),
		\]
		then
		\[
		\mathbb E
		\sup_{|x|\le M_n}
		\sup_{z\in\mathbb R}
		|F_{n,x}(z)-G_{n,x}(z)|
		=o(n^{-1/2}).
		\]
	\end{lemma}
	
	\begin{proof}
		Since $p_{1,n}(x)=\sum_{\nu=0}^m a_{\nu n}x^\nu$ and $C_n=\sum_{\nu=0}^m|a_{\nu n}|$, we have for every $x\in\mathbb{R}$ that
		$$ |p_{1,n}(x)| \le C_n(1+|x|^m). $$
		Consequently,
		\[
		|1-G_n(x)|
		\le
		1-\Phi(x)
		+
		\frac{C_n}{\sqrt n}(1+x^m)\phi(x),
		\qquad x\ge M_n.
		\]
		Using
		\[
		1-\Phi(x)\le\frac{\phi(x)}{x},
		\qquad x>0,
		\]
		and the fact that $(1+x^m)\phi(x)$ is decreasing on
		$[M_n,\infty)$ for all sufficiently large $n$, we obtain
		\[
		\sup_{x\ge M_n}|1-G_n(x)|
		\le
		\frac{\phi(M_n)}{M_n}
		+
		\frac{C_n}{\sqrt n}(1+M_n^m)\phi(M_n).
		\]
		
		For every realization, the monotonicity of $F_n$ gives
		\[
		\begin{aligned}
			\sup_{x\ge M_n}|F_n(x)-G_n(x)|
			&\le
			\Delta_n^{(c)}
			+
			2\sup_{x\ge M_n}|1-G_n(x)|.
		\end{aligned}
		\]
		The same argument on $(-\infty,-M_n]$ gives
		\[
		\sup_{x\le-M_n}|F_n(x)-G_n(x)|
		\le
		\Delta_n^{(c)}
		+
		2\sup_{x\le-M_n}|G_n(x)|.
		\]
		Since $\Phi(-x)=1-\Phi(x)$ and $\phi(-x)=\phi(x)$, the two tails
		satisfy the same bound. Therefore,
		\[
		\sup_{x\in\mathbb R}|F_n(x)-G_n(x)|
		\le
		\Delta_n^{(c)}
		+
		2\left\{
		\frac{\phi(M_n)}{M_n}
		+
		\frac{C_n}{\sqrt n}(1+M_n^m)\phi(M_n)
		\right\}.
		\]
		Taking expectations yields
		\[
		\begin{aligned}
			\mathbb E\sup_{x\in\mathbb R}|F_n(x)-G_n(x)|
			\le
			\mathbb E\Delta_n^{(c)}
			+
			\frac{2\phi(M_n)}{M_n}+
			\frac{2\mathbb EC_n}{\sqrt n}
			(1+M_n^m)\phi(M_n).
		\end{aligned}
		\]
		Because $M_n=K\sqrt{\log n}$ with $K>1$,
		\[
		\phi(M_n)
		=
		\frac1{\sqrt{2\pi}}n^{-K^2/2}.
		\]
		Hence, using $\mathbb EC_n=O(1)$,
		\[
		\frac{\phi(M_n)}{M_n}
		=o(n^{-1/2}),
		\qquad
		\frac{\mathbb EC_n}{\sqrt n}
		(1+M_n^m)\phi(M_n)
		=o(n^{-1/2}).
		\]
		Together with
		$\mathbb E\Delta_n^{(c)}=o(n^{-1/2})$, this proves
		\[
		\mathbb E\sup_{x\in\mathbb R}|F_n(x)-G_n(x)|
		=o(n^{-1/2}).
		\]
		
		If the quantities also depend on $u$, the same proof applies after
		taking the supremum over $|u|\le M_n$, using
		\[
		\mathbb E\sup_{|u|\le M_n}C_n(x)=O(1).
		\]
	\end{proof}

\begin{lemma}\label{integration}
	Let $X_1,\ldots,X_m$ be independent and $Z=\sum_{|S|\le d}Z_S$,
	where $d\ge0$ is an integer and each integrable $Z_S$ depends only
	on the coordinates in $S\subseteq\{1,\ldots,m\}$.
	Let $\mathbb E_I$ average over the coordinates in $I$ only.
	
	\textup{(i)} For $|f_j(X_j)|=1$ and $\alpha_j=\mathbb Ef_j$,
	\[
	\mathbb E\Big[Z_S\prod_jf_j\Big]
	=\mathbb E\Big[Z_S\prod_{j\in S}f_j\Big]
	\prod_{j\notin S}\alpha_j.
	\]
	If $Z\in L^2$, then
	\[
	\Big|\mathbb E\Big[Z\prod_jf_j\Big]\Big|
	\le 2^{d/2}\|Z\|_2
	\prod_j\Big(\frac{1+|\alpha_j|^2}{2}\Big)^{1/2}.
	\]
	\textup{(ii)} If $I_1,\ldots,I_{d+1}\subseteq\{1,\ldots,m\}$ are disjoint
	and, for real $W$,
	$|\mathbb E_{I_k}e^{itW}|\le D(t)$ almost surely for every $k$,
	with deterministic $D(t)$, then
	\[
	|\mathbb E[Ze^{itW}]|\le(2^{d+1}-1)D(t)\mathbb E|Z|.
	\]
	Both statements hold conditionally when independence is preserved
	and the sets and bounds are fixed under the conditioning.
\end{lemma}
\begin{proof}
	The identity in (i) follows from independence. In expanding
	$\prod_j\{\alpha_j+(f_j-\alpha_j)\}$, only terms with at most $d$
	centered factors contribute against $Z$. Their squared $L^2$ norm is
	\[
	\sum_{|S|\le d}\prod_{j\notin S}|\alpha_j|^2
	\prod_{j\in S}(1-|\alpha_j|^2)
	\le 2^d\prod_j\frac{1+|\alpha_j|^2}{2},
	\]
	so Cauchy--Schwarz gives the bound.
	For (ii), each $Z_S$ omits one of the $d+1$ sets, whence
	\[
	Z=\sum_{\varnothing\ne T\subseteq\{1,\ldots,d+1\}}
	(-1)^{|T|+1}\mathbb E_{\cup_{k\in T}I_k}Z.
	\]
	Each summand is constant on at least one $I_k$ and has $L^1$ norm
	at most $\mathbb E|Z|$. Integrating over that set proves (ii).
\end{proof}

\subsection{Truncation and standardization}
\label{trunca}

Fix $
5/12<\alpha<1/2$.
For the entries of $\Z$, define
\[
\widehat Z_{ij}
=
Z_{ij}\mathbf 1_{\{|Z_{ij}|\le n^\alpha\}},
~~
\mu_n=\mathbb E\widehat Z_{11},
~~
s_n^2=\operatorname{Var}(\widehat Z_{11}),\quad
Z_{ij,n}
=
\frac{\widehat Z_{ij}-\mu_n}{s_n},
~~
\Z_n=(Z_{ij,n}).
\]
Define $\widehat{\Y}$ and $\Y_n$ analogously. The corresponding
sample covariance matrices are
\[
\mathbf S
=
\frac1n\bm\Sigma^{1/2}\Z'\Z\bm\Sigma^{1/2},
\qquad
\widehat{\mathbf S}
=
\frac1n\bm\Sigma^{1/2}
\widehat{\Z}'\widehat{\Z}\bm\Sigma^{1/2},
\qquad
\mathbf S_n
=
\frac1n\bm\Sigma^{1/2}\Z_n'\Z_n\bm\Sigma^{1/2}.
\]

\begin{lemma}\label{qtruncation}
	Under Assumptions (a)--(c), the standardized truncated entries
	have mean zero, variance one, uniformly bounded sixth moments,
	and absolute values at most $Cn^\alpha$. Moreover,
	\begin{equation}\label{truncation}
		\mathbb P\{\Z\ne\widehat{\Z}\ \text{or}\ \Y\ne\widehat{\Y}\}
		=o(n^{-1/2}).
	\end{equation}
	
	Consider any prescribed polynomial-size family of deterministic
	row hybrids and deletions from finitely many independent copies,
	each retaining $n-o(n)$ rows with normalization $n$.
	For sufficiently small fixed $\varepsilon,c>0$, put
	$\lambda_{+,n}=(1+\sqrt{\gamma_n})^2$ and
	\begin{equation}\label{eq:truncation-spectral-gap}
		\mathcal G_n=\bigcap_{\B_n,k}
		\left\{
		\begin{aligned}
			\|\B_n\|^2/n\le\lambda_{+,n}+\varepsilon,\operatorname{dist}\!\left(
			\rho_{nk},
			\operatorname{spec}(n^{-1}\B_n\V_k\D_k\V_k'\B_n')
			\right)\ge c
		\end{aligned}
		\right\}.
	\end{equation}
	Then
	\[
	\mathbb P(\mathcal G_n^c)
	\le Ce^{-cn^{1-2\alpha}}=O_D(n^{-D})
	\]
	for every fixed $D>0$.
	On $\mathcal G_n$, all required resolvents and fixed-order
	derivatives are uniformly bounded, and
	$|\eta_k|\wedge|\eta_{ky}|\ge c$.
	Together with the anti-concentration bounds used above,
	these conclusions justify replacing the original arrays by
	$\Z_n,\Y_n$ in the preceding distributional comparisons
	with error $o(n^{-1/2})$.
	
	In the single-spike case, there exists a fixed $c_0>0$
	such that the event
	\[
	\mathcal E_n=
	\left\{
	|\hat\rho-\rho_n|\le c_0,\quad
	\lambda_{\max}(n^{-1}\V_2'\Z'\Z\V_2)
	\le \rho_n-2c_0
	\right\}
	\]
	satisfies $\mathbb P(\mathcal E_n^c)=o(n^{-1/2})$.
	This holds for both the original and the standardized
	truncated array, with the corresponding sample eigenvalue.
\end{lemma}

\begin{proof}
	The finite sixth moment gives
	\[
	\mathbb E\!\left[
	|Z_{11}|^j\mathbf1_{\{|Z_{11}|>n^\alpha\}}
	\right]
	=o(n^{-(6-j)\alpha}),\qquad 0\le j\le6.
	\]
	Thus a union bound proves \eqref{truncation}, since
	$n(p+r)=O(n^2)$ and $\alpha>5/12$, while
	\[
	|\mu_n|=o(n^{-5\alpha}),\qquad
	|s_n^2-1|=o(n^{-4\alpha}).
	\]
	The asserted entrywise and moment bounds follow immediately.
	
	For the spectral event, convex Lipschitz concentration gives,
	for each fixed $u>0$,
	\begin{equation}\label{eq:truncation-largest-singular-value}
		\mathbb P\{|F-\operatorname{Med}(F)|>u\sqrt n\}
		\le4e^{-c_u n^{1-2\alpha}},
	\end{equation}
	where $F$ is $\|\B_n\|$ or the sum of the largest $j\le r$
	singular values of $\B_n\V_k\D_k^{1/2}$.
	The spectral limits of \cite{baik2006eigenvalues} locate
	their medians uniformly when the retained row count satisfies
	$N/n\to1$; the zero population eigenvalue is handled by
	positive regularization and Weyl's inequality.
	The truncation and normalization estimates above preserve
	these limits.
	Taking differences of consecutive singular-value sums,
	using the fixed separation of the outliers and bulk edge,
	and taking a union bound proves the bound for
	$\mathbb P(\mathcal G_n^c)$.
	The resolvent bounds follow from the spectral gap, and
	the block inverse identity
	$[(\rho \I-A)^{-1}]_{ii}=(\rho\eta_i)^{-1}$
	gives the denominator bounds.
	
	For $\mathcal E_n$, choose $c_0>0$ such that
	$\rho_n-\lambda_{+,n}\ge4c_0$ for all sufficiently large $n$.
	The preceding spectral limits show that the medians of
	$n^{-1/2}\|\Z_n\bm\Sigma^{1/2}\|$ and
	$n^{-1/2}\|\Z_n\|$ are
	$\sqrt{\rho_n}+o(1)$ and
	$\sqrt{\lambda_{+,n}}+o(1)$, respectively.
	Applying the same convex Lipschitz concentration bound
	to these norms, and using $\|\Z_n\V_2\|\le\|\Z_n\|$,
	gives $\mathbb P(\mathcal E_n^c)\le Ce^{-cn^{1-2\alpha}}$
	for the standardized truncated array.
	Repeating the argument for $\widehat\Z$ and using
	\eqref{truncation} gives the result for the original array.
	
	On $\mathcal G_n$, when no entry is truncated,
	$\B=s_n\B_n+\mu_n\mathbf1\mathbf1'$ and hence
	\begin{equation}\label{eq:truncation-array-transfer}
		n^{-1/2}\|\B_n-\B\|
		\le C|s_n-1|+\sqrt{p+r}\,|\mu_n|
		=o(n^{-1}),
	\end{equation}
	uniformly over the prescribed arrays.
	The inverse identity and Weyl's inequality therefore give
	$o(n^{-1})$ perturbations of the covariance matrices,
	their eigenvalues, and the required resolvents.
	The corresponding normalized statistics change by
	$o(n^{-1/2})$.
	The preceding anti-concentration bounds and
	\eqref{truncation} then give the claimed distributional error.
	Finally,
	\[
	\mathbb EZ_{11,n}^3-\Gamma=o(n^{-3\alpha}),\quad
	\mathbb EZ_{11,n}^4-(\beta+3)=o(n^{-2\alpha}),\quad
	\mathbb EZ_{11,n}^6-\Delta=o(1),
	\]
	so replacing the transformed moment parameters in the
	Edgeworth expansion by the original ones also incurs
	only $o(n^{-1/2})$ error.
	All row expectations are taken under the product law,
	using bounded spectral extensions rather than conditioning
	on $\mathcal G_n$.
\end{proof}

	\subsection{Some useful equations}
In this section, we extend the calculations from the proof of Theorem~\ref{G4MT}. Through straightforward computations, we derive the following result:
\begin{align*}
	&\Omega_{s_1}( \widetilde{\Z}_1,\Y_{k+1})-\Omega_{s_1}( \Z_{k0}\V_1,\Y_{k0})\\
	&=\frac{-c_n }{\theta_k\sqrt n}\bigg(1-\V_1'\z_k\z_k'\V_1-\frac 1n\V_1'\Z_{0k}'\Big(\rho_n\I_{n-1}-\frac 1n\Y_{0k}\V_2\V_2'\Y_{0k}'\Big)^{-1}\Y_{0k}\V_2\V_2'\z_k\z_k'\V_1\\
	&\quad+\frac {1}{n^2}\z_k'\V_2\V_2'\Y_{0k}'\Big(\rho_n\I_{n-1}-\frac 1n\Y_{0k}\V_2\V_2'\Y_{0k}'\Big)^{-2}\Y_{0k}\V_2\V_2'\z_k\\
	&\quad-\frac 1n\V'_1\z_k\z_k'\V_2\V_2'\Y_{0k}'\Big(\rho_n\I_{n-1}-\frac 1n\Y_{0k}\V_2\V_2'\Y_{0k}'\Big)^{-1}\Z_{0k}\V_1\\
	&\quad-\frac{1}{n^2}\V_1'\Z_{0k}'\Big(\rho_n\I_{n-1}-\frac 1n\Y_{0k}\V_2\V_2'\Y_{0k}'\Big)^{-1} \Y_{0k}\V_2\V_2'\z_k\z_k'\V_2\V_2'\Y_{0k}'\\
	&\quad\times\Big(\rho_n\I_{n-1}-\frac 1n\Y_{0k}\V_2\V_2'\Y_{0k}'\Big)^{-1}\Z_{0k}\V_1\bigg)+\frac{c_n}{\rho_n\sqrt n}\\
	&=\frac{c_n}{\rho_n\sqrt n}+\sum_{i=1}^5\sum_{j=0}^2\alpha_{kij},
\end{align*}
where
\begin{align*}
	\alpha_{k10}&=-\frac{c_n}{n^{5/2}\rho_n\bar\eta_k}\z_k'\p_k\z_k,~ \alpha_{k11}=-\frac{c_n\epsilon_k}{n^{5/2}\rho_n\bar\eta_k^2}\z_k'\p_k\z_k,~\alpha_{k12}=-\frac{c_n\epsilon_k^2}{n^{5/2}\rho_n\eta_k\bar\eta_k^2}\z_k'\p_k\z_k,\\
	\alpha_{k20}&=\frac{c_n}{n^{3/2}\rho_n\bar\eta_k}\s_k'\z_k\z_k'\V_1,\quad\alpha_{k21}=\frac{c_n\epsilon_k}{n^{3/2}\rho_n\bar\eta_k^2}\s_k'\z_k\z_k'\V_1,\\
	\alpha_{k22}&=\frac{c_n\epsilon_k^2}{n^{3/2}\eta_k\rho_n\bar\eta_k^2}\s_k'\z_k\z_k'\V_1,
	~\alpha_{k30}=\frac{c_n}{n^{3/2}\rho_n\bar\eta_k}\V_1'\z_k\z_k'\s_k,\\
	\alpha_{k31}&=\frac{c_n\epsilon_k}{n^{3/2}\rho_n\bar\eta_k^2}\V_1'\z_k\z_k'\s_k,~ \quad\alpha_{k32}=\frac{c_n\epsilon_k^2}{n^{3/2}\rho_n\eta_k\bar\eta_k^2}\V_1'\z_k\z_k'\s_k,\\
	\alpha_{k40}&=\frac{c_n}{n^{5/2}\rho_n\bar\eta_k}\s_k'\z_k\z_k'\s_k,~
	\alpha_{k41}=\frac{c_n\epsilon_k}{n^{5/2}\rho_n\bar\eta_k^2}\s_k'\z_k\z_k'\s_k,~
	\alpha_{k42}=\frac{c_n\epsilon_k^2}{n^{5/2}\rho_n\eta_k\bar\eta_k^2}\s_k'\z_k\z_k'\s_k,\\
	\alpha_{k50}&=-\frac{c_n}{\sqrt n \rho_n\bar\eta_k}(1-\V_1'\z_k\z_k'\V_1),~\quad\alpha_{k51}=-\frac{c_n\epsilon_k }{\sqrt n \rho_n\bar\eta_k^2}(1-\V_1'\z_k\z_k'\V_1),\\
	\alpha_{k52}&=-\frac{c_n \epsilon_k^2}{\sqrt n\rho_n\eta_k \bar\eta_k^2}(1-\V_1'\z_k\z_k'\V_1).
\end{align*}
Using the same method, we can also obtain the following result:
\begin{align*}
	&\Omega_{s_1}( \widetilde{\Z}_1,\Y_{k})-\Omega_{s_1}( \Z_{k0}\V_1,\Y_{k0})=\frac{c_n}{\rho_n\sqrt n}+\sum_{i=1}^5\sum_{j=0}^2\alpha_{kijy},
\end{align*}
where
\begin{align*}
	\alpha_{k10y}&=-\frac{c_n}{n^{5/2}\rho_n\bar\eta_{k}}\y_k'\p_k\y_k,\alpha_{k11y}=-\frac{c_n\epsilon_{ky}}{n^{5/2}\rho_n\bar\eta_{k}^2}\y_k'\p_k\y_k,\\
	\alpha_{k12y}&=-\frac{c_n\epsilon_{ky}^2}{n^{5/2}\rho_n\eta_{ky}\bar\eta_{k}^2}\y_k'\p_k\y_k,\\
	\alpha_{k20y}&=\frac{c_n}{n^{3/2}\rho_n\bar\eta_{k}}\s_k'\y_k\z_k'\V_1,\quad\alpha_{k21y}=\frac{c_n\epsilon_{ky}}{n^{3/2}\rho_n\bar\eta_{k}^2}\s_k'\y_k\z_k'\V_1,\\
	\alpha_{k22y}&=\frac{c_n\epsilon_{ky}^2}{n^{3/2}\eta_{ky}\rho_n\bar\eta_{k}^2}\s_k'\y_k\z_k'\V_1,
	~\alpha_{k30y}=\frac{c_n}{n^{3/2}\rho_n\bar\eta_{k}}\V_1'\z_k\y_k'\s_k,\\
	\alpha_{k31y}&=\frac{c_n\epsilon_{ky}}{n^{3/2}\rho_n\bar\eta_{k}^2}\V_1'\z_k\y_k'\s_k,~ \quad\alpha_{k32y}=\frac{c_n\epsilon_{ky}^2}{n^{3/2}\rho_n\eta_{ky}\bar\eta_{k}^2}\V_1'\z_k\y_k'\s_k,\\
	\alpha_{k40y}&=\frac{c_n}{n^{5/2}\rho_n\bar\eta_{k}}\s_k'\y_k\y_k'\s_k,
	\alpha_{k41y}=\frac{c_n\epsilon_{ky}}{n^{5/2}\rho_n\bar\eta_{k}^2}\s_k'\y_k\y_k'\s_k,\\
	\alpha_{k42y}&=\frac{c_n\epsilon_{ky}^2}{n^{5/2}\rho_n\eta_{ky}\bar\eta_{k}^2}\s_k'\y_k\y_k'\s_k,\\
	\alpha_{k50y}&=-\frac{c_n}{\sqrt n \rho_n\bar\eta_{k}}(1-\V_1'\z_k\z_k'\V_1),~\quad\alpha_{k51y}=-\frac{c_n\epsilon_{ky} }{\sqrt n \rho_n\bar\eta_{k}^2}(1-\V_1'\z_k\z_k'\V_1),\\
	\alpha_{k52y}&=-\frac{c_n \epsilon_{ky}^2}{\sqrt n\rho_n\eta_{ky} \bar\eta_{k}^2}(1-\V_1'\z_k\z_k'\V_1).
\end{align*}
Using the same method, decompose the second term into the equation below. Through a  series of straightforward calculations, we can obtain the following result:
\begin{align*}
	&\qquad\Omega_{s_{21}}( \widetilde{\Z}_1,\Y_{k+1})-\Omega_{s_{21}}( \Z_{k0}\V_1,\Y_{k0})\\
	&=\frac{\tau_n}{\theta_k^2n^2}\bigg(\z_k'\V_2\V_2'\z_k+\frac 1n\z_k'\V_2\V_2'\Y_{0k}'\Big(\rho_n\I_{n-1}-\frac 1n\Y_{0k}\V_2\V_2'\Y_{0k}'\Big)^{-1}\Y_{0k}\V_2\V_2'\z_k\\
	&\quad+\frac {1}{n^2}\z_k'\V_2\V_2'\z_k\z_k'\V_2\V_2'\Y_{0k}'\Big(\rho_n\I_{n-1}-\frac {1}{n}\Y_{0k}\V_2\V_2'\Y_{0k}'\Big)^{-2}\Y_{0k}\V_2\V_2'\z_k\\
	&\quad+\frac {\theta_k}{ n}\z_k'\V_2\V_2'\Y_{0k}'\Big(\rho_n\I_{n-1}-\frac {1}{n}\Y_{0k}\V_2\V_2'\Y_{0k}'\Big)^{-2}\Y_{0k}\V_2\V_2'\z_k\\
	&\quad+\frac{1}{n^3}\z_k'\V_2\V_2'\Y_{0k}'\Big(\rho_n\I_{n-1}-\frac 1n\Y_{0k}\V_2\V_2'\Y_{0k}'\Big)^{-1}\Y_{0k}\V_2\V_2'\z_k\z_k'\V_2\V_2'\Y_{0k}'\\
	&\quad\times\Big(\rho_n\I_{n-1}-\frac 1n\Y_{0k}\V_2\V_2'\Y_{0k}'\Big)^{-2}\Y_{0k}\V_2\V_2'\z_k\\
	&\quad+\tr\Big\{\frac{1}{n}\Y_{0k}\V_2\V_2'\z_k\z_k'\V_2\V_2'\Y_{0k}'\Big(\rho_n\I_{n-1}-\frac 1n\Y_{0k}\V_2\V_2'\Y_{0k}'\Big)^{-1}\\
	&\quad+\frac{\theta_k}{ n}\Y_{0k}\V_2\V_2'\z_k\z_k'\V_2\V_2'\Y_{0k}' \Big(\rho_n\I_{n-1}-\frac 1n\Y_{0k}\V_2\V_2'\Y_{0k}'\Big)^{-2}\\
	&\quad +\frac{1}{n^3}\Y_{0k}\V_2\V_2'\z_k\z_k'\V_2\V_2'\Y_{0k}'\Big(\rho_n\I_{n-1}-\frac 1n\Y_{0k}\V_2\V_2'\Y_{0k}'\Big)^{-2}\Y_{0k}\V_2\V_2'\z_k\z_k'\V_2\V_2'\\
	&\quad\times\Y_{0k}'\Big(\rho_n\I_{n-1}-\frac 1n\Y_{0k}\V_2\V_2'\Y_{0k}'\Big)^{-1}\\
	&\quad +\frac{1}{n^2}\Y_{0k}\V_2\V_2'\Y_{0k}'\Big(\rho_n\I_{n-1} -\frac 1n\Y_{0k}\V_2\V_2'\Y_{0k}'\Big)^{-1}\Y_{0k}\V_2\V_2'\z_k\z_k'\V_2\V_2'\Y_{0k}'\\
	&\quad\times\Big(\rho_n\I_{n-1}-\frac 1n\Y_{0k}\V_2\V_2'\Y_{0k}'\Big)^{-1}\\
	&\quad +\frac{1}{n^4}\Y_{0k}\V_2\V_2'\Y_{0k}'\Big(\rho_n\I_{n-1} -\frac 1n\Y_{0k}\V_2\V_2'\Y_{0k}'\Big)^{-1}\Y_{0k}\V_2\V_2'\z_k\z_k'\V_2\V_2'\Y_{0k}'\\
	&\quad\times\Big(\rho_n\I_{n-1}-\frac 1n\Y_{0k}\V_2\V_2'\Y_{0k}'\Big)^{-2}\Y_{0k}\V_2\V_2'\z_k\z_k'\V_2\V_2'\Y_{0k}'\\
	&\quad\times\Big(\rho_n\I_{n-1}-\frac 1n\Y_{0k}\V_2\V_2'\Y_{0k}'\Big)^{-1}\\
	&\quad +\frac{\theta_k}{n^2}\Y_{0k}\V_2\V_2'\Y_{0k}'\Big(\rho_n\I_{n-1}-\frac 1n\Y_{0k}\V_2\V_2'\Y_{0k}'\Big)^{-2}\Y_{0k}\V_2\V_2'\z_k\z_k'\V_2\V_2'\Y_{0k}'\\
	&\quad\times\Big(\rho_n\I_{n-1}-\frac 1n\Y_{0k}\V_2\V_2'\Y_{0k}'\Big)^{-1}\\
	&\quad +\frac{\theta_k}{n^2}\Y_{0k}\V_2\V_2'\Y_{0k}'\Big(\rho_n\I_{n-1}-\frac 1n\Y_{0k}\V_2\V_2'\Y_{0k}'\Big)^{-1}\Y_{0k}\V_2\V_2'\z_k\z_k'\V_2\V_2'\Y_{0k}'\\
	&\quad\times\Big(\rho_n\I_{n-1}-\frac 1n\Y_{0k}\V_2\V_2'\Y_{0k}'\Big)^{-2}\Big\}\bigg)\\
	&=\sum_{i=1}^{12}\sum_{j=0}^2\beta_{kij},
\end{align*}
where
\begin{align*}
	\beta_{k10}&=\frac{\tau_n}{\rho_n^2\eta_k\bar\eta_k n^4}\z_k'\V_2\V_2'\z_k \z_k'\m_k\m_k'\z_k,~
	\beta_{k11}=\frac{\tau_n\epsilon_k}{\rho_n^2\eta_k\bar\eta_k^2 n^4}\z_k'\V_2\V_2'\z_k \z_k'\m_k\m_k'\z_k,\\
	\beta_{k12}&=\frac{\tau_n\epsilon_k^2}{\rho_n^2\eta_k^2\bar\eta_k^2 n^4}\z_k'\V_2\V_2'\z_k \z_k'\m_k\m_k'\z_k,\\
	\beta_{k20}&=\frac {\tau_n}{\rho_n^2\eta_k\bar\eta_k n^3}\z_k'\q_{k}\z_k,~\beta_{k21}=\frac {\tau_n\epsilon_k}{\rho_n^2\eta_k\bar\eta_k^2 n^3}\z_k'\q_{k}\z_k,
	~\beta_{k22}=\frac {\tau_n\epsilon_k^2}{\rho_n^2\eta_k^2\bar\eta_k^2 n^3}\z_k'\q_{k}\z_k,\\
	\beta_{k30}&=\frac{\tau_n\theta_k}{\rho_n^2\eta_k\bar\eta_k n^3}\z_k'\p_k\z_k,~
	\beta_{k31}=\frac{\tau_n\theta_k\epsilon_k}{\rho_n^2\eta_k\bar\eta_k^2 n^3}\z_k'\p_k\z_k,~\beta_{k32}=\frac{\tau_n\theta_k\epsilon_k^2}{\rho_n^2\eta_k^2\bar\eta_k^2 n^3}\z_k'\p_k\z_k,\\
	\beta_{k40}&=\frac{\tau_n}{\rho_n^2\eta_k\bar\eta_k n^5}\z_k'\q_{k}\V_2\V_2'\z_k\z_k'\p_{k}\z_k,~\beta_{k41}=\frac{\tau_n\epsilon_k}{\rho_n^2\eta_k\bar\eta_k^2 n^5}\z_k'\q_{k}\V_2\V_2'\z_k\z_k'\p_{k}\z_k,\\
	\beta_{k42}&=\frac{\tau_n\epsilon_k^2}{\rho_n^2\eta_k^2\bar\eta_k^2 n^5}\z_k'\q_{k}\V_2\V_2'\z_k\z_k'\p_{k}\z_k,~\beta_{k50}=\frac{\tau_n}{\rho_n^2\eta_k\bar\eta_k n^3}\tr(\Y_{0k}\V_2\V_2'\z_k\z_k'\m_k),\\
	\beta_{k51}&=\frac{\tau_n\epsilon_k}{\rho_n^2\eta_k\bar\eta_k^2 n^3}\tr(\Y_{0k}\V_2\V_2'\z_k\z_k'\m_k),~\beta_{k52}=\frac{\tau_n\epsilon_k^2}{\rho_n^2\eta_k^2\bar\eta_k^2 n^3}\tr(\Y_{0k}\V_2\V_2'\z_k\z_k'\m_k),\\
	\beta_{k60}&=\frac{\tau_n\theta_k}{\rho_n^2\eta_k\bar\eta_k n^3}\tr(\Y_{0k}\V_2\V_2'\z_k\z_k'\e_k),~\beta_{k61}=\frac{\tau_n\theta_k\epsilon_k}{\rho_n^2\eta_k\bar\eta_k^2 n^3}\tr(\Y_{0k}\V_2\V_2'\z_k\z_k'\e_k),\\
	\beta_{k62}&=\frac{\tau_n\theta_k\epsilon_k^2}{\rho_n^2\eta_k^2\bar\eta_k^2 n^3}\tr(\Y_{0k}\V_2\V_2'\z_k\z_k'\e_k),\\
	\beta_{k70}&=\frac{\tau_n}{\rho_n^2\eta_k\bar\eta_k n^5}\tr(\Y_{0k}\V_2\V_2'\z_k\z_k'\p_k\z_k\z_k'\m_k),\\
	\beta_{k71}&=\frac{\tau_n\epsilon_k}{\rho_n^2\eta_k\bar\eta_k^2 n^5}\tr(\Y_{0k}\V_2\V_2'\z_k\z_k'\p_k\z_k\z_k'\m_k),\\
	\beta_{k72}&=\frac{\tau_n\epsilon_k^2}{\rho_n^2\eta_k^2\bar\eta_k^2 n^5}\tr(\Y_{0k}\V_2\V_2'\z_k\z_k'\p_k\z_k\z_k'\m_k),\\
	\beta_{k80}&=\frac{\tau_n}{\rho_n^2\eta_k\bar\eta_k n^4}\tr(\Y_{0k}\q_k\z_k\z_k'\m_k),~\beta_{k81}=\frac{\tau_n\epsilon_k}{\rho_n^2\eta_k\bar\eta_k^2 n^4}\tr(\Y_{0k}\q_k\z_k\z_k'\m_k),\\
	\beta_{k82}&=\frac{\tau_n\epsilon_k^2}{\rho_n^2\eta_k^2\bar\eta_k^2 n^4}\tr(\Y_{0k}\q_k\z_k\z_k'\m_k),~\beta_{k90}=\frac{\tau_n}{\rho_n^2\eta_k\bar\eta_k n^6}\tr(\Y_{0k}\q_k\z_k\z_k'\p_k\z_k\z_k'\m_k),\\
	\beta_{k91}&=\frac{\tau_n\epsilon_k}{\rho_n^2\eta_k\bar\eta_k^2 n^6}\tr(\Y_{0k}\q_k\z_k\z_k'\p_k\z_k\z_k'\m_k),\\
	\beta_{k92}&=\frac{\tau_n\epsilon_k^2}{\rho_n^2\eta_k^2\bar\eta_k^2 n^6}\tr(\Y_{0k}\q_k\z_k\z_k'\p_k\z_k\z_k'\m_k),\\
	\beta_{k100}&=\frac{\tau_n\theta_k}{\rho_n^2\eta_k\bar\eta_k n^4}\tr(\Y_{0k}\p_k\z_k\z_k'\m_k),~\beta_{k101}=\frac{\tau_n\theta_k\epsilon_k}{\rho_n^2\eta_k\bar\eta_k^2 n^4}\tr(\Y_{0k}\p_k\z_k\z_k'\m_k),\\
	\beta_{k102}&=\frac{\tau_n\theta_k\epsilon_k^2}{\rho_n^2\eta_k^2\bar\eta_k^2 n^4}\tr(\Y_{0k}\p_k\z_k\z_k'\m_k),~\beta_{k110}=\frac{\tau_n\theta_k}{\rho_n^2\eta_k\bar\eta_k n^4}\tr(\Y_{0k}\q_k\z_k\z_k'\e_k),\\
	\beta_{k111}&=\frac{\tau_n\theta_k\epsilon_k}{\rho_n^2\eta_k\bar\eta_k^2 n^4}\tr(\Y_{0k}\q_k\z_k\z_k'\e_k),~\beta_{k112}=\frac{\tau_n\theta_k\epsilon_k^2}{\rho_n^2\eta_k^2\bar\eta_k^2 n^4}\tr(\Y_{0k}\q_k\z_k\z_k'\e_k),\\
	\beta_{k120}&=\frac {\tau_n}{\rho_n^2\eta_k\bar\eta_k n^2}\z_k'\V_2\V_2'\z_k,
	~\beta_{k121}=\frac {\tau_n\epsilon_k}{\rho_n^2\eta_k\bar\eta_k^2 n^2}\z_k'\V_2\V_2'\z_k,\\
	\beta_{k122}&=\frac{\tau_n\epsilon^2_k}{\rho_n^2\eta_k^2\bar\eta_k^2 n^2}\z_k'\V_2\V_2'\z_k.
\end{align*}
Using the same method, we can also obtain the following result:
\[\Omega_{s_{21}}( \widetilde{\Z}_1,\Y_{k})-\Omega_{s_{21}}( \Z_{k0}\V_1,\Y_{k0})=\sum_{i=1}^{12}\sum_{j=0}^2\beta_{kijy},\]
where 
\begin{align*}
	\beta_{k10y}&=\frac{\tau_n}{\rho_n^2\eta_{ky}\bar\eta_{k} n^4}\y_k'\V_2\V_2'\y_k \y_k'\m_k\m_k'\y_k,\\
	\beta_{k11y}&=\frac{\tau_n\epsilon_{ky}}{\rho_n^2\eta_{ky}\bar\eta_{k}^2 n^4}\y_k'\V_2\V_2'\y_k \y_k'\m_k\m_k'\y_k,\\
	\beta_{k12y}&=\frac{\tau_n\epsilon_{ky}^2}{\rho_n^2\eta_{ky}^2\bar\eta_{k}^2 n^4}\y_k'\V_2\V_2'\y_k \y_k'\m_k\m_k'\y_k,\\
	\beta_{k20y}&=\frac {\tau_n}{\rho_n^2\eta_{ky}\bar\eta_{k} n^3}\y_k'\q_{k}\y_k,~\beta_{k21y}=\frac {\tau_n\epsilon_{ky}}{\rho_n^2\eta_{ky}\bar\eta_{k}^2 n^3}\y_k'\q_{k}\y_k,\\
	\beta_{k22y}&=\frac {\tau_n\epsilon_{ky}^2}{\rho_n^2\eta_{ky}^2\bar\eta_{k}^2 n^3}\y_k'\q_{k}\y_k,~
	\beta_{k30y}=\frac{\tau_n\theta_{ky}}{\rho_n^2\eta_{ky}\bar\eta_{k} n^3}\y_k'\p_k\y_k,\\
	\beta_{k31y}&=\frac{\tau_n\theta_{ky}\epsilon_{ky}}{\rho_n^2\eta_{ky}\bar\eta_{k}^2 n^3}\y_k'\p_k\y_k,~\beta_{k32y}=\frac{\tau_n\theta_{ky}\epsilon_{ky}^2}{\rho_n^2\eta_{ky}^2\bar\eta_{k}^2 n^3}\y_k'\p_k\y_k,\\
	\beta_{k40y}&=\frac{\tau_n}{\rho_n^2\eta_{ky}\bar\eta_{k} n^5}\y_k'\q_{k}\V_2\V_2'\y_k\y_k'\p_{k}\y_k,~\beta_{k41y}=\frac{\tau_n\epsilon_{ky}}{\rho_n^2\eta_{ky}\bar\eta_{k}^2 n^5}\y_k'\q_{k}\V_2\V_2'\y_k\y_k'\p_{k}\y_k,\\
	\beta_{k42y}&=\frac{\tau_n\epsilon_{ky}^2}{\rho_n^2\eta_{ky}^2\bar\eta_{k}^2 n^5}\y_k'\q_{k}\V_2\V_2'\y_k\y_k'\p_{k}\y_k,~\beta_{k50y}=\frac{\tau_n}{\rho_n^2\eta_{ky}\bar\eta_{k} n^3}\tr(\Y_{0k}\V_2\V_2'\y_k\y_k'\m_k),\\
	\beta_{k51y}&=\frac{\tau_n\epsilon_{ky}}{\rho_n^2\eta_{ky}\bar\eta_{k}^2 n^3}\tr(\Y_{0k}\V_2\V_2'\y_k\y_k'\m_k),\\
	\beta_{k52y}&=\frac{\tau_n\epsilon_{ky}^2}{\rho_n^2\eta_{ky}^2\bar\eta_{k}^2 n^3}\tr(\Y_{0k}\V_2\V_2'\y_k\y_k'\m_k),\\
	\beta_{k60y}&=\frac{\tau_n\theta_{ky}}{\rho_n^2\eta_{ky}\bar\eta_{k} n^3}\tr(\Y_{0k}\V_2\V_2'\y_k\y_k'\e_k),~\beta_{k61y}=\frac{\tau_n\theta_{ky}\epsilon_{ky}}{\rho_n^2\eta_{ky}\bar\eta_{k}^2 n^3}\tr(\Y_{0k}\V_2\V_2'\y_k\y_k'\e_k),\\
	\beta_{k62y}&=\frac{\tau_n\theta_{ky}\epsilon_{ky}^2}{\rho_n^2\eta_{ky}^2\bar\eta_{k}^2 n^3}\tr(\Y_{0k}\V_2\V_2'\y_k\y_k'\e_k),\\
	\beta_{k70y}&=\frac{\tau_n}{\rho_n^2\eta_{ky}\bar\eta_{k} n^5}\tr(\Y_{0k}\V_2\V_2'\y_k\y_k'\p_k\y_k\y_k'\m_k),\\
	\beta_{k71y}&=\frac{\tau_n\epsilon_{ky}}{\rho_n^2\eta_{ky}\bar\eta_{k}^2 n^5}\tr(\Y_{0k}\V_2\V_2'\y_k\y_k'\p_k\y_k\y_k'\m_k),\\
	\beta_{k72y}&=\frac{\tau_n\epsilon_{ky}^2}{\rho_n^2\eta_{ky}^2\bar\eta_{k}^2 n^5}\tr(\Y_{0k}\V_2\V_2'\y_k\y_k'\p_k\y_k\y_k'\m_k),\\
	\beta_{k80y}&=\frac{\tau_n}{\rho_n^2\eta_{ky}\bar\eta_{k} n^4}\tr(\Y_{0k}\q_k\y_k\y_k'\m_k),~\beta_{k81y}=\frac{\tau_n\epsilon_{ky}}{\rho_n^2\eta_{ky}\bar\eta_{k}^2 n^4}\tr(\Y_{0k}\q_k\y_k\y_k'\m_k),\\
	\beta_{k82y}&=\frac{\tau_n\epsilon_{ky}^2}{\rho_n^2\eta_{ky}^2\bar\eta_{k}^2 n^4}\tr(\Y_{0k}\q_k\y_k\y_k'\m_k),\\
	\beta_{k90y}&=\frac{\tau_n}{\rho_n^2\eta_{ky}\bar\eta_{k} n^6}\tr(\Y_{0k}\q_k\y_k\y_k'\p_k\y_k\y_k'\m_k),\\
	\beta_{k91y}&=\frac{\tau_n\epsilon_{ky}}{\rho_n^2\eta_{ky}\bar\eta_{k}^2 n^6}\tr(\Y_{0k}\q_k\y_k\y_k'\p_k\y_k\y_k'\m_k),\\
	\beta_{k92y}&=\frac{\tau_n\epsilon_{ky}^2}{\rho_n^2\eta_{ky}^2\bar\eta_{k}^2 n^6}\tr(\Y_{0k}\q_k\y_k\y_k'\p_k\y_k\y_k'\m_k),\\
	\beta_{k100y}&=\frac{\tau_n\theta_{ky}}{\rho_n^2\eta_{ky}\bar\eta_{k} n^4}\tr(\Y_{0k}\p_k\y_k\y_k'\m_k),~\beta_{k101y}=\frac{\tau_n\theta_{ky}\epsilon_{ky}}{\rho_n^2\eta_{ky}\bar\eta_{k}^2 n^4}\tr(\Y_{0k}\p_k\y_k\y_k'\m_k),\\
	\beta_{k102y}&=\frac{\tau_n\theta_{ky}\epsilon_{ky}^2}{\rho_n^2\eta_{ky}^2\bar\eta_{k}^2 n^4}\tr(\Y_{0k}\p_k\y_k\y_k'\m_k),~\beta_{k110y}=\frac{\tau_n\theta_{ky}}{\rho_n^2\eta_{ky}\bar\eta_{k} n^4}\tr(\Y_{0k}\q_k\y_k\y_k'\e_k),\\
	\beta_{k111y}&=\frac{\tau_n\theta_{ky}\epsilon_{ky}}{\rho_n^2\eta_{ky}\bar\eta_{k}^2 n^4}\tr(\Y_{0k}\q_k\y_k\y_k'\e_k),~\beta_{k112y}=\frac{\tau_n\theta_{ky}\epsilon_{ky}^2}{\rho_n^2\eta_{ky}^2\bar\eta_{k}^2 n^4}\tr(\Y_{0k}\q_k\y_k\y_k'\e_k),\\
	\beta_{k120y}&=\frac {\tau_n}{\rho_n^2\eta_{ky}\bar\eta_{k} n^2}\y_k'\V_2\V_2'\y_k,
	~\beta_{k121y}=\frac {\tau_n\epsilon_{ky}}{\rho_n^2\eta_{ky}\bar\eta_{k}^2 n^2}\y_k'\V_2\V_2'\y_k,\\
	\beta_{k122y}&=\frac{\tau_n\epsilon^2_{ky}}{\rho_n^2\eta_{ky}^2\bar\eta_{k}^2 n^2}\y_k'\V_2\V_2'\y_k.
\end{align*}

Next, we continue to split the third term. Through a  series of straightforward calculations, we can obtain the following result:
\begin{align*}
	&\quad\Omega_{s_{22}}( \widetilde{\Z}_1,\Y_{k+1})-\Omega_{s_{22}}( \Z_{k0}\V_1,\Y_{k0})\\
	&=-\frac{\tau_n}{\theta_k^2n^2}\bigg(\V_1'\z_k\z_k'\V_2\V_2'\z_k\z_k'\V_1\\
	&\quad+\frac {1}{n^2}\V_1'\z_k\z_k'\V_2\V_2'\z_k\z_k'\V_2\V_2'\Y_{0k}'\Big(\rho_n\I_{n-1}-\frac {1}{n}\Y_{0k}\V_2\V_2'\Y_{0k}'\Big)^{-2}\Y_{0k}\V_2\V_2'\z_k\z_k'\V_1\\
	&\quad+\frac 1n\V_1'\z_k\z_k'\V_2\V_2'\Y_{0k}'\Big(\rho_n\I_{n-1}-\frac 1n\Y_{0k}\V_2\V_2'\Y_{0k}'\Big)^{-1}\Y_{0k}\V_2\V_2'\z_k\z_k'\V_1\\
	&\quad+\frac {\theta_k }{n}\V_1'\z_k\z_k'\V_2\V_2'\Y_{0k}'\Big(\rho_n\I_{n-1}-\frac {1}{n}\Y_{0k}\V_2\V_2'\Y_{0k}'\Big)^{-2}\Y_{0k}\V_2\V_2'\z_k\z_k'\V_1\\
	&\quad+\frac{1}{n^3}\V_1'\z_k\z_k'\V_2\V_2'\Y_{0k}'\Big(\rho_n\I_{n-1}-\frac 1n\Y_{0k}\V_2\V_2'\Y_{0k}'\Big)^{-1}\Y_{0k}\V_2\V_2'\z_k\z_k'\V_2\V_2'\Y_{0k}'\\
	&\quad\times\Big(\rho_n\I_{n-1}-\frac 1n\Y_{0k}\V_2\V_2'\Y_{0k}'\Big)^{-2}\Y_{0k}\V_2\V_2'\z_k\z_k'\V_1 +\V_1'\Z_{0k}'\Y_{0k}\V_2\V_2'\z_k\z_k'\V_1\\
	&\quad+\frac{1}{n^2}\V_1'\Z_{0k}'\Y_{0k}\V_2\V_2'\z_k\z_k'\V_2\V_2'\Y_{0k}' \Big(\rho_n\I_{n-1}-\frac {1}{n}\Y_{0k}\V_2\V_2'\Y_{0k}'\Big)^{-2}\Y_{0k}\V_2\V_2'\z_k\z_k'\V_1\\
	&\quad+\frac 1n \V_1'\Z_{0k}'\Y_{0k}\V_2\V_2'\Y_{0k}'\Big(\rho_n\I_{n-1}-\frac {1}{n}\Y_{0k}\V_2\V_2'\Y_{0k}'\Big)^{-1}\Y_{0k}\V_2\V_2'\z_k\z_k'\V_1\\
	&\quad+\frac{\theta_k}{ n}\V_1'\Z_{0k}'\Y_{0k}\V_2\V_2'\Y_{0k}'\Big(\rho_n\I_{n-1}-\frac {1}{ n}\Y_{0k}\V_2\V_2'\Y_{0k}'\Big)^{-2}\Y_{0k}\V_2\V_2'\z_k\z_k'\V_1\\
	&\quad+\frac{1}{n^3}\V_1'\Z_{0k}'\Y_{0k}\V_2\V_2'\Y_{0k}'\Big(\rho_n\I_{n-1}-\frac {1}{n}\Y_{0k}\V_2\V_2'\Y_{0k}'\Big)^{-1}\Y_{0k}\V_2\V_2'\z_k\z_k'\V_2\V_2'\Y_{0k}'\\
	&\quad\times\Big(\rho_n\I_{n-1}-\frac {1}{n}\Y_{0k}\V_2\V_2'\Y_{0k}'\Big)^{-2}\Y_{0k}\V_2\V_2'\z_k\z_k'\V_1\\
	&\quad+\frac 1n\V_1'\z_k\z_k'\V_2\V_2'\z_k\z_k'\V_2\V_2'\Y_{0k}'\Big(\rho_n\I_{n-1}-\frac 1n\Y_{0k}\V_2\V_2'\Y_{0k}'\Big)^{-1}\Z_{0k}\V_1\\
	&\quad+\frac{\theta_k}{n}\V_1'\z_k\z_k'\V_2\V_2'\z_k\z_k'\V_2\V_2'\Y_{0k}'\Big(\rho_n\I_{n-1}-\frac 1n\Y_{0k}\V_2\V_2'\Y_{0k}'\Big)^{-2}\Z_{0k}\V_1\\
	&\quad+\frac{1}{n^3}\V_1'\z_k\z_k'\V_2\V_2'\z_k\z_k'\V_2\V_2'\Y_{0k}'\Big(\rho_n\I_{n-1}-\frac 1n\Y_{0k}\V_2\V_2'\Y_{0k}'\Big)^{-2}\Y_{0k}\V_2\V_2'\z_k\z_k'\\
	&\quad\times\V_2\V_2'\Y_{0k}'(\rho_n\I_{n-1}-\frac 1n\Y_{0k}\V_2\V_2'\Y_{0k}'\Big)^{-1}\Z_{0k}\V_1\\
	&\quad+\frac{1}{n^2}\V_1'\z_k\z_k'\V_2\V_2'\Y_{0k}'\Big(\rho_n\I_{n-1}-\frac 1n\Y_{0k}\V_2\V_2'\Y_{0k}'\Big)^{-1}\Y_{0k}\V_2\V_2'\z_k\z_k'\V_2\V_2'\Y_{0k}'\\
	&\quad\times\Big(\rho_n\I_{n-1}-\frac 1n\Y_{0k}\V_2\V_2'\Y_{0k}'\Big)^{-1}\Z_{0k}\V_1\\
	&\quad+\theta_k^2\V_1'\z_k\z_k'\V_2\V_2'\Y_{0k}'\Big(\rho_n\I_{n-1}-\frac 1n\Y_{0k}\V_2\V_2'\Y_{0k}'\Big)^{-2}\Z_{0k}\V_1\\
	&\quad+\frac{1}{n^4}\V_1'\z_k\z_k'\V_2\V_2'\Y_{0k}'\Big(\rho_n\I_{n-1}-\frac 1n\Y_{0k}\V_2\V_2'\Y_{0k}'\Big)^{-1}\Y_{0k}\V_2\V_2'\z_k\z_k'\V_2\V_2'\Y_{0k}'\\
	&\quad\times\Big(\rho_n\I_{n-1}-\frac 1n\Y_{0k}\V_2\V_2'\Y_{0k}'\Big)^{-2}\Y_{0k}\V_2\V_2'\z_k\z_k'\V_2\V_2'\Y_{0k}'\\
	&\quad\times (\rho_n\I_{n-1}-\frac 1n\Y_{0k}\V_2\V_2'\Y_{0k}'\Big)^{-1}\Z_{0k}\V_1\\
	&\quad+\frac{\theta_k}{ n^2}\V_1'\z_k\z_k'\V_2\V_2'\Y_{0k}'\Big(\rho_n\I_{n-1}-\frac 1n\Y_{0k}\V_2\V_2'\Y_{0k}'\Big)^{-2}\Y_{0k}\V_2\V_2'\z_k\z_k'\V_2\V_2'\Y_{0k}'\\
	&\quad\times (\rho_n\I_{n-1}-\frac 1n\Y_{0k}\V_2\V_2'\Y_{0k}'\Big)^{-1}\Z_{0k}\V_1\\
	&\quad+\frac{\theta_k}{ n^2}\V_1'\z_k\z_k'\V_2\V_2'\Y_{0k}' \Big(\rho_n\I_{n-1}-\frac 1n\Y_{0k}\V_2\V_2'\Y_{0k}'\Big)^{-1}\Y_{0k}\V_2\V_2'\z_k\z_k'\V_2\V_2'\Y_{0k}'\\
	&\quad\times\Big(\rho_n\I_{n-1}-\frac 1n\Y_{0k}\V_2\V_2'\Y_{0k}'\Big)^{-2}\Z_{0k}\V_1\\
	&\quad+\frac 1n\V_1'\Z_{0k}'\Y_{0k}\V_2\V_2'\z_k\z_k'\V_2\V_2'\Y_{0k}'\Big(\rho_n\I_{n-1}-\frac 1n\Y_{0k}\V_2\V_2'\Y_{0k}'\Big)^{-1}\Z_{0k}\V_1\\
	&\quad+\frac{\theta_k }{n}\V_1'\Z_{0k}'\Y_{0k}\V_2\V_2'\z_k\z_k'\V_2\V_2'\Y_{0k}'\Big(\rho_n\I_{n-1}-\frac 1n\Y_{0k}\V_2\V_2'\Y_{0k}'\Big)^{-2}\Z_{0k}\V_1\\
	&\quad+\frac{1}{n^3}\V_1'\Z_{0k}'\Y_{0k}\V_2\V_2'\z_k\z_k'\V_2\V_2'\Y_{0k}'\Big(\rho_n\I_{n-1}-\frac 1n\Y_{0k}\V_2\V_2'\Y_{0k}'\Big)^{-2}\Y_{0k}\V_2\V_2' \\
	&\quad\times\z_k\z_k'\V_2\V_2'\Y_{0k}'(\rho_n\I_{n-1}-\frac 1n\Y_{0k}\V_2\V_2'\Y_{0k}'\Big)^{-1}\Z_{0k}\V_1\\
	&\quad+\frac{1}{n^2}\V_1'\Z_{0k}'\Y_{0k}\V_2\V_2'\Y_{0k}'\Big(\rho_n\I_{n-1}-\frac 1n\Y_{0k}\V_2\V_2'\Y_{0k}'\Big)^{-1}\Y_{0k}\V_2\V_2'\z_k\z_k'\V_2\V_2'\Y_{0k}' \\
	&\quad\times\Big(\rho_n\I_{n-1}-\frac 1n\Y_{0k}\V_2\V_2'\Y_{0k}'\Big)^{-1}\Z_{0k}\V_1\\
	&\quad+\frac{1}{n^4}\V_1'\Z_{0k}'\Y_{0k}\V_2\V_2'\Y_{0k}'\Big(\rho_n\I_{n-1}-\frac 1n\Y_{0k}\V_2\V_2'\Y_{0k}'\Big)^{-1} \Y_{0k}\V_2\V_2'\z_k\z_k'\V_2\V_2'\Y_{0k}'\\
	&\quad\times\Big(\rho_n\I_{n-1}-\frac 1n\Y_{0k}\V_2\V_2'\Y_{0k}'\Big)^{-2}\Y_{0k}\V_2\V_2'\z_k\z_k'\V_2\V_2'\Y_{0k}'\\
	&\quad\times(\rho_n\I_{n-1}-\frac 1n\Y_{0k}\V_2\V_2'\Y_{0k}'\Big)^{-1}\Z_{0k}\V_1\\
	&\quad+\frac{\theta_k}{ n^2}\V_1'\Z_{0k}'\Y_{0k}\V_2\V_2'\Y_{0k}'\Big(\rho_n\I_{n-1}-\frac 1n\Y_{0k}\V_2\V_2'\Y_{0k}'\Big)^{-2}\Y_{0k}\V_2\V_2'\z_k\z_k'\V_2\V_2'\Y_{0k}'\\
	&\quad\times(\rho_n\I_{n-1}-\frac 1n\Y_{0k}\V_2\V_2'\Y_{0k}'\Big)^{-1}\Z_{0k}\V_1\\
	&\quad+\frac{\theta_k}{ n^2}\V_1'\Z_{0k}'\Y_{0k}\V_2\V_2'\Y_{0k}'\Big(\rho_n\I_{n-1}-\frac 1n\Y_{0k}\V_2\V_2'\Y_{0k}'\Big)^{-1}\Y_{0k}\V_2\V_2'\z_k\z_k'\V_2\V_2'\Y_{0k}'\\
	&\quad\times\Big(\rho_n\I_{n-1}-\frac 1n\Y_{0k}\V_2\V_2'\Y_{0k}'\Big)^{-2}\Z_{0k}\V_1\bigg)\\
	&=-\sum_{i=1}^{25}\sum_{j=0}^2\Xi_{kij}
\end{align*}
where
\begin{align*}
	\Xi_{k10}&=\frac{\tau_n}{\rho_n^2\eta_k\bar\eta_k n^2}\V_1'\z_k\z_k'\V_2\V_2'\z_k\z_k'\V_1,
	\Xi_{k11}=\frac{\tau_n\epsilon_k}{\rho_n^2\eta_k\bar\eta_k^2 n^2}\V_1'\z_k\z_k'\V_2\V_2'\z_k\z_k'\V_1,\\
	\Xi_{k12}&=\frac{\tau_n\epsilon_k^2}{\rho_n^2\eta_k^2\bar\eta_k^2 n^2}\V_1'\z_k\z_k'\V_2\V_2'\z_k\z_k'\V_1,\\
	\Xi_{k20}&=\frac{\tau_n}{\rho_n^2\eta_k\bar\eta_k n^4}\V_1'\z_k\z_k'\V_2\V_2'\z_k\z_k'\p_k\z_k\z_k'\V_1,\\
	\Xi_{k21}&=\frac{\tau_n\epsilon_k}{\rho_n^2\eta_k\bar\eta_k^2 n^4}\V_1'\z_k\z_k'\V_2\V_2'\z_k\z_k'\p_k\z_k\z_k'\V_1,\\
	\Xi_{k22}&=\frac{\tau_n\epsilon_k^2}{\rho_n^2\eta_k^2\bar\eta_k^2 n^4}\V_1'\z_k\z_k'\V_2\V_2'\z_k\z_k'\p_k\z_k\z_k'\V_1,~\Xi_{k30}=\frac{\tau_n}{\rho_n^2\eta_k\bar\eta_k n^3}\V_1'\z_k\z_k'\q_k\z_k\z_k'\V_1,\\
	\Xi_{k31}&=\frac{\tau_n\epsilon_k}{\rho_n^2\eta_k\bar\eta_k^2 n^3}\V_1'\z_k\z_k'\q_k\z_k\z_k'\V_1,~\Xi_{k32}=\frac{\tau_n\epsilon_k^2}{\rho_n^2\eta_k^2\bar\eta_k^2 n^3}\V_1'\z_k\z_k'\q_k\z_k\z_k'\V_1,\\
	\Xi_{k40}&=\frac{\tau_n\theta_k}{\rho_n^2\eta_k\bar\eta_k n^3}\V_1'\z_k\z_k'\p_k\z_k\z_k'\V_1,~\Xi_{k41}=\frac{\tau_n\epsilon_k\theta_k}{\rho_n^2\eta_k\bar\eta_k^2 n^3}\V_1'\z_k\z_k'\p_k\z_k\z_k'\V_1,\\
	\Xi_{k42}&=\frac{\tau_n\epsilon_k^2\theta_k}{\rho_n^2\eta_k^2\bar\eta_k^2 n^3}\V_1'\z_k\z_k'\p_k\z_k\z_k'\V_1,~\Xi_{k50}=\frac{\tau_n}{\rho_n^2\eta_k\bar\eta_k n^5}\V_1'\z_k\z_k'\q_k\z_k\z_k'\p_k\z_k\z_k'\V_1,\\
	\Xi_{k51}&=\frac{\tau_n\epsilon_k}{\rho_n^2\eta_k\bar\eta_k^2 n^5}\V_1'\z_k\z_k'\q_k\z_k\z_k'\p_k\z_k\z_k'\V_1,\\
	\Xi_{k52}&=\frac{\tau_n\epsilon_k^2}{\rho_n^2\eta_k^2\bar\eta_k^2 n^5}\V_1'\z_k\z_k'\q_k\z_k\z_k'\p_k\z_k\z_k'\V_1,\\
	\Xi_{k60}&=\frac{\tau_n}{\rho_n^2\eta_k\bar\eta_k n^4}\V_1'\Z_{0k}'\Y_{0k}\V_2\V_2'\z_k\z_k'\p_k\z_k\z_k'\V_1,\\
	\Xi_{k61}&=\frac{\tau_n\epsilon_k}{\rho_n^2\eta_k\bar\eta_k^2 n^4}\V_1'\Z_{0k}'\Y_{0k}\V_2\V_2'\z_k\z_k'\p_k\z_k\z_k'\V_1,\\
	\Xi_{k62}&=\frac{\tau_n\epsilon_k^2}{\rho_n^2\eta_k^2\bar\eta_k^2 n^4}\V_1'\Z_{0k}'\Y_{0k}\V_2\V_2'\z_k\z_k'\p_k\z_k\z_k'\V_1,\\
	\Xi_{k70}&=\frac{\tau_n}{\rho_n^2\eta_k\bar\eta_k n^3}\V_1'\Z_{0k}'\Y_{0k}\q_k\z_k\z_k'\V_1,\\
	\Xi_{k71}&=\frac{\tau_n\epsilon_k}{\rho_n^2\eta_k\bar\eta_k^2 n^3}\V_1'\Z_{0k}'\Y_{0k}\q_k\z_k\z_k'\V_1,~\Xi_{k72}=\frac{\tau_n\epsilon_k^2}{\rho_n^2\eta_k^2\bar\eta_k^2 n^3}\V_1'\Z_{0k}'\Y_{0k}\q_k\z_k\z_k'\V_1,\\
	\Xi_{k80}&=\frac{\tau_n\theta_k}{\rho_n^2\eta_k\bar\eta_k n^3}\V_1'\Z_{0k}'\Y_{0k}\p_k\z_k\z_k'\V_1,~\Xi_{k81}=\frac{\tau_n\epsilon_k\theta_k}{\rho_n^2\eta_k\bar\eta_k^2 n^3}\V_1'\Z_{0k}'\Y_{0k}\p_k\z_k\z_k'\V_1,\\
	\Xi_{k82}&=\frac{\tau_n\epsilon_k^2\theta_k}{\rho_n^2\eta_k^2\bar\eta_k^2 n^3}\V_1'\Z_{0k}'\Y_{0k}\p_k\z_k\z_k'\V_1,\\
	\Xi_{k90}&=\frac{\tau_n}{\rho_n^2\eta_k\bar\eta_k n^5}\V_1'\Z_{0k}'\Y_{0k}\q_k\z_k\z_k'\p_k\z_k\z_k'\V_1,\\
	\Xi_{k91}&=\frac{\tau_n\epsilon_k}{\rho_n^2\eta_k\bar\eta_k^2 n^5}\V_1'\Z_{0k}'\Y_{0k}\q_k\z_k\z_k'\p_k\z_k\z_k'\V_1,\\
	\Xi_{k92}&=\frac{\tau_n\epsilon_k^2}{\rho_n^2\eta_k^2\bar\eta_k^2 n^5}\V_1'\Z_{0k}'\Y_{0k}\q_k\z_k\z_k'\p_k\z_k\z_k'\V_1,\\
	\Xi_{k100}&=\frac{\tau_n}{\rho_n^2\eta_k\bar\eta_k n^3}\V_1'\z_k\z_k'\V_2\V_2'\z_k\z_k'\m_k\Z_{0k}\V_1,\\
	\Xi_{k101}&=\frac{\tau_n\epsilon_k}{\rho_n^2\eta_k\bar\eta_k^2 n^3}\V_1'\z_k\z_k'\V_2\V_2'\z_k\z_k'\m_k\Z_{0k}\V_1,\\
	\Xi_{k102}&=\frac{\tau_n\epsilon_k^2}{\rho_n^2\eta_k^2\bar\eta_k^2 n^3}\V_1'\z_k\z_k'\V_2\V_2'\z_k\z_k'\m_k\Z_{0k}\V_1,\\
	\Xi_{k110}&=\frac{\tau_n\theta_k}{\rho_n^2\eta_k\bar\eta_k n^3}\V_1'\z_k\z_k'\V_2\V_2'\z_k\z_k'\e_k\Z_{0k}\V_1,\\
	\Xi_{k111}&=\frac{\tau_n\epsilon_k\theta_k}{\rho_n^2\eta_k\bar\eta_k^2 n^3}\V_1'\z_k\z_k'\V_2\V_2'\z_k\z_k'\e_k\Z_{0k}\V_1,\\
	\Xi_{k112}&=\frac{\tau_n\epsilon_k^2\theta_k}{\rho_n^2\eta_k^2\bar\eta_k^2 n^3}\V_1'\z_k\z_k'\V_2\V_2'\z_k\z_k'\e_k\Z_{0k}\V_1,\\
	\Xi_{k120}&=\frac{\tau_n}{\rho_n^2\eta_k\bar\eta_k n^5}\V_1'\z_k\z_k'\V_2\V_2'\z_k\z_k'\p_k\z_k\z_k'\m_k\Z_{0k}\V_1,\\
	\Xi_{k121}&=\frac{\tau_n\epsilon_k}{\rho_n^2\eta_k\bar\eta_k^2 n^5}\V_1'\z_k\z_k'\V_2\V_2'\z_k\z_k'\p_k\z_k\z_k'\m_k\Z_{0k}\V_1,\\
	\Xi_{k122}&=\frac{\tau_n\epsilon_k^2}{\rho_n^2\eta_k^2\bar\eta_k^2 n^5}\V_1'\z_k\z_k'\V_2\V_2'\z_k\z_k'\p_k\z_k\z_k'\m_k\Z_{0k}\V_1,\\
	\Xi_{k130}&=\frac{\tau_n}{\rho_n^2\eta_k\bar\eta_k n^4}\V_1'\z_k\z_k'\q_k\z_k\z_k'\m_k\Z_{0k}\V_1,\\
	\Xi_{k131}&=\frac{\tau_n\epsilon_k}{\rho_n^2\eta_k\bar\eta_k^2 n^4}\V_1'\z_k\z_k'\q_k\z_k\z_k'\m_k\Z_{0k}\V_1,\\
	\Xi_{k132}&=\frac{\tau_n\epsilon_k^2}{\rho_n^2\eta_k^2\bar\eta_k^2 n^4}\V_1'\z_k\z_k'\q_k\z_k\z_k'\m_k\Z_{0k}\V_1,~\Xi_{k140}=\frac{\tau_n\theta_k^2}{\rho_n^2\eta_k\bar\eta_k n^2}\V_1'\z_k\z_k'\e_k\Z_{0k}\V_1,\\
	\Xi_{k141}&=\frac{\tau_n\epsilon_k\theta_k^2}{\rho_n^2\eta_k\bar\eta_k^2 n^2}\V_1'\z_k\z_k'\e_k\Z_{0k}\V_1,~\Xi_{k142}=\frac{\tau_n\epsilon_k^2\theta_k^2}{\rho_n^2\eta_k^2\bar\eta_k^2 n^2}\V_1'\z_k\z_k'\e_k\Z_{0k}\V_1,\\
	\Xi_{k150}&=\frac{\tau_n}{\rho_n^2\eta_k\bar\eta_k n^6}\V_1'\z_k\z_k'\q_k\z_k\z_k'\p_k\z_k\z_k'\m_k\Z_{0k}\V_1\\
	\Xi_{k151}&=\frac{\tau_n\epsilon_k}{\rho_n^2\eta_k\bar\eta_k^2 n^6}\V_1'\z_k\z_k'\q_k\z_k\z_k'\p_k\z_k\z_k'\m_k\Z_{0k}\V_1,\\
	\Xi_{k152}&=\frac{\tau_n\epsilon_k^2}{\rho_n^2\eta_k^2\bar\eta_k^2 n^6}\V_1'\z_k\z_k'\q_k\z_k\z_k'\p_k\z_k\z_k'\m_k\Z_{0k}\V_1,\\
	\Xi_{k160}&=\frac{\tau_n\theta_k}{\rho_n^2\eta_k\bar\eta_k n^4}\V_1'\z_k\z_k'\p_k\z_k\z_k'\m_k\Z_{0k}\V_1,\\
	\Xi_{k161}&=\frac{\tau_n\epsilon_k\theta_k}{\rho_n^2\eta_k\bar\eta_k^2 n^4}\V_1'\z_k\z_k'\p_k\z_k\z_k'\m_k\Z_{0k}\V_1,\\
	\Xi_{k162}&=\frac{\tau_n\epsilon_k^2\theta_k}{\rho_n^2\eta_k^2\bar\eta_k^2 n^4}\V_1'\z_k\z_k'\p_k\z_k\z_k'\m_k\Z_{0k}\V_1,\\
	\Xi_{k170}&=\frac{\tau_n\theta_k}{\rho_n^2\eta_k\bar\eta_k n^4}\V_1'\z_k\z_k'\q_k\z_k\z_k'\e_k\Z_{0k}\V_1,\\
	\Xi_{k171}&=\frac{\tau_n\epsilon_k\theta_k}{\rho_n^2\eta_k\bar\eta_k^2 n^4}\V_1'\z_k\z_k'\q_k\z_k\z_k'\e_k\Z_{0k}\V_1,\\
	\Xi_{k172}&=\frac{\tau_n\epsilon_k^2\theta_k}{\rho_n^2\eta_k^2\bar\eta_k^2 n^4}\V_1'\z_k\z_k'\q_k\z_k\z_k'\e_k\Z_{0k}\V_1,\\
	\Xi_{k180}&=\frac{\tau_n}{\rho_n^2\eta_k\bar\eta_k n^3}\V_1'\Z_{0k}'\Y_{0k}\V_2\V_2'\z_k\z_k'\m_k\Z_{0k}\V_1,\\
	\Xi_{k181}&=\frac{\tau_n\epsilon_k}{\rho_n^2\eta_k\bar\eta_k^2 n^3}\V_1'\Z_{0k}'\Y_{0k}\V_2\V_2'\z_k\z_k'\m_k\Z_{0k}\V_1,\\
	\Xi_{k182}&=\frac{\tau_n\epsilon_k^2}{\rho_n^2\eta_k^2\bar\eta_k^2 n^3}\V_1'\Z_{0k}'\Y_{0k}\V_2\V_2'\z_k\z_k'\m_k\Z_{0k}\V_1,\\
	\Xi_{k190}&=\frac{\tau_n\theta_k}{\rho_n^2\eta_k\bar\eta_k n^3}\V_1'\Z_{0k}'\Y_{0k}\V_2\V_2'\z_k\z_k'\e_k\Z_{0k}\V_1,\\
	\Xi_{k191}&=\frac{\tau_n\epsilon_k\theta_k}{\rho_n^2\eta_k\bar\eta_k^2 n^3}\V_1'\Z_{0k}'\Y_{0k}\V_2\V_2'\z_k\z_k'\e_k\Z_{0k}\V_1,\\
	\Xi_{k192}&=\frac{\tau_n\epsilon_k^2\theta_k}{\rho_n^2\eta_k^2\bar\eta_k^2 n^3}\V_1'\Z_{0k}'\Y_{0k}\V_2\V_2'\z_k\z_k'\e_k\Z_{0k}\V_1,\\
	\Xi_{k200}&=\frac{\tau_n}{\rho_n^2\eta_k\bar\eta_k n^5}\V_1'\Z_{0k}'\Y_{0k}\V_2\V_2'\z_k\z_k'\p_k\z_k\z_k'\m_k\Z_{0k}\V_1,\\
	\Xi_{k201}&=\frac{\tau_n\epsilon_k}{\rho_n^2\eta_k\bar\eta_k^2 n^5}\V_1'\Z_{0k}'\Y_{0k}\V_2\V_2'\z_k\z_k'\p_k\z_k\z_k'\m_k\Z_{0k}\V_1,\\
	\Xi_{k202}&=\frac{\tau_n\epsilon_k^2}{\rho_n^2\eta_k^2\bar\eta_k^2 n^5}\V_1'\Z_{0k}'\Y_{0k}\V_2\V_2'\z_k\z_k'\p_k\z_k\z_k'\m_k\Z_{0k}\V_1,\\
	\Xi_{k210}&=\frac{\tau_n}{\rho_n^2\eta_k\bar\eta_k n^4}\V_1'\Z_{0k}'\Y_{0k}\q_k\z_k\z_k'\m_k\Z_{0k}\V_1,\\
	\Xi_{k211}&=\frac{\tau_n\epsilon_k}{\rho_n^2\eta_k\bar\eta_k^2 n^4}\V_1'\Z_{0k}'\Y_{0k}\q_k\z_k\z_k'\m_k\Z_{0k}\V_1,\\
	\Xi_{k212}&=\frac{\tau_n\epsilon_k^2}{\rho_n^2\eta_k^2\bar\eta_k^2 n^4}\V_1'\Z_{0k}'\Y_{0k}\q_k\z_k\z_k'\m_k\Z_{0k}\V_1,\\
	\Xi_{k220}&=\frac{\tau_n}{\rho_n^2\eta_k\bar\eta_k n^6}\V_1'\Z_{0k}'\Y_{0k}\q_k\z_k\z_k'\p_k\z_k\z_k'\m_k\Z_{0k}\V_1,\\
	\Xi_{k221}&=\frac{\tau_n\epsilon_k}{\rho_n^2\eta_k\bar\eta_k^2 n^6}\V_1'\Z_{0k}'\Y_{0k}\q_k\z_k\z_k'\p_k\z_k\z_k'\m_k\Z_{0k}\V_1,\\
	\Xi_{k222}&=\frac{\tau_n\epsilon_k^2}{\rho_n^2\eta_k^2\bar\eta_k^2 n^6}\V_1'\Z_{0k}'\Y_{0k}\q_k\z_k\z_k'\p_k\z_k\z_k'\m_k\Z_{0k}\V_1,\\
	\Xi_{k230}&=\frac{\tau_n\theta_k}{\rho_n^2\eta_k\bar\eta_k n^4}\V_1'\Z_{0k}'\Y_{0k}\p_k\z_k\z_k'\m_k\Z_{0k}\V_1,\\
	\Xi_{k231}&=\frac{\tau_n\epsilon_k\theta_k}{\rho_n^2\eta_k\bar\eta_k^2 n^4}\V_1'\Z_{0k}'\Y_{0k}\p_k\z_k\z_k'\m_k\Z_{0k}\V_1,\\
	\Xi_{k232}&=\frac{\tau_n\epsilon_k^2\theta_k}{\rho_n^2\eta_k^2\bar\eta_k^2 n^4}\V_1'\Z_{0k}'\Y_{0k}\p_k\z_k\z_k'\m_k\Z_{0k}\V_1,\\
	\Xi_{k240}&=\frac{\tau_n\theta_k}{\rho_n^2\eta_k\bar\eta_k n^4}\V_1'\Z_{0k}'\Y_{0k}\q_k\z_k\z_k'\e_k\Z_{0k}\V_1,\\
	\Xi_{k241}&=\frac{\tau_n\epsilon_k\theta_k}{\rho_n^2\eta_k\bar\eta_k^2 n^4}\V_1'\Z_{0k}'\Y_{0k}\q_k\z_k\z_k'\e_k\Z_{0k}\V_1,\\
	\Xi_{k242}&=\frac{\tau_n\epsilon_k^2\theta_k}{\rho_n^2\eta_k^2\bar\eta_k^2 n^4}\V_1'\Z_{0k}'\Y_{0k}\q_k\z_k\z_k'\e_k\Z_{0k}\V_1,\\
	\Xi_{k250}&=\frac{\tau_n}{\rho_n^2\eta_k\bar\eta_k n^2}\V_1'\Z_{0k}'\Y_{0k}\V_2\V_2'\z_k\z_k'\V_1,\\
	\Xi_{k251}&=\frac{\tau_n\epsilon_k}{\rho_n^2\eta_k\bar\eta_k^2 n^2}\V_1'\Z_{0k}'\Y_{0k}\V_2\V_2'\z_k\z_k'\V_1,\\
	\Xi_{k252}&=\frac{\tau_n\epsilon_k^2}{\rho_n^2\eta_k^2\bar\eta_k^2 n^2}\V_1'\Z_{0k}'\Y_{0k}\V_2\V_2'\z_k\z_k'\V_1.
\end{align*}
Using the same method, we can also obtain the following result:
\[\Omega_{s_{22}}( \widetilde{\Z}_1,\Y_{k})-\Omega_{s_{22}}( \Z_{k0}\V_1,\Y_{k0})=-\sum_{i=1}^{25}\sum_{j=0}^2\Xi_{kijy},\]
where 
\begin{align*}
	\Xi_{k10y}&=\frac{\tau_n}{\rho_n^2\eta_{ky}\bar\eta_{k} n^2}\V_1'\z_k\y_k'\V_2\V_2'\y_k\z_k'\V_1,
	\Xi_{k11y}=\frac{\tau_n\epsilon_{ky}}{\rho_n^2\eta_{ky}\bar\eta_{k}^2 n^2}\V_1'\z_k\y_k'\V_2\V_2'\y_k\z_k'\V_1,\\
	\Xi_{k12y}&=\frac{\tau_n\epsilon_{ky}^2}{\rho_n^2\eta_{ky}^2\bar\eta_{k}^2 n^2}\V_1'\z_k\y_k'\V_2\V_2'\y_k\z_k'\V_1,\\
	\Xi_{k20y}&=\frac{\tau_n}{\rho_n^2\eta_{ky}\bar\eta_{k} n^4}\V_1'\z_k\y_k'\V_2\V_2'\y_k\y_k'\p_k\y_k\z_k'\V_1,\\
	\Xi_{k21y}&=\frac{\tau_n\epsilon_{ky}}{\rho_n^2\eta_{ky}\bar\eta_{k}^2 n^4}\V_1'\z_k\y_k'\V_2\V_2'\y_k\y_k'\p_k\y_k\z_k'\V_1,\\
	\Xi_{k22y}&=\frac{\tau_n\epsilon_{ky}^2}{\rho_n^2\eta_{ky}^2\bar\eta_{k}^2 n^4}\V_1'\z_k\y_k'\V_2\V_2'\y_k\y_k'\p_k\y_k\z_k'\V_1,\\
	\Xi_{k30y}&=\frac{\tau_n}{\rho_n^2\eta_{ky}\bar\eta_{k} n^3}\V_1'\z_k\y_k'\q_k\y_k\z_k'\V_1,\\
	\Xi_{k31y}&=\frac{\tau_n\epsilon_{ky}}{\rho_n^2\eta_{ky}\bar\eta_{k}^2 n^3}\V_1'\z_k\y_k'\q_k\y_k\z_k'\V_1,~\Xi_{k32y}=\frac{\tau_n\epsilon_{ky}^2}{\rho_n^2\eta_{ky}^2\bar\eta_{k}^2 n^3}\V_1'\z_k\y_k'\q_k\y_k\z_k'\V_1,\\
	\Xi_{k40y}&=\frac{\tau_n\theta_{ky}}{\rho_n^2\eta_{ky}\bar\eta_{k} n^3}\V_1'\z_k\y_k'\p_k\y_k\z_k'\V_1,~\Xi_{k41y}=\frac{\tau_n\epsilon_{ky}\theta_{ky}}{\rho_n^2\eta_{ky}\bar\eta_{k}^2 n^3}\V_1'\z_k\y_k'\p_k\y_k\z_k'\V_1,\\
	\Xi_{k42y}&=\frac{\tau_n\epsilon_{ky}^2\theta_{ky}}{\rho_n^2\eta_{ky}^2\bar\eta_{k}^2 n^3}\V_1'\z_k\y_k'\p_k\y_k\z_k'\V_1,\\
	\Xi_{k50y}&=\frac{\tau_n}{\rho_n^2\eta_{ky}\bar\eta_{k} n^5}\V_1'\z_k\y_k'\q_k\y_k\y_k'\p_k\y_k\z_k'\V_1,\\
	\Xi_{k51y}&=\frac{\tau_n\epsilon_{ky}}{\rho_n^2\eta_{ky}\bar\eta_{k}^2 n^5}\V_1'\z_k\y_k'\q_k\y_k\y_k'\p_k\y_k\z_k'\V_1,\\
	\Xi_{k52y}&=\frac{\tau_n\epsilon_{ky}^2}{\rho_n^2\eta_{ky}^2\bar\eta_{k}^2 n^5}\V_1'\z_k\y_k'\q_k\y_k\y_k'\p_k\y_k\z_k'\V_1,\\
	\Xi_{k60y}&=\frac{\tau_n}{\rho_n^2\eta_{ky}\bar\eta_{k} n^4}\V_1'\Z_{0k}'\Y_{0k}\V_2\V_2'\y_k\y_k'\p_k\y_k\z_k'\V_1,\\
	\Xi_{k61y}&=\frac{\tau_n\epsilon_{ky}}{\rho_n^2\eta_{ky}\bar\eta_{k}^2 n^4}\V_1'\Z_{0k}'\Y_{0k}\V_2\V_2'\y_k\y_k'\p_k\y_k\z_k'\V_1,\\
	\Xi_{k62y}&=\frac{\tau_n\epsilon_{ky}^2}{\rho_n^2\eta_{ky}^2\bar\eta_{k}^2 n^4}\V_1'\Z_{0k}'\Y_{0k}\V_2\V_2'\y_k\y_k'\p_k\y_k\z_k'\V_1,\\
	\Xi_{k70y}&=\frac{\tau_n}{\rho_n^2\eta_{ky}\bar\eta_{k} n^3}\V_1'\Z_{0k}'\Y_{0k}\q_k\y_k\z_k'\V_1,\\
	\Xi_{k71y}&=\frac{\tau_n\epsilon_{ky}}{\rho_n^2\eta_{ky}\bar\eta_{k}^2 n^3}\V_1'\Z_{0k}'\Y_{0k}\q_k\y_k\z_k'\V_1,~\Xi_{k72y}=\frac{\tau_n\epsilon_{ky}^2}{\rho_n^2\eta_{ky}^2\bar\eta_{k}^2 n^3}\V_1'\Z_{0k}'\Y_{0k}\q_k\y_k\z_k'\V_1,\\
	\Xi_{k80y}&=\frac{\tau_n\theta_{ky}}{\rho_n^2\eta_{ky}\bar\eta_{k} n^3}\V_1'\Z_{0k}'\Y_{0k}\p_k\y_k\z_k'\V_1,~\Xi_{k81y}=\frac{\tau_n\epsilon_{ky}\theta_{ky}}{\rho_n^2\eta_{ky}\bar\eta_{k}^2 n^3}\V_1'\Z_{0k}'\Y_{0k}\p_k\y_k\z_k'\V_1,\\
	\Xi_{k82y}&=\frac{\tau_n\epsilon_{ky}^2\theta_{ky}}{\rho_n^2\eta_{ky}^2\bar\eta_{k}^2 n^3}\V_1'\Z_{0k}'\Y_{0k}\p_k\y_k\z_k'\V_1,\\
	\Xi_{k90y}&=\frac{\tau_n}{\rho_n^2\eta_{ky}\bar\eta_{k} n^5}\V_1'\Z_{0k}'\Y_{0k}\q_k\y_k\y_k'\p_k\y_k\z_k'\V_1,\\
	\Xi_{k91y}&=\frac{\tau_n\epsilon_{ky}}{\rho_n^2\eta_{ky}\bar\eta_{k}^2 n^5}\V_1'\Z_{0k}'\Y_{0k}\q_k\y_k\y_k'\p_k\y_k\z_k'\V_1,\\
	\Xi_{k92y}&=\frac{\tau_n\epsilon_{ky}^2}{\rho_n^2\eta_{ky}^2\bar\eta_{k}^2 n^5}\V_1'\Z_{0k}'\Y_{0k}\q_k\y_k\y_k'\p_k\y_k\z_k'\V_1,\\
	\Xi_{k100y}&=\frac{\tau_n}{\rho_n^2\eta_{ky}\bar\eta_{k} n^3}\V_1'\z_k\y_k'\V_2\V_2'\y_k\y_k'\m_k\Z_{0k}\V_1,\\
	\Xi_{k101y}&=\frac{\tau_n\epsilon_{ky}}{\rho_n^2\eta_{ky}\bar\eta_{k}^2 n^3}\V_1'\z_k\y_k'\V_2\V_2'\y_k\y_k'\m_k\Z_{0k}\V_1,\\
	\Xi_{k102y}&=\frac{\tau_n\epsilon_{ky}^2}{\rho_n^2\eta_{ky}^2\bar\eta_{k}^2 n^3}\V_1'\z_k\y_k'\V_2\V_2'\y_k\y_k'\m_k\Z_{0k}\V_1,\\
	\Xi_{k110y}&=\frac{\tau_n\theta_{ky}}{\rho_n^2\eta_{ky}\bar\eta_{k} n^3}\V_1'\z_k\y_k'\V_2\V_2'\y_k\y_k'\e_k\Z_{0k}\V_1,\\
	\Xi_{k111y}&=\frac{\tau_n\epsilon_{ky}\theta_{ky}}{\rho_n^2\eta_{ky}\bar\eta_{k}^2 n^3}\V_1'\z_k\y_k'\V_2\V_2'\y_k\y_k'\e_k\Z_{0k}\V_1,\\
	\Xi_{k112y}&=\frac{\tau_n\epsilon_{ky}^2\theta_{ky}}{\rho_n^2\eta_{ky}^2\bar\eta_{k}^2 n^3}\V_1'\z_k\y_k'\V_2\V_2'\y_k\y_k'\e_k\Z_{0k}\V_1,\\
	\Xi_{k120y}&=\frac{\tau_n}{\rho_n^2\eta_{ky}\bar\eta_{k} n^5}\V_1'\z_k\y_k'\V_2\V_2'\y_k\y_k'\p_k\y_k\y_k'\m_k\Z_{0k}\V_1,\\
	\Xi_{k121y}&=\frac{\tau_n\epsilon_{ky}}{\rho_n^2\eta_{ky}\bar\eta_{k}^2 n^5}\V_1'\z_k\y_k'\V_2\V_2'\y_k\y_k'\p_k\y_k\y_k'\m_k\Z_{0k}\V_1,\\
	\Xi_{k122y}&=\frac{\tau_n\epsilon_{ky}^2}{\rho_n^2\eta_{ky}^2\bar\eta_{k}^2 n^5}\V_1'\z_k\y_k'\V_2\V_2'\y_k\y_k'\p_k\y_k\y_k'\m_k\Z_{0k}\V_1,\\
	\Xi_{k130y}&=\frac{\tau_n}{\rho_n^2\eta_{ky}\bar\eta_{k} n^4}\V_1'\z_k\y_k'\q_k\y_k\y_k'\m_k\Z_{0k}\V_1,\\
	\Xi_{k131y}&=\frac{\tau_n\epsilon_{ky}}{\rho_n^2\eta_{ky}\bar\eta_{k}^2 n^4}\V_1'\z_k\y_k'\q_k\y_k\y_k'\m_k\Z_{0k}\V_1,\\
	\Xi_{k132y}&=\frac{\tau_n\epsilon_{ky}^2}{\rho_n^2\eta_{ky}^2\bar\eta_{k}^2 n^4}\V_1'\z_k\y_k'\q_k\y_k\y_k'\m_k\Z_{0k}\V_1,\\
	\Xi_{k140y}&=\frac{\tau_n\theta_{ky}^2}{\rho_n^2\eta_{ky}\bar\eta_{k} n^2}\V_1'\z_k\y_k'\e_k\Z_{0k}\V_1,\\
	\Xi_{k141y}&=\frac{\tau_n\epsilon_{ky}\theta_{ky}^2}{\rho_n^2\eta_{ky}\bar\eta_{k}^2 n^2}\V_1'\z_k\y_k'\e_k\Z_{0k}\V_1,~\Xi_{k142y}=\frac{\tau_n\epsilon_{ky}^2\theta_{ky}^2}{\rho_n^2\eta_{ky}^2\bar\eta_{k}^2 n^2}\V_1'\z_k\y_k'\e_k\Z_{0k}\V_1,\\
	\Xi_{k150y}&=\frac{\tau_n}{\rho_n^2\eta_{ky}\bar\eta_{k} n^6}\V_1'\z_k\y_k'\q_k\y_k\y_k'\p_k\y_k\y_k'\m_k\Z_{0k}\V_1\\
	\Xi_{k151y}&=\frac{\tau_n\epsilon_{ky}}{\rho_n^2\eta_{ky}\bar\eta_{k}^2 n^6}\V_1'\z_k\y_k'\q_k\y_k\y_k'\p_k\y_k\y_k'\m_k\Z_{0k}\V_1,\\
	\Xi_{k152y}&=\frac{\tau_n\epsilon_{ky}^2}{\rho_n^2\eta_{ky}^2\bar\eta_{k}^2 n^6}\V_1'\z_k\y_k'\q_k\y_k\y_k'\p_k\y_k\y_k'\m_k\Z_{0k}\V_1,\\
	\Xi_{k160y}&=\frac{\tau_n\theta_{ky}}{\rho_n^2\eta_{ky}\bar\eta_{k} n^4}\V_1'\z_k\y_k'\p_k\y_k\y_k'\m_k\Z_{0k}\V_1,\\
	\Xi_{k161y}&=\frac{\tau_n\epsilon_{ky}\theta_{ky}}{\rho_n^2\eta_{ky}\bar\eta_{k}^2 n^4}\V_1'\z_k\y_k'\p_k\y_k\y_k'\m_k\Z_{0k}\V_1,\\
	\Xi_{k162y}&=\frac{\tau_n\epsilon_{ky}^2\theta_{ky}}{\rho_n^2\eta_{ky}^2\bar\eta_{k}^2 n^4}\V_1'\z_k\y_k'\p_k\y_k\y_k'\m_k\Z_{0k}\V_1,\\
	\Xi_{k170y}&=\frac{\tau_n\theta_{ky}}{\rho_n^2\eta_{ky}\bar\eta_{k} n^4}\V_1'\z_k\y_k'\q_k\y_k\y_k'\e_k\Z_{0k}\V_1,\\
	\Xi_{k171y}&=\frac{\tau_n\epsilon_{ky}\theta_{ky}}{\rho_n^2\eta_{ky}\bar\eta_{k}^2 n^4}\V_1'\z_k\y_k'\q_k\y_k\y_k'\e_k\Z_{0k}\V_1,\\
	\Xi_{k172y}&=\frac{\tau_n\epsilon_{ky}^2\theta_{ky}}{\rho_n^2\eta_{ky}^2\bar\eta_{k}^2 n^4}\V_1'\z_k\y_k'\q_k\y_k\y_k'\e_k\Z_{0k}\V_1,\\
	\Xi_{k180y}&=\frac{\tau_n}{\rho_n^2\eta_{ky}\bar\eta_{k} n^3}\V_1'\Z_{0k}'\Y_{0k}\V_2\V_2'\y_k\y_k'\m_k\Z_{0k}\V_1,\\
	\Xi_{k181y}&=\frac{\tau_n\epsilon_{ky}}{\rho_n^2\eta_{ky}\bar\eta_{k}^2 n^3}\V_1'\Z_{0k}'\Y_{0k}\V_2\V_2'\y_k\y_k'\m_k\Z_{0k}\V_1,\\
	\Xi_{k182y}&=\frac{\tau_n\epsilon_{ky}^2}{\rho_n^2\eta_{ky}^2\bar\eta_{k}^2 n^3}\V_1'\Z_{0k}'\Y_{0k}\V_2\V_2'\y_k\y_k'\m_k\Z_{0k}\V_1,\\
	\Xi_{k190y}&=\frac{\tau_n\theta_{ky}}{\rho_n^2\eta_{ky}\bar\eta_{k} n^3}\V_1'\Z_{0k}'\Y_{0k}\V_2\V_2'\y_k\y_k'\e_k\Z_{0k}\V_1,\\
	\Xi_{k191y}&=\frac{\tau_n\epsilon_{ky}\theta_{ky}}{\rho_n^2\eta_{ky}\bar\eta_{k}^2 n^3}\V_1'\Z_{0k}'\Y_{0k}\V_2\V_2'\y_k\y_k'\e_k\Z_{0k}\V_1,\\
	\Xi_{k192y}&=\frac{\tau_n\epsilon_{ky}^2\theta_{ky}}{\rho_n^2\eta_{ky}^2\bar\eta_{k}^2 n^3}\V_1'\Z_{0k}'\Y_{0k}\V_2\V_2'\y_k\y_k'\e_k\Z_{0k}\V_1,\\
	\Xi_{k200y}&=\frac{\tau_n}{\rho_n^2\eta_{ky}\bar\eta_{k} n^5}\V_1'\Z_{0k}'\Y_{0k}\V_2\V_2'\y_k\y_k'\p_k\y_k\y_k'\m_k\Z_{0k}\V_1,\\
	\Xi_{k201y}&=\frac{\tau_n\epsilon_{ky}}{\rho_n^2\eta_{ky}\bar\eta_{k}^2 n^5}\V_1'\Z_{0k}'\Y_{0k}\V_2\V_2'\y_k\y_k'\p_k\y_k\y_k'\m_k\Z_{0k}\V_1,\\
	\Xi_{k202y}&=\frac{\tau_n\epsilon_{ky}^2}{\rho_n^2\eta_{ky}^2\bar\eta_{k}^2 n^5}\V_1'\Z_{0k}'\Y_{0k}\V_2\V_2'\y_k\y_k'\p_k\y_k\y_k'\m_k\Z_{0k}\V_1,\\
	\Xi_{k210y}&=\frac{\tau_n}{\rho_n^2\eta_{ky}\bar\eta_{k} n^4}\V_1'\Z_{0k}'\Y_{0k}\q_k\y_k\y_k'\m_k\Z_{0k}\V_1,\\
	\Xi_{k211y}&=\frac{\tau_n\epsilon_{ky}}{\rho_n^2\eta_{ky}\bar\eta_{k}^2 n^4}\V_1'\Z_{0k}'\Y_{0k}\q_k\y_k\y_k'\m_k\Z_{0k}\V_1,\\
	\Xi_{k212y}&=\frac{\tau_n\epsilon_{ky}^2}{\rho_n^2\eta_{ky}^2\bar\eta_{k}^2 n^4}\V_1'\Z_{0k}'\Y_{0k}\q_k\y_k\y_k'\m_k\Z_{0k}\V_1,\\
	\Xi_{k220y}&=\frac{\tau_n}{\rho_n^2\eta_{ky}\bar\eta_{k} n^6}\V_1'\Z_{0k}'\Y_{0k}\q_k\y_k\y_k'\p_k\y_k\y_k'\m_k\Z_{0k}\V_1,\\
	\Xi_{k221y}&=\frac{\tau_n\epsilon_{ky}}{\rho_n^2\eta_{ky}\bar\eta_{k}^2 n^6}\V_1'\Z_{0k}'\Y_{0k}\q_k\y_k\y_k'\p_k\y_k\y_k'\m_k\Z_{0k}\V_1,\\
	\Xi_{k222y}&=\frac{\tau_n\epsilon_{ky}^2}{\rho_n^2\eta_{ky}^2\bar\eta_{k}^2 n^6}\V_1'\Z_{0k}'\Y_{0k}\q_k\y_k\y_k'\p_k\y_k\y_k'\m_k\Z_{0k}\V_1,\\
	\Xi_{k230y}&=\frac{\tau_n\theta_{ky}}{\rho_n^2\eta_{ky}\bar\eta_{k} n^4}\V_1'\Z_{0k}'\Y_{0k}\p_k\y_k\y_k'\m_k\Z_{0k}\V_1,\\
	\Xi_{k231y}&=\frac{\tau_n\epsilon_{ky}\theta_{ky}}{\rho_n^2\eta_{ky}\bar\eta_{k}^2 n^4}\V_1'\Z_{0k}'\Y_{0k}\p_k\y_k\y_k'\m_k\Z_{0k}\V_1,\\
	\Xi_{k232y}&=\frac{\tau_n\epsilon_{ky}^2\theta_{ky}}{\rho_n^2\eta_{ky}^2\bar\eta_{k}^2 n^4}\V_1'\Z_{0k}'\Y_{0k}\p_k\y_k\y_k'\m_k\Z_{0k}\V_1,\\
	\Xi_{k240y}&=\frac{\tau_n\theta_{ky}}{\rho_n^2\eta_{ky}\bar\eta_{k} n^4}\V_1'\Z_{0k}'\Y_{0k}\q_k\y_k\y_k'\e_k\Z_{0k}\V_1,\\
	\Xi_{k241y}&=\frac{\tau_n\epsilon_{ky}\theta_{ky}}{\rho_n^2\eta_{ky}\bar\eta_{k}^2 n^4}\V_1'\Z_{0k}'\Y_{0k}\q_k\y_k\y_k'\e_k\Z_{0k}\V_1,\\
	\Xi_{k242y}&=\frac{\tau_n\epsilon_{ky}^2\theta_{ky}}{\rho_n^2\eta_{ky}^2\bar\eta_{k}^2 n^4}\V_1'\Z_{0k}'\Y_{0k}\q_k\y_k\y_k'\e_k\Z_{0k}\V_1,\\
	\Xi_{k250y}&=\frac{\tau_n}{\rho_n^2\eta_{ky}\bar\eta_{k} n^2}\V_1'\Z_{0k}'\Y_{0k}\V_2\V_2'\y_k\z_k'\V_1,\\
	\Xi_{k251y}&=\frac{\tau_n\epsilon_{ky}}{\rho_n^2\eta_{ky}\bar\eta_{k}^2 n^2}\V_1'\Z_{0k}'\Y_{0k}\V_2\V_2'\y_k\z_k'\V_1,\\
	\Xi_{k252y}&=\frac{\tau_n\epsilon_{ky}^2}{\rho_n^2\eta_{ky}^2\bar\eta_{k}^2 n^2}\V_1'\Z_{0k}'\Y_{0k}\V_2\V_2'\y_k\z_k'\V_1.
\end{align*}

\section{Additional numerical results}\label{app:simulations}
This section presents supplementary numerical results, including:  
(i) empirical density histograms of the spiked eigenvalues, and  
(ii) comprehensive tables quantifying spike number estimation accuracy.  
These analyses were deferred from Section~\ref{simulation} due to space limitations.

\begin{table}[H]
	\centering
	\caption{Estimation accuracy under $U(-\sqrt{3},\sqrt{3})$ of diagonal covariance matrix (All values in percentages)}
	\label{sphericcase2}
	\begin{tabular}{@{}c*{7}{@{\hspace{0.5cm}}c}@{}}
		\hline
		$(q,n)$ &J\&B-E&Y\&J-E&J\&B&Y\&J & P\&Y& BCF(AIC)&K\&N\\[6pt]
		(5,50) & \textbf{100.00} & \textbf{100.00} & 61.30 & \textbf{100.00} & 53.60 & 91.20 & \textbf{100.00} \\
		(10,100) & \textbf{100.00} & \textbf{100.00} & 92.80 & \textbf{100.00} & 62.60 & 90.80 & \textbf{100.00} \\
		(20,200) & \textbf{100.00} & \textbf{100.00} & 96.90 & \textbf{100.00} & 76.80 & 94.60 & \textbf{100.00} \\
		(40,400) & \textbf{100.00} & \textbf{100.00} & 99.40 & \textbf{100.00} & 88.80 & 96.60 & \textbf{100.00} \\
		(10,50) & \textbf{99.90} & \textbf{99.90} & 64.70 & \textbf{99.90} & 50.50 & 91.10 & \textbf{99.90} \\
		(20,100) & \textbf{99.90} & \textbf{99.90} & 92.00 & \textbf{99.90} & 58.00 & 95.70 & \textbf{99.90} \\
		(40,200) & \textbf{99.90} & \textbf{99.90} & 97.60 & \textbf{99.90} & 72.60 & 97.20 & \textbf{99.90} \\
		(80,400) & \textbf{100.00} & \textbf{100.00} & 99.70 & \textbf{100.00} & 86.30 & 99.20 & 99.80 \\
		(15,50) & \textbf{99.80} & 99.40 & 62.90 & 99.40 & 42.30 & 93.70 & 99.30 \\
		(30,100) & \textbf{99.90} & \textbf{99.90} & 92.70 & \textbf{99.90} & 60.00 & 96.50 & \textbf{99.90} \\
		(60,200) & 99.80 & \textbf{100.00} & 98.00 & \textbf{100.00} & 73.10 & 99.20 & 99.80 \\
		(120,400) & \textbf{100.00} & 99.90 & 99.90 & 99.90 & 85.90 & 99.90 & 99.70 \\
		\hline
	\end{tabular}
\end{table}

\begin{table}
	\centering
	\caption{Estimation accuracy under $N(0,1)$ of diagonal
		covariance matrix (All values in percentages)}
	\label{sphericcase3}
	\begin{tabular}{@{}c*{7}{@{\hspace{0.5cm}}c}@{}}
		\hline
		$(q,n)$ &J\&B-E&Y\&J-E&J\&B&Y\&J & P\&Y& BCF(AIC)&K\&N\\[6pt]
		(5,50) & \textbf{99.90} & 99.80 & 74.60 & 99.80 & 56.60 & 83.70 & 99.70 \\
		(10,100) & 99.50 & 99.60 & 93.30 & 99.60 & 66.00 & 84.80 & \textbf{99.80} \\
		(20,200) & 99.80 & \textbf{99.90} & 95.70 & \textbf{99.90} & 74.30 & 89.90 & \textbf{99.90} \\
		(40,400) & \textbf{99.90} & 99.30 & 97.30 & 99.30 & 83.80 & 92.30 & 99.00 \\
		(10,50) & 99.20 & \textbf{99.30} & 74.90 & \textbf{99.30} & 50.10 & 81.10 & 99.20 \\
		(20,100) & 99.70 & \textbf{99.90} & 91.40 & \textbf{99.90} & 63.70 & 90.30 & \textbf{99.90} \\
		(40,200) & \textbf{99.60} & \textbf{99.60} & 95.20 & \textbf{99.60} & 72.70 & 95.50 & \textbf{99.60} \\
		(80,400) & \textbf{100.00} & 99.80 & 99.60 & 99.80 & 84.00 & 99.10 & 99.30 \\
		(15,50) & \textbf{98.50} & 98.00 & 77.30 & 98.00 & 42.20 & 86.50 & 97.90 \\
		(30,100) & 99.50 & \textbf{99.80} & 93.60 & \textbf{99.80} & 62.80 & 94.20 & 99.70 \\
		(60,200) & 99.70 & \textbf{99.90} & 95.80 & \textbf{99.90} & 73.50 & 98.40 & 99.60 \\
		(120,400) & \textbf{100.00} & 99.80 & 99.50 & 99.80 & 83.60 & 99.90 & 99.30 \\
		\hline
	\end{tabular}
\end{table}

\begin{table}[H]
	\centering
	\caption{Estimation accuracy under Gamma, Uniform and Gaussian assumptions under general covariance matrix (All values in percentages)}
	\label{sphericcase4-654}
	\begin{tabular}{c c*{7}{@{\hspace{0.30cm}}c}}
		\hline
		Distribution & $(p,n)$ & J\&B-E & Y\&J-E & J\&B & Y\&J & P\&Y & BCF(AIC) & K\&N\\[6pt]
		$2Ga(1,2)-1$
		& (5,50) & 94.30 & \textbf{94.60} & 83.70 & \textbf{94.60} & 51.60 & 69.80 & 93.50 \\
		& (10,100) & \textbf{95.80} & 94.30 & 87.20 & 94.30 & 70.00 & 47.00 & 95.20 \\
		& (20,200) & \textbf{93.90} & 88.20 & 76.30 & 88.20 & 77.80 & 44.00 & 89.10 \\
		& (40,400) & \textbf{97.50} & 86.90 & 70.50 & 86.90 & 84.60 & 55.30 & 85.60 \\
		& (10,50) & 92.20 & 93.50 & 80.00 & 93.50 & 46.40 & 46.70 & \textbf{93.90} \\
		& (20,100) & \textbf{95.00} & 92.30 & 80.00 & 92.30 & 66.50 & 51.50 & 92.30 \\
		& (40,200) & \textbf{94.50} & 87.60 & 70.20 & 87.60 & 74.40 & 59.00 & 85.60 \\
		& (80,400) & \textbf{99.10} & 90.70 & 69.20 & 90.70 & 85.70 & 82.30 & 84.60 \\
		& (15,50) & 89.00 & \textbf{90.50} & 80.10 & \textbf{90.50} & 37.20 & 48.60 & 90.40 \\
		& (30,100) & \textbf{93.70} & 90.20 & 76.00 & 90.20 & 62.50 & 52.80 & 88.90 \\
		& (60,200) & \textbf{96.90} & 91.00 & 70.80 & 91.00 & 74.40 & 74.00 & 86.30 \\
		& (120,400) & \textbf{99.80} & 94.30 & 70.70 & 94.30 & 82.10 & 94.40 & 87.80 \\
		\hline
		$U(-\sqrt{3},\sqrt{3})$
		& (5,50) & \textbf{100.00} & 99.90 & 62.30 & 99.90 & 53.10 & 89.10 & 99.80 \\
		& (10,100) & \textbf{100.00} & \textbf{100.00} & 92.90 & \textbf{100.00} & 63.40 & 88.80 & \textbf{100.00} \\
		& (20,200) & \textbf{99.90} & \textbf{99.90} & 97.10 & \textbf{99.90} & 71.60 & 94.10 & \textbf{99.90} \\
		& (40,400) & \textbf{99.80} & \textbf{99.80} & 98.50 & \textbf{99.80} & 83.40 & 96.00 & 99.70 \\
		& (10,50) & \textbf{99.60} & 99.10 & 68.70 & 99.10 & 52.90 & 88.20 & 99.10 \\
		& (20,100) & \textbf{100.00} & \textbf{100.00} & 92.90 & \textbf{100.00} & 60.50 & 93.40 & \textbf{100.00} \\
		& (40,200) & 99.80 & \textbf{100.00} & 97.80 & \textbf{100.00} & 75.00 & 97.90 & \textbf{100.00} \\
		& (80,400) & \textbf{100.00} & 99.90 & 99.20 & 99.90 & 86.00 & 99.00 & 99.70 \\
		& (15,50) & \textbf{98.90} & 97.80 & 67.40 & 97.80 & 44.00 & 93.50 & 97.80 \\
		& (30,100) & 99.60 & \textbf{99.90} & 92.80 & \textbf{99.90} & 62.30 & 96.50 & 99.80 \\
		& (60,200) & 99.80 & \textbf{100.00} & 97.90 & \textbf{100.00} & 71.70 & 98.90 & 99.80 \\
		& (120,400) & \textbf{100.00} & 99.90 & 99.30 & 99.90 & 83.40 & 99.90 & 99.40 \\
		\hline
		$N(0,1)$
		& (5,50) & \textbf{100.00} & \textbf{100.00} & 72.40 & \textbf{100.00} & 53.90 & 84.00 & 99.70 \\
		& (10,100) & 99.80 & 99.80 & 93.30 & 99.80 & 65.70 & 85.50 & \textbf{99.90} \\
		& (20,200) & 99.30 & \textbf{99.60} & 95.80 & \textbf{99.60} & 75.00 & 90.60 & \textbf{99.60} \\
		& (40,400) & \textbf{99.60} & 99.50 & 98.60 & 99.50 & 85.80 & 92.30 & 99.40 \\
		& (10,50) & \textbf{99.00} & 98.70 & 76.60 & 98.70 & 51.50 & 84.80 & 98.30 \\
		& (20,100) & 99.70 & \textbf{99.90} & 92.80 & \textbf{99.90} & 64.90 & 91.00 & \textbf{99.90} \\
		& (40,200) & 99.50 & \textbf{99.60} & 95.80 & \textbf{99.60} & 74.20 & 94.50 & 99.50 \\
		& (80,400) & \textbf{100.00} & 99.80 & 99.20 & 99.80 & 84.50 & 98.90 & 99.50 \\
		& (15,50) & \textbf{98.70} & 96.90 & 76.60 & 96.90 & 42.20 & 85.70 & 96.80 \\
		& (30,100) & 99.40 & \textbf{99.80} & 91.30 & \textbf{99.80} & 63.20 & 93.80 & 99.70 \\
		& (60,200) & 99.60 & \textbf{99.80} & 96.40 & \textbf{99.80} & 75.10 & 98.30 & 99.60 \\
		& (120,400) & \textbf{100.00} & 99.90 & 99.50 & 99.90 & 85.10 & \textbf{100.00} & 99.00 \\
		\hline
	\end{tabular}
\end{table}

\begin{table}[H]
	\centering
	\caption{Estimation accuracy under Gamma, Uniform and Gaussian assumptions under block covariance matrix (All values in percentages)}
	\label{sphericcase5-654}
	\begin{tabular}{c c*{7}{@{\hspace{0.30cm}}c}}
		\hline
		Distribution & $(p,n)$ & J\&B-E & Y\&J-E & J\&B & Y\&J & P\&Y & BCF(AIC) & K\&N\\[6pt]
		$2Ga(1,2)-1$
		& (5,50) & \textbf{92.30} & 91.00 & 80.10 & 91.00 & 44.80 & 57.90 & 91.10 \\
		& (10,100) & \textbf{88.70} & 86.30 & 76.80 & 86.30 & 67.00 & 39.40 & 88.50 \\
		& (20,200) & \textbf{85.60} & 82.10 & 70.30 & 82.10 & 80.60 & 38.50 & 83.40 \\
		& (40,400) & 88.50 & 82.50 & 66.50 & 82.50 & \textbf{89.40} & 50.60 & 81.60 \\
		& (10,50) & \textbf{85.20} & 83.70 & 77.10 & 83.70 & 32.10 & 36.80 & 83.80 \\
		& (20,100) & \textbf{86.30} & 84.50 & 72.00 & 84.50 & 63.00 & 42.80 & 84.50 \\
		& (40,200) & \textbf{85.80} & 83.40 & 67.00 & 83.40 & 79.80 & 53.40 & 81.90 \\
		& (80,400) & \textbf{95.60} & 89.20 & 65.80 & 89.20 & 87.50 & 80.80 & 83.40 \\
		& (15,50) & \textbf{78.10} & 76.00 & 71.10 & 76.00 & 25.80 & 40.30 & 75.60 \\
		& (30,100) & \textbf{87.00} & 85.30 & 72.90 & 85.30 & 58.30 & 48.10 & 84.10 \\
		& (60,200) & \textbf{89.70} & 87.60 & 67.80 & 87.60 & 79.80 & 72.30 & 83.10 \\
		& (120,400) & \textbf{97.20} & 93.10 & 66.50 & 93.10 & 88.20 & 93.30 & 84.20 \\
		\hline
		$U(-\sqrt{3},\sqrt{3})$
		& (5,50) & \textbf{100.00} & \textbf{100.00} & 58.10 & \textbf{100.00} & 48.60 & 92.80 & \textbf{100.00} \\
		& (10,100) & \textbf{100.00} & \textbf{100.00} & 94.30 & \textbf{100.00} & 60.80 & 93.00 & \textbf{100.00} \\
		& (20,200) & 99.90 & \textbf{100.00} & 97.20 & \textbf{100.00} & 69.70 & 94.60 & \textbf{100.00} \\
		& (40,400) & \textbf{99.90} & \textbf{99.90} & 99.80 & \textbf{99.90} & 83.40 & 95.70 & \textbf{99.90} \\
		& (10,50) & \textbf{100.00} & \textbf{100.00} & 61.80 & \textbf{100.00} & 45.70 & 92.30 & 99.90 \\
		& (20,100) & \textbf{100.00} & \textbf{100.00} & 90.20 & \textbf{100.00} & 55.70 & 95.50 & \textbf{100.00} \\
		& (40,200) & \textbf{100.00} & \textbf{100.00} & 97.40 & \textbf{100.00} & 70.40 & 97.50 & 99.90 \\
		& (80,400) & \textbf{100.00} & \textbf{100.00} & 99.20 & \textbf{100.00} & 81.40 & 99.30 & 99.80 \\
		& (15,50) & \textbf{99.80} & 99.50 & 60.40 & 99.50 & 40.90 & 92.40 & 99.40 \\
		& (30,100) & 99.90 & \textbf{100.00} & 91.00 & \textbf{100.00} & 56.00 & 96.80 & 99.90 \\
		& (60,200) & \textbf{99.90} & \textbf{99.90} & 97.90 & \textbf{99.90} & 67.50 & 99.10 & 99.80 \\
		& (120,400) & \textbf{100.00} & 99.90 & 99.30 & 99.90 & 79.70 & 99.90 & 99.60 \\
		\hline
		$N(0,1)$
		& (5,50) & \textbf{99.90} & \textbf{99.90} & 75.20 & \textbf{99.90} & 59.00 & 86.40 & 99.80 \\
		& (10,100) & 99.40 & 99.60 & 92.20 & 99.60 & 65.90 & 84.10 & \textbf{99.70} \\
		& (20,200) & \textbf{99.70} & \textbf{99.70} & 97.30 & \textbf{99.70} & 78.70 & 87.90 & \textbf{99.70} \\
		& (40,400) & \textbf{99.60} & 99.50 & 98.20 & 99.50 & 85.20 & 91.50 & 99.40 \\
		& (10,50) & \textbf{98.90} & 98.80 & 73.80 & 98.80 & 46.90 & 80.90 & 98.40 \\
		& (20,100) & 99.70 & \textbf{99.80} & 94.40 & \textbf{99.80} & 68.00 & 89.00 & \textbf{99.80} \\
		& (40,200) & 99.80 & \textbf{99.90} & 95.20 & \textbf{99.90} & 75.70 & 94.00 & 99.70 \\
		& (80,400) & \textbf{99.90} & \textbf{99.90} & 99.30 & \textbf{99.90} & 85.20 & 98.60 & 99.30 \\
		& (15,50) & \textbf{97.20} & 95.60 & 74.40 & 95.60 & 40.40 & 83.00 & 95.50 \\
		& (30,100) & 99.30 & \textbf{99.80} & 92.40 & \textbf{99.80} & 61.10 & 92.60 & 99.60 \\
		& (60,200) & 99.70 & \textbf{99.90} & 95.90 & \textbf{99.90} & 72.00 & 97.90 & 99.50 \\
		& (120,400) & \textbf{100.00} & 99.80 & 99.60 & 99.80 & 82.70 & 99.60 & 99.00 \\
		\hline
	\end{tabular}
\end{table}

\begin{figure}[H]
	
	\begin{subfigure}{0.33\textwidth}
		\centering
		\includegraphics[width=\linewidth]{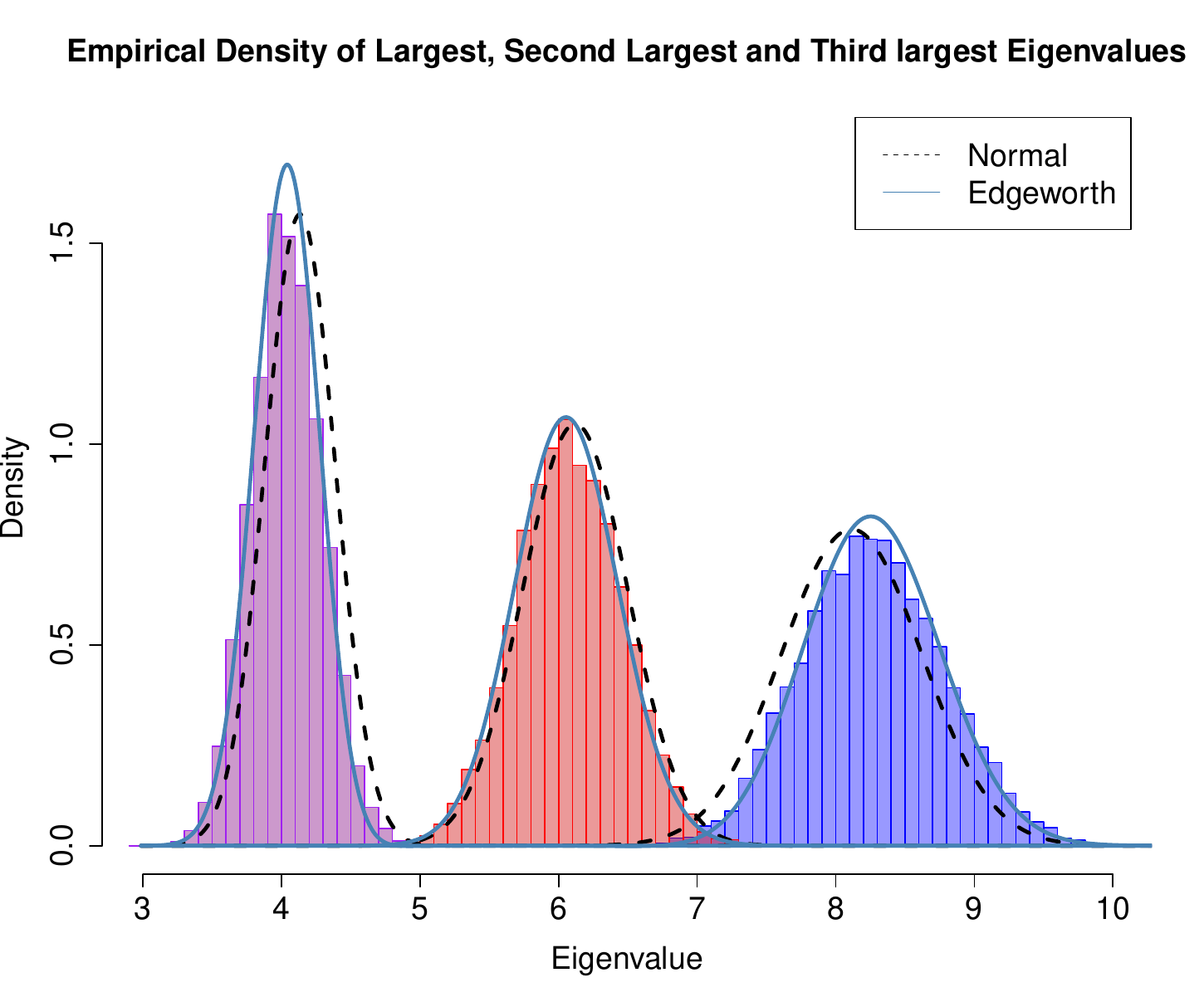}
		\caption{$U(-\sqrt{3}, \sqrt{3})$}
	\end{subfigure}%
	\begin{subfigure}{0.33\textwidth}
		\centering
		\includegraphics[width=\linewidth]{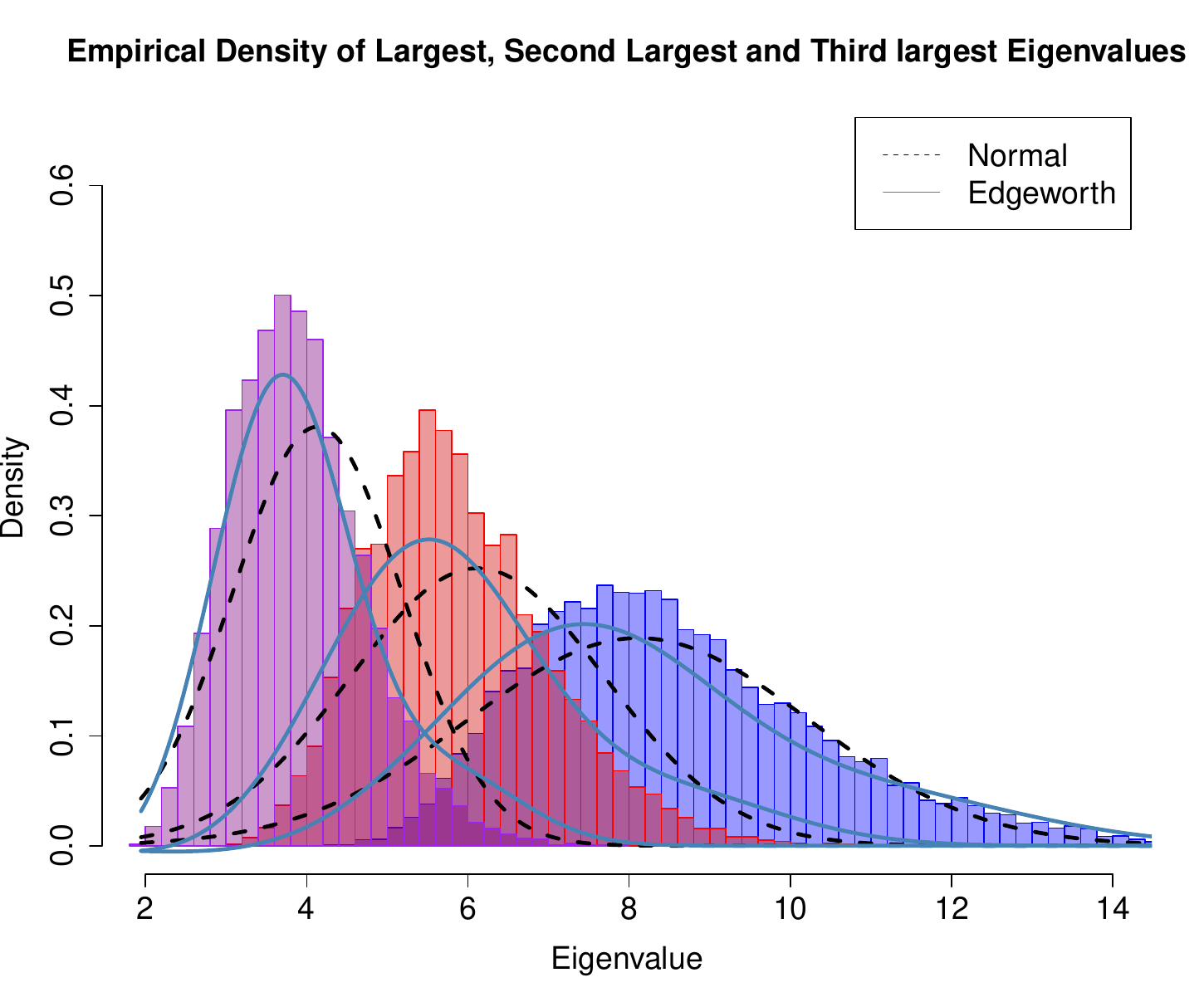}
		\caption{$\chi^2(1)$}
	\end{subfigure}%
	\begin{subfigure}{0.33\textwidth}
		\centering
		\includegraphics[width=\linewidth]{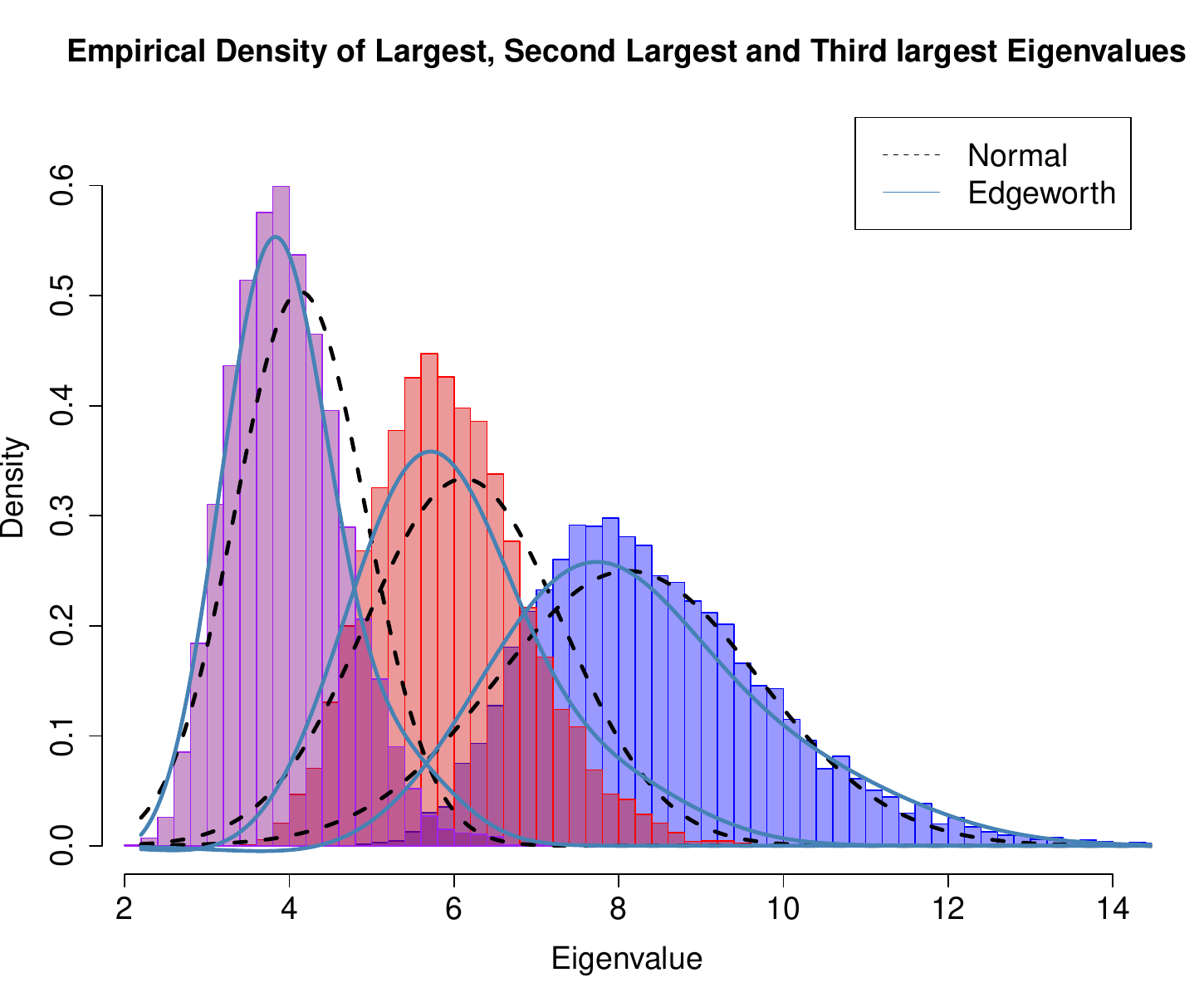}
		\caption{ $Ga(1,1)$}
	\end{subfigure}%

	\begin{subfigure}{0.33\textwidth}
		\centering
		\includegraphics[width=\linewidth]{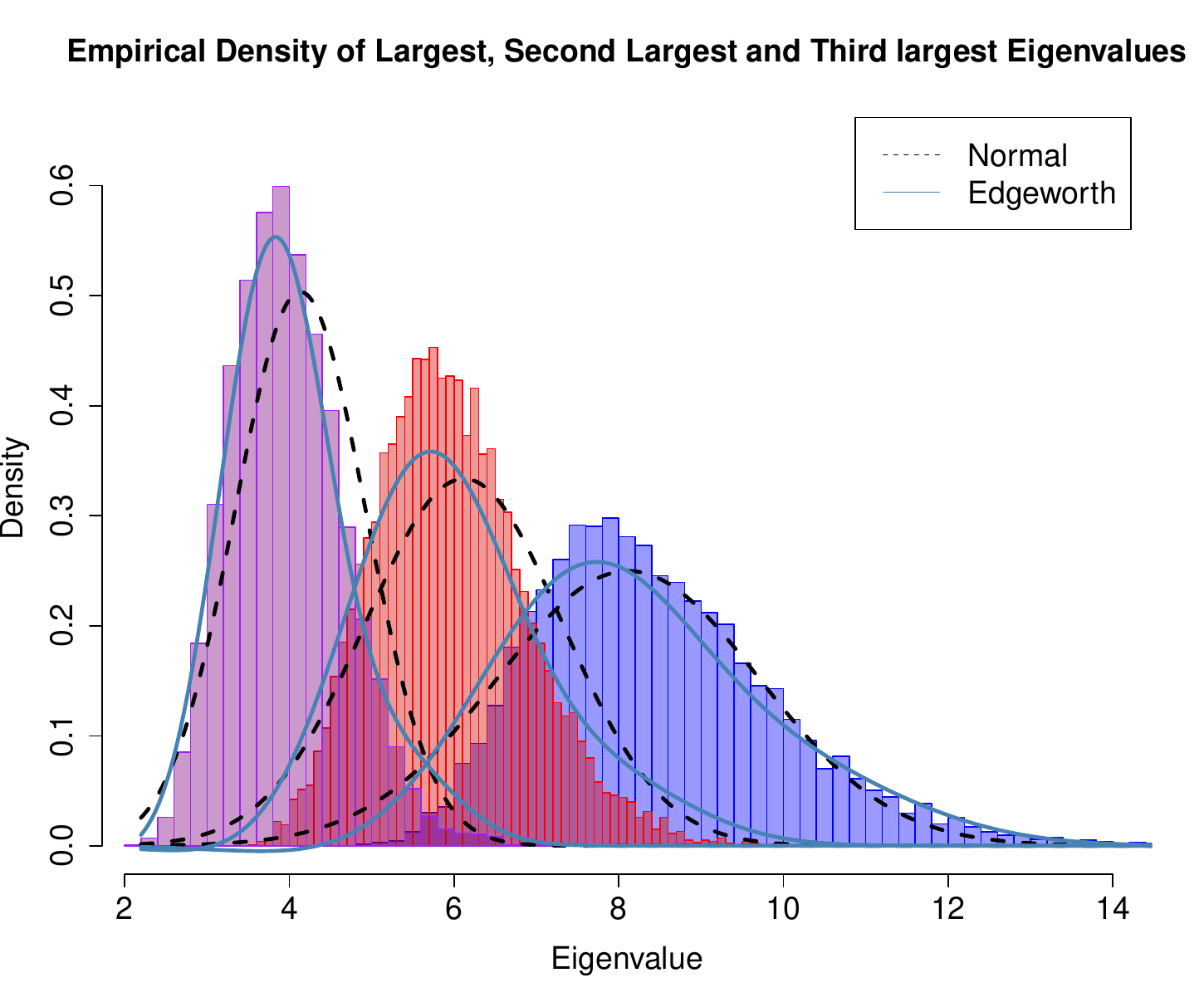}
		\caption{ $Ga(1,2)$}
	\end{subfigure}%
	\begin{subfigure}{0.33\textwidth}
		\centering
		\includegraphics[width=\linewidth]{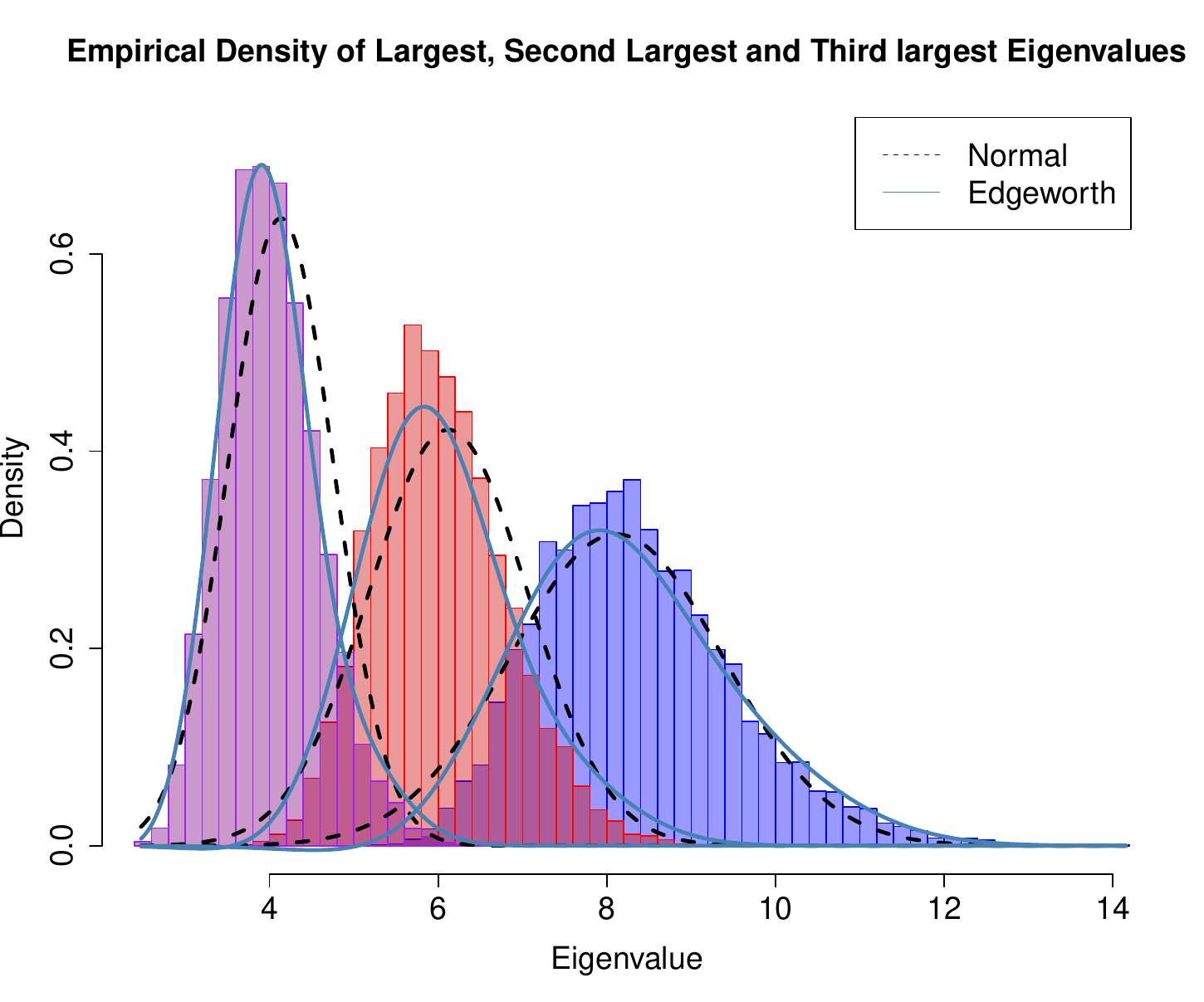}
		\caption{$Ga(2,2)$}
	\end{subfigure}
	\begin{subfigure}{0.33\textwidth}
		\centering
		\includegraphics[width=\linewidth]{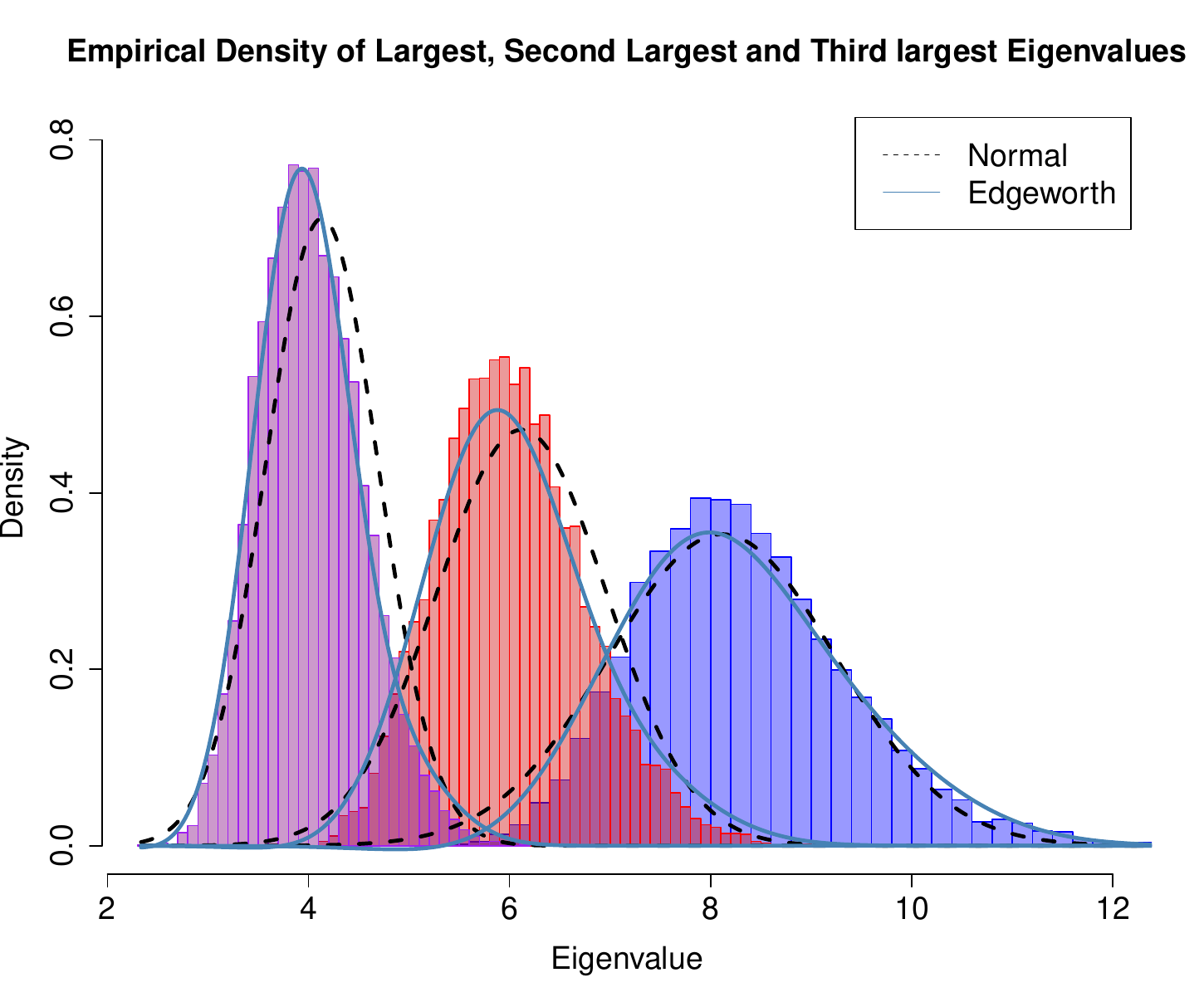}
		\caption{ $Ga(3,3)$}
	\end{subfigure}
	\caption{ Edgeworth expansion  for different samples under Setting 3}
	\label{fig:both_images3}
\end{figure}

\begin{figure}[H]
	
	\begin{subfigure}{0.33\textwidth}
		\centering
		\includegraphics[width=\linewidth]{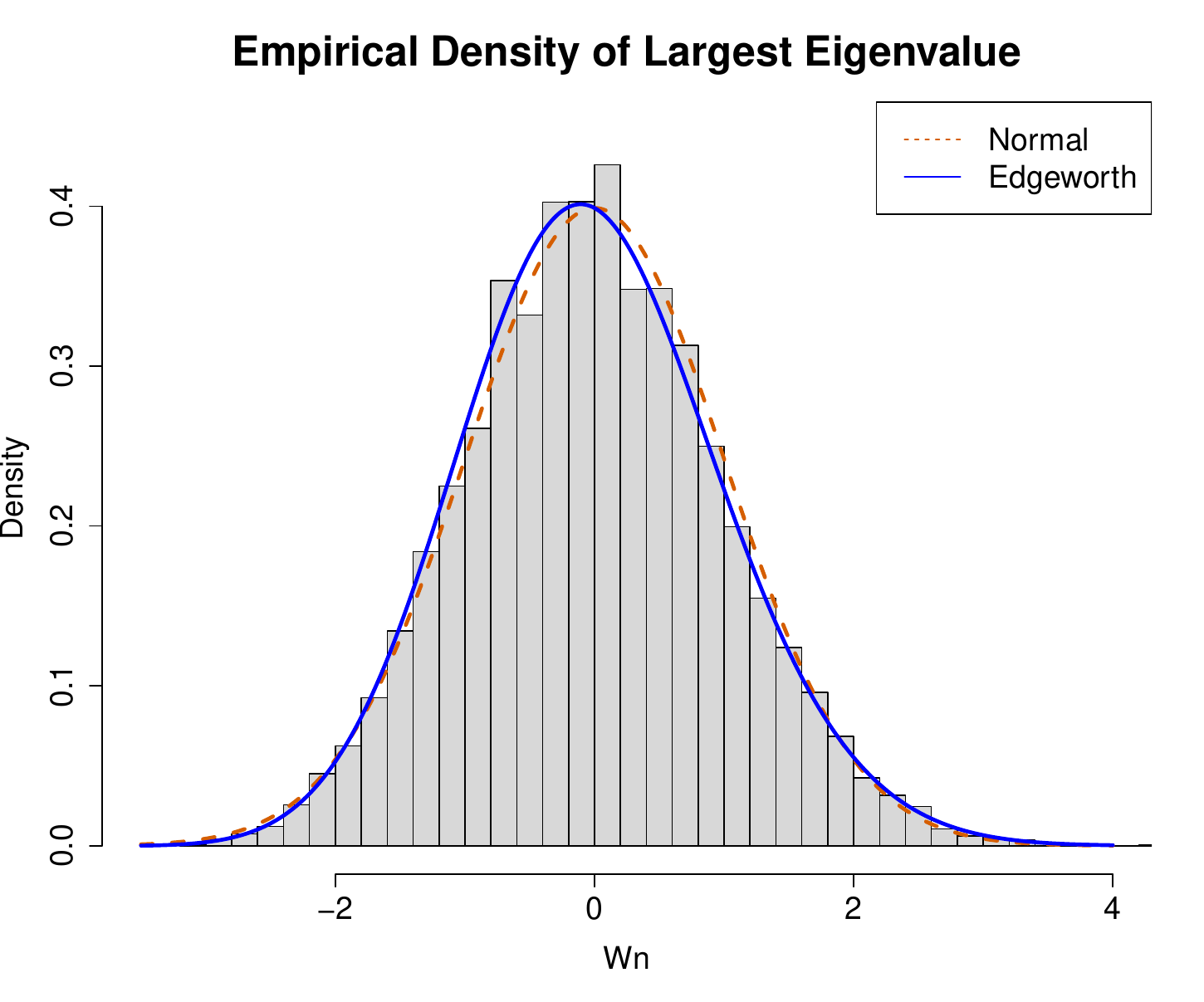}
		\caption{ $U(-\sqrt{3},\sqrt{3}),~ l=4$}
	\end{subfigure}%
	\begin{subfigure}{0.33\textwidth}
		\centering
		\includegraphics[width=\linewidth]{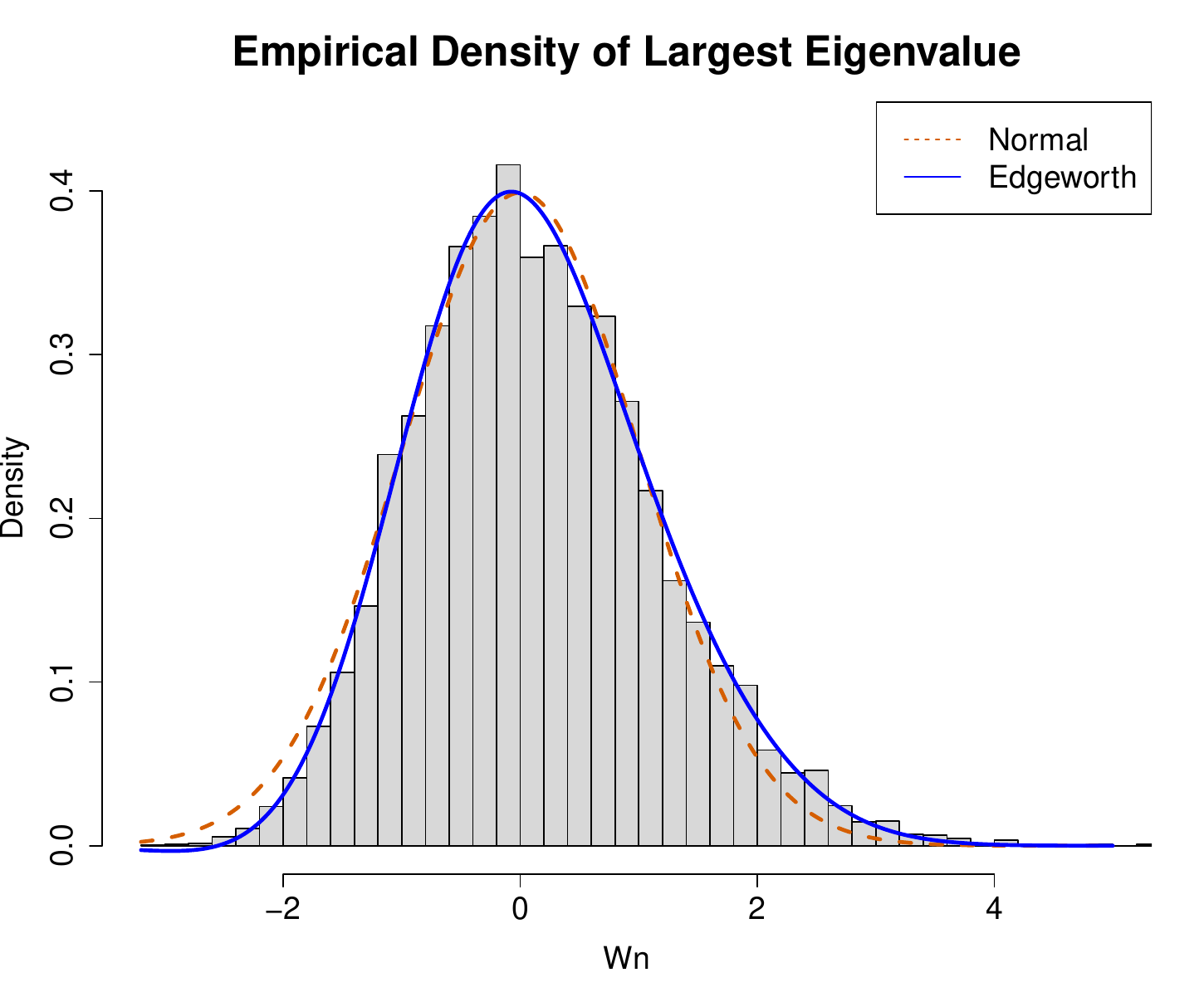}
		\caption{$\chi^2(1),~ l=4$}
	\end{subfigure}%
	\begin{subfigure}{0.33\textwidth}
		\centering
		\includegraphics[width=\linewidth]{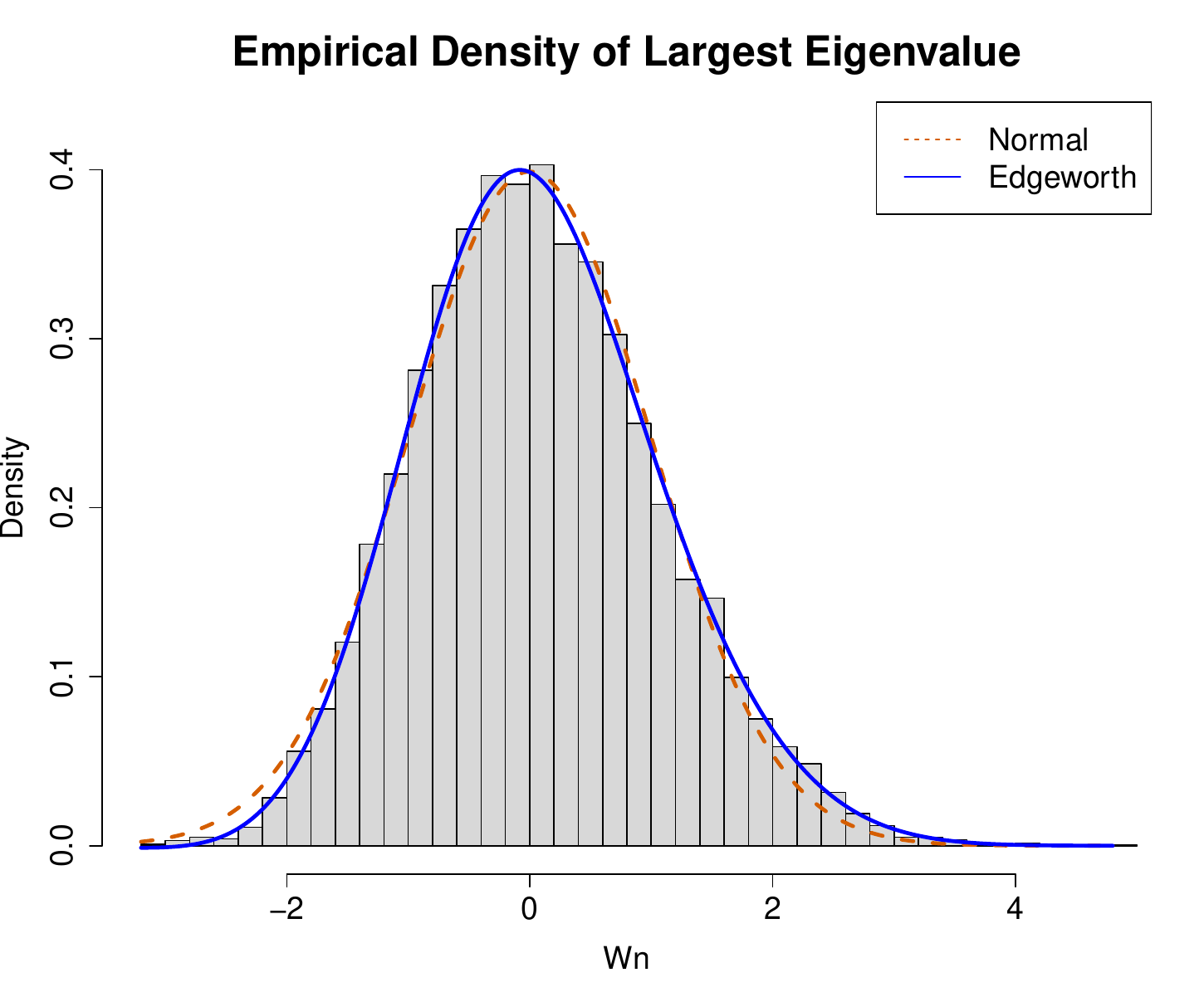}
		\caption{  $Ga(1,1),~ l=4$}
	\end{subfigure}%

	\begin{subfigure}{0.33\textwidth}
		\centering
		\includegraphics[width=\linewidth]{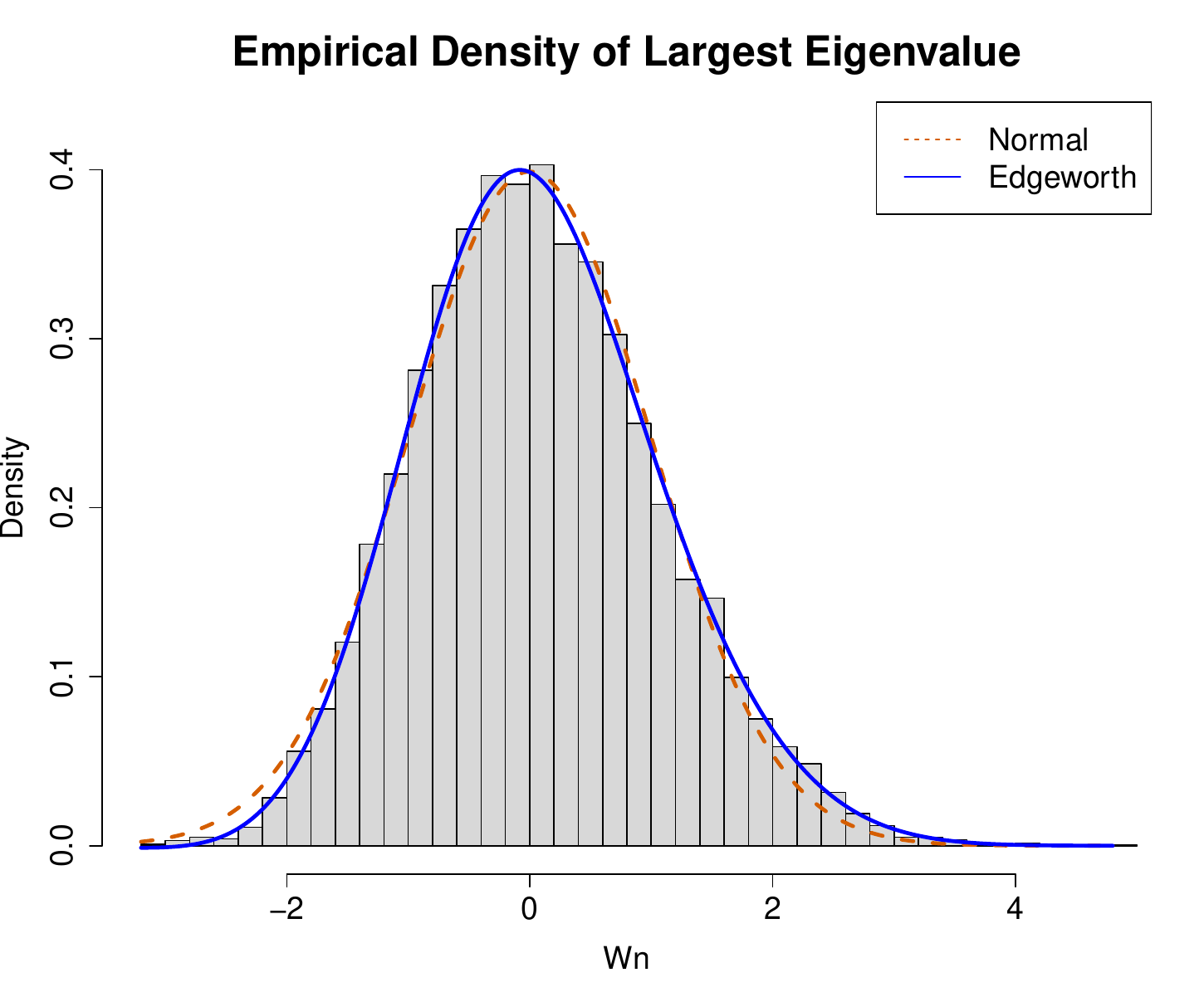}
		\caption{ $Ga(1,2),~ l=4$}
	\end{subfigure}%
	\begin{subfigure}{0.33\textwidth}
		\centering
		\includegraphics[width=\linewidth]{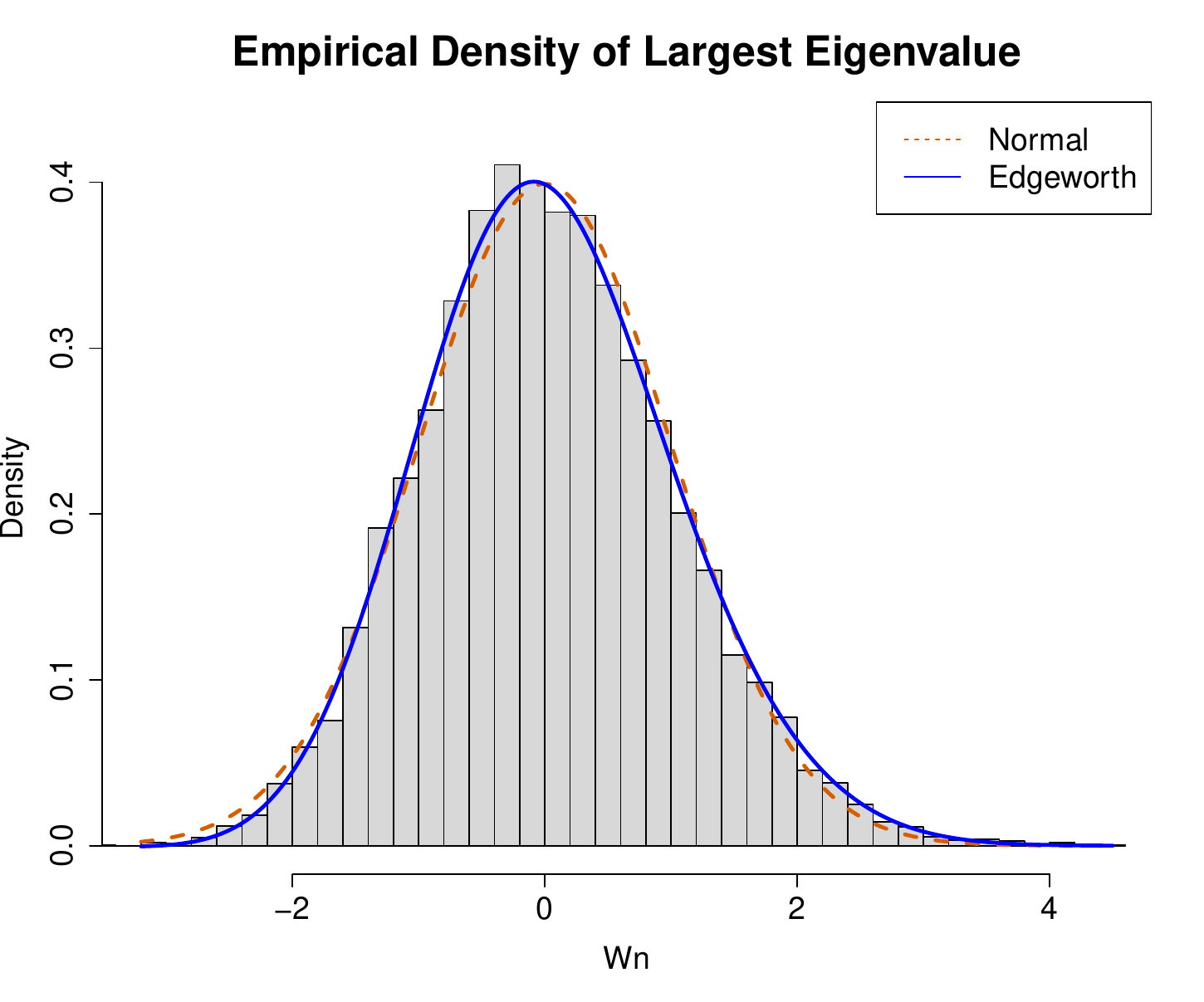}
		\caption{$Ga(2,2),~ l=4$}
	\end{subfigure}
	\begin{subfigure}{0.33\textwidth}
		\centering
		\includegraphics[width=\linewidth]{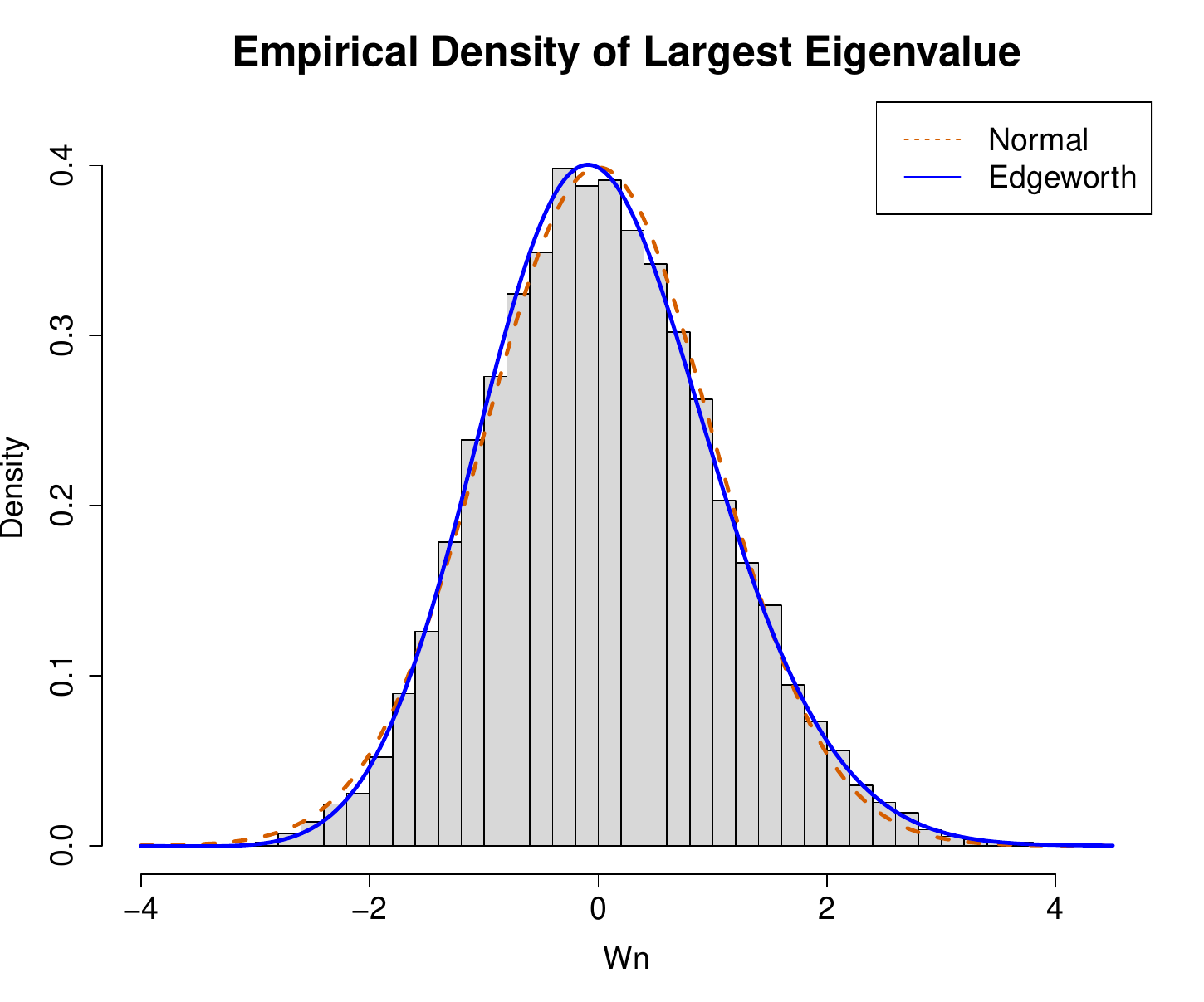}
		\caption{$Ga(3,3),~ l=4$}
	\end{subfigure}
	\caption{ Edgeworth expansion  for different samples under Setting 4}
	\label{fig:both_images4}
\end{figure}

\begin{figure}[H]
	
	\begin{subfigure}{0.33\textwidth}
		\centering
		\includegraphics[width=\linewidth]{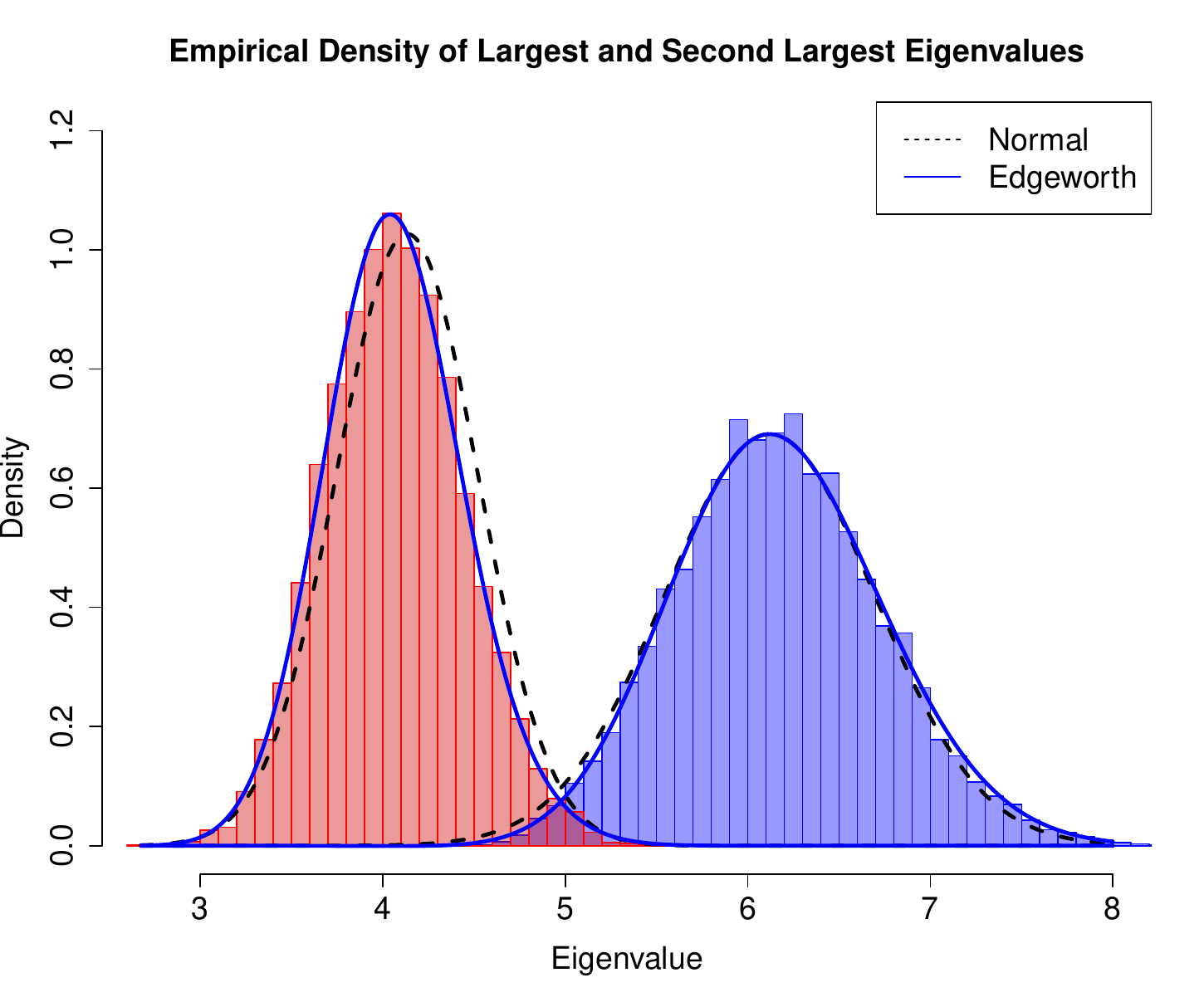}
		\caption{  $U(-\sqrt{3},\sqrt{3}),~l_1=6,~l_2=4$}
	\end{subfigure}%
	\begin{subfigure}{0.33\textwidth}
		\centering
		\includegraphics[width=\linewidth]{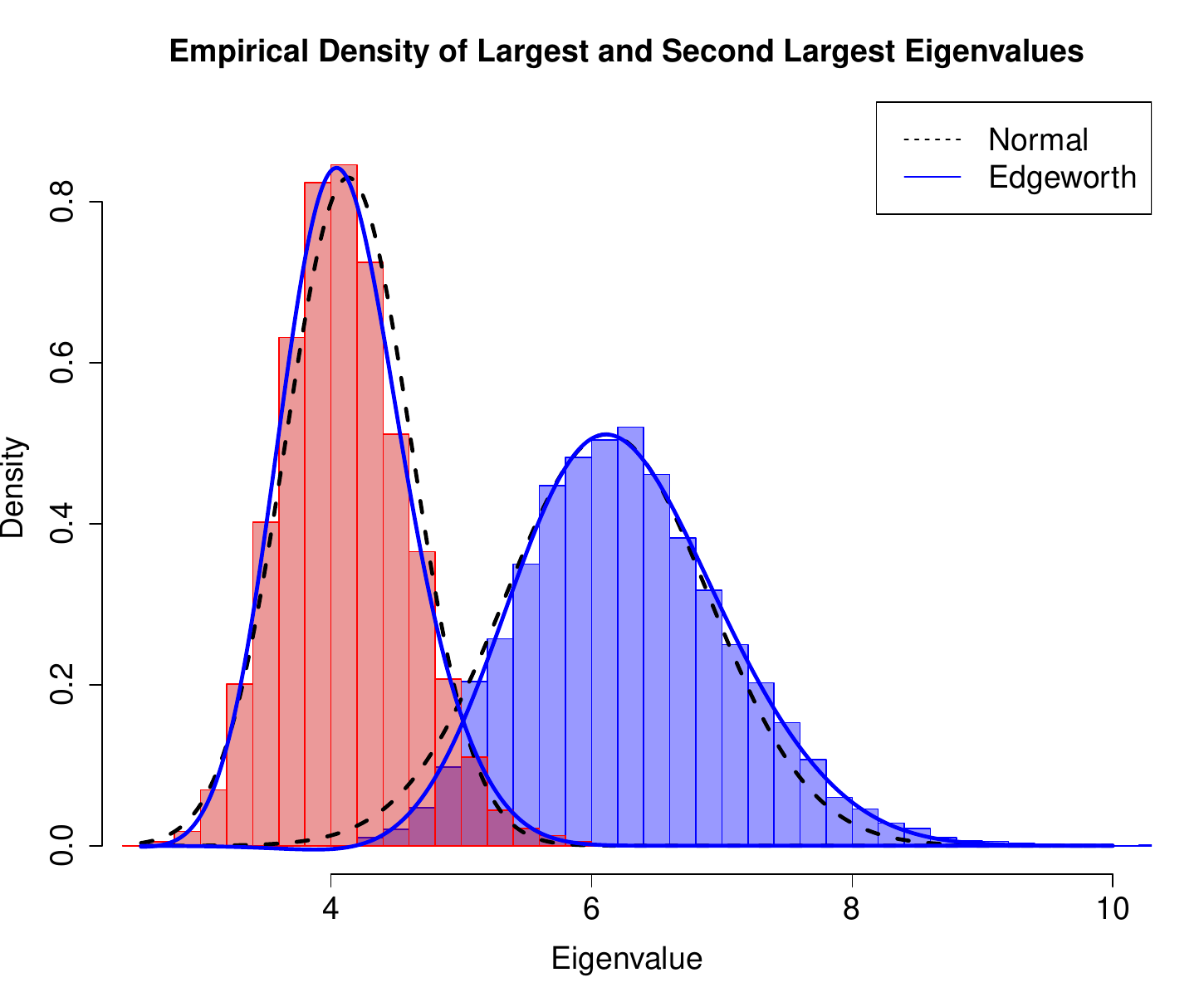}
		\caption{ $\chi^2(1),~l_1=6,~l_2=4$}
	\end{subfigure}%
	\begin{subfigure}{0.33\textwidth}
		\centering
		\includegraphics[width=\linewidth]{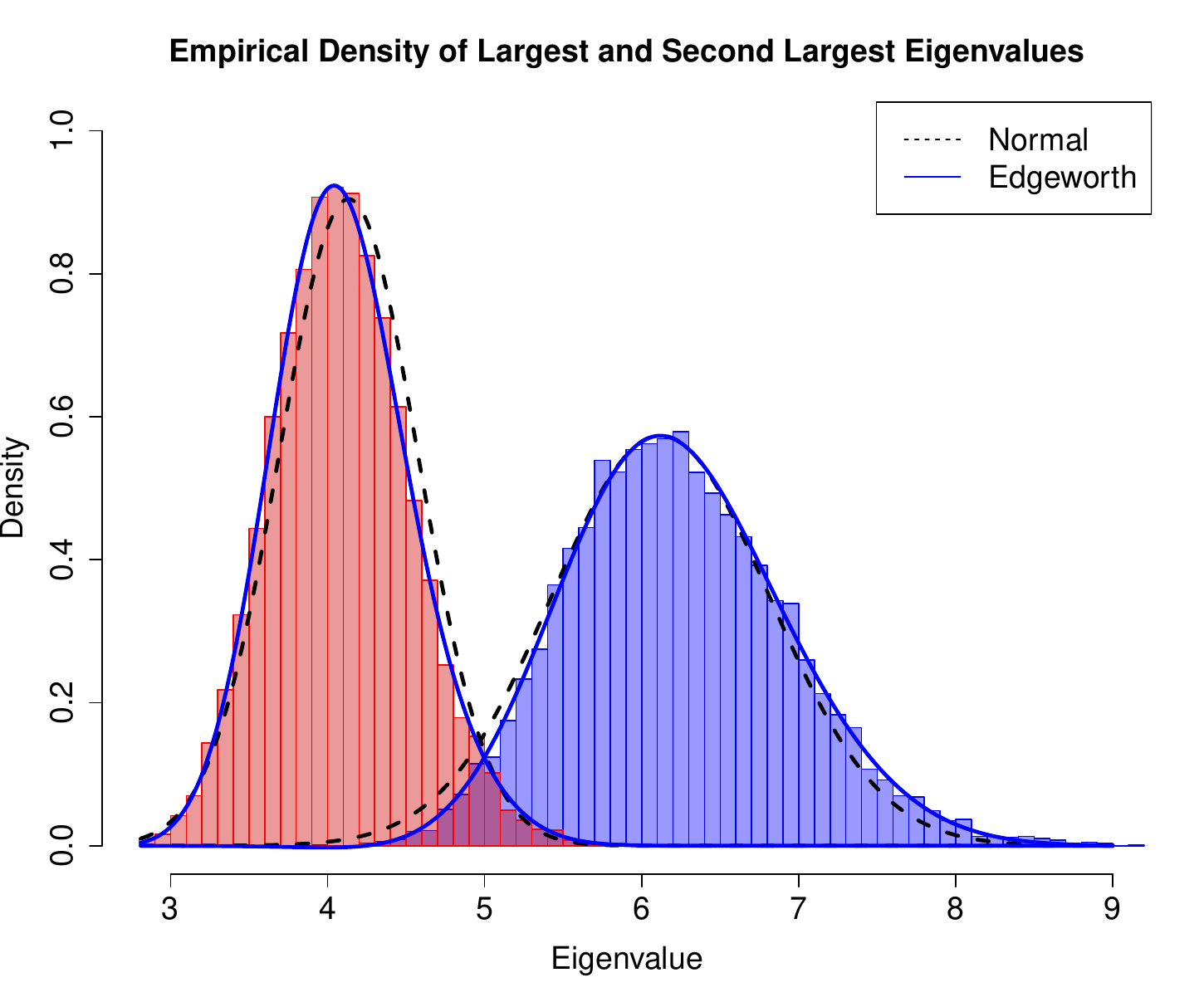}
		\caption{ $Ga(1,1),~l_1=6,~l_2=4$}
	\end{subfigure}%

	\begin{subfigure}{0.33\textwidth}
		\centering
		\includegraphics[width=\linewidth]{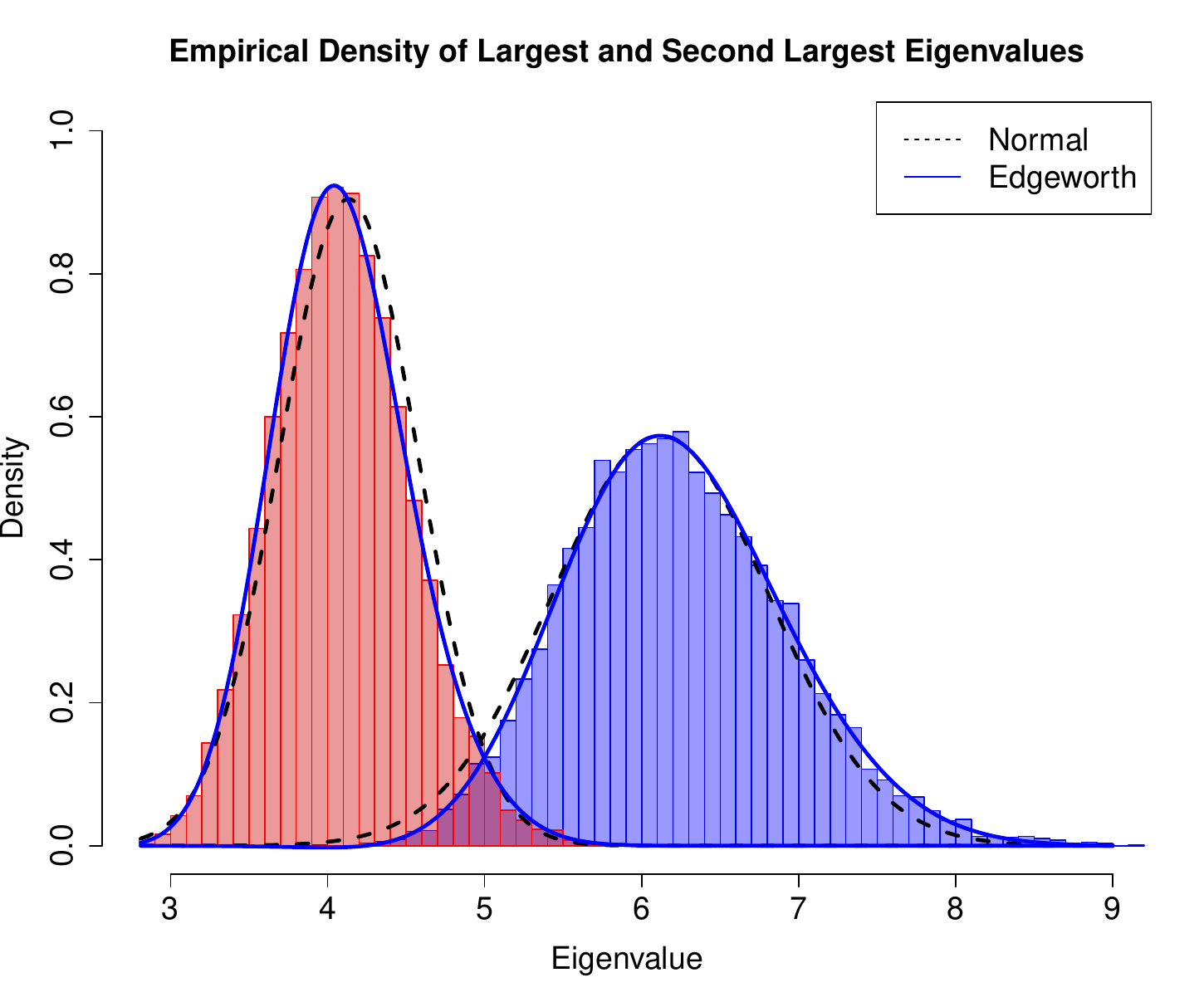}
		\caption{ $Ga(1,2),~ l_1=6,~l_2=4$}
	\end{subfigure}%
	\begin{subfigure}{0.33\textwidth}
		\centering
		\includegraphics[width=\linewidth]{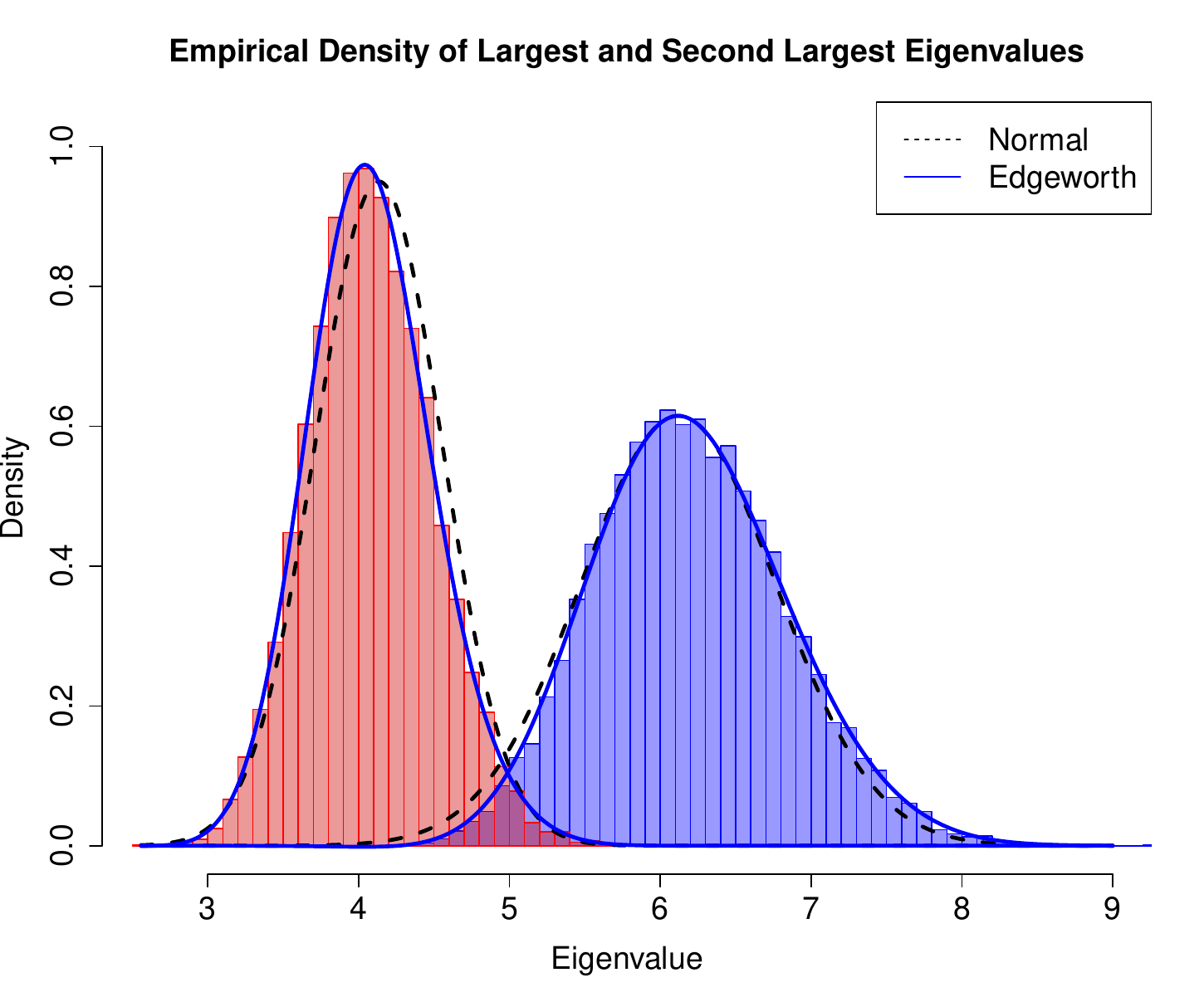}
		\caption{$Ga(2,2),~ l_1=6,~l_2=4$}
	\end{subfigure}
	\begin{subfigure}{0.33\textwidth}
		\centering
		\includegraphics[width=\linewidth]{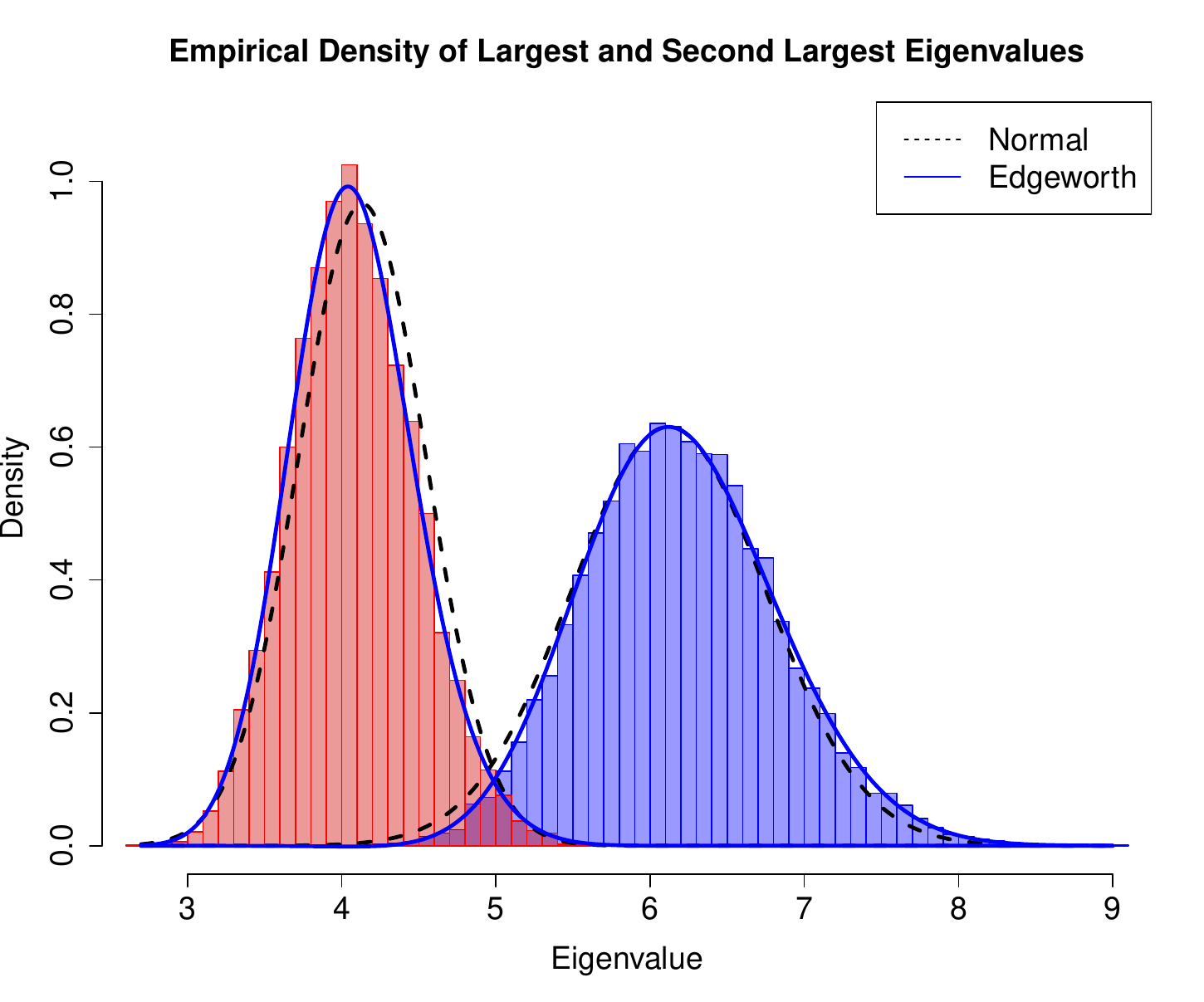}
		\caption{$Ga(3,3),~ l_1=6,~l_2=4$}
	\end{subfigure}
	\caption{ Edgeworth expansion  for different samples under Setting 5}
	\label{fig:both_images5}
\end{figure}

\begin{figure}[H]
	
	\begin{subfigure}{0.33\textwidth}
		\centering
		\includegraphics[width=\linewidth]{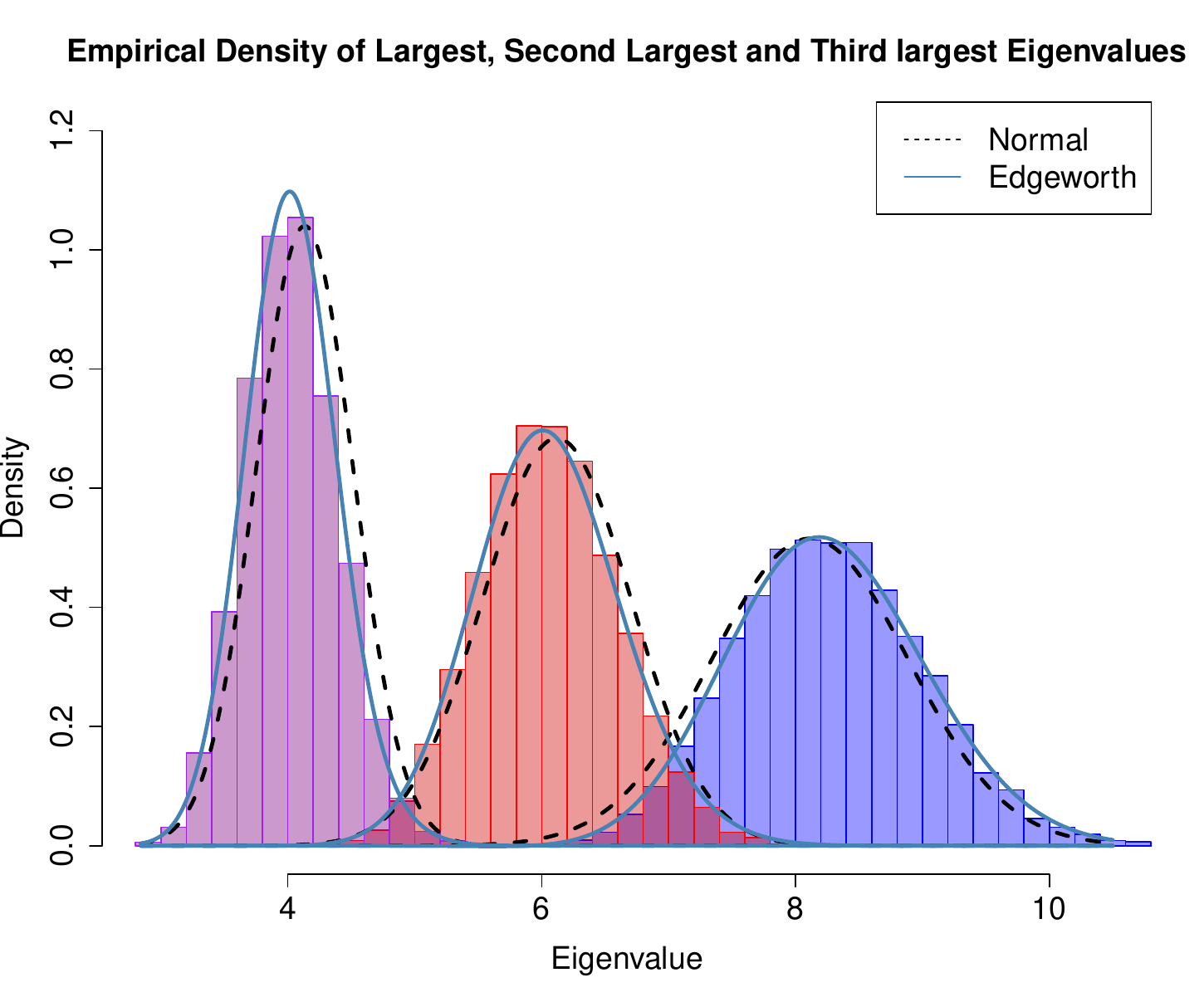}
		\caption{ $U(-\sqrt{3},\sqrt{3})$}
	\end{subfigure}%
	\begin{subfigure}{0.33\textwidth}
		\centering
		\includegraphics[width=\linewidth]{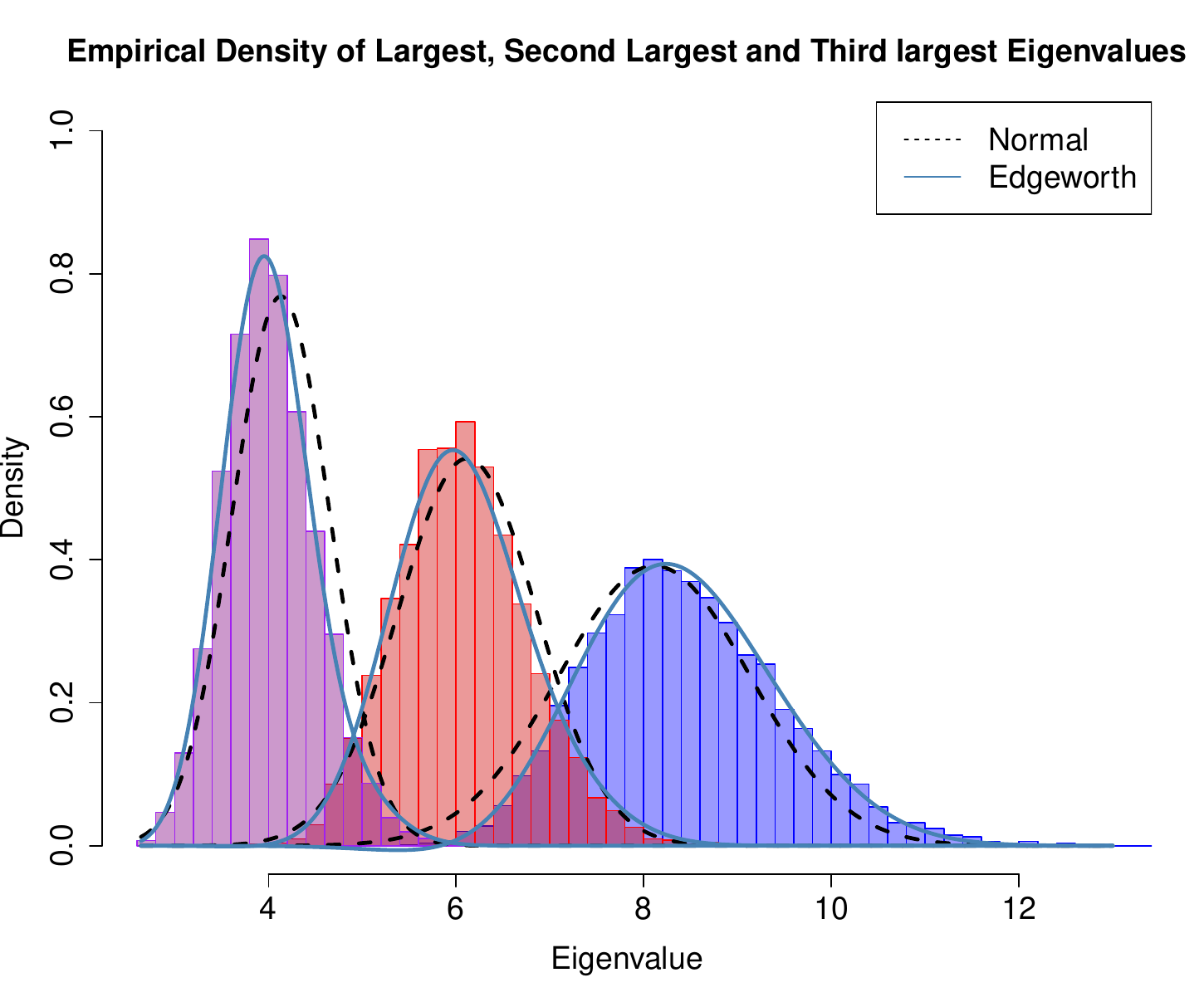}
		\caption{ $\chi^2(1)$}
	\end{subfigure}%
	\begin{subfigure}{0.33\textwidth}
		\centering
		\includegraphics[width=\linewidth]{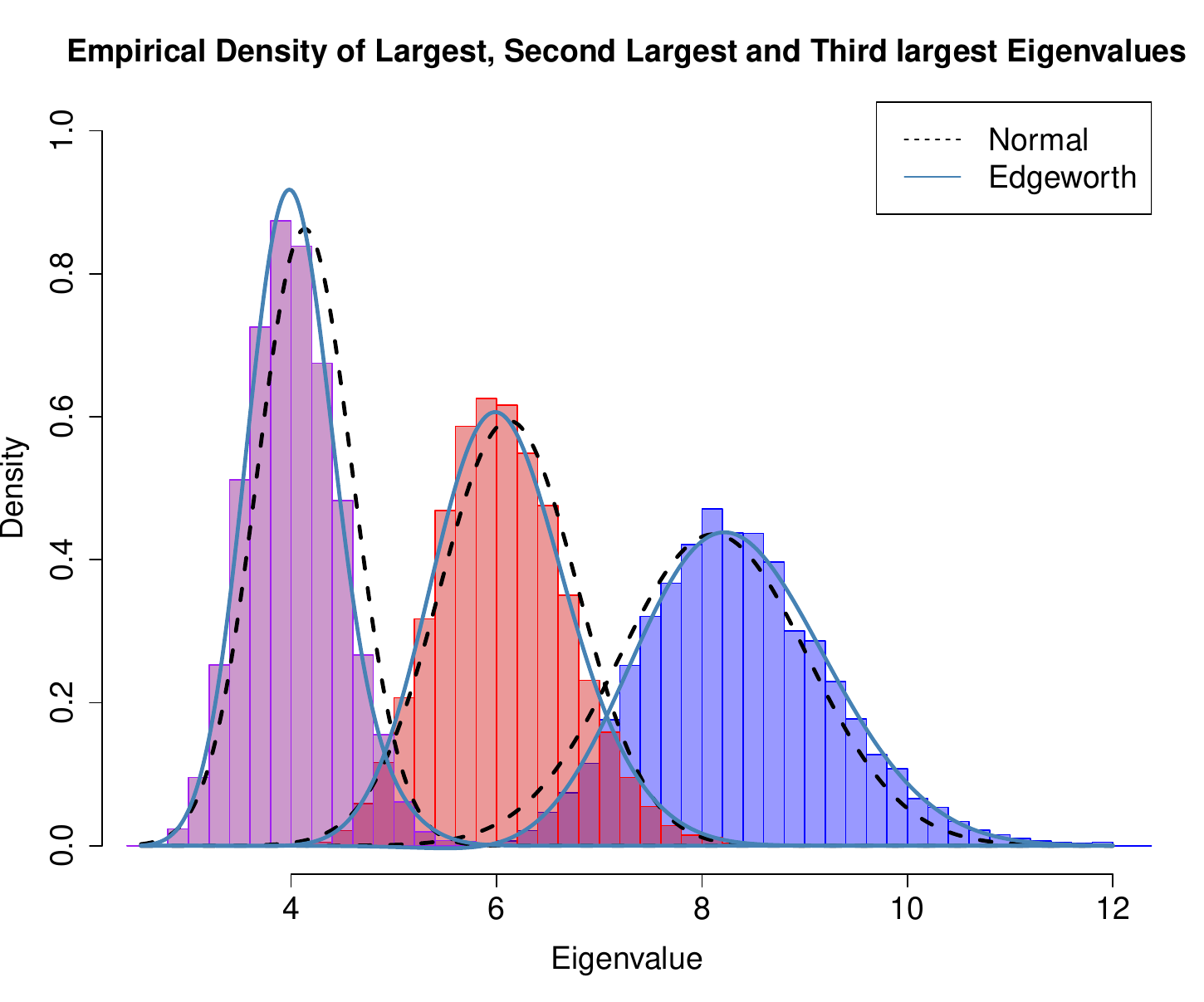}
		\caption{ $Ga(1,1)$}
	\end{subfigure}%

	\begin{subfigure}{0.33\textwidth}
		\centering
		\includegraphics[width=\linewidth]{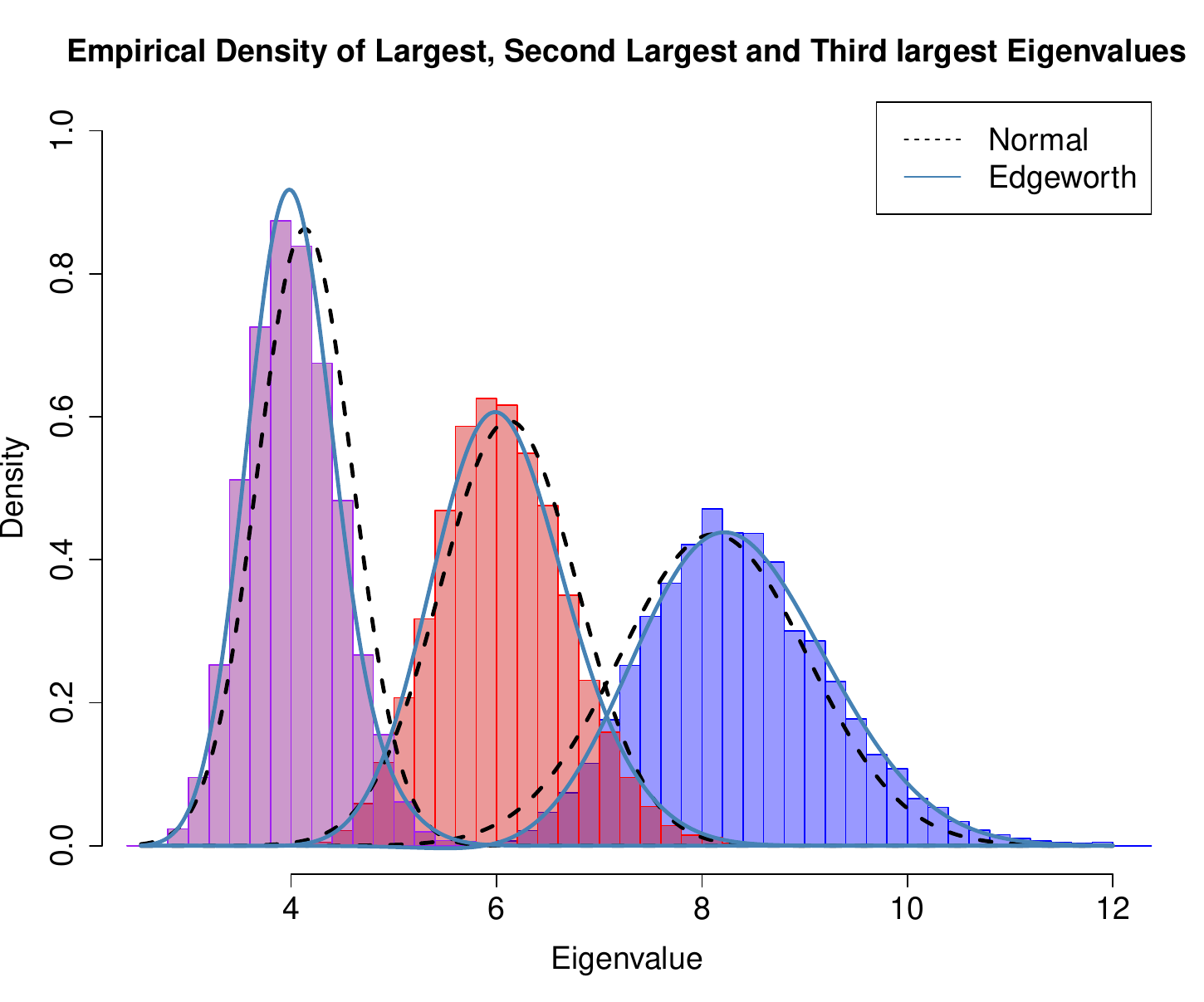}
		\caption{ $Ga(1,2)$}
	\end{subfigure}%
	\begin{subfigure}{0.33\textwidth}
		\centering
		\includegraphics[width=\linewidth]{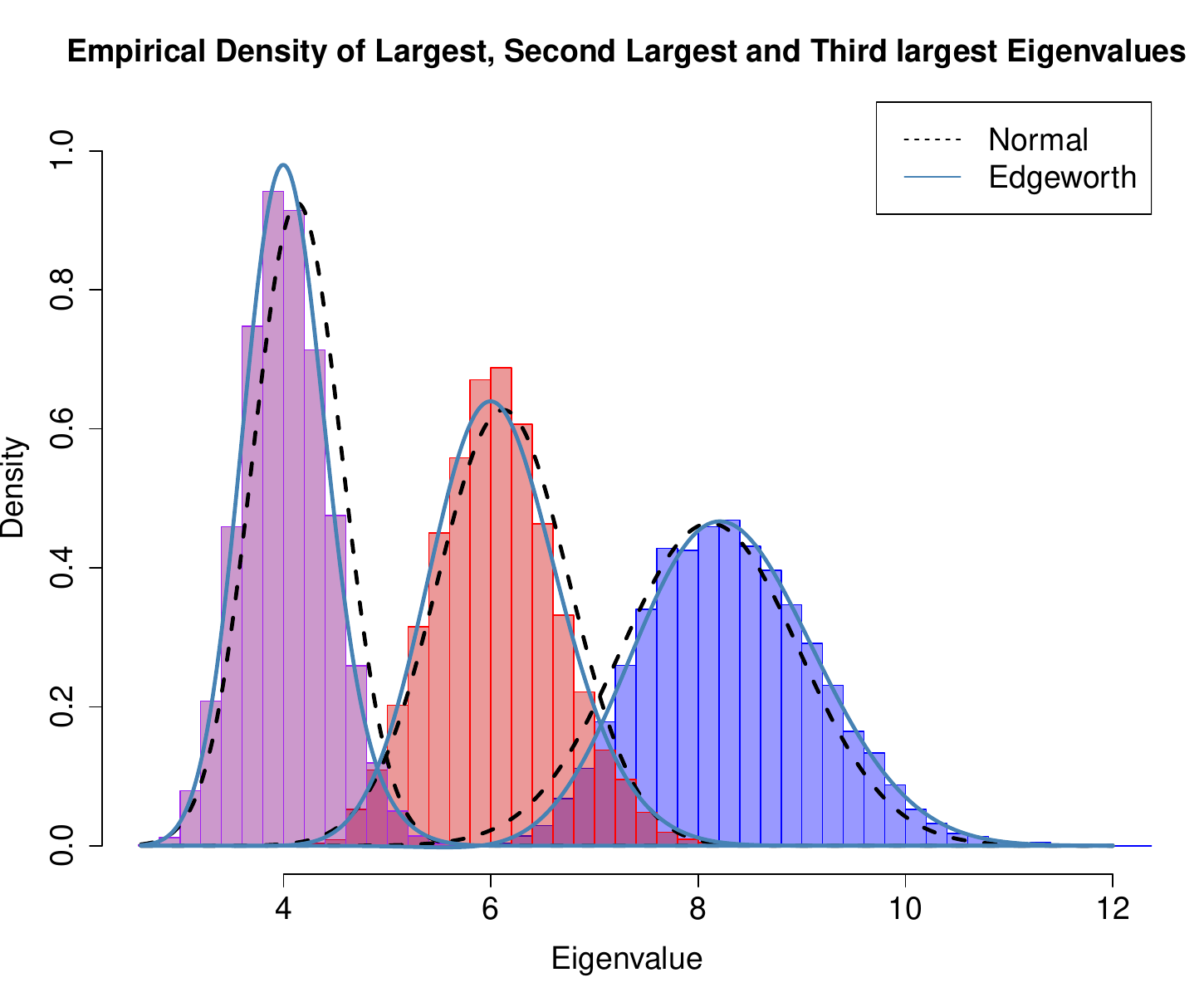}
		\caption{$Ga(2,2)$}
	\end{subfigure}
	\begin{subfigure}{0.33\textwidth}
		\centering
		\includegraphics[width=\linewidth]{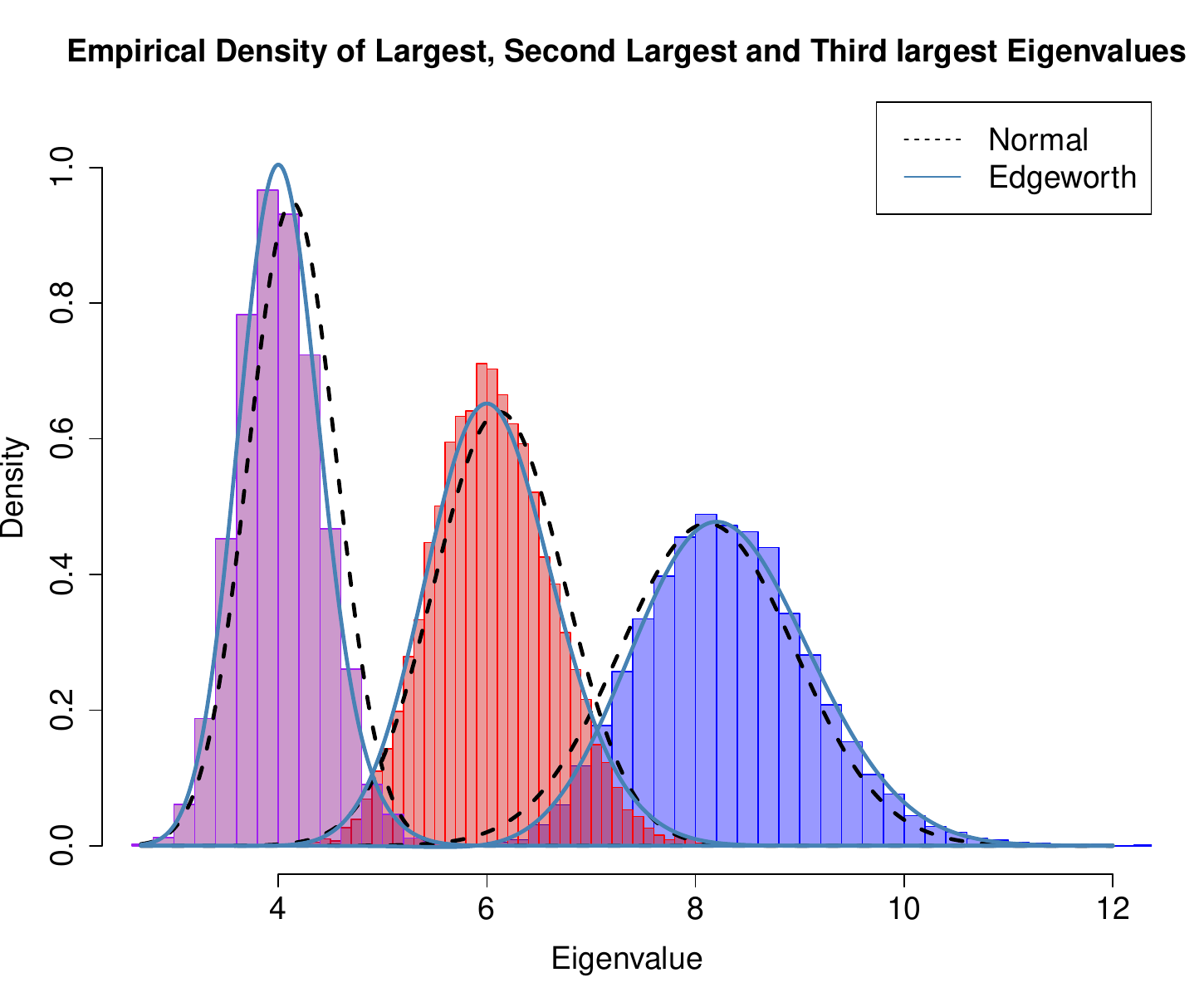}
		\caption{$Ga(3,3)$}
	\end{subfigure}
	\caption{ Edgeworth expansion  for different samples under Setting 6}
	\label{fig:both_images6}
\end{figure}

\begin{figure}[H]
	
	\begin{subfigure}{0.33\textwidth}
		\centering
		\includegraphics[width=\linewidth]{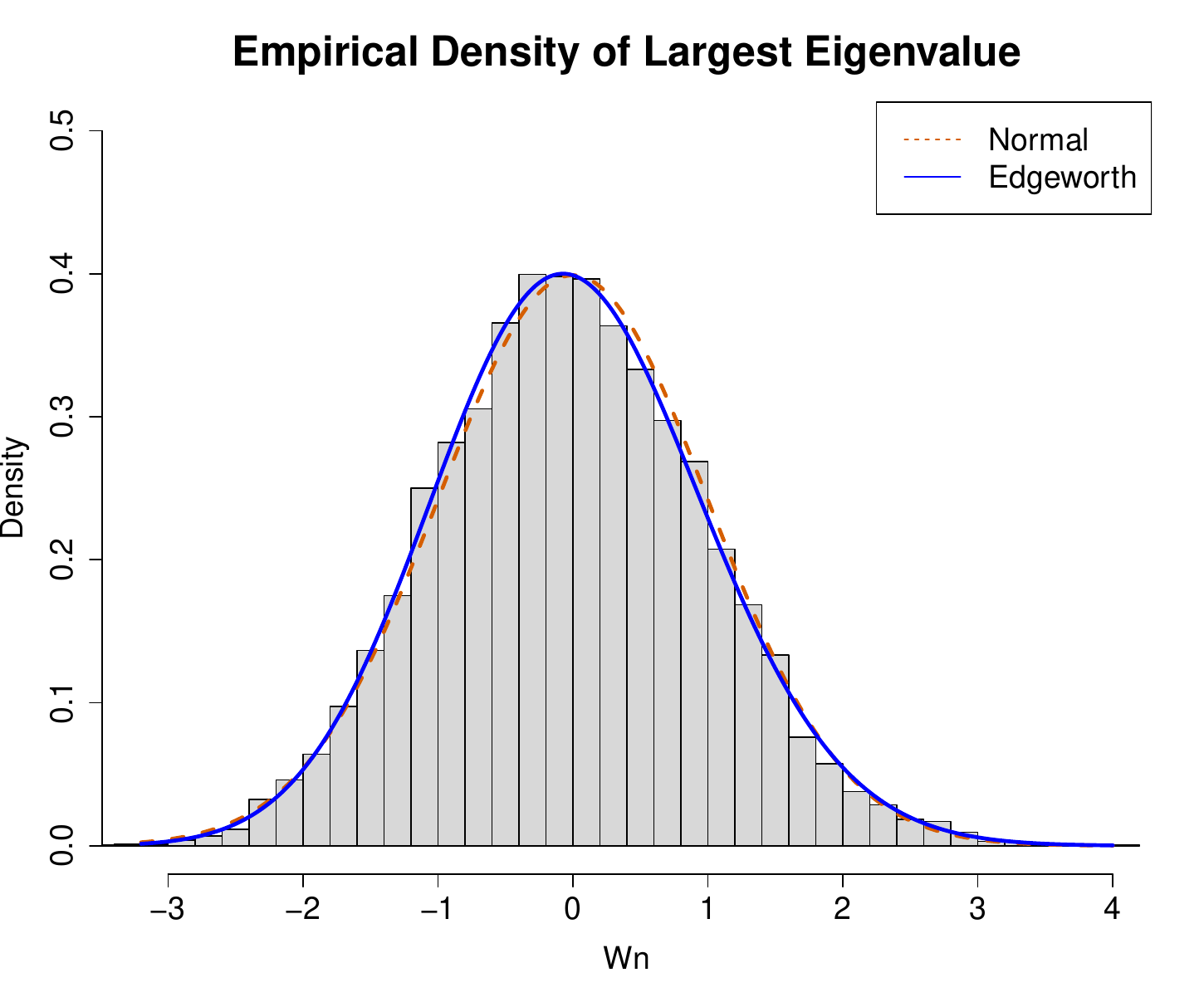}
		\caption{ $U(-\sqrt{3},\sqrt{3}),~ l=4$}
	\end{subfigure}%
	\begin{subfigure}{0.33\textwidth}
		\centering
		\includegraphics[width=\linewidth]{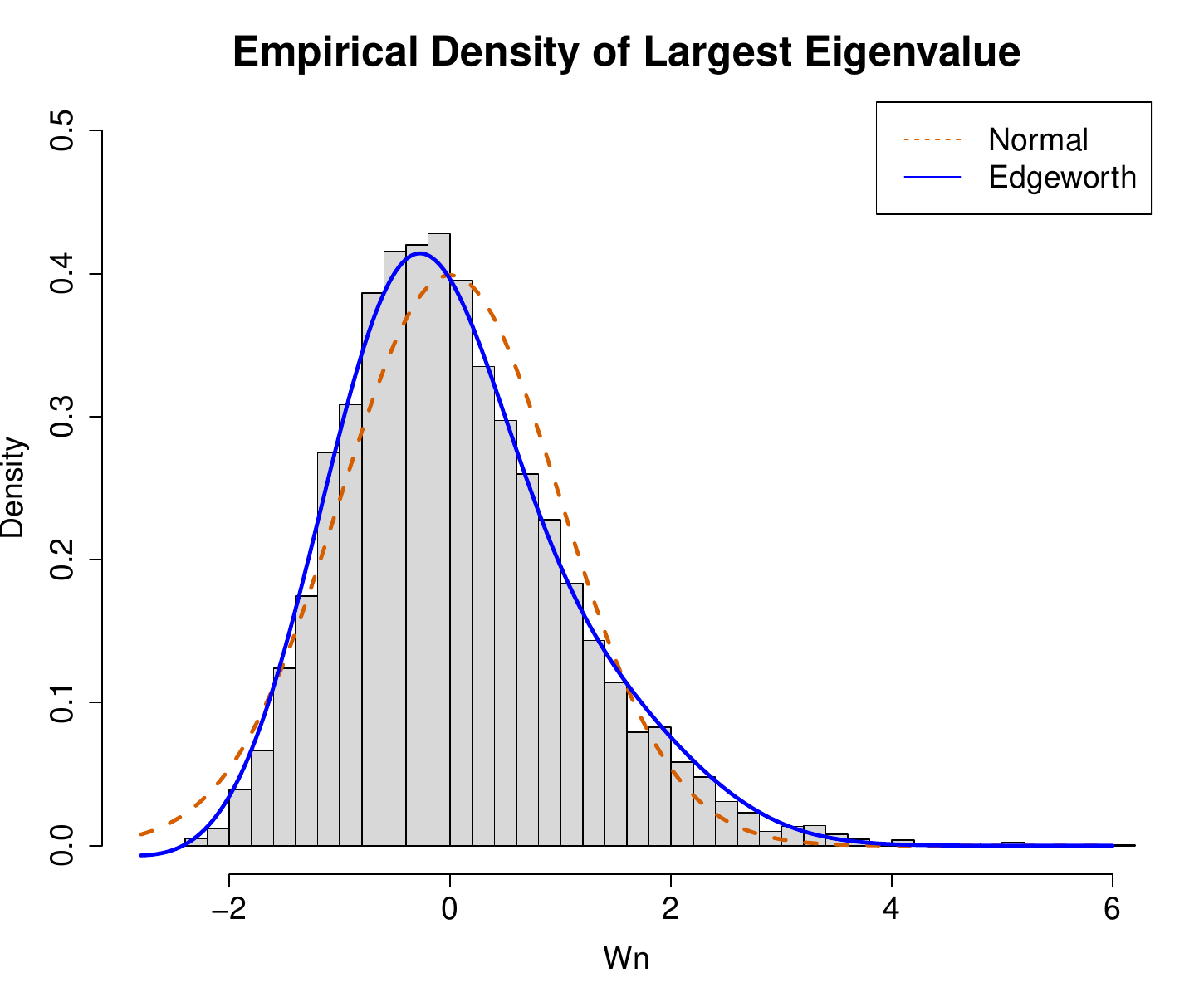}
		\caption{ $\chi^2(1),~ l=4$}
	\end{subfigure}%
	\begin{subfigure}{0.33\textwidth}
		\centering
		\includegraphics[width=\linewidth]{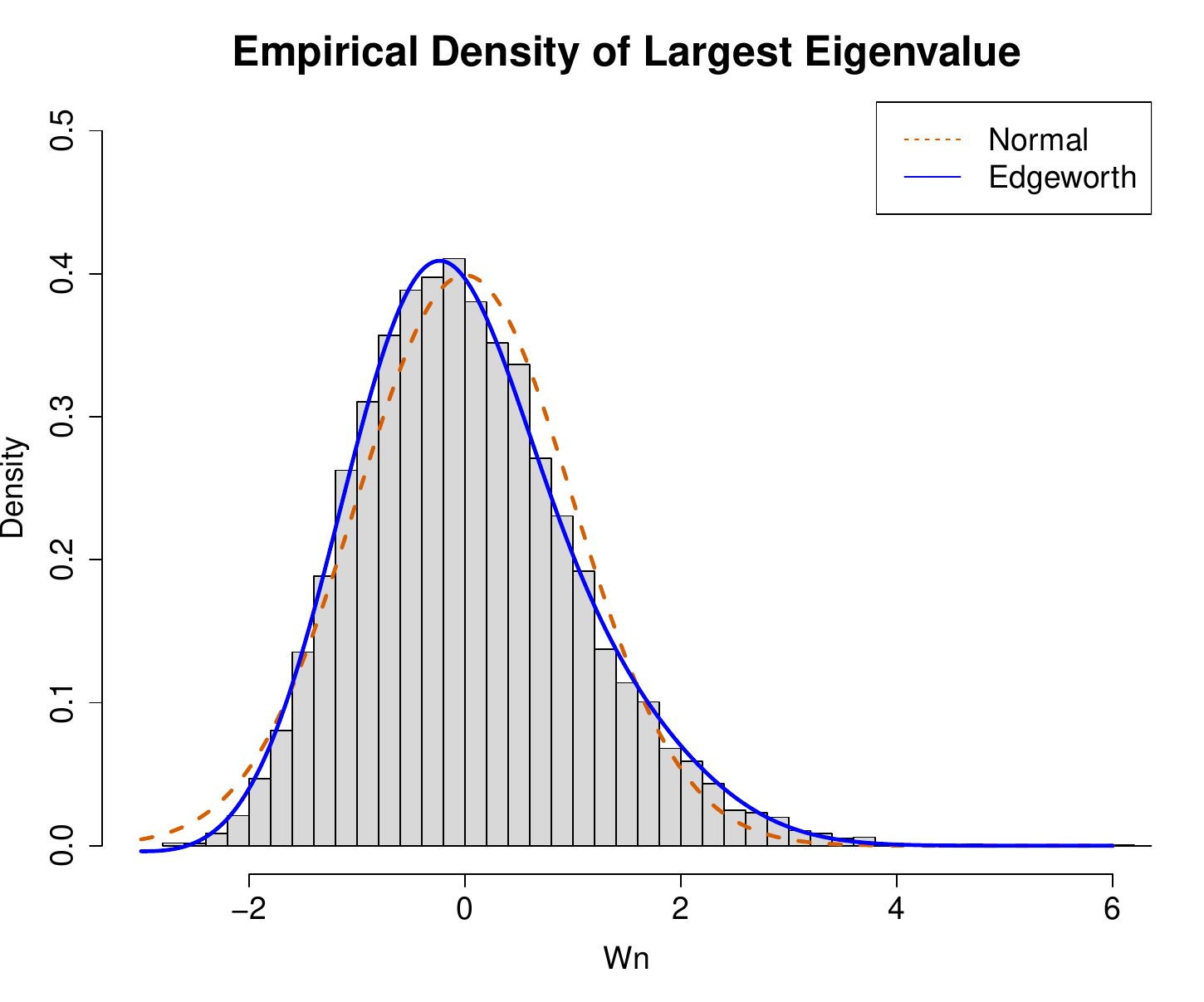}
		\caption{ $Ga(1,1),~ l=4$}
	\end{subfigure}%

	\begin{subfigure}{0.33\textwidth}
		\centering
		\includegraphics[width=\linewidth]{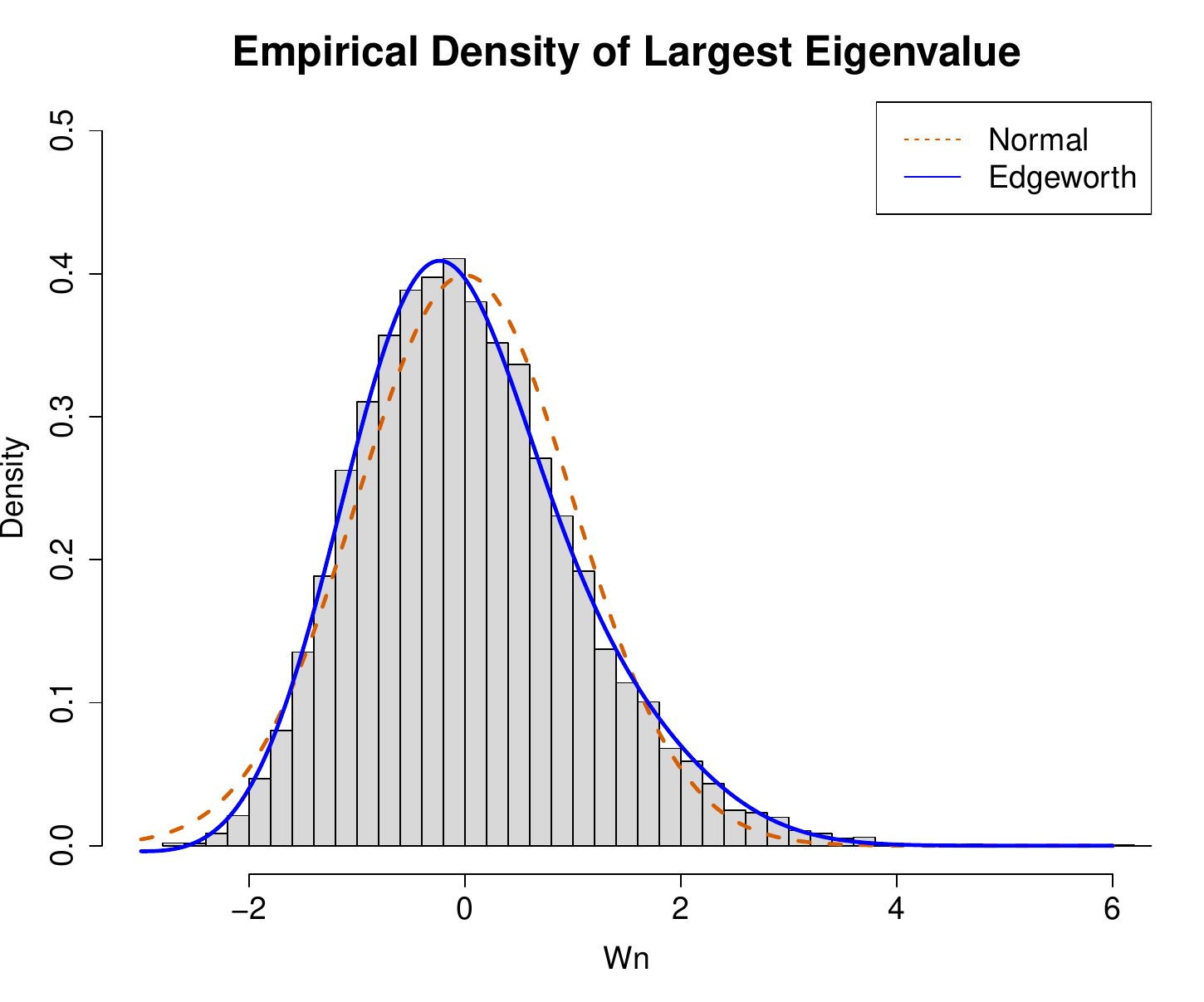}
		\caption{ $Ga(1,2),~ l=4$}
	\end{subfigure}%
	\begin{subfigure}{0.33\textwidth}
		\centering
		\includegraphics[width=\linewidth]{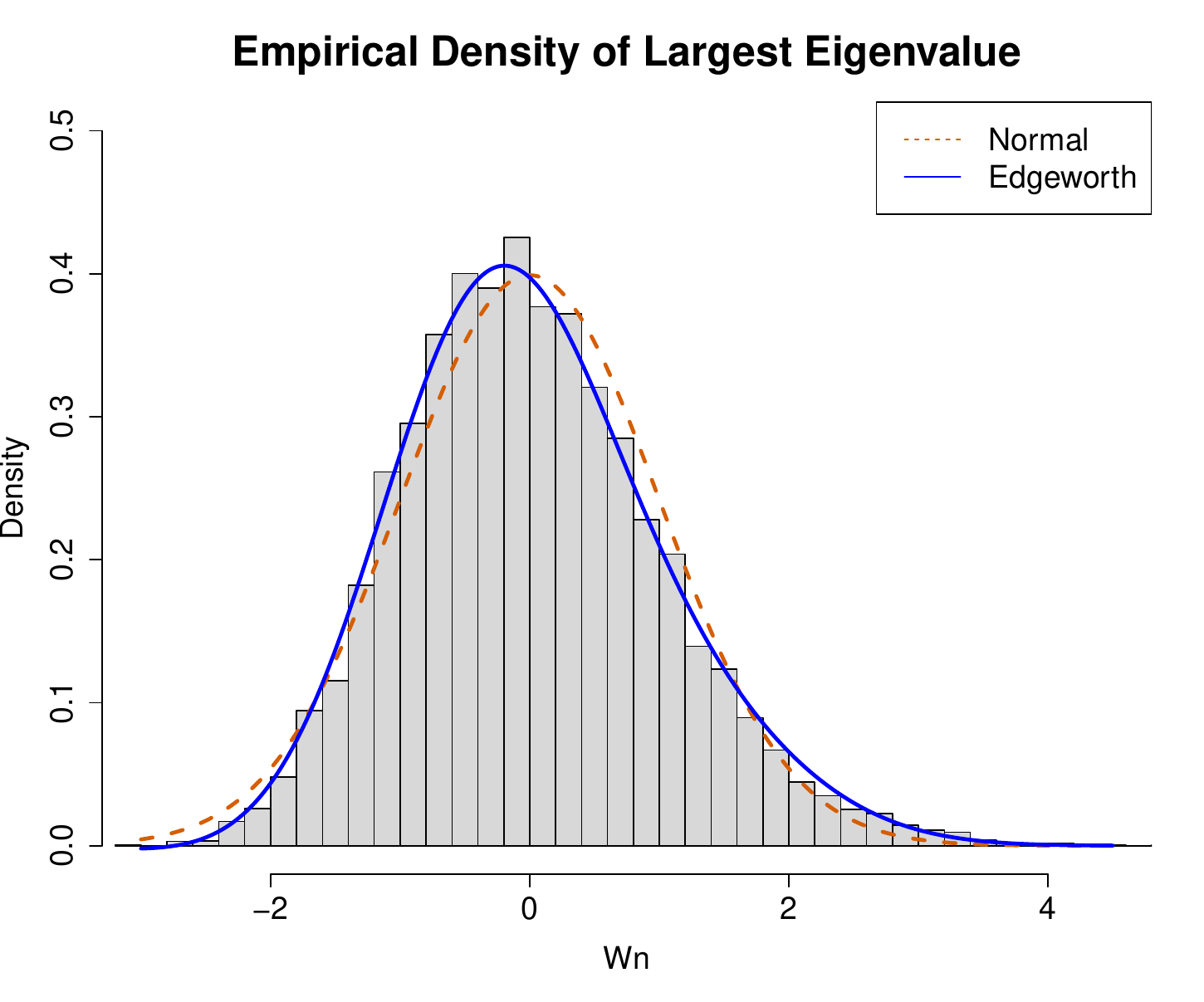}
		\caption{$Ga(2,2),~ l=4$}
	\end{subfigure}
	\begin{subfigure}{0.33\textwidth}
		\centering
		\includegraphics[width=\linewidth]{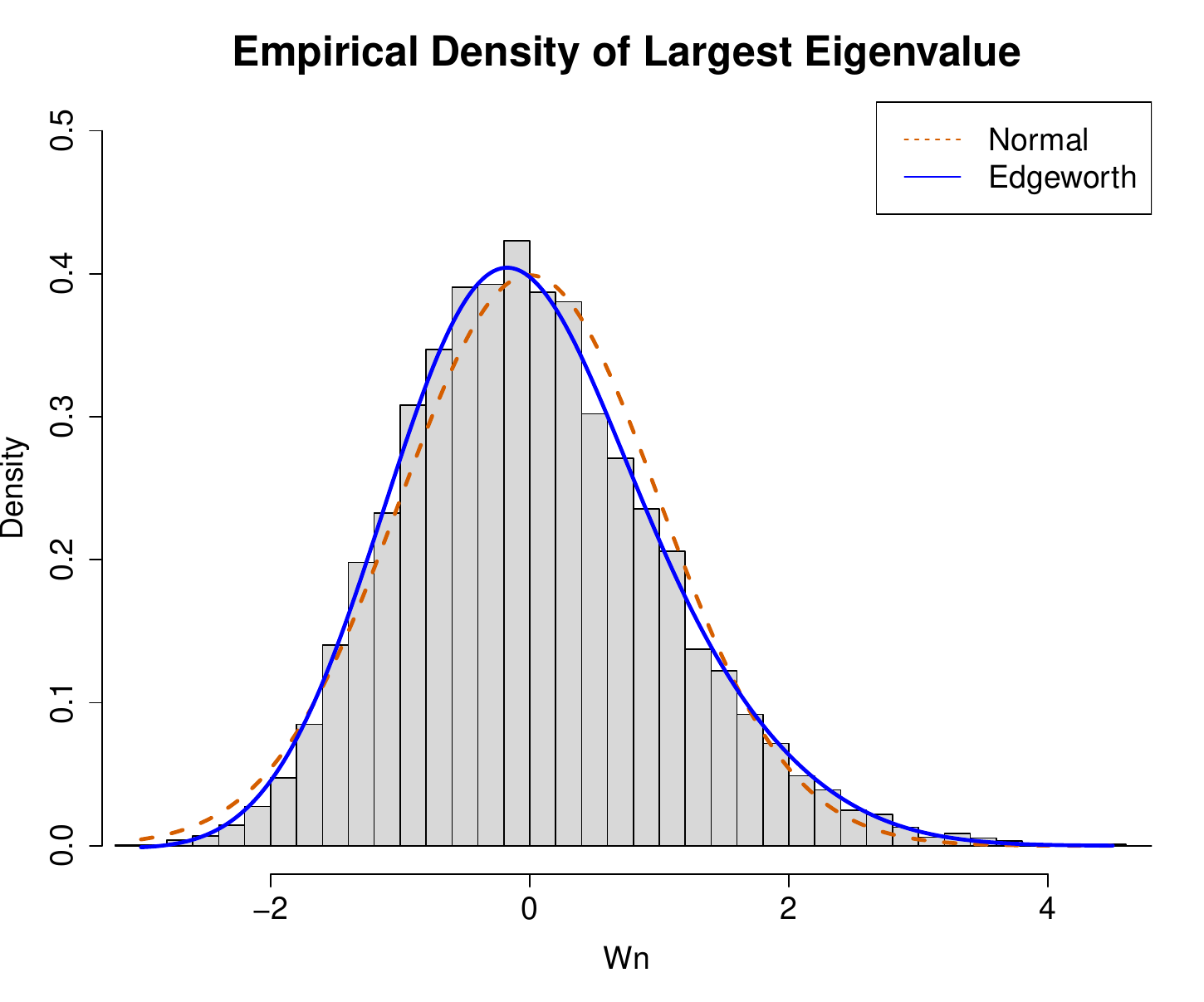}
		\caption{$Ga(3,3),~ l=4$}
	\end{subfigure}
	\caption{ Edgeworth expansion  for different samples under Setting 7}
	\label{fig:both_images7}
\end{figure}

\begin{figure}[H]
	
	\begin{subfigure}{0.33\textwidth}
		\centering
		\includegraphics[width=\linewidth]{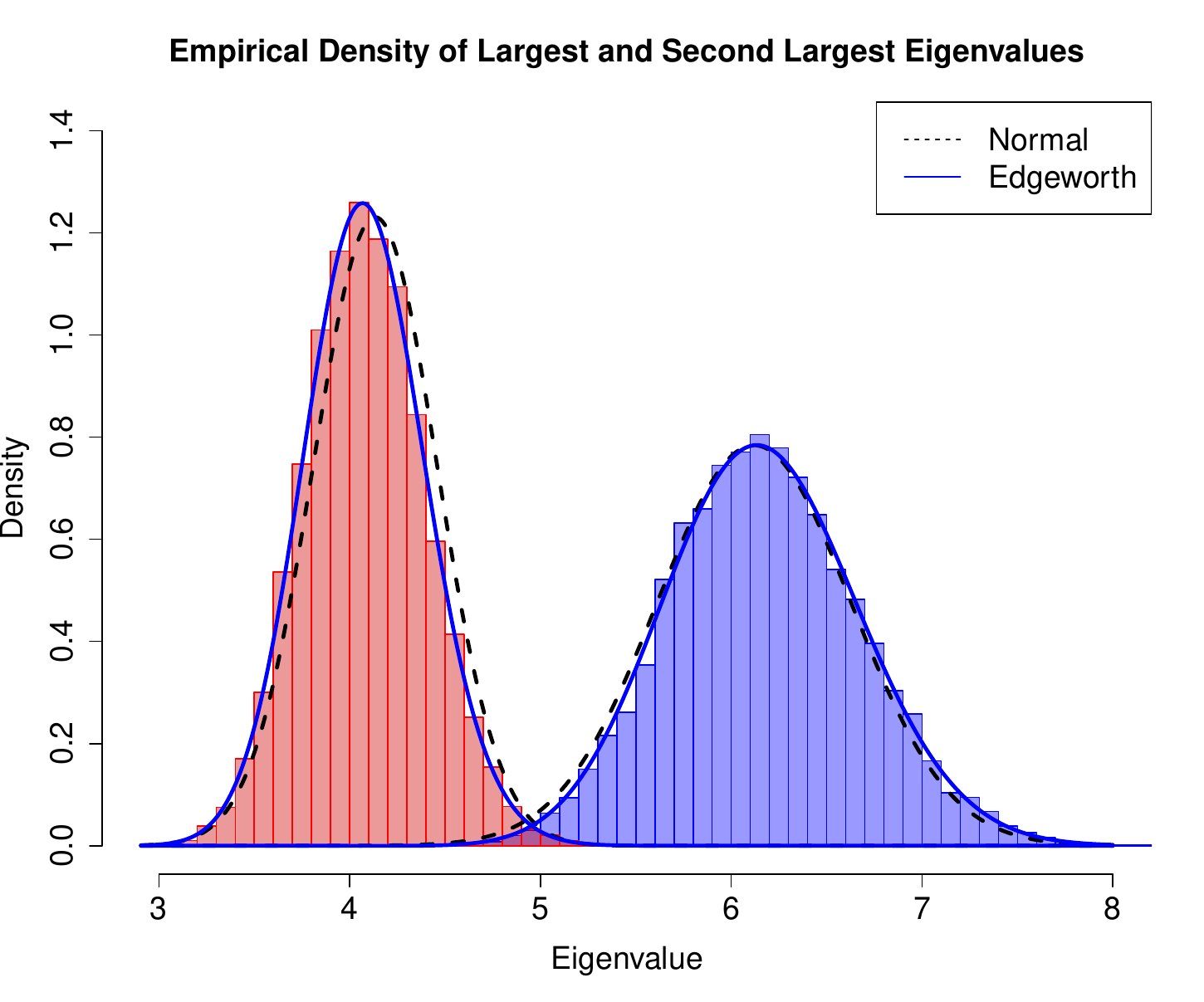}
		\caption{ $U(-\sqrt{3},\sqrt{3}),l_1=6,l_2=4$}
	\end{subfigure}%
	\begin{subfigure}{0.33\textwidth}
		\centering
		\includegraphics[width=\linewidth]{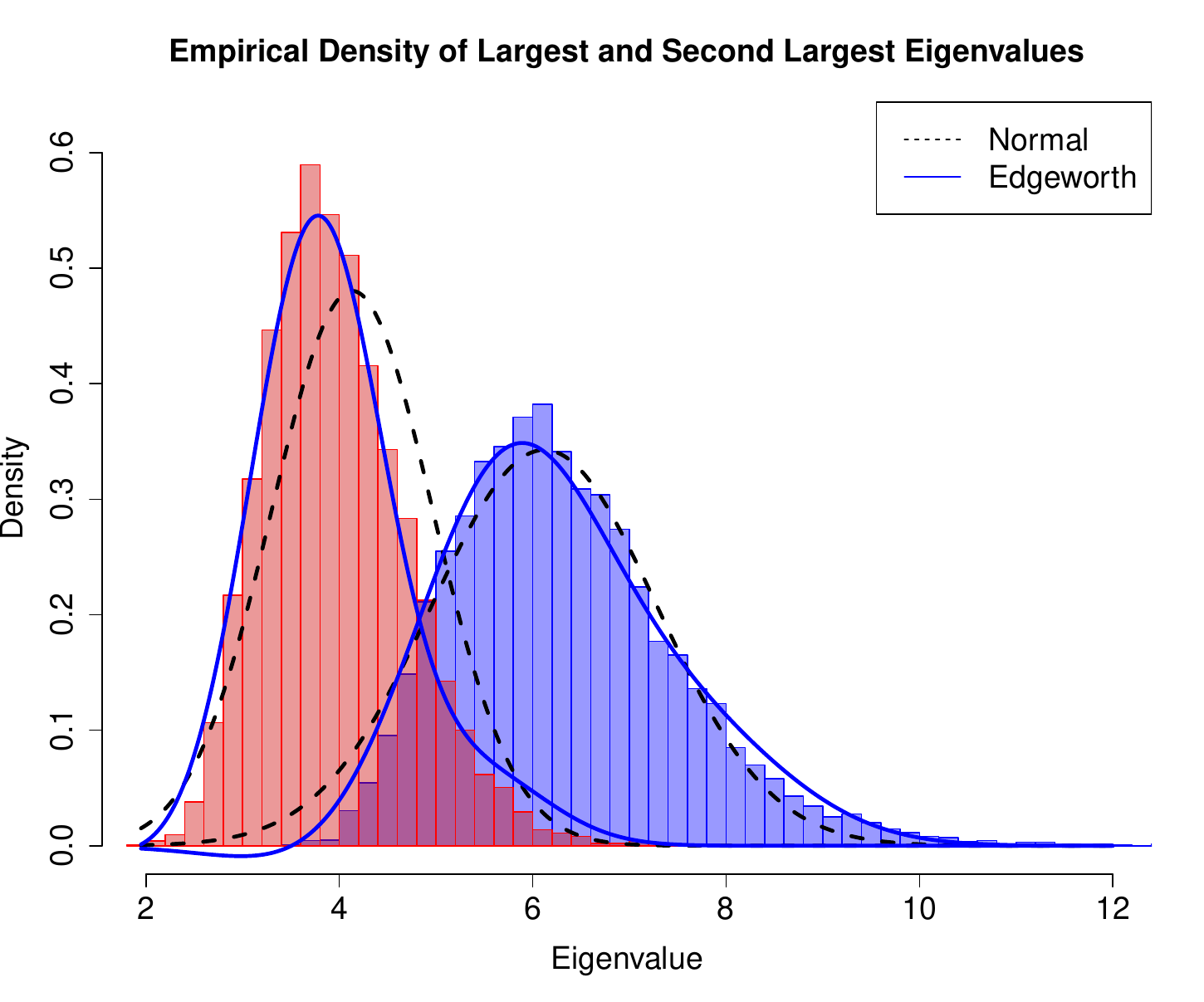}
		\caption{ $\chi^2(1),~ l_1=6,~l_2=4$}
	\end{subfigure}%
	\begin{subfigure}{0.33\textwidth}
		\centering
		\includegraphics[width=\linewidth]{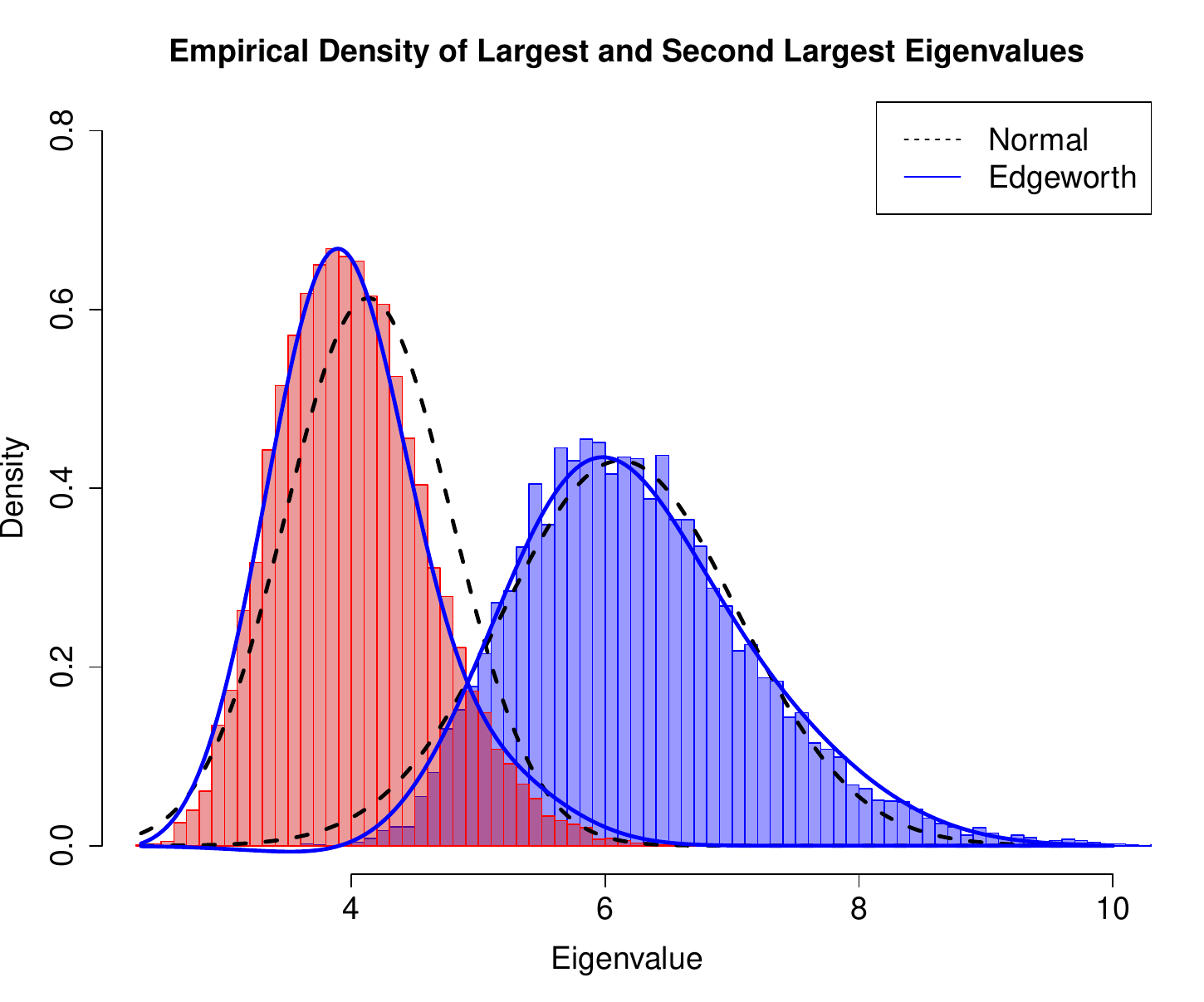}
		\caption{ $Ga(1,1),~ l_1=6,~l_2=4$}
	\end{subfigure}%

	\begin{subfigure}{0.33\textwidth}
		\centering
		\includegraphics[width=\linewidth]{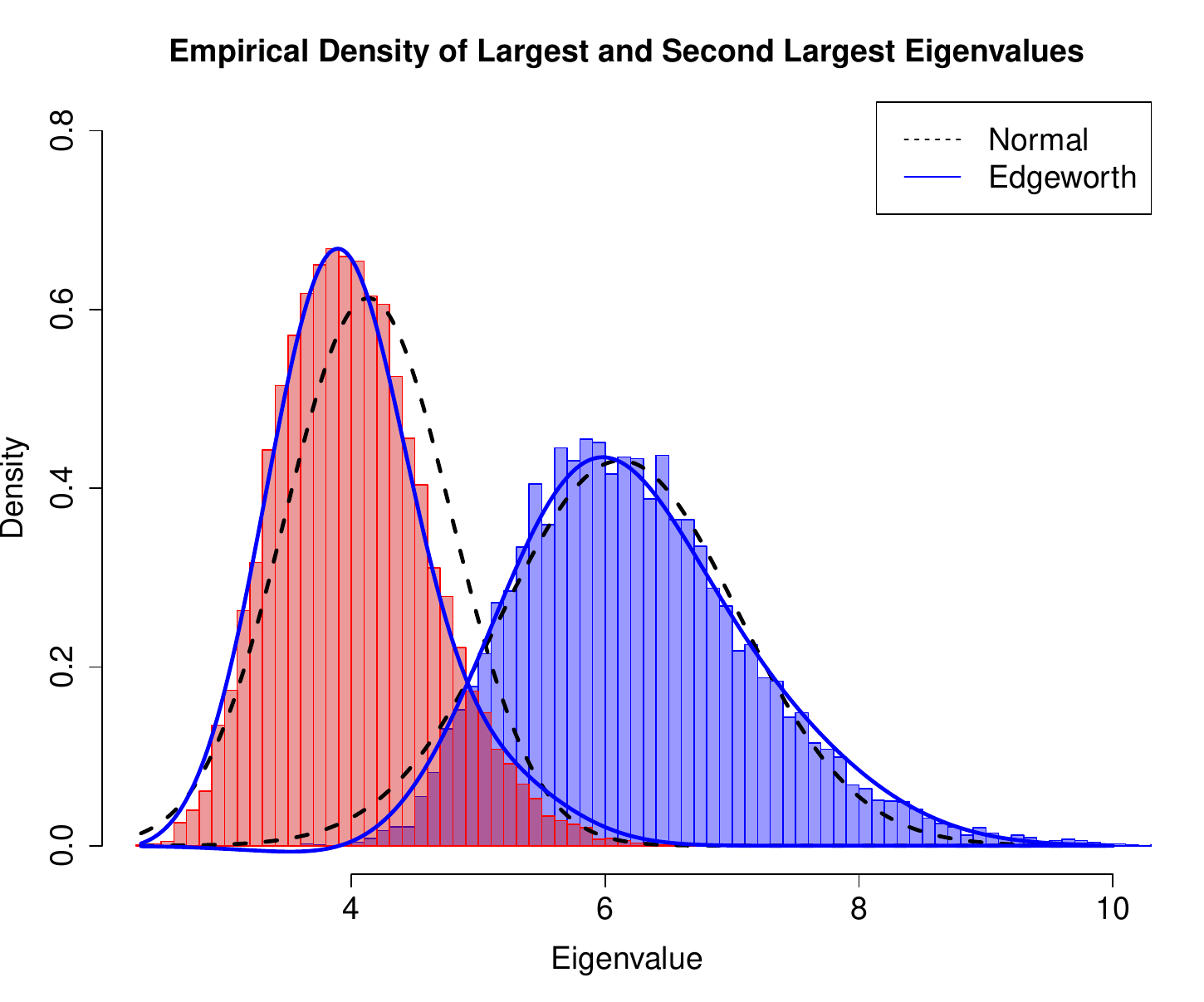}
		\caption{ $Ga(1,2),~ l_1=6,~l_2=4$}
	\end{subfigure}%
	\begin{subfigure}{0.33\textwidth}
		\centering
		\includegraphics[width=\linewidth]{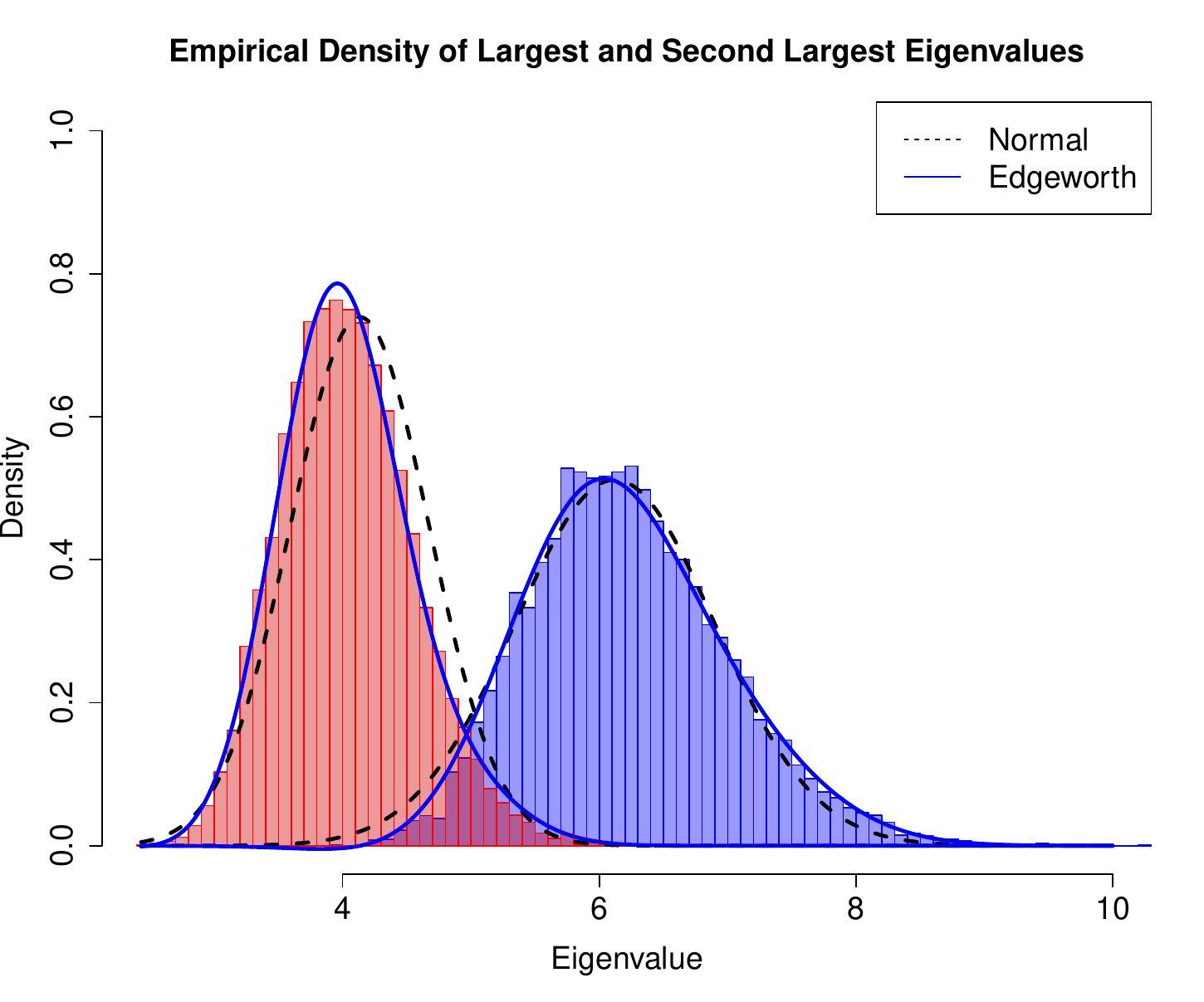}
		\caption{$Ga(2,2),~ l_1=6,~l_2=4$}
	\end{subfigure}
	\begin{subfigure}{0.33\textwidth}
		\centering
		\includegraphics[width=\linewidth]{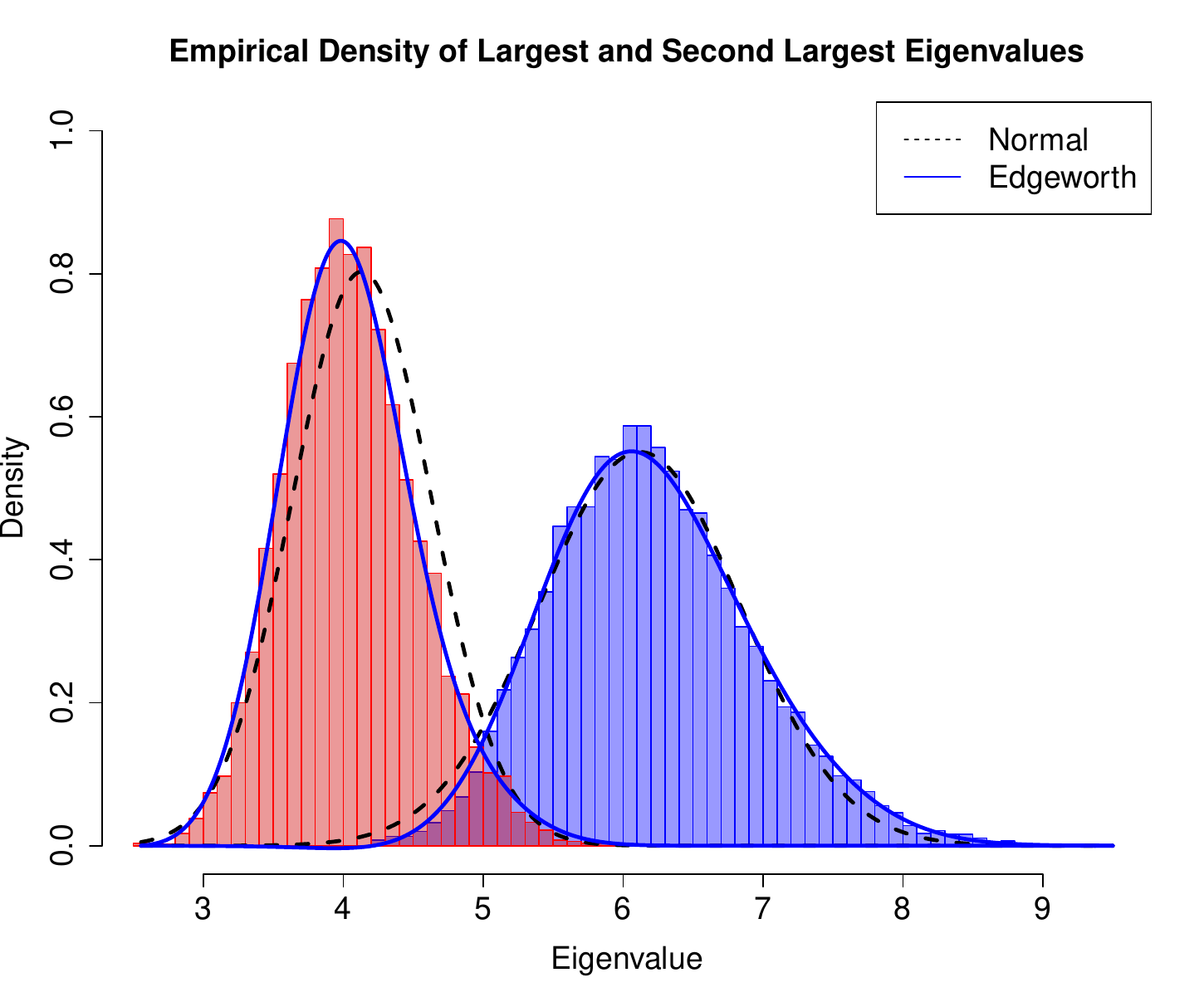}
		\caption{$Ga(3,3),~ l_1=6,~l_2=4$}
	\end{subfigure}
	\caption{ Edgeworth expansion  for different samples under Setting 8}
	\label{fig:both_images8}
\end{figure}

\begin{figure}[H]
	
	\begin{subfigure}{0.33\textwidth}
		\centering
		\includegraphics[width=\linewidth]{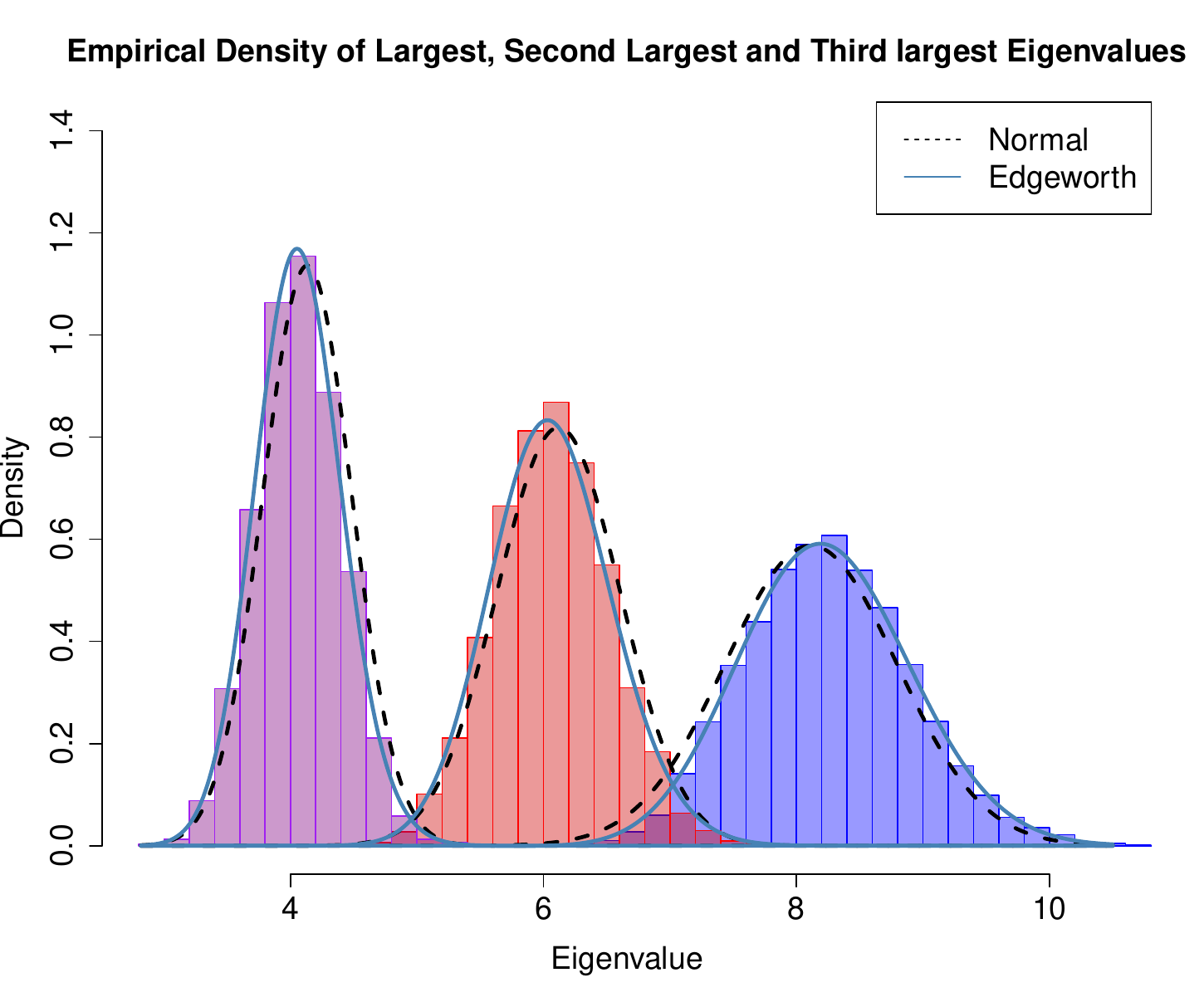}
		\caption{  $U(-\sqrt{3},\sqrt{3})$}
	\end{subfigure}%
	\begin{subfigure}{0.33\textwidth}
		\centering
		\includegraphics[width=\linewidth]{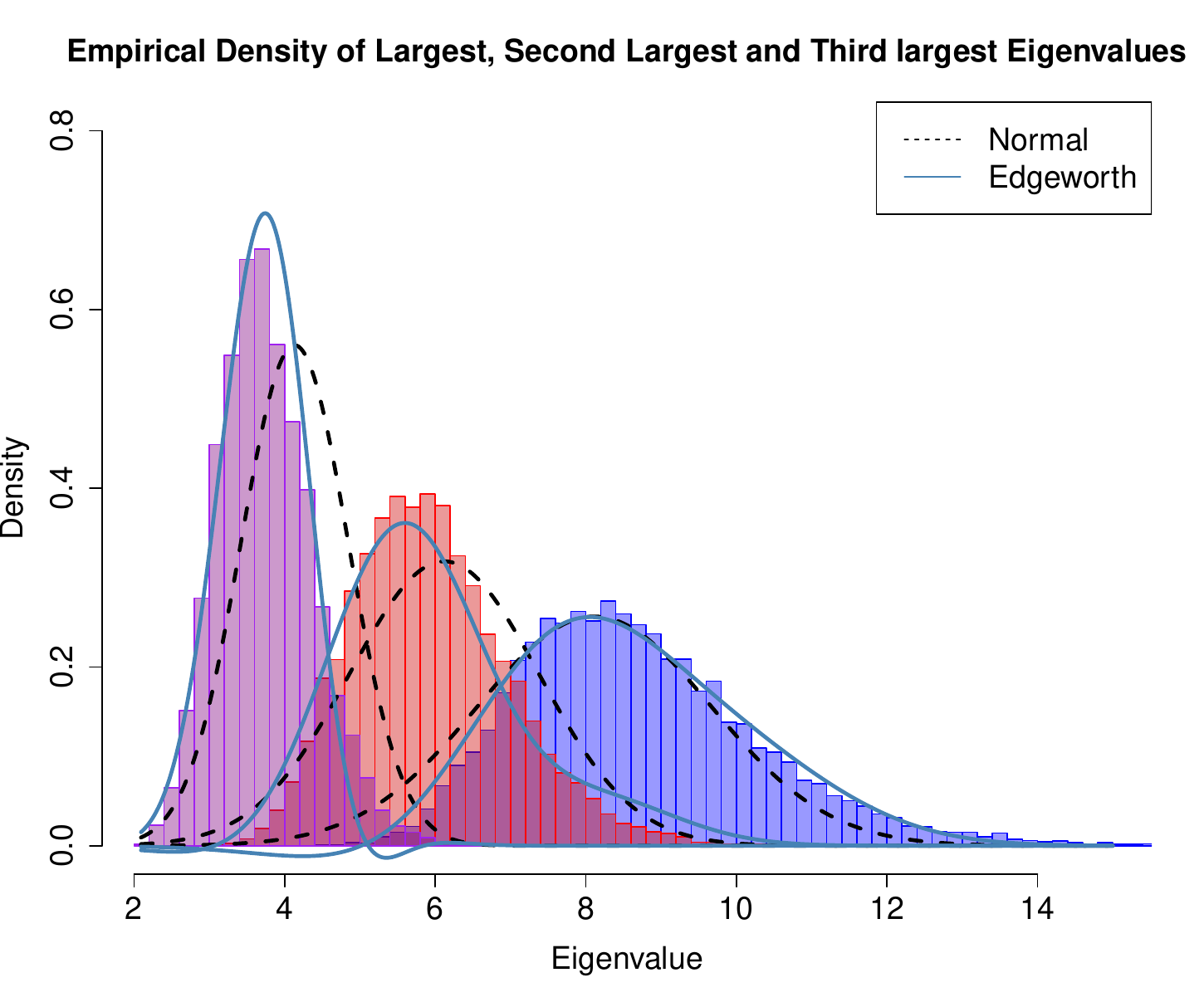}
		\caption{ $\chi^2(1)$}
	\end{subfigure}%
	\begin{subfigure}{0.33\textwidth}
		\centering
		\includegraphics[width=\linewidth]{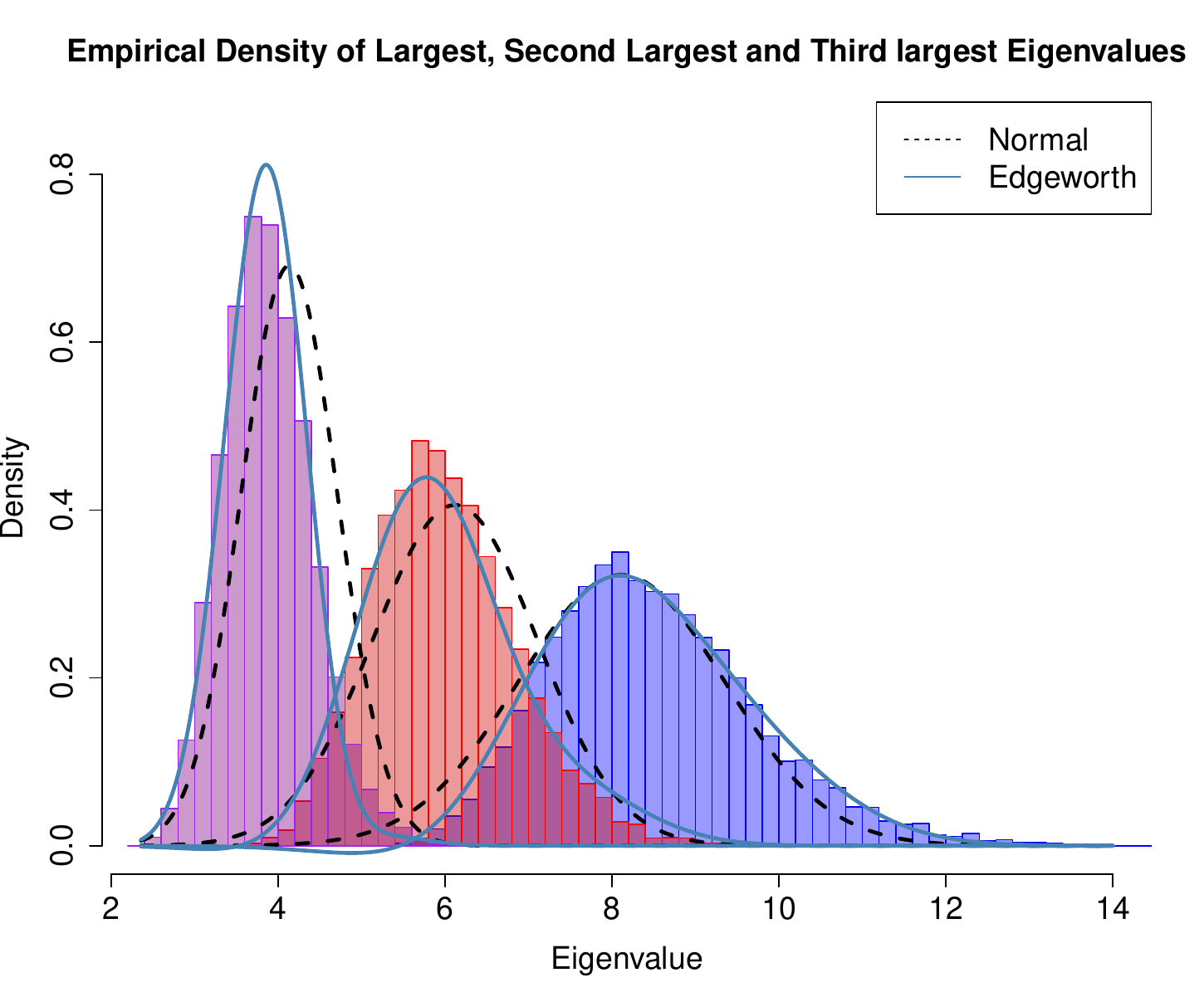}
		\caption{  $Ga(1,1)$}
	\end{subfigure}%

	\begin{subfigure}{0.33\textwidth}
		\centering
		\includegraphics[width=\linewidth]{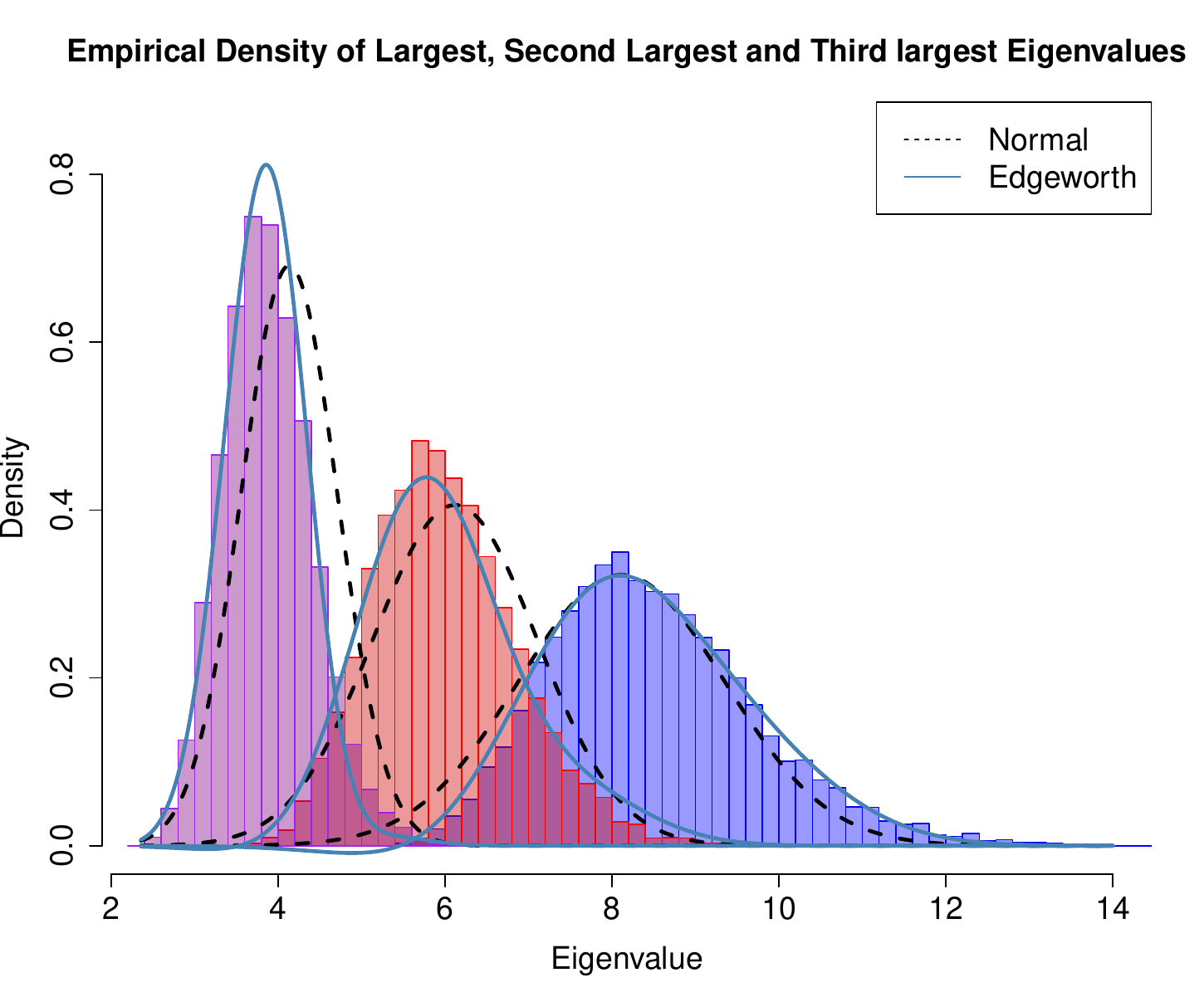}
		\caption{$Ga(1,2)$}
	\end{subfigure}%
	\begin{subfigure}{0.33\textwidth}
		\centering
		\includegraphics[width=\linewidth]{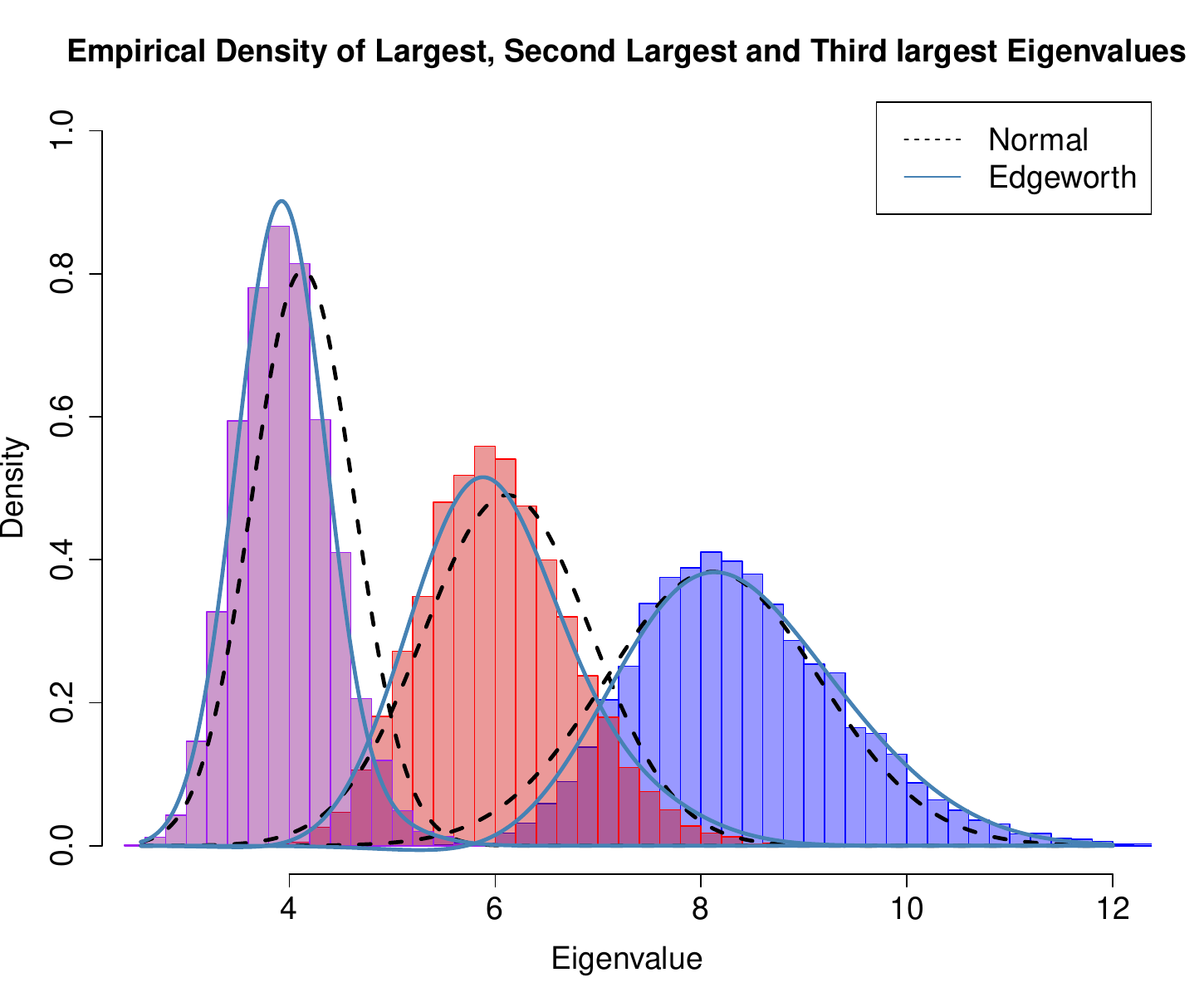}
		\caption{$Ga(2,2)$}
	\end{subfigure}
	\begin{subfigure}{0.33\textwidth}
		\centering
		\includegraphics[width=\linewidth]{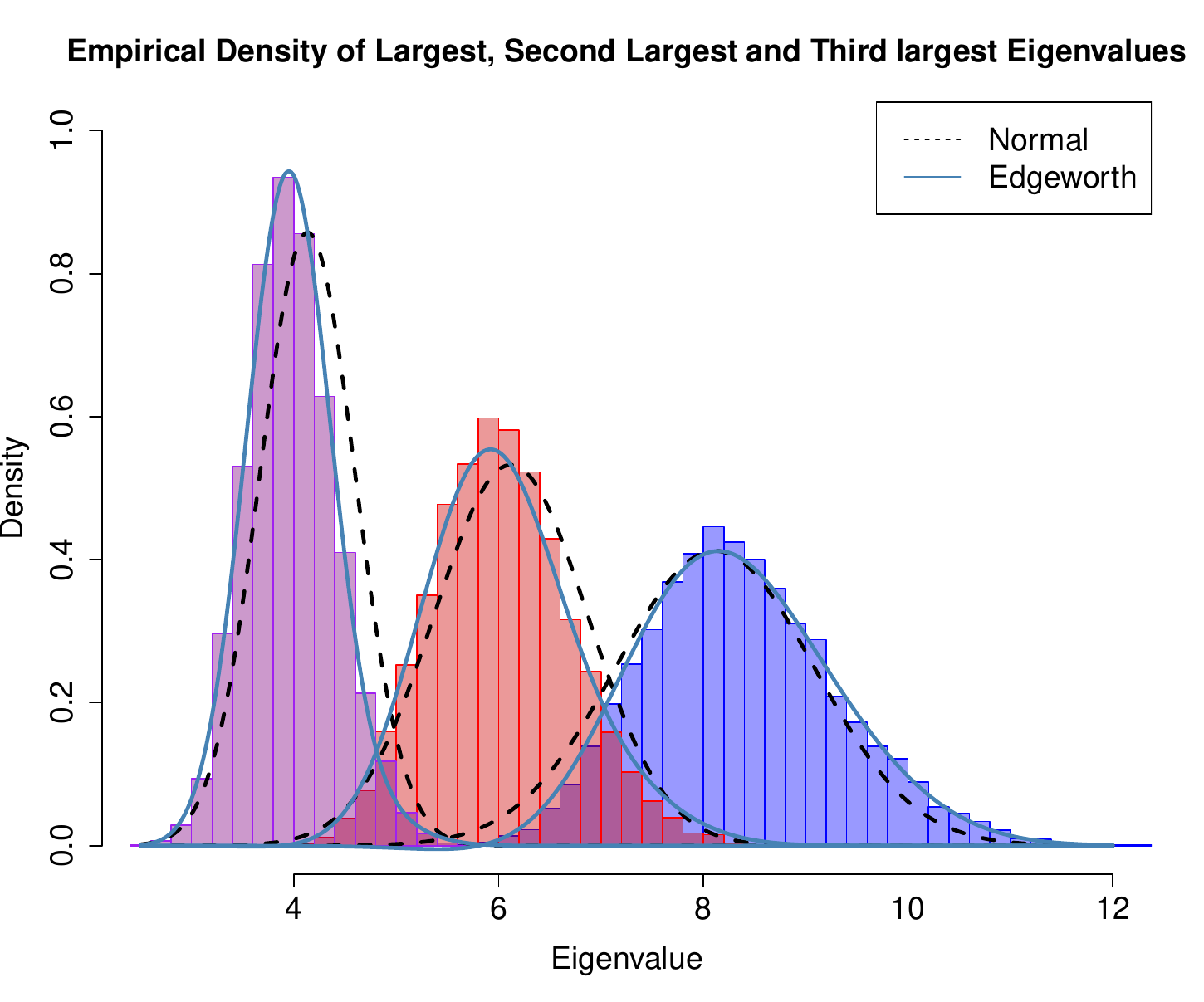}
		\caption{$Ga(3,3)$}
	\end{subfigure}
	\caption{ Edgeworth expansion  for different samples under Setting 9}
	\label{fig:both_images9}
\end{figure}



\end{document}